\documentclass[aap,preprint]{imsart}
\makeatletter
\def\journal@name{Preprint}
\makeatother
\RequirePackage[T1]{fontenc}
\RequirePackage{amsthm,amsmath,amsfonts,amssymb,mathtools}
\RequirePackage{bm,booktabs,array}
\RequirePackage[numbers]{natbib}
\RequirePackage[hidelinks]{hyperref}
\hypersetup{pdfauthor={Ziqian Lin and Huangang Lu},
  pdftitle={Sharp detection and triangle likelihood limits for quadratic spherical kernel graphs}}
\startlocaldefs
\theoremstyle{plain}
\newtheorem{theorem}{Theorem}[section]
\newtheorem{proposition}[theorem]{Proposition}
\newtheorem{lemma}[theorem]{Lemma}
\newtheorem{corollary}[theorem]{Corollary}

\theoremstyle{definition}

\numberwithin{equation}{section}
\numberwithin{table}{section}

\newcommand{\Sph}{\mathbb S}
\newcommand{\E}{\mathbb E}
\newcommand{\Pp}{\mathsf P}
\newcommand{\Qq}{\mathsf Q}
\newcommand{\TV}{\operatorname{TV}}
\newcommand{\tr}{\operatorname{tr}}
\newcommand{\HS}{\mathrm{HS}}

\newcommand{\Kop}{\mathsf K}
\newcommand{\Adj}{\bm A}
\newcommand{\cH}{\mathcal H}

\newcommand{\Normal}{\mathcal N}

\endlocaldefs

\begin{document}
\begin{frontmatter}
\title{Sharp detection and triangle likelihood limits for quadratic spherical kernel graphs}
\runtitle{Detection of quadratic spherical geometry}
\begin{aug}
\author[A]{\fnms{Ziqian}~\snm{Lin}\ead[label=e1]{linziqian@fzu.edu.cn}}
\author[B]{\fnms{Huangang}~\snm{Lu}\ead[label=e2]{luhuangang@pku.edu.cn}}
\address[A]{School of Mathematics and Statistics, Fuzhou University\printead[presep={,\ }]{e1}}
\address[B]{School of Physics, Peking University\printead[presep={,\ }]{e2}}
\end{aug}
\begin{abstract}
We study whether quadratic spherical geometry can be detected from a
single observed graph. Latent positions are independent and uniform on
\(\mathbb S^{d-1}\); conditional on them, edges are independent with
connection probability \(K_{d,\beta}(t)=[1+\beta(dt^2-1)/(d-1)]/2\),
where \(0<|\beta|\le1\). We test this model against \(G(n,1/2)\).
Writing \(N=(d-1)(d+2)/2\) for the degree-two harmonic multiplicity,
we prove that the total variation distance tends to one when
\(N\ll n^{3/4}|\beta|^{3/2}\) and to zero when
\(N\gg n^{3/4}|\beta|^{3/2}\). At criticality, if \(N\to\infty\) and
\(\Delta_n=\sqrt{\binom n3}\,|\beta|^3/N^2\to\Delta\in(0,\infty)\),
the signed triangle captures all asymptotic testing information and the
total variation distance tends to \(2\Phi(\Delta/2)-1\), where \(\Phi\)
is the standard normal distribution function. Polynomial finite-sample
error bounds hold under polynomial signal scaling. The proof controls
the likelihood of the dependent quadratic feature law through spectral
recursions, covariance truncation and interpolation from a symmetric prior.
\end{abstract}
\end{frontmatter}

\tableofcontents
\section{Introduction}
\label{sec:introduction}

\subsection{Hidden geometry as a matched-marginal testing problem}

Many network models attach an unobserved position to each vertex and make
nearby or similar vertices more likely to connect.  This mechanism appears in
latent-space models for social networks \cite{HoffRafteryHandcock2002}, random
dot-product graphs \cite{AthreyaEtAl2018}, and classical random geometric
graphs, where an edge is present when two latent points lie within a fixed
distance \cite{Penrose2003}.  Random connection models replace the hard
distance threshold by a connection probability
\cite{DettmannGeorgiou2016}.  The spherical soft model considered here is a
high-dimensional instance of this latter construction and belongs to the
smooth random geometric graph framework of \cite{LiuRaczLatent2023,MaoWuXu2026}.

Recovering latent coordinates or estimating a connection function presupposes
a more basic statistical fact: the observed graph law must be detectably
different from a law with no latent geometry.  Detection is weaker than either
estimation problem and therefore supplies an information-theoretic baseline.
Below its threshold, no graph-based procedure can even certify that geometry
is present.  Above the threshold, one may meaningfully ask for more detailed
structural information.

We formulate this baseline as a matched-marginal test.  Let
\(X_1,\ldots,X_n\) be independent and uniformly distributed on
\(\Sph^{d-1}\), and let \(K:[-1,1]\to[0,1]\) be a measurable connection
function.  Conditional on the latent points, the upper-triangular entries of
the adjacency matrix \(\Adj=(a_{ij})\) are independent and satisfy
\[
  \mathbb P(a_{ij}=1\mid X_1,\ldots,X_n)
  =
  K(\langle X_i,X_j\rangle).
\]
Write \(\Pp_{n,d,K}\) for the resulting graph law and
\[
  p=\mathbb E K(\langle X_1,X_2\rangle)
\]
for its marginal edge density.  The null law
\(\Qq_{n,p}=G(n,p)\) has independent
\(\operatorname{Bernoulli}(p)\) edges.  Thus the two models agree on the
law of every individual edge.  Any evidence for geometry must be contained in
the dependence created when several observed edges reuse the same latent
vertices.

For a test \(\varphi\), with \(\varphi=1\) declaring the geometric
alternative,
\[
  \inf_\varphi
  \bigl\{
    \Qq_{n,p}(\varphi=1)
    +
    \Pp_{n,d,K}(\varphi=0)
  \bigr\}
  =
  1-\TV(\Pp_{n,d,K},\Qq_{n,p}).
\]
Consequently, total variation tending to zero means that no test
asymptotically distinguishes the graph laws, while total variation tending to
one is equivalent to the existence of a test with both errors tending to
zero.  A limit strictly between these endpoints describes the finite amount
of information retained by a critical limiting experiment.  Our aim is to
determine all three behaviors for a canonical even harmonic model.

\subsection{The quadratic layer and the cube-trace scale}
\label{subsec:quadratic-spectral-scale}

The spectral origin of the testing threshold becomes visible after centering
and normalizing the connection function.  Assume \(0<p<1\) and put
\[
  \kappa(t)=\frac{K(t)-p}{\sqrt{p(1-p)}}
\]
and let \(\Kop_\kappa\) be the integral operator on
\(L^2(\Sph^{d-1},\sigma_d)\) with kernel
\((x,y)\mapsto\kappa(\langle x,y\rangle)\).  If
\(Y_{ij}=(a_{ij}-p)/\sqrt{p(1-p)}\), then conditional edge independence and
kernel composition give, for every simple cycle on distinct vertices,
\begin{equation}
  \mathbb E_{\Pp_{n,d,K}}
  \prod_{r=1}^{\ell}Y_{i_ri_{r+1}}
  =
  \tr(\Kop_\kappa^\ell),
  \qquad i_{\ell+1}=i_1.
  \label{eq:intro-cycle-trace}
\end{equation}
Section~\ref{subsec:spherical-spectral-background} explains this identity; the full derivation is in Section~S5 of the Supplementary Material. The corresponding expectation under the matched null is zero. The triangle
is the shortest simple cycle, so its mean shift is governed by
\(\tr(\Kop_\kappa^3)\).  Summing over all triangles gives alternative
mean \(\binom n3\tr(\Kop_\kappa^3)\) and exact null variance
\(\binom n3\).  Consequently, the squared mean shift divided by the null
variance is
\[
  \binom n3\tr^2(\Kop_\kappa^3)
  =
  \frac{n^3}{6}\bigl(1+O(n^{-1})\bigr)
  \tr^2(\Kop_\kappa^3).
\]
Thus, up to the constant factor \(6\), the squared signal-to-noise scale is
\(n^3\tr^2(\Kop_\kappa^3)\), which Mao, Wu, and Xu
\cite{MaoWuXu2026} proposed as the spectral scale for
distinguishing smooth spherical geometric graphs from their matched
Erd\H{o}s--R\'enyi laws.

For arbitrary kernels, the two directions suggested by this scale involve
different difficulties.  A divergent signed-triangle signal gives a
candidate statistic, although its alternative variance must still be
controlled.  A vanishing triangle signal, by contrast, says only that one
statistic fails; it does not by itself rule out information in longer cycles
or in localized graph structure.  The present paper isolates a single
harmonic block, thereby removing cancellation between different spectral
components while retaining a genuinely non-Gaussian latent feature law.

Specifically, let
\[
  Z_{2,d}(t)=\frac{dt^2-1}{d-1},
  \qquad
  K_{d,\beta}(t)=\frac{1+\beta Z_{2,d}(t)}2,
  \qquad 0<|\beta|\le1.
\]
Here \(Z_{2,d}\) is the degree-two zonal harmonic normalized by
\(Z_{2,d}(1)=1\); the formula follows from the spherical-harmonic addition
theorem \cite[Chapter~1]{DaiXu2013}.
The degree-two harmonic space has multiplicity
\[
  N=N_{2,d}=\frac{(d-1)(d+2)}2.
\]
Since \(Z_{2,d}\) is centered, the edge density is \(p=1/2\); throughout
the pure model we therefore write
\[
  \Pp_n=\Pp_{n,d,K_{d,\beta}},
  \qquad
  \Qq_n=G(n,1/2).
\]
The centered operator has the single nonzero eigenvalue \(\beta/N\), with
multiplicity \(N\), and hence
\[
  \tr(\Kop_\kappa^\ell)
  =
  \frac{\beta^\ell}{N^{\ell-1}},
  \qquad \ell\ge2.
\]
The model is even and nonmonotone, so it has no degree-one component that
could reduce detection to a linear correlation.  At the same time,
\(Z_{2,d}(\langle x,y\rangle)\) admits an exact inner-product
representation in the space of traceless symmetric matrices.  This
combination makes the quadratic layer a canonical nontrivial setting in which
the spectral prediction can be tested at the level of the full graph
likelihood. The main difficulty is ruling out every test when the
cube-trace signal vanishes. At criticality, one must also determine whether
the full graph contains information beyond the triangle.

\subsection{Results and interpretation}

Our first result proves both strict sides of the spectral prediction.  For any
parameter sequence with \(d_n\in\{2,3,\ldots\}\) and
\(0<|\beta_n|\le1\), set \(N=N_{2,d_n}\) and abbreviate
\(\beta=\beta_n\).  Then
\begin{equation}
\begin{aligned}
  N\ll n^{3/4}|\beta|^{3/2}
  &\quad\Longrightarrow\quad
  \TV(\Pp_n,\Qq_n)\to1,\\
  N\gg n^{3/4}|\beta|^{3/2}
  &\quad\Longrightarrow\quad
  \TV(\Pp_n,\Qq_n)\to0.
\end{aligned}
  \label{eq:intro-pure-dichotomy}
\end{equation}
For this single harmonic layer, the cube-trace scale is explicit:
\[
  n^3\tr^2(\Kop_\kappa^3)
  =
  \frac{n^3\beta^6}{N^4}
  =
  \left(
    \frac{n^{3/4}|\beta|^{3/2}}{N}
  \right)^4.
\]
Thus the first regime in \eqref{eq:intro-pure-dichotomy} makes the squared
signal-to-noise scale tend to infinity, while the second makes it tend to
zero.  Under both graph laws, each edge has marginal mean \(1/2\), so its
centered, variance-one version is
\(Y_{ij}=(a_{ij}-1/2)/\sqrt{(1/2)(1-1/2)}=2a_{ij}-1\).
Define the signed-triangle sum
\[
  T_n=\sum_{i<j<k}Y_{ij}Y_{jk}Y_{ki}.
\]
Unlike the ordinary triangle count, this statistic includes every vertex
triple.  Each triple is weighted by the product of its three centered edge
signs \(Y_{ij},Y_{jk},Y_{ki}\in\{-1,1\}\), which explains the term
``signed triangle.''
In the first regime of \eqref{eq:intro-pure-dichotomy}, the exact moment
calculation yields
\[
  \operatorname{Var}_{\Qq_n}(T_n)
  =o\!\left((\E_{\Pp_n}T_n)^2\right),
  \qquad
  \operatorname{Var}_{\Pp_n}(T_n)
  =o\!\left((\E_{\Pp_n}T_n)^2\right).
\]
Thus an oriented signed-triangle test has vanishing type-I and type-II errors,
which makes the first implication constructive.  The second implication is
information-theoretic: it rules out every graph statistic on the opposite
strict side and closes the polynomial gap left by general polynomial-kernel
bounds in this degree-two layer.

We next identify the likelihood limit at the boundary in
\eqref{eq:intro-pure-dichotomy}.  Assume \(N\to\infty\) and define
\[
  \Delta_n
  =
  \frac{\sqrt{\binom n3}\,|\beta|^3}{N^2},
  \qquad
  W_n
  =
  \frac{\operatorname{sgn}(\beta)}
       {\sqrt{\binom n3}}
  T_n.
\]
If \(\Delta_n\to\Delta\in(0,\infty)\), then \(W_n\) converges to
\(\Normal(0,1)\) under \(\Qq_n\) and to \(\Normal(\Delta,1)\) under
\(\Pp_n\).  The proof also constructs latent events of probability
\(1-o(1)\) for which the conditioned likelihood ratio
\(L_n^{\mathcal E}\) satisfies
\begin{equation}
  \log L_n^{\mathcal E}
  =
  \Delta_nW_n-\frac12\Delta_n^2+o_{\Qq_n}(1),
  \qquad
  \TV(\Pp_n,\Qq_n)
  \longrightarrow
  2\Phi(\Delta/2)-1,
  \label{eq:intro-high-multiplicity-LAN}
\end{equation}
where \(\Phi\) is the standard normal distribution function. For each
specified critical alternative, a threshold test based on \(W_n\) attains
the limiting minimum sum of type-I and type-II errors among all graph
tests. The likelihood expansion establishes this optimality by showing that
no additional likelihood fluctuation survives beyond the triangle.
The TV limit is the distance between \(\Normal(0,1)\) and
\(\Normal(\Delta,1)\); together with the likelihood expansion, this gives
the one-dimensional local asymptotic normality (LAN) structure used here.
When \(\beta^2\asymp n^{-r}\),
\(0\le r<1\), we also give a polynomial finite-sample bound for this Gaussian
approximation and its log-likelihood remainder.

Table~\ref{tab:intro-regimes} summarizes these conclusions.

\begin{table}[t]
\caption{The strict detection regimes and the critical likelihood limit.}
\label{tab:intro-regimes}
\centering
\small
\renewcommand{\arraystretch}{1.18}
\begin{tabular}{@{}>{\raggedright\arraybackslash}p{0.21\textwidth}
                    >{\raggedright\arraybackslash}p{0.35\textwidth}
                    >{\raggedright\arraybackslash}p{0.35\textwidth}@{}}
\toprule
Regime & Assumptions & Statistical conclusion \\
\midrule
Detection & \(N/(n^{3/4}|\beta|^{3/2})\to0\) & \(\TV(\Pp_n,\Qq_n)\to1\). \\
Indistinguishability & \(N/(n^{3/4}|\beta|^{3/2})\to\infty\) & \(\TV(\Pp_n,\Qq_n)\to0\). \\
Critical window & \(N\to\infty\), \(\Delta_n\to\Delta\in(0,\infty)\)
 & One-dimensional Gaussian shift.
   \(\TV\to2\Phi(\Delta/2)-1\). \\
\bottomrule
\end{tabular}
\end{table}

\subsection{Overview of the proofs}

Triangle moments identify the candidate threshold. The main proof then
controls the full likelihood at two levels: its conditioned second moment
tends to one in the indistinguishable regime, whereas its exact nontrivial
limit identifies the optimal testing risk at criticality.

\paragraph*{Detection}
The signed-triangle sum has null mean zero and alternative mean
\(\binom n3\tr(\Kop_\kappa^3)\).  Distinct triangle terms are uncorrelated
under the null; under the alternative, their covariance vanishes unless they
share an edge.  The resulting variance formulas show that, when
\(N\ll n^{3/4}|\beta|^{3/2}\), the variance under either law is negligible
relative to the squared mean shift.  Chebyshev's inequality then gives a
consistent test.

\paragraph*{Indistinguishability}
For \(N\gg n^{3/4}|\beta|^{3/2}\), the quadratic feature representation
rewrites the likelihood second moment
as an interaction between two independent latent configurations.  Rare
configurations can dominate this integral, so we condition the prior on an
event \(\mathcal E_n\) of probability \(1-o(1)\) that controls the empirical
covariance and Gram-matrix norms.  A spectral deletion recursion handles
\(N/(n\beta^2)\to\infty\).  For \(N/(n\beta^2)=O(1)\), we first control a
symmetric feature prior and then interpolate back to the quadratic feature
law, using covariance truncation and a cavity expansion to bound the error.
Both arguments yield \(\E_{\Qq_n}(L_n^{\mathcal E})^2=1+o(1)\) for the
conditioned likelihood ratio.  This implies
\(\TV(\Pp_n^{\mathcal E_n},\Qq_n)=o(1)\), while
\(\TV(\Pp_n,\Pp_n^{\mathcal E_n})\le\mathbb P(\mathcal E_n^c)=o(1)\)
removes the conditioning.

\paragraph*{The critical likelihood}
Assume \(N\to\infty\) and \(\Delta_n\to\Delta\in(0,\infty)\).  We now
need the exact limiting second moment.  The covariance cutoff must remain
compatible with the recursion and have negligible error under the exponential
weight.  Concentration on the product of spheres supplies a tail bound strong
enough to meet both requirements.

We then expand the conditioned second moment over edge sets.  For each fixed
number of edges, conditioning can be removed at an \(o(1)\) cost.  Centering
eliminates vertices of degree one, isotropy contracts vertices of degree two,
and quadratic-form moment bounds with H\"older's inequality
show that each fixed connected graph other than a triangle contributes
\(o(1)\).
A uniform bound on the tail of the expansion allows the contributions of
vertex-disjoint triangles to be summed, giving
\(\E_{\Qq_n}(L_n^{\mathcal E})^2\to e^{\Delta^2}\).

A de Jong central limit argument \cite{deJong1990} identifies the Gaussian
shift carried by \(W_n\) under the two graph laws.  Its limiting likelihood
ratio has the same second moment \(e^{\Delta^2}\).
Lemma~\ref{lem:hilbert-likelihood-projection} uses equality in conditional
Jensen's inequality to identify every joint subsequential limit of
\((W_n,L_n^{\mathcal E})\) with this Gaussian likelihood. The lemma
applies to general dominated experiments whenever a statistic converges
under both laws and its limiting likelihood has a finite second moment
equal to the limit of the full likelihood second moments.
This provides a criterion for transferring a statistic's limit law to the
optimal testing risk. Here its application depends on the quadratic
moment and concentration estimates developed earlier.
Removing the conditioning yields \eqref{eq:intro-high-multiplicity-LAN}.
Quantitative normal approximation
\cite{DoblerPeccati2017} and quantitative bounds for the edge-set expansion
give the finite-sample estimates.

\subsection{Relation to previous work}

The quadratic problem connects sharp geometric detection thresholds with
critical likelihood limits.

\paragraph*{Geometric graph testing}
Bubeck, Ding, Eldan, and R\'acz
\cite{BubeckDingEldanRacz2016} established the dense hard spherical
threshold \(d\asymp n^3\): signed triangles detect below this scale, and
a Wishart--GOE comparison gives indistinguishability above it.
Brennan, Bresler, and Nagaraj
\cite{BrennanBreslerNagaraj2020} obtained density-dependent
total-variation bounds, and Liu, Mohanty, Schramm, and Yang
\cite{LiuMohantySchrammYang2024} studied the sparse spherical model.
More recently, Du, Mao, Sun, Wu, and Xu \cite{DuMaoSunWuXu2026}
obtained the conjectured density-dependent indistinguishability bound
subject to \(d\ge(1+\epsilon)n\), for any fixed \(\epsilon>0\).

The role of spectral information also appears in anisotropic Gaussian
geometric graphs.  Eldan and Mikulincer \cite{EldanMikulincer2020}
identified detection scales through moments of the covariance spectrum;
Brennan, Bresler, and Huang \cite{BrennanBreslerHuang2024} established
the sharp threshold, confirming the optimality of signed triangles there.
Triangle optimality depends on the model: Bangachev and Bresler
\cite{BangachevBresler2024} showed that signed four-cycles outperform
triangles in \(L_\infty\) torus geometry and are optimal among low-degree
tests, while their random-algebraic framework
\cite{BangachevBresler2025} includes connections with vanishing triangle
signal.  These results motivate identifying the kernel structure that makes
triangles optimal.  Here the centered connection kernel consists of a single
degree-two spherical harmonic layer, observed through a soft Bernoulli channel.

\paragraph*{Soft kernels and the polynomial gap}
Liu and R\'acz studied noisy hard graphs and monotone inner-product connection
functions \cite{LiuRacz2023,LiuRaczLatent2023}.  Our quadratic kernel is
even and nonmonotone.  Mao, Wu, and Xu \cite{MaoWuXu2026} developed
posterior-overlap methods for smooth kernels.  Their Conjecture~1 and the
discussion following their equation~(7) explicitly predict the quadratic
scale \(d\asymp n^{3/8}\) for
\(K(t)=(1+t^2-1/d)/2\), which is our model with \(\beta=(d-1)/d\).
The detectable direction follows the signed-triangle method analyzed in
their Theorem~7; Section~\ref{sec:cycles} gives the exact quadratic moments.

The remaining gap concerns indistinguishability.  For any fixed
\(0<\epsilon<1/6\) and \(n^{-1+2\epsilon}\le|\beta|\le1\),
\cite[Theorem~8]{MaoWuXu2026} gives the sufficient condition
\(d\gg n^{3/4}|\beta|^{3/4}\).\footnote{Indeed,
\(\kappa(t)=\beta d\,t^2/(d-1)-\beta/(d-1)\).  In the notation of that
theorem, take \(L=2\), \(B=3\), and \(r=\sqrt{|\beta|}\).
The coefficient bounds hold for \(d\ge2\), and the remaining dimension
conditions follow from the stated parameter range and sufficient condition.}
Theorem~\ref{thm:pure-dichotomy} proves indistinguishability under
\[
  N_{2,d}\gg n^{3/4}|\beta|^{3/2},
\]
which matches the signed-triangle boundary.  For \(|\beta|\asymp1\),
this improves the sufficient dimension scale from \(d\gg n^{3/4}\) to
\(d\gg n^{3/8}\). The new threshold result is the matching
indistinguishability bound. The critical likelihood analysis goes further:
it identifies the information retained at the boundary and proves the
asymptotic optimality of the triangle test there. The finite-sample bounds
quantify these critical conclusions for the pure quadratic family.

\paragraph*{Matrix comparison and critical limits}
The Wishart--GOE literature provides a related picture of the transition
from geometric dependence to independent noise.  Bubeck and Ganguly
\cite{BubeckGanguly2018} extended matrix comparison beyond Gaussian entries.
R\'acz and Richey \cite[Theorem~1.2]{RaczRichey2019} computed the limiting
total variation in the critical window \(d/n^3\to c\in(0,\infty)\),
using a likelihood expansion with a Gaussian limit.  Our critical profile
has the same Gaussian-shift form.  Its derivation requires control of the
quadratic feature law and the Bernoulli observation channel.  As shown in
Lemma~\ref{lem:feature-lift}, the features lie on an orthogonal-group orbit
in the space of traceless symmetric matrices; their coordinates are dependent,
and only the random edges are observed.  Existing Wishart comparisons therefore do not directly
yield our graph likelihood limit.  The likelihood analysis proves that the
signed triangle captures all asymptotic testing information in this critical
window and quantifies the approximation error.

\subsection{Organization}

Sections~\ref{sec:model} and \ref{sec:results} give the model and results;
Section~\ref{sec:cycles} proves detection. Section~\ref{sec:quadratic}
reduces indistinguishability to likelihood bounds, and
Sections~\ref{sec:dense-likelihood} and \ref{sec:likelihood-interpolation}
develop compensated deletion and interpolation.
Section~\ref{sec:critical-window} identifies the critical likelihood through
the graph expansion, two-model CLT and projection;
Section~\ref{sec:quantitative-main} gives the finite-sample bounds.

The main article retains the estimates needed to control rare configurations
under the likelihood weight and to sum the critical graph expansion.
Appendices~\ref{app:complete-preliminaries} and
\ref{app:complete-spectral-recursions} provide the feature estimates and
sparse recursion. Appendices~\ref{app:complete-adaptive-even} and
\ref{app:complete-buffered-cavity} control deletion remainders,
covariance-boundary errors and cubic moments.
Appendix~\ref{app:critical-likelihood-analysis} closes the uniform critical
bound with a common truncation sequence, and
Appendix~\ref{app:quantitative-LAN} proves quantitative likelihood projection.
The Supplementary Material collects standard feature, matrix and likelihood
calculations, spectral background, critical truncation constructions,
repeated scale checks, and quantitative versions of the graph-expansion
and triangle-approximation proofs. References with an S prefix point there.

\paragraph*{Notation}
For a positive integer \(m\), write \([m]=\{1,\ldots,m\}\), and let
\(\Phi\) denote the standard normal distribution function.  For nonnegative
quantities \(a\) and \(b\), the relation \(a\lesssim b\) means
\(a\le Cb\) for a constant independent of the asymptotic indices.  We use
\(a\gtrsim b\) for the reverse inequality, \(a\asymp b\) when both bounds
hold, and \(a\ll b\) when \(a/b\to0\).  Lowercase \(c\) and uppercase \(C\) may denote positive constants that change from line to line,
and a subscript records any permitted dependence.

The notation \(o_{\Qq_n}(1)\) means convergence to zero in
\(\Qq_n\)-probability.  We write \(I_N\) for the \(N\times N\) identity
matrix and \(\mathbb S_N^+\) for the cone of real symmetric positive
semidefinite \(N\times N\) matrices.  The norms
\(\|\cdot\|_{\HS}\) and \(\|\cdot\|_{\mathrm{op}}\) are the
Hilbert--Schmidt and operator norms, respectively.  In matrix coordinates,
the Hilbert--Schmidt norm is the Frobenius norm.  Finally,
\(x_+=\max\{x,0\}\).
\section{The quadratic model and testing notation}
\label{sec:model}

We now fix the pure degree-two model. Let \(d\ge2\), let \(\sigma_d\)
be uniform probability measure on \(\Sph^{d-1}\), and sample
\(X_1,\ldots,X_n\) independently from \(\sigma_d\). Put
\[
 Z_{2,d}(t)=\frac{dt^2-1}{d-1},\qquad
 N_{2,d}=\frac{(d-1)(d+2)}2.
\]
The connection function is
\begin{equation}
 K_{d,\beta}(t)=\frac{1+\beta Z_{2,d}(t)}2,
 \qquad 0<|\beta|\le1.
 \label{eq:pure-quadratic-model}
\end{equation}
Since \(|Z_{2,d}(t)|\le1\), this is a valid connection probability.
Conditional on the latent points, the upper-triangular entries of the
symmetric adjacency matrix \(\Adj=(a_{ij})\), with \(a_{ii}=0\), are
independent and satisfy
\[
 a_{ij}\mid X_1,\ldots,X_n
 \sim\operatorname{Bernoulli}\bigl(K_{d,\beta}(\langle X_i,X_j\rangle)\bigr),
 \qquad i<j.
\]
This is the pure second-harmonic soft random geometric graph from
\cite[Section~2.5]{MaoWuXu2026}. Since
\(\mathbb E\langle X_1,X_2\rangle^2=1/d\), its marginal edge probability
is \(p=1/2\). We use
\begin{equation}
 N:=N_{2,d},\qquad
 \Pp_n:=\Pp_{n,d,K_{d,\beta}},\qquad
 \Qq_n:=\Qq_{n,1/2}=G(n,1/2).
 \label{eq:pure-model-abbreviations}
\end{equation}
Under the Erd\H{o}s--R\'enyi law \(\Qq_n\), all upper-triangular entries
are independent \(\operatorname{Bernoulli}(1/2)\) variables
\cite{JansonLuczakRucinski2000}. The centered kernel and edge variables are
\(\kappa_{d,\beta}=\beta Z_{2,d}\) and \(Y_{ij}=2a_{ij}-1\).
The operator \(\Kop_{\kappa_{d,\beta}}\) has eigenvalue \(\beta/N\)
with multiplicity \(N\), and all its other eigenvalues vanish.
Section~\ref{subsec:spherical-spectral-background} derives this spectrum
from the feature representation and records the cycle--trace identity.
The general harmonic background is in Section~S5
of the Supplementary Material.

\subsection{Testing and likelihood conventions}
\label{subsec:testing-convention}

For probability laws \(P,Q\) on graphs with vertex set \([n]\), use
\[
 \TV(P,Q)=\sup_E|P(E)-Q(E)|=\frac12\int|dP-dQ|.
\]
The infimum of the sum of testing errors is \(1-\TV(P,Q)\)
\cite[Chapter~2]{Tsybakov2009}. We call the laws asymptotically
indistinguishable when their TV distance tends to zero, and say that strong
detection is possible when it tends to one. If \(P\ll Q\), write
\(L=dP/dQ\). Since \(\mathbb E_QL=1\), the chi-squared divergence and
Cauchy--Schwarz bound are
\begin{equation}
\begin{aligned}
 \chi^2(P\|Q)&=\mathbb E_Q(L-1)^2=\mathbb E_QL^2-1,\\
 \TV(P,Q)&=\tfrac12\mathbb E_Q|L-1|
 \le\tfrac12\sqrt{\chi^2(P\|Q)}.
\end{aligned}
 \label{eq:tv-chi-square-bound}
\end{equation}
For a latent-prior event \(\mathcal E\) of positive probability,
\(\Pp^{\mathcal E}\) denotes the graph law obtained by conditioning that
prior, \(\mathbb P(\mathcal E)\) its prior probability, and
\(L^{\mathcal E}=d\Pp^{\mathcal E}/d\Qq\) the conditioned likelihood
ratio. We suppress model indices when they are clear.
For log-likelihood statements, we use the convention \(\log0=-\infty\).

For a random variable \(Z\), \(\mathcal L_P(Z)\) is its law under \(P\).
We write \(\Normal(m,\Sigma)\) for a normal law with mean \(m\) and
covariance \(\Sigma\); in one dimension the second argument is the variance.
All limits are as \(n\to\infty\), allowing \(d=d_n\), \(\beta=\beta_n\),
and hence \(N=N_{2,d_n}\) to vary.
The signed triangle is computed from the adjacency matrix without knowing
\(d,N\), or \(|\beta|\). A specified sign of \(\beta\) orients the
one-sided test; if both signs are admitted, its absolute value gives a
two-sided test. The critical likelihood and TV formulas refer to a specified
alternative.
\section{Main results}
\label{sec:results}

The degree-two layer admits a sharp dichotomy away from the critical scale.
Theorem~\ref{thm:pure-dichotomy} states this dichotomy.
Theorem~\ref{thm:growing-critical-main} identifies the critical likelihood
limit and gives a finite-sample approximation under polynomial signal scaling.

In the pure quadratic model, the centered kernel is supported on a single
degree-two block, so the cubic trace has no cancellation between harmonic
blocks.  In this setting, the cube-trace criterion reduces to comparing
\(N=N_{2,d}\) with
\(n^{3/4}|\beta|^{3/2}\).  Theorem~\ref{thm:pure-dichotomy} proves the
corresponding dichotomy whenever one of these quantities asymptotically
dominates the other.

\begin{theorem}[Pure quadratic sharp dichotomy away from the critical window]
\label{thm:pure-dichotomy}
Let $\Pp_n$ be generated by the pure degree-two single-layer kernel
\(\kappa_{d,\beta}=\beta Z_{2,d}\), equivalently by the connection function
in \eqref{eq:pure-quadratic-model}, and let $\Qq_n=G(n,1/2)$.  Set
\[
  N=N_{2,d}=\frac{(d-1)(d+2)}2.
\]
Then
\[
  \frac{N}{n^{3/4}|\beta|^{3/2}}\to0
  \quad\Longrightarrow\quad
  \TV(\Pp_n,\Qq_n)\to1,
\]
and
\[
  \frac{N}{n^{3/4}|\beta|^{3/2}}\to\infty
  \quad\Longrightarrow\quad
  \TV(\Pp_n,\Qq_n)\to0.
\]
\end{theorem}

By \eqref{eq:single-layer-traces},
\[
  n^3\tr^2(\Kop_\kappa^3)
  =\frac{n^3\beta^6}{N^4}
  =\left(\frac{n^{3/4}|\beta|^{3/2}}{N}\right)^4.
\]
Equivalently, within the pure quadratic model,
Theorem~\ref{thm:pure-dichotomy} proves both implications of the cube-trace
criterion:
\[
\begin{aligned}
  n^3\tr^2(\Kop_\kappa^3)\to\infty
  &\quad\Longrightarrow\quad \TV(\Pp_n,\Qq_n)\to1,\\
  n^3\tr^2(\Kop_\kappa^3)\to0
  &\quad\Longrightarrow\quad \TV(\Pp_n,\Qq_n)\to0.
\end{aligned}
\]
The two implications have different statistical content.  The first is
witnessed by the polynomial-time signed-triangle statistic, whose mean shift
is larger than its standard deviation under both models.  The second
controls the full likelihood ratio and therefore rules out every test, not
only tests based on a bounded collection of subgraph counts.  Thus the same
scale governs both detectability and indistinguishability.  The constant-order
case requires the separate likelihood analysis below.

At the critical scale \(N\asymp n^{3/4}|\beta|^{3/2}\), neither strict
implication determines the testing distance.  The next theorem identifies
its limit through the normalized signed triangle.

\begin{theorem}[Critical profile and finite-sample bounds]
\label{thm:growing-critical-main}
Let \(\Pp_n\), \(\Qq_n\), and \(N=N_{2,d}\) be as in
Theorem~\ref{thm:pure-dichotomy}, and allow \(d=d_n\) and
\(\beta=\beta_n\) to depend on \(n\).
\par\smallskip
\noindent\emph{(i) Asymptotic total-variation limit.} Assume
\[
  N\to\infty,
  \qquad
  \Delta_n
  :=
  \frac{\sqrt{\binom n3}\,|\beta_n|^3}{N^2}
  \longrightarrow\Delta\in(0,\infty).
\]
Then
\[
  \TV(\Pp_n,\Qq_n)
  \longrightarrow
  2\Phi(\Delta/2)-1.
\]

\par\smallskip
\noindent\emph{(ii) Uniform finite-sample approximation.}\par
\noindent Recall the normalized signed-triangle statistic
\[
  W_n
  :=
  \operatorname{sgn}(\beta_n)
  \frac{1}{\sqrt{\binom n3}}
  \sum_{i<j<k}Y_{ij}Y_{jk}Y_{ki}.
\]
Fix \(0<a<b<\infty\), \(0\le r<1\), and
\(0<c_-\le c_+<\infty\).  If
\[
  a\le\Delta_n\le b,
  \qquad
  c_-n^{-r}\le\beta_n^2\le c_+n^{-r},
\]
then there are constants \(\gamma,C>0\) and \(n_0\), depending only on
\(a,b,r,c_-,c_+\), such that, for every \(n\ge n_0\),
\[
  \left|
    \TV(\Pp_n,\Qq_n)-\{2\Phi(\Delta_n/2)-1\}
  \right|
  \le Cn^{-\gamma}.
\]
Moreover, if \(L_n=d\Pp_n/d\Qq_n\), then, for every \(0<t\le1\),
\[
  \Qq_n\left\{
    \left|\log L_n-(\Delta_nW_n-\Delta_n^2/2)\right|>t
  \right\}
  \le Ct^{-2}n^{-\gamma}.
\]
\end{theorem}

Here the signed triangle is the statistic
\(T_n=\sum_{i<j<k}Y_{ij}Y_{jk}Y_{ki}\), rather than the ordinary triangle
count.  The variable \(W_n\) is its null-variance normalization, oriented by
\(\operatorname{sgn}(\beta_n)\).  Under the stated assumptions, \(W_n\)
carries all limiting testing information, and under polynomial signal scaling
the approximation is uniform at a polynomial rate.
\section{Signed triangles in the quadratic layer}
\label{sec:cycles}

The cycle--trace identity \eqref{eq:cycle-trace-identity} identifies the
signed triangle as the shortest simple cycle with a nonzero mean under the
quadratic alternative. We compute its exact moments and prove the detectable implication of
Theorem~\ref{thm:pure-dichotomy}.

Put
\[
 Y_{ij}=2a_{ij}-1,
 \qquad
 T_n=\sum_{1\le i<j<k\le n}Y_{ij}Y_{jk}Y_{ki}.
\]
For the degree-two kernel, \(\Kop_\kappa\) has eigenvalue \(\beta/N\)
with multiplicity \(N=N_{2,d}\).  Hence
\begin{equation}
 \tau:=\tr(\Kop_\kappa^3)=\frac{\beta^3}{N^2},
 \qquad
 \zeta:=\tr(\Kop_\kappa^4)=\frac{\beta^4}{N^3}.
 \label{eq:quadratic-triangle-traces}
\end{equation}

\subsection{Exact signed-triangle moments}
\label{subsec:triangle-moment-calculation}

\begin{proposition}[Exact quadratic triangle moments]
\label{prop:quadratic-triangle-moments}
Under \(\Qq_n=G(n,1/2)\),
\[
 \E_{\Qq_n}T_n=0,
 \qquad
 \operatorname{Var}_{\Qq_n}(T_n)=\binom n3.
\]
Under the quadratic geometric law \(\Pp_n\),
\begin{align}
 \E_{\Pp_n}T_n
 &=\binom n3\tau,                                      \label{eq:quadratic-triangle-mean}\\
 \operatorname{Var}_{\Pp_n}(T_n)
 &=\binom n3(1-\tau^2)
   +12\binom n4(\zeta-\tau^2).                         \label{eq:quadratic-triangle-variance}
\end{align}
\end{proposition}

\begin{proof}
Under the null, distinct triangle monomials are orthogonal: their product
contains an edge that occurs exactly once unless the two triangles coincide.
Every monomial has square one, which gives the null formulas.

Under \(\Pp_n\), \eqref{eq:cycle-trace-identity} gives mean \(\tau\) for
each triangle.  If two triangles share at most one vertex, integrating the
other latent points gives \(\tau\) for each cycle, uniformly in the possible
shared vertex by rotational invariance; their product expectation is therefore
\(\tau^2\).  The only remaining case is a pair sharing one edge.  If the
triangles are \((1,2,3)\) and \((1,2,4)\), then \(Y_{12}^2=1\), and conditional
edge independence followed by integration over \(X_3\) and \(X_4\) gives
\[
 \E_{\Pp_n}
 \{Y_{12}Y_{23}Y_{31}Y_{12}Y_{24}Y_{41}\}
 =\tr(\Kop_\kappa^4)=\zeta.
\]
Thus an edge-sharing pair has covariance \(\zeta-\tau^2\).  There are
\(12\binom n4\) ordered pairs of distinct triangles sharing an edge, which
proves \eqref{eq:quadratic-triangle-variance}.
\end{proof}

\subsection{Detection from the mean shift}
The variance has only two contributions: individual triangles and pairs
sharing an edge. We compare both with the squared mean.

\begin{proposition}[Quadratic detectable direction]
\label{prop:single-positive}
If
\[
 \frac{n^3\beta^6}{N^4}\longrightarrow\infty,
\]
then the signed-triangle test oriented by \(\operatorname{sgn}(\beta)\) is
strongly consistent and
\[
 \TV(\Pp_n,\Qq_n)\longrightarrow1.
\]
\end{proposition}

\begin{proof}
Let \(\mu_n=\E_{\Pp_n}T_n\).  Proposition
\ref{prop:quadratic-triangle-moments} and
\eqref{eq:quadratic-triangle-traces} give
\[
 \mu_n^2\asymp\frac{n^6\beta^6}{N^4},
 \qquad
 \operatorname{Var}_{\Qq_n}(T_n)\asymp n^3,
 \qquad
 \operatorname{Var}_{\Pp_n}(T_n)
 =O\!\left(n^3+\frac{n^4\beta^4}{N^3}\right).
\]
Consequently,
\[
 \frac{\operatorname{Var}_{\Qq_n}(T_n)}{\mu_n^2}
 =O\!\left(\frac{N^4}{n^3\beta^6}\right)=o(1)
\]
and
\[
 \frac{\operatorname{Var}_{\Pp_n}(T_n)}{\mu_n^2}
 =O\!\left(
   \frac{N^4}{n^3\beta^6}+\frac{N}{n^2\beta^2}
 \right).
\]
The assumption implies eventually
\(N\le n^{3/4}|\beta|^{3/2}\).  Since \(N\ge1\), it also implies
\(|\beta|\ge n^{-1/2}\), and therefore
\[
 \frac{N}{n^2\beta^2}
 \le n^{-5/4}|\beta|^{-1/2}
 \le n^{-1}\longrightarrow0.
\]
Both variances are thus \(o(\mu_n^2)\).  Chebyshev's inequality applied to
\[
 \varphi_n
 =\mathbf 1\!\left\{
   \operatorname{sgn}(\mu_n)T_n>\frac{|\mu_n|}{2}
 \right\}
\]
shows that
\(\Qq_n(\varphi_n=1)+\Pp_n(\varphi_n=0)\to0\).  The variational identity for
testing error and total variation completes the proof.
\end{proof}

Since \(N^4/(n^3\beta^6)\to0\) is equivalent to
\(N/(n^{3/4}|\beta|^{3/2})\to0\), the proposition proves precisely the
detectable half of the strict quadratic boundary.  If the sign of \(\beta\)
is not specified, replacing the oriented rejection event by its two-sided
analogue based on \(|T_n|\) gives the same strong-detection conclusion.
\section{Likelihood analysis and indistinguishability}
\label{sec:quadratic}

We prove the indistinguishability implication of
Theorem~\ref{thm:pure-dichotomy} through the full likelihood ratio.
The quadratic feature representation turns its second moment into an
interaction between two independent latent configurations. Conditioning the
prior controls rare configurations, after which the sparse and dense
recursions cover the required parameter range.

\subsection{The quadratic feature representation}
\label{subsec:quadratic-feature-lift}

Use the traceless symmetric matrix space
\(\mathbb V_d=\{H\in\mathbb R^{d\times d}:H^\top=H,\ \tr H=0\}\),
of dimension \(N\), with the Hilbert--Schmidt inner product.
We also write \(V=\mathbb V_d\). Define
\[
 \varphi_d(x)=\sqrt{\frac d{d-1}}
       \left(xx^\top-\frac1dI_d\right),\qquad q_i=\varphi_d(X_i).
\]
Lemma~\ref{lem:feature-lift} verifies
\[
\begin{gathered}
 \langle\varphi_d(x),\varphi_d(y)\rangle_{\HS}
 =Z_{2,d}(\langle x,y\rangle),\\
 \|q_i\|_{\HS}=1,\qquad
 \mathbb E q_i=0,\qquad
 \mathbb E(q_i\otimes q_i)=I_{\mathbb V_d}/N.
\end{gathered}
\]
Let \(\mu\) denote the law of \(q_i\).
Here \(q\otimes q\) is the rank-one operator
\(H\mapsto\langle q,H\rangle_{\HS}q\). These identities make the model a
random dot-product graph with a centered isotropic feature law.
For \(d\ge3\), the features lie on a proper orthogonal-group orbit in
the unit sphere of \(\mathbb V_d\). Thus the uniform-spherical theorem
\cite[Theorem~3]{MaoWuXu2026} does not apply by replacing the latent
dimension with \(N\). The moment and concentration estimates in
Appendix~\ref{app:complete-preliminaries} retain this orbit structure.

\subsection{Likelihood second moments and prior truncation}
\label{subsec:quadratic-second-moment}

Write \(q=(q_1,\ldots,q_n)\), and let \(B=B(q)\) be its hollow Gram matrix:
\(B_{ii}=0\) and \(B_{ij}=\langle q_i,q_j\rangle_{\HS}\) for \(i\ne j\).
The conditional edge likelihood factor is \(1+\beta Y_{ij}B_{ij}\), so
integrating over the latent configuration gives
\begin{equation}
  L_n(\Adj)=\frac{d\Pp_n}{d\Qq_n}(\Adj)
  =
  \E_q\prod_{i<j}
  \left(1+\beta Y_{ij}B_{ij}\right).
  \label{eq:pure-mixture-likelihood}
\end{equation}
Let \(q^\ast\) be an independent copy of \(q\), with Gram matrix
\(B^\ast=B(q^\ast)\). Squaring the latent integral introduces these two
replicas. Independence and centering of the edge signs under \(\Qq_n\)
then give
\begin{equation}
  \E_{\Qq_n}\left[
    L_n(\Adj)^2
  \right]
  =
  \E_{q,q^\ast}\prod_{i<j}
  \left(
    1+\beta^2
    \langle q_i,q_j\rangle_{\HS}
    \langle q_i^\ast,q_j^\ast\rangle_{\HS}
  \right).
  \label{eq:pure-second-moment-replica}
\end{equation}
Lemma~\ref{lem:truncated-likelihood} derives the corresponding identity with
prior conditioning; setting its event to the full prior space gives these
two formulas. For fixed \(q^\ast\), the remaining expectation is a partition
function with interaction matrix \(A=\beta^2B^\ast\).

Rare configurations may dominate the unconditioned second moment. It
suffices instead to find a prior event \(\mathcal E_n\) with probability
\(1-o(1)\) such that \(\mathbb E_{\Qq_n}(L_n^{\mathcal E})^2=1+o(1)\).
Here \(L_n^{\mathcal E}=d\Pp_n^{\mathcal E_n}/d\Qq_n\).
Indeed, prior conditioning changes the graph law by at most
\(\mathbb P(\mathcal E_n^c)\) in total variation. The triangle inequality
and \eqref{eq:tv-chi-square-bound} yield
\[
 \TV(\Pp_n,\Qq_n)
 \le\mathbb P(\mathcal E_n^c)
   +\tfrac12\sqrt{\mathbb E_{\Qq_n}(L_n^{\mathcal E})^2-1}=o(1).
\]

\begin{proposition}[Pure quadratic indistinguishability]
\label{prop:pure-negative}
If
\[
  \frac{N}{n^{3/4}|\beta|^{3/2}}\to\infty,
\]
then
\[
  \TV(\Pp_n,\Qq_n)\to0.
\]
\end{proposition}

\begin{proof}
At indices for which \(\beta=0\), the two graph laws coincide.  Otherwise set
\(R_n=N/(n\beta^2)\).  If \(R_n\to\infty\),
Proposition~\ref{prop:sparse-signal-complete} gives the conclusion.  If
\(R_n=O(1)\), define
\(h_n=N/(n^{3/4}|\beta|^{3/2})\).  The hypothesis gives \(h_n\to\infty\),
whereas \(R_n=O(1)\) gives \(N\lesssim n\beta^2\).  Therefore
the conditional second-moment estimate in
Proposition~\ref{prop:dense-branch-second-moment}, proved in
Sections~\ref{sec:dense-likelihood} and~\ref{sec:likelihood-interpolation}, applies.  Hence the truncated
\(\chi^2\)-divergence tends to zero.  Since
\(\mathbb P(q\notin\mathcal E_n)=o(1)\), Lemma~\ref{lem:truncated-likelihood} and
\eqref{eq:tv-chi-square-bound} give \(\TV(\Pp_n,\Qq_n)\to0\).

For a general sequence, every subsequence has a further subsequence on which
either \(R_n\to\infty\) or \(R_n=O(1)\).  The two cases above give
total-variation convergence on every such further subsequence, which proves the
claim for the original sequence.
\end{proof}

\begin{proof}[Proof of Theorem~\ref{thm:pure-dichotomy}]
If
$N/(n^{3/4}|\beta|^{3/2})\to0$, then
$n^3\beta^6/N^4\to\infty$, so
Proposition~\ref{prop:single-positive} gives $\TV\to1$.
If
$N/(n^{3/4}|\beta|^{3/2})\to\infty$, then
Proposition~\ref{prop:pure-negative} gives $\TV\to0$.
\end{proof}

Theorem~\ref{thm:pure-dichotomy} therefore shows that the same dimensionless
quantity \(n^3\beta^6/N^4\) governs both strict regimes.  This quantity has the
order of the squared triangle mean divided by its null variance.  Hence, for
the pure quadratic model, the signed-triangle threshold agrees with the
information-theoretic threshold away from the critical scale.

\section{The symmetric endpoint of the dense likelihood bound}
\label{sec:dense-likelihood}

\subsection{Conditional target and interpolation}

By Lemma~\ref{lem:truncated-likelihood}, the dense branch of
Proposition~\ref{prop:pure-negative} reduces to the following uniform bound
after fixing the first latent replica.

\begin{proposition}[Dense-branch conditional second moment]
\label{prop:dense-branch-second-moment}
Suppose that
\begin{equation}
  \frac{N}{n^{3/4}|\beta|^{3/2}}\longrightarrow\infty,
  \qquad N\lesssim n\beta^2.
  \label{eq:dense-master-regime}
\end{equation}
There exists a sequence of measurable subsets \(\mathcal E_n\) of the latent
configuration space such that \(\mathbb P(q\in\mathcal E_n)=1-o(1)\) and,
uniformly over
\(q^\ast\in\mathcal E_n\),
\begin{equation}
  \E_q\left[
    \mathbf1_{\{q\in\mathcal E_n\}}
    \prod_{i<j}
    \left\{1+\beta^2
      \langle q_i,q_j\rangle_{\HS}
      \langle q_i^\ast,q_j^\ast\rangle_{\HS}
    \right\}
  \right]\le1+o(1).
  \label{eq:pure-negative-conditional-target}
\end{equation}
\end{proposition}

Recall the hollow Gram matrix \(B(u)\) from
Section~\ref{subsec:quadratic-second-moment}, and put \(A(u)=\beta^2B(u)\).
Fix \(q^\ast\in\mathcal E_n\), with \(\mathcal E_n\) defined below, and
write \(B^\ast=B(q^\ast)\), \(B=B(q)\), and \(A=\beta^2B^\ast\).
In this section, Section~\ref{sec:likelihood-interpolation}, and Appendices~\ref{app:complete-adaptive-even}--\ref{app:quantitative-LAN},
\(A\), \(A_s\), and \(A^{(m)}\) refer to the fixed first replica;
\(A(u)\) retains its argument for a generic configuration, and \(B=B(q)\)
belongs to the replica being integrated. The compensated exponent is
\begin{equation}
 K_A=\sum_{i<j}\left(A_{ij}B_{ij}-\tfrac12A_{ij}^2B_{ij}^2\right).
 \label{eq:dense-compensated-exponent}
\end{equation}
Choose \(C_{\mathrm{mat}}\) so that Proposition~\ref{prop:row-gram-environment}
gives \(\|A\|_{\HS}^2\le C_{\mathrm{mat}}\beta^4n^2/N\). For sufficiently
large absolute constants \(K_0,C_{\mathrm{cov}}\), set
\begin{equation}
 T=K_0\left\{C_{\mathrm{mat}}\frac{\beta^4n^2}{N^2}+\log(nN)\right\},
 \qquad \Lambda=C_{\mathrm{cov}}\left(\frac nN+T\right).
 \label{eq:adaptive-cutoff-parameters}
\end{equation}
For an index set \(I\), write
\[
 \mathcal C_I(\lambda)=\left\{\left\|\sum_{i\in I}q_i\otimes q_i\right\|_{\mathrm{op}}\le\lambda\right\},
 \qquad \mathcal C_m(\lambda)=\mathcal C_{[m]}(\lambda).
\]
The deterministic parameters \(T,\Lambda\) depend only on \(n,N,\beta\), and
\(T\ge K_0\{\|A\|_{\HS}^2/N+\log(nN)\}\).
Define \(\mathcal E_n\) as the intersection of the covariance event
\(\mathcal C_n(\Lambda+1)\) from Section~\ref{sec:covariance-truncation},
the event giving the bounds of Proposition~\ref{prop:row-gram-environment}
for \(A(u)\), and the maximum-coefficient event of Lemma~\ref{lem:gmax}.
Those two results and Lemma~\ref{lem:uniform-covariance-tail} give
\(\mathbb P(\mathcal E_n)=1-o(1)\). We use the matrix bounds
\eqref{eq:environment-FR}--\eqref{eq:environment-L} with the constants fixed
by these events. The row and entry bounds control local remainders; the
Hilbert--Schmidt and operator bounds control the trace potential and
exponential weights. The sixth row moment is used in
Section~\ref{sec:return-simple-edge-likelihood}, which
proves the uniform comparison
\begin{equation}
 \mathbf1_{\{q\in\mathcal E_n\}}\prod_{i<j}(1+A_{ij}B_{ij})
 \le e^{o(1)}\mathbf1_{\mathcal C_n(\Lambda+1)}e^{K_A}.
 \label{eq:dense-product-comparison}
\end{equation}
To bound its expectation, interpolate from the symmetric feature law:
\[
 \bar\mu=\frac{\mu+\mu^-}{2},\qquad
 \eta=\frac{\mu-\mu^-}{2},\qquad \mu_t=\bar\mu+t\eta,\quad 0\le t\le1,
\]
where \(\mu^-\) is the law of \(-q\) for \(q\sim\mu\). Define
\begin{equation}
 Z(t)=\mathbb E_{\mu_t^{\otimes n}}
       [\mathbf1_{\mathcal C_n(\Lambda+1)}e^{K_A}].
 \label{eq:adaptive-interpolating-partition}
\end{equation}
The proof has three steps:
\[
\begin{aligned}
 Z(0)&\le1+o(1)
 &&\text{(Proposition~\ref{prop:symmetric-endpoint})},\\
 |Z'(t)|&\le o(1)\max\{1,Z(t)\}
 &&\text{(Sections~\ref{sec:buffered-cavity-identity}--\ref{sec:differential-inequality})},\\
 Z(1)&\le1+o(1)\ \Longrightarrow\ \eqref{eq:pure-negative-conditional-target}
 &&\text{(Section~\ref{sec:return-simple-edge-likelihood})}.
\end{aligned}
\]
Gronwall's inequality connects the second and third steps. Lemma
\ref{lem:truncated-likelihood} and \eqref{eq:tv-chi-square-bound} then give
\(\TV(\Pp_n,\Qq_n)\to0\). This section proves the symmetric endpoint;
Section~\ref{sec:likelihood-interpolation} controls the interpolation and
returns to the exact likelihood product. The supporting moment and
concentration estimates are proved in Appendices~\ref{app:complete-adaptive-even}
and~\ref{app:complete-buffered-cavity}.

Unless stated otherwise, all \(o(1)\) terms in these two sections and their supporting appendices are
uniform for \(q^\ast\in\mathcal E_n\) and \(t\in[0,1]\), along arbitrary
sequences satisfying \eqref{eq:dense-master-regime} with \(0<|\beta|\le1\).
For deterministic \(A\), this means uniformity over
\eqref{eq:environment-FR}--\eqref{eq:environment-L} and, when invoked,
Lemma~\ref{lem:gmax}. All implicit constants are independent of the varying
parameters unless a dependence is displayed. In these two sections, expectations concern the second replica. Write
\(h_n=N/(n^{3/4}|\beta|^{3/2})\); then \(h_n\to\infty\) and
\(n\beta^2\gtrsim h_n^4\).

\subsection{The compensated deletion recursion}

We bound the symmetric endpoint by integrating \(q_n,q_{n-1},\ldots,q_2\)
successively. Integrating one feature creates an interaction among the
remaining features, which determines the next coefficient matrix. The
negative quadratic term cancels the leading local variance, but its
coefficients must be tracked separately from the updated interaction.
For symmetric hollow matrices \(A,P\) of the same size, with
\(P_{ij}\ge0\), define
\begin{equation}
 K_{A,P}(q)=\sum_{i<j}A_{ij}B_{ij}-\tfrac12\sum_{i<j}P_{ij}B_{ij}^2.
 \label{eq:two-matrix-exponent}
\end{equation}
Here \(A\) records the linear interaction and \(P\) the quadratic
compensation. We derive their update from a conditional second-moment
calculation, then identify the errors that must be summed to close the
recursion.

\subsubsection{Conditional integration and the matrix update}
\label{sec:two-matrix-compensation}

Let \(\operatorname{off}(M)\) set the diagonal of a square matrix to zero,
and put \(R_0=R_2(A)=\max_i\|A_{i\cdot}\|_2\).
Start with \(A^{(n)}=A\) and \(P^{(n)}=A^{\circ2}\), where
\(A^{\circ2}=(A_{ij}^2)_{ij}\), so that
\(K_{A^{(n)},P^{(n)}}=K_A\). At stage \(m\), write
\[
  A^{(m)}
  =
  \begin{pmatrix}
  \widehat A^{(m)}&a^{(m)}\\
  (a^{(m)})^\top&0
  \end{pmatrix},
\]
where \(\widehat A^{(m)}\) is the leading
\((m-1)\times(m-1)\) principal block and
\(a_j^{(m)}=A_{jm}^{(m)}\) for \(j<m\). The retained quadratic
coefficients are unchanged by deletion: obtain \(P^{(m-1)}\) by
restricting \(P^{(m)}\) to \([m-1]\). The update of \(A^{(m)}\) will
follow from the conditional calculation below. Let
\[
  r_j^{(m)\,2}
  =
  \sum_{\substack{k\le m\\k\ne j}}
  (A_{jk}^{(m)})^2,
\]
so \(r_j^{(m)\,2}\) is the squared \(\ell_2\)-norm of row \(j\) of
\(A^{(m)}\).  We abbreviate $r_m=r_m^{(m)}$ for the eliminated row.

At a fixed deletion step, write $a_j=a_j^{(m)}=A_{jm}^{(m)}$,
$r_m^2=\sum_{j<m}a_j^2$, and set
\begin{equation}
  H_m=\sum_{j<m}a_jq_j,\qquad D_m=\langle H_m,q_m\rangle_{\HS}.
  \label{eq:adaptive-local-field-defs}
\end{equation}
Throughout this subsection, $\mathbb E_m$ integrates only
$q_m\sim\bar\mu$, conditional on $q_1,\ldots,q_{m-1}$.  Define
\[
\begin{gathered}
  p_m=\sum_{j<m}P_{jm}^{(m)},\qquad
  Q_m^P=\sum_{j<m}P_{jm}^{(m)}B_{jm}^2,\qquad
  W_m^P=\tfrac12(D_m^2-Q_m^P),\\
  \delta_m=\frac{r_m^2-p_m}{2N},\qquad
  L_m=\frac1N\sum_{j<k<m}a_ja_kB_{jk}.
\end{gathered}
\]
The scalar \(\delta_m\) measures the difference between the current
squared row norm and the inherited compensation mass. The term \(L_m\)
is an interaction among retained features. Unit feature norms and isotropy give
$\|H_m\|_{\HS}^2=r_m^2+2NL_m$,
$\mathbb E_mD_m^2=\|H_m\|_{\HS}^2/N$, and
$\mathbb E_mQ_m^P=p_m/N$.  Hence
\begin{equation}
  \mathbb E_mW_m^P=\frac{r_m^2-p_m}{2N}+L_m=\delta_m+L_m.
  \label{eq:conditional-compensation-mean}
\end{equation}
The conditional mean thus contains the retained interaction \(L_m\).
To absorb it into the next exponent, define
\begin{equation}
  A^{(m-1)}
  =
  \operatorname{off}\left(
  \widehat A^{(m)}+\frac1Na^{(m)}(a^{(m)})^\top
  \right).
  \label{eq:hollow-schur}
\end{equation}
This hollow--Schur update adds exactly \(L_m\) to the retained linear
exponent, so
\begin{equation}
  K_{\widehat A^{(m)},P^{(m-1)}}+L_m
  =K_{A^{(m-1)},P^{(m-1)}}.
  \label{eq:two-matrix-identity}
\end{equation}
Separating the interactions with $q_m$ therefore gives
\begin{equation}
  K_{A^{(m)},P^{(m)}}=K_{A^{(m-1)},P^{(m-1)}}
  +D_m-\tfrac12Q_m^P-L_m.
  \label{eq:two-matrix-recursion}
\end{equation}

Under $q_m\mapsto-q_m$, $D_m$ is odd and $Q_m^P$ is even.  Symmetry of
$\bar\mu$ and $\cosh x\le e^{x^2/2}$ yield
\[
  \mathbb E_m e^{D_m-Q_m^P/2}
  =\mathbb E_m[e^{-Q_m^P/2}\cosh D_m]\le\mathbb E_m e^{W_m^P}.
\]
Using $e^w\le1+w+Cw^2e^{w_+}$ and
\eqref{eq:conditional-compensation-mean}, we obtain
\begin{equation}
\begin{aligned}
  \widetilde{\mathcal E}_m&:=C\mathbb E_m[(W_m^P)^2e^{(W_m^P)_+}],\\
  \mathbb E_m e^{D_m-Q_m^P/2}&\le1+\delta_m+L_m+\widetilde{\mathcal E}_m.
\end{aligned}
  \label{eq:even-one-step}
\end{equation}
Combining this with \eqref{eq:two-matrix-recursion} and $1+x\le e^x$ gives
\begin{equation}
\begin{aligned}
  \mathbb E_m e^{K_{A^{(m)},P^{(m)}}}
  &=e^{K_{A^{(m-1)},P^{(m-1)}}-L_m}\mathbb E_m e^{D_m-Q_m^P/2}\\
  &\le e^{K_{A^{(m-1)},P^{(m-1)}}+\delta_m+\widetilde{\mathcal E}_m}.
\end{aligned}
  \label{eq:one-step-exponent-recursion}
\end{equation}
The Schur correction thus absorbs \(L_m\); the terms left to sum are
\(\delta_m\) and \(\widetilde{\mathcal E}_m\).
Because \(P^{(m-1)}\) is only restricted while \(A^{(m-1)}\) is updated,
\(P^{(m-1)}\) generally differs from \((A^{(m-1)})^{\circ2}\).
This explains the use of two matrices in \eqref{eq:two-matrix-exponent}.

\subsubsection{Remainder control and iteration}

The defect \(\delta_m\) is deterministic once the first replica is fixed;
\(\widetilde{\mathcal E}_m\) still depends on the retained features.
We bound the latter by a deterministic majorant on a good event before
iterating \eqref{eq:one-step-exponent-recursion}.

Since $Q_m^P\ge0$,
$(W_m^P)^2e^{(W_m^P)_+}\le C(D_m^4+(Q_m^P)^2)e^{D_m^2/2}$.
To control this remainder, choose thresholds $k_{m,\mathrm{ad}}\ge1$ and set
\begin{equation}
  \mathcal G_m=\{\|H_m\|_{\HS}^2\le k_{m,\mathrm{ad}}r_m^2\},\qquad
  \mathcal G^{\mathrm{rec}}=\bigcap_{m=2}^n\mathcal G_m.
  \label{eq:recursive-good-event}
\end{equation}
The event \(\mathcal G_m\) depends only on \(q_1,\ldots,q_{m-1}\),
so it can be imposed before integrating \(q_m\).
For $r_m=0$, set $k_{m,\mathrm{ad}}=1$; then $H_m=0$ and $\mathcal G_m$
is the whole space.  On $\mathcal G_m$, both $\|H_m\|_{\HS}^2$ and
$\|H_m\|_{\mathrm{op}}^2$ are at most $k_{m,\mathrm{ad}}r_m^2$.  Define
\[
  U_m=k_{m,\mathrm{ad}}^2(r_m^2/N)^2,\qquad
  V_m=(p_m/N)^2,\qquad \mathcal E_m=C_{\mathrm{rem}}(U_m+V_m),
\]
where $C_{\mathrm{rem}}$ is a sufficiently large absolute constant.
The terms \(U_m\) and \(V_m\) control, respectively, the \(D_m^4\)
and \((Q_m^P)^2\) contributions to the remainder.
Sections~\ref{sec:adaptive-local-fields}--\ref{sec:adaptive-exponential-remainders}
construct the thresholds with exceptional
probability at most $ne^{-T}$ on $\mathcal C_n(\Lambda)$, prove
$\sum_m\mathcal E_m=o(1)$ and
$\max_m k_{m,\mathrm{ad}}r_m^2/d=o(1)$, and establish
$\widetilde{\mathcal E}_m\le\mathcal E_m$ on $\mathcal G_m$.
The three estimates needed to close the recursion are therefore
\begin{equation}
\begin{gathered}
 \sum_{m=2}^n\mathcal E_m=o(1),\qquad \sum_{m=2}^n\delta_m=o(1),\\
 \mathbb E_{\bar\mu^{\otimes n}}[
   \mathbf1_{\mathcal C_n(\Lambda)}
   \mathbf1_{(\mathcal G^{\mathrm{rec}})^c}e^{K_A}]=o(1).
\end{gathered}
\label{eq:adaptive-proof-targets}
\end{equation}
The first is supplied by the remainder estimates just cited.
Section~\ref{sec:compensation-trace-potential} sums \(\delta_m\) by a
telescoping trace potential. For the third, the good-event probability
bound is combined with the spectral second moment from
Theorem~\ref{thm:l2-spectral-recursion}.
Section~\ref{sec:symmetric-endpoint-closure} then iterates the one-step
inequality, adds the discarded contribution, and extends the estimate
from \(\Lambda\) to \(\Lambda+1\), obtaining \(Z(0)\le1+o(1)\).

\subsection{The compensation defect and its trace potential}
\label{sec:compensation-trace-potential}

It remains to sum the compensation defect $\delta_m=(r_m^2-p_m)/(2N)$.
We reduce it to quadratic forms and identify their cancellation through
a telescoping trace potential.  Expanding the hollow--Schur update gives
\[
  \|A^{(m-1)}\|_{\HS}^2-\|A^{(m)}\|_{\HS}^2
  =-2r_m^2+\frac2N(a^{(m)})^\top\widehat A^{(m)}a^{(m)}
  +\frac{r_m^4-\sum_{j<m}(a_j^{(m)})^4}{N^2}.
\]
Sum over $m$ and use $A^{(1)}=0$ and
\textup{(S3.7)} to obtain
\begin{equation}
  \sum_{m=2}^n\delta_m
  =\sum_{m=2}^n\frac{(a^{(m)})^\top\widehat A^{(m)}a^{(m)}}{2N^2}
  +\operatorname{Rem}_n,
  \label{eq:delta-quadratic-reduction}
\end{equation}
where
\[
  0\le\operatorname{Rem}_n
  :=\sum_{m=2}^n\frac{r_m^4-\sum_{j<m}(a_j^{(m)})^4}{4N^3}
  \le\frac1{4N}\sum_{m=2}^n(r_m^2/N)^2=o(1)
\]
by \textup{(S3.3)}.  The quadratic-form sum is
controlled by the following potential.

For $\|M\|_{\mathrm{op}}<N$, define
\begin{equation}
  \mathcal P_N(M)
  :=
  -\frac N2\log\det(I-M/N)
  -\frac12\operatorname{tr}M
  -\frac1{4N}\operatorname{tr}(M^2)
  =
  \sum_{\ell\ge3}
  \frac{\operatorname{tr}(M^\ell)}{2\ell N^{\ell-1}}.
  \label{eq:trace-potential-definition}
\end{equation}

\begin{lemma}[One-step trace-potential approximation]
\label{lem:trace-potential-step}
Let $M_0\in\mathbb R^{\ell\times\ell}$ be symmetric and hollow, let
$a=(a_1,\ldots,a_\ell)^\top\in\mathbb R^\ell$, and set
\[
  M=\begin{pmatrix}M_0&a\\a^\top&0\end{pmatrix},
  \qquad
  M'=\operatorname{off}(M_0+aa^\top/N),
\]
and suppose that
$\|M_0\|_{\mathrm{op}}/N\le1/4$ and $\|a\|_2^2/N^2\le1/16$.  Then
\[
\begin{aligned}
  \left|
  \frac{a^\top M_0a}{2N^2}
  -\bigl\{\mathcal P_N(M)-\mathcal P_N(M')\bigr\}
  \right|
  \le
  \frac C{N^3}
  \left(
  \|a\|_2^4+
  \max_j\|M_0e_j\|_2^2\,\|a\|_2^2
  \right).
\end{aligned}
\]
\end{lemma}

\begin{proof}
Let $\mathcal R=(I-M_0/N)^{-1}$ and
$s=N^{-2}a^\top\mathcal Ra$.  The assumptions imply
$\|\mathcal R\|_{\mathrm{op}}\le4/3$ and $s\le1/12$.  The block
determinant formula gives
\[
  \det(I-M/N)=\det(I-M_0/N)(1-s).
\]
Since
$\operatorname{tr}(M^2)-\operatorname{tr}(M_0^2)=2\|a\|_2^2$,
\begin{equation}
  \mathcal P_N(M)-\mathcal P_N(M_0)
  =
  -\frac N2\log(1-s)-\frac{\|a\|_2^2}{2N}.
  \label{eq:trace-potential-block}
\end{equation}
Define the higher-order resolvent remainder by
\[
  \mathcal R_{\ge2}\coloneqq\mathcal R-I-\frac{M_0}N.
\]
Because
$\mathcal R=I+M_0/N+\mathcal R_{\ge2}$, the linear part of the logarithm in
\eqref{eq:trace-potential-block} satisfies
\[
\begin{aligned}
  \frac N2s-\frac{\|a\|_2^2}{2N}
  &=\frac{a^\top M_0a}{2N^2}
  +\frac{a^\top\mathcal R_{\ge2}a}{2N}.
\end{aligned}
\]
Moreover, $|{-\log(1-s)-s}|\le Cs^2$, so the nonlinear logarithmic
remainder is at most $CNs^2\le C\|a\|_2^4/N^3$.  We have therefore shown
\begin{equation}
  \frac{a^\top M_0a}{2N^2}
  -\bigl\{\mathcal P_N(M)-\mathcal P_N(M_0)\bigr\}
  =-\frac{a^\top\mathcal R_{\ge2}a}{2N}+O(\|a\|_2^4/N^3).
  \label{eq:trace-potential-first-comparison}
\end{equation}

Now write
\[
  E=M'-M_0
  =\frac{aa^\top}{N}-\operatorname{diag}\left(\frac{a_1^2}{N},
  \ldots,\frac{a_\ell^2}{N}\right).
\]
The Fr\'echet derivative of the trace potential $\mathcal P_N$ at $M_0$ is
\begin{equation}
  D\mathcal P_N(M_0)[E]
  =\frac12\operatorname{tr}(\mathcal R_{\ge2}E)
  =\frac{a^\top\mathcal R_{\ge2}a}{2N}
  -\frac1{2N}\sum_ja_j^2(\mathcal R_{\ge2})_{jj}.
  \label{eq:trace-potential-derivative}
\end{equation}
The path $M_0+tE$, $0\le t\le1$, has operator norm at most $5N/16$, and
  \(\|E\|_{\HS}\le \|a\|_2^2/N\).  Along this path, differentiating the resolvent
shows that
\[
  |D^2\mathcal P_N(M_0+tE)[E,E]|
  \le\frac C N\|E\|_{\HS}^2
  \le\frac{C\|a\|_2^4}{N^3}.
\]
Taylor's formula therefore adds only $O(\|a\|_2^4/N^3)$ to
\eqref{eq:trace-potential-derivative}.

It remains to bound the diagonal term.  Since
$\mathcal R_{\ge2}=(M_0/N)^2\mathcal R$ and $M_0$ commutes with $\mathcal R$,
\[
  |(\mathcal R_{\ge2})_{jj}|
  \le\frac C{N^2}\|M_0e_j\|_2^2.
\]
Hence
\[
  \frac1{2N}\sum_ja_j^2|(\mathcal R_{\ge2})_{jj}|
  \le\frac C{N^3}
  \max_j\|M_0e_j\|_2^2\,\|a\|_2^2.
\]
Adding the Taylor expansion of $\mathcal P_N(M')-\mathcal P_N(M_0)$ to
\eqref{eq:trace-potential-first-comparison} gives
\[
  \frac{a^\top M_0a}{2N^2}
  -\bigl\{\mathcal P_N(M)-\mathcal P_N(M')\bigr\}
  =
  -\frac1{2N}\sum_ja_j^2(\mathcal R_{\ge2})_{jj}
  +O(\|a\|_2^4/N^3).
\]
The preceding bound on the diagonal term proves the stated estimate.
\end{proof}

Since \(\|A\|_{\HS}^2/N^2=o(1)\) by
\textup{(S3.2)}, Proposition~\ref{prop:global-row-quartic} applies for all sufficiently large \(n\) and gives
$\|\widehat A^{(m)}\|_{\mathrm{op}}\le N/\sqrt{20}<N/4$ and
$r_m^2\le N^2/20<N^2/16$.  We may therefore apply
Lemma~\ref{lem:trace-potential-step} with
$(M,M_0,a,M')=(A^{(m)},\widehat A^{(m)},a^{(m)},A^{(m-1)})$.
The row and quartic bounds of the same proposition give the explicit
summed error
\begin{equation}
\begin{aligned}
  &\sum_{m=2}^n\frac C{N^3}
  \left(r_m^4+\max_{j<m}\|\widehat A^{(m)}e_j\|_2^2\,r_m^2\right)\\
  &\qquad\le\frac{CR_0^2\|A\|_{\HS}^2}{N^3}
  \exp\{C\|A\|_{\HS}^2/N^2\}=o(1).
\end{aligned}
  \label{eq:trace-potential-summed-error}
\end{equation}
Here \textup{(S3.3)} and
$\|A\|_{\HS}^2/N^2=o(1)$ give the last equality.  Since $A^{(1)}=0$,
the potential differences telescope:
\begin{equation}
  \sum_{m=2}^n\frac{(a^{(m)})^\top\widehat A^{(m)}a^{(m)}}{2N^2}
  =\mathcal P_N(A)+o(1).
  \label{eq:trace-potential-telescope}
\end{equation}
Finally, $|\operatorname{tr}(A^\ell)|\le
\|A\|_{\HS}^2\|A\|_{\mathrm{op}}^{\ell-2}$ and
$\|A\|_{\mathrm{op}}/N\le1/\sqrt{20}$ bound the potential series by
$C\|A\|_{\HS}^2\|A\|_{\mathrm{op}}/N^2$.
Substituting \eqref{eq:environment-FR} and \eqref{eq:environment-L} gives
\[
\begin{aligned}
  |\mathcal P_N(A)|
  &\lesssim\frac{\beta^6n^3}{N^4}
  +\frac{\beta^6n^{5/2}\sqrt{\log N}}{N^{7/2}}
  +\frac{\beta^6n^2\log N}{N^3}\\
  &=h_n^{-4}+\frac{|\beta|^{3/4}\sqrt{\log N}}{h_n^{7/2}n^{1/8}}
  +\frac{|\beta|^{3/2}\log N}{h_n^3n^{1/4}}=o(1),
\end{aligned}
\]
because $|\beta|\le1$, $N\lesssim n$, and $h_n\to\infty$.
Equations~\eqref{eq:delta-quadratic-reduction} and
\eqref{eq:trace-potential-telescope} now imply
\begin{equation}
  \sum_{m=2}^n\delta_m=o(1).
  \label{eq:delta-sum-small}
\end{equation}

\subsection{Closure of the symmetric endpoint}
\label{sec:symmetric-endpoint-closure}

\begin{proposition}[Symmetric compensated endpoint]
\label{prop:symmetric-endpoint}
Under \eqref{eq:dense-master-regime}, uniformly on the event in
Proposition~\ref{prop:row-gram-environment} and over
\(\Lambda'\in[\Lambda,\Lambda+1]\),
\[
  \mathbb E_{\bar\mu^{\otimes n}}\left[
  \mathbf1_{\mathcal C_n(\Lambda')}
  e^{K_A}
  \right]\le1+o(1).
\]
\end{proposition}

\begin{proof}
We iterate \eqref{eq:one-step-exponent-recursion} on the recursive good
event and use the spectral MGF to bound its complement.  Define
\begin{equation}
  \mathcal Z_m^{\mathrm{comp}}(A^{(m)},P^{(m)})
  =\mathbb E_{\bar\mu^{\otimes m}}\left[
  \mathbf1_{\mathcal C_m(\Lambda)}
  \prod_{\ell=2}^m\mathbf1_{\mathcal G_\ell}
  e^{K_{A^{(m)},P^{(m)}}}\right].
  \label{eq:compensated-recursive-partition}
\end{equation}
Since $\mathcal G_m$ is measurable in $q_1,\ldots,q_{m-1}$ and
$\mathcal C_m(\Lambda)\subseteq\mathcal C_{m-1}(\Lambda)$, integrating
$q_m$ first and using $\widetilde{\mathcal E}_m\le\mathcal E_m$ yields
\[
  \mathcal Z_m^{\mathrm{comp}}(A^{(m)},P^{(m)})
  \le e^{\delta_m+\mathcal E_m}
  \mathcal Z_{m-1}^{\mathrm{comp}}(A^{(m-1)},P^{(m-1)}).
\]
Iteration to $A^{(1)}=P^{(1)}=0$, with
\eqref{eq:even-remainders} and \eqref{eq:delta-sum-small}, gives
\begin{equation}
  \mathbb E_{\bar\mu^{\otimes n}}\left[
  \mathbf1_{\mathcal C_n(\Lambda)}\mathbf1_{\mathcal G^{\mathrm{rec}}}e^{K_A}\right]
  \le\exp\left\{\sum_{m=2}^n(\delta_m+\mathcal E_m)\right\}\le1+o(1).
  \label{eq:symmetric-endpoint-good-part}
\end{equation}

For the complement, write $\mathcal C=\mathcal C_n(\Lambda)$ and
$\mathcal G=\mathcal G^{\mathrm{rec}}$, with expectations under
$\bar\mu^{\otimes n}$.  The bounds in
\textup{(S3.1)}--\textup{(S3.2)}
and the definition of $\Lambda$ imply
\begin{equation}
  \frac{\|2A\|_{\mathrm{op}}}{N}\le\frac{2\|A\|_{\HS}}{N}
  \lesssim h_n^{-2},\qquad
  \frac{(\Lambda+1)R_0^2}{N}
  \lesssim h_n^{-4}+\frac{\log(nN)}{h_n^2\sqrt n}=o(1).
  \label{eq:dense-recursion-smallness}
\end{equation}
For the second bound, the contributions of $n/N$, $T$, and $1$ in
$\Lambda+1\lesssim n/N+T$ are bounded using
$T\lesssim(n\beta^2)^{1/2}/h_n^2+\log(nN)$ and
$h_n^4\lesssim n\beta^2\le n$.
Since $R_2(2A)=2R_0$ and $d^2\asymp N$, these bounds verify both
hypotheses of Theorem~\ref{thm:l2-spectral-recursion}, even at cutoff
$\Lambda+1$.  Dropping the nonpositive compensation gives
\[
  \mathbb E[\mathbf1_{\mathcal C}e^{2K_A}]
  \le\mathbb E\left[\mathbf1_{\mathcal C}
  e^{2\sum_{i<j}A_{ij}B_{ij}}\right]
  \le\exp\{C\|A\|_{\HS}^2/N\}.
\]
Cauchy--Schwarz and \eqref{eq:all-bad-rows-probability} now yield
\begin{equation}
  \log\mathbb E[\mathbf1_{\mathcal C}\mathbf1_{\mathcal G^c}e^{K_A}]
  \le\frac{C\|A\|_{\HS}^2}{2N}+\tfrac12\log n-\tfrac12T\longrightarrow-\infty,
  \label{eq:symmetric-endpoint-bad-part}
\end{equation}
because $T\ge K_0\{\|A\|_{\HS}^2/N+\log(nN)\}$ with $K_0$ sufficiently large.

For $\Lambda'\in[\Lambda,\Lambda+1]$, redefine $\theta(a)$ and
$k_{m,\mathrm{ad}}$ with cutoff $\Lambda'$.  Since $\Lambda\ge1$ and
$\Lambda'/\Lambda\le2$, the threshold estimates, remainder bounds, and
bad-row probability remain uniform; \eqref{eq:dense-recursion-smallness}
also holds at every such cutoff.  Repeating the two preceding bounds
therefore proves the claimed uniform endpoint estimate.

\end{proof}

Taking \(\Lambda'=\Lambda+1\) in
Proposition~\ref{prop:symmetric-endpoint} and setting \(t=0\) in
\eqref{eq:adaptive-interpolating-partition} gives
\[
  Z(0)\le1+o(1).
\]
Thus Proposition~\ref{prop:symmetric-endpoint} is precisely the
symmetric-endpoint conclusion identified in the proof map for
Proposition~\ref{prop:dense-branch-second-moment}.

\section{Interpolation and completion of indistinguishability}
\label{sec:likelihood-interpolation}

We retain the fixed matrix \(A\), cutoff, interpolation and uniformity
convention of Section~\ref{sec:dense-likelihood}. Starting from its bound
\(Z(0)\le1+o(1)\), we control \(Z'(t)\) by deleting one coordinate.
The odd local expansion has a cubic leading term. Summing over the deleted
coordinate produces a canonical weighted \(U\)-statistic; its high moments,
combined with a second moment of the likelihood weight, yield a uniform
differential inequality. Gronwall's inequality and an exact-product
comparison then prove Proposition~\ref{prop:dense-branch-second-moment}.

\subsection{Fixed interpolation and coordinate deletion with an enlarged cutoff}
\label{sec:buffered-cavity-identity}

In \eqref{eq:adaptive-interpolating-partition}, only the product measure
$\mu_t^{\otimes n}$ depends on $t$.  Differentiation replaces one factor
$\mu_t$ by $\eta$.  To integrate that coordinate conditionally, use
$q_i\otimes q_i\preceq I$ and hence
\[
  \mathcal C_{-i}(\Lambda)\subseteq\mathcal C_n(\Lambda+1)
  \subseteq\mathcal C_{-i}(\Lambda+1).
\]
Thus the full cutoff disappears from the $q_i$-integral on
$\mathcal C_{-i}(\Lambda)$; the discrepancy is confined to a covariance
boundary event.  Lemma~\ref{lem:uniform-covariance-tail} and
\eqref{eq:adaptive-cutoff-parameters} give
\begin{equation}
  \sup_{i,\,0\le t\le1}\mathbb P_{\mu_t^{\otimes(n-1)}}
  \{\mathcal C_{-i}(\Lambda)^c\}\le e^{-cT}.
  \label{eq:covariance-tail-dense}
\end{equation}
We combine this tail with a common second-moment bound below.

For each \(i\), separate the terms in \(K_A\) that involve \(q_i\) by defining
\[
\begin{aligned}
  K_{-i}
  &:={}
  \sum_{\substack{j<k\\j,k\ne i}}
  \left(A_{jk}B_{jk}-\frac12A_{jk}^2B_{jk}^2\right),\\
  D_i&:=\sum_{j\ne i}A_{ij}B_{ij},
  &Q_i&:=\sum_{j\ne i}A_{ij}^2B_{ij}^2.
\end{aligned}
\]
Then \(K_A=K_{-i}+D_i-Q_i/2\).  Conditional on \(q_{-i}\), define the
\(q_i\)-integrals against \(\mu_t\) and \(\eta\), respectively, by
\[
  Z_i^{\mathrm{loc}}(t):=\mathbb E_{\mu_t,i}e^{D_i-Q_i/2},
  \qquad
  O_i:=\int e^{D_i-Q_i/2}\,d\eta(q_i).
\]
The integrand in \(Z_i^{\mathrm{loc}}(t)\) is positive, so
\(Z_i^{\mathrm{loc}}(t)>0\).  Moreover,
\(d\mu_t/dt=\eta\) gives \((Z_i^{\mathrm{loc}})'(t)=O_i\).  Hence the one-coordinate score is
the logarithmic derivative
\[
  \rho_i(t):=\frac{d}{dt}\log Z_i^{\mathrm{loc}}(t)
  =\frac{O_i}{Z_i^{\mathrm{loc}}(t)}.
\]

Lemma~\ref{lem:buffered-cavity-complete} now expresses \(Z'(t)\) as a sum of
these conditional one-coordinate scores.  The discrepancy between the full
and reduced covariance cutoffs becomes the remainder
\(\operatorname{Rem}_{\mathrm{cov}}(t)\).

\begin{lemma}[Coordinate deletion with an enlarged cutoff]
\label{lem:buffered-cavity-complete}
Uniformly in $0\le t\le1$,
\begin{equation}
\begin{split}
  Z'(t)
  =
  \mathbb E_{\mu_t^{\otimes n}}\left[
  \mathbf1_{\mathcal C_n(\Lambda+1)}
  e^{K_A}\sum_{i=1}^n\rho_i(t)
  \right]
  +\operatorname{Rem}_{\mathrm{cov}}(t),
  \label{eq:buffered-cavity-identity}
\end{split}
\end{equation}
where
\[
  \sup_{0\le t\le1}|\operatorname{Rem}_{\mathrm{cov}}(t)|=o(1).
\]
The identity also has the following finite-parameter form, independent of
the dense-regime choice of cutoff.  For any \(\Lambda\) such that
\(2A\) satisfies the hypotheses of
Theorem~\ref{thm:l2-spectral-recursion} at cutoff \(\Lambda+1\), put
\[
  p_\Lambda:=\sup_{i,\,0\le t\le1}
     \mathbb P_{\mu_t^{\otimes(n-1)}}
        \{\mathcal C_{-i}(\Lambda)^c\}.
\]
Taking its smallness constant sufficiently small also ensures the
one-coordinate MGF hypothesis, and
\begin{equation}
  \sup_t|\operatorname{Rem}_{\mathrm{cov}}(t)|
  \le 2n p_\Lambda^{1/2}
    \exp\left\{
       \frac C N\|A\|_{\HS}^2
       +\frac{C(\Lambda+1)}N R_2(A)^2
    \right\}.
  \label{eq:buffered-cavity-general-remainder}
\end{equation}
\begin{samepage}
More generally, let $\mathcal U_i$ be an event measurable with respect to
$q_{-i}$.  If, for every $i$ and $t$,
\[
  \mathbb P_{\mu_t^{\otimes(n-1)}}\left\{
    \mathcal C_{-i}(\Lambda+1)\cap\mathcal U_i
  \right\}
  \le e^{-u}
\]
for some $u\ge0$, then
\[
\begin{aligned}
  \sum_i\left|
  \int
  \mathbf1_{\mathcal C_n(\Lambda+1)}
  \mathbf1_{\mathcal U_i}e^{K_A}
  \,d\eta(q_i)\prod_{j\ne i}d\mu_t(q_j)
  \right|
  \le
  Cn\exp\left\{
    \frac C{2N}\|A\|_{\HS}^2-\frac u2
  \right\}.
\end{aligned}
\]
\end{samepage}
\end{lemma}

The proof in Section~\ref{sec:covariance-boundary-proof} derives the identity
by conditional integration and bounds both boundary errors using the
same spectral second moment. The finite-parameter assertion will also be
used at the critical scale.

\subsection{The cubic term in the local score}
\label{sec:main-local-score}
Write
\[
 \mathcal K_3(q_1,q_2,q_3)
 =\mathbb E_{q\sim\mu}\prod_{a=1}^3\langle q,q_a\rangle_{\HS}.
\]
The third-moment identity \eqref{eq:third-moment-tensor} identifies this
kernel with \(\gamma_d\operatorname{tr}_{\mathrm{sym}}(q_1q_2q_3)\).
We localize each row before expanding its score.

For each \(i\), let \(a^{(i)}=(A_{ij})_{j\ne i}\) be the \(i\)-th row with
its diagonal entry removed, and define its weighted feature sum by
\[
  H_i:=\sum_{j\ne i}A_{ij}q_j.
\]
Since \(A_{ii}=0\), we have \(\|a^{(i)}\|_2=\|A_{i\cdot}\|_2\).
Choose the local tail level
\[
  T_{\mathrm{loc}}
  :=K_1\left(\frac{\|A\|_{\HS}^2}{N}+1\right)\log n.
\]
When \(\|A_{i\cdot}\|_2>0\), set
\(k_{i,\mathrm{ad}}:=k_{\mathrm{ad}}(a^{(i)},T_{\mathrm{loc}})\) using the threshold
from Lemma~\ref{lem:adaptive-threshold} with covariance bound
\(\Lambda+1\).  When \(\|A_{i\cdot}\|_2=0\), set \(k_{i,\mathrm{ad}}:=1\).  In either
case, define
\[
  \mathcal G_i=
  \{\|H_i\|_{\HS}^2\le k_{i,\mathrm{ad}}\|A_{i\cdot}\|_2^2\}.
\]
Then, uniformly in \(i\) and \(t\in[0,1]\),
\begin{equation}
  \mathbb P_{\mu_t^{\otimes(n-1)}}\left\{
  \mathcal C_{-i}(\Lambda+1)\cap\mathcal G_i^c
  \right\}\le e^{-T_{\mathrm{loc}}}.
  \label{eq:bad-row-probability}
\end{equation}
\begin{lemma}[One-coordinate score expansion]
\label{lem:cavity-score-expansion}
For all sufficiently large \(n\), on \(\mathcal G_i\), uniformly in
\(0\le t\le1\),
\[
  \rho_i(t)=T_i+\varepsilon_i(t),
\]
where
\begin{equation}
  T_i=
  \sum_{\substack{j<k<\ell\\j,k,\ell\ne i}}
  A_{ij}A_{ik}A_{i\ell}\mathcal K_3(q_j,q_k,q_\ell)
  -\frac13\sum_{j\ne i}A_{ij}^3\mathcal K_3(q_j,q_j,q_j),
  \label{eq:local-cubic-source}
\end{equation}
and
\begin{equation}
  |\varepsilon_i(t)|
  \le
  C\left(
    \frac{k_{i,\mathrm{ad}}\|A_{i\cdot}\|_2^2}{N}
    +\frac{\|A\|_{\max}\|A_{i\cdot}\|_2^2}{N^{3/2}}
  \right)^2.
  \label{eq:score-error}
\end{equation}
The pointwise expansion and \eqref{eq:score-error} also hold for any
deterministic row thresholds \(k_{i,\mathrm{ad}}\ge1\) satisfying
\begin{equation}
  \max_i\frac{k_{i,\mathrm{ad}}\|A_{i\cdot}\|_2^2}{N}\le c_\ast,
  \qquad
  \max_i\frac{k_{i,\mathrm{ad}}\|A_{i\cdot}\|_2^2}{d}\le c_\ast,
  \label{eq:local-score-finite-conditions}
\end{equation}
with an absolute, sufficiently small \(c_\ast>0\), using the corresponding
events \(\mathcal G_i\). Extend \(\varepsilon_i(t)\) by zero on
\(\mathcal G_i^c\). For the dense-regime parameters of this section,
\[
  \sum_i\sup_t|\varepsilon_i(t)|=o(1).
\]
\end{lemma}

To see the cancellation underlying the score expansion, note that
integration against \(\eta\) extracts the odd part of the local weight.
Since \(\mathbb E_{\mu,i}D_i=0\), its leading term is
\(\mathbb E_{\mu,i}(D_i^3/6-D_iQ_i/2)\). Put $a_j=A_{ij}$ and $B_j=B_{ij}$, with indices
different from $i$.  Expanding the two products cancels all terms with
exactly two equal indices:
\[
  \frac16\left(\sum_j a_jB_j\right)^3
  -\frac12\left(\sum_j a_jB_j\right)\left(\sum_j a_j^2B_j^2\right)
  =\sum_{j<k<\ell}a_ja_ka_\ell B_jB_kB_\ell-\frac13\sum_j a_j^3B_j^3.
\]
Conditional expectation gives exactly $T_i$ in
\eqref{eq:local-cubic-source}.  
The error bound and division by the positive local denominator are proved
in Section~\ref{sec:two-sided-local-expansion}. In particular, the error
is defined on the entire latent space by zero extension, while the score
expansion is asserted only on \(\mathcal G_i\). The discarded rows are
controlled separately in Section~\ref{sec:differential-inequality}.

\subsection{Aggregation of the cubic scores}
\label{sec:full-row-aggregation}

After the local expansion, the leading score is the complete
sum \(\sum_iT_i\), still indexed by the deleted vertex \(i\).  Reordering its
terms indexed by pairwise distinct
$j,k,\ell$ converts them into a canonical weighted $U$-statistic of order
three.  Lemma~\ref{lem:full-row-coefficient-norms} records the coefficient
norms needed to apply the $U$-statistic moment inequality to that statistic.

Define
\begin{equation}
  g_{jk\ell}
  =
  \sum_{i=1}^nA_{ij}A_{ik}A_{i\ell},
  \qquad
  \mathcal T_g
  =
  \sum_{j<k<\ell}
  g_{jk\ell}\mathcal K_3(q_j,q_k,q_\ell).
  \label{eq:full-row-tensor}
\end{equation}
Because $A_{ii}=0$, summing \eqref{eq:local-cubic-source} over the deleted
index gives the exact identity
\begin{equation}
  \sum_iT_i
  =\mathcal T_g
  -\frac13\sum_{i\ne j}A_{ij}^3
  \mathcal K_3(q_j,q_j,q_j).
  \label{eq:full-row-aggregation}
\end{equation}

The kernel is canonical under each \(\mu_t\), since
\(\mathbb E_{\mu_t}q=0\). We use the exponent
\(p_{\mathrm H}=2\vee K_{\mathrm H}\|A\|_{\HS}^2/N\), with
\(K_{\mathrm H}\) a sufficiently large absolute constant. The coefficient
and kernel bounds in Appendix~\ref{app:complete-buffered-cavity}, applied
to Adamczak's inequality in Section~\ref{sec:adamczak-exponent-audit}, give
\begin{equation}
 \sup_{0\le t\le1}\|\mathcal T_g\|_{L^{p_{\mathrm H}}(\mu_t^{\otimes n})}=o(1),
 \qquad N^{-3/2}\sum_{i,j}|A_{ij}|^3=o(1).
 \label{eq:dense-master-odd-ledger}
\end{equation}
These are \eqref{eq:Tg-Lp-small} and \eqref{eq:repeated-index-small}.
The full coefficient array must be aggregated before taking its norm;
the mixed three-block estimate in Lemma~\ref{lem:mixed-cubic-singletons}
uses the operator norm of \(A\) to control this term. The nineteen scale
substitutions are recorded in Table~S3.1
of the Supplementary Material.
We now transfer these unweighted estimates to the likelihood weight.

\subsection{Differential inequality}
\label{sec:differential-inequality}

We combine the unweighted moment bounds with a spectral second moment of
the likelihood weight.  Equation~\eqref{eq:dense-recursion-smallness}
verifies the hypotheses of Theorem~\ref{thm:l2-spectral-recursion} for
$2A$ at cutoff $\Lambda+1$.  Since
$2K_A\le2\sum_{i<j}A_{ij}B_{ij}$, this gives
\begin{equation}
  Z^{(2)}(t):=\mathbb E_{\mu_t^{\otimes n}}
  [\mathbf1_{\mathcal C_n(\Lambda+1)}e^{2K_A}]
  \le\exp\{C\|A\|_{\HS}^2/N\}.
  \label{eq:weighted-second-moment}
\end{equation}
Recall $p_{\mathrm H}=2\vee K_{\mathrm H}\|A\|_{\HS}^2/N$ and set
$r=p_{\mathrm H}/(p_{\mathrm H}-1)$.  With expectations under
$\mu_t^{\otimes n}$, H\"older's inequality and interpolation give
\begin{equation}
\begin{aligned}
  \mathbb E[\mathbf1_{\mathcal C_n(\Lambda+1)}e^{K_A}|Y|]
  &\le\bigl(\mathbb E[\mathbf1_{\mathcal C_n(\Lambda+1)}e^{rK_A}]
  \bigr)^{1/r}\|Y\|_{p_{\mathrm H}}\\
  &\le Z(t)^{1-2/p_{\mathrm H}}Z^{(2)}(t)^{1/p_{\mathrm H}}\|Y\|_{p_{\mathrm H}}
  \le C\max\{1,Z(t)\}\|Y\|_{p_{\mathrm H}}.
\end{aligned}
  \label{eq:weighted-holder-closure}
\end{equation}
Here $Z^{(2)}(t)^{1/p_{\mathrm H}}\le e^{C/K_{\mathrm H}}$ and
$0\le1-2/p_{\mathrm H}\le1$.

To remove the exceptional rows, define
\[
  \mathcal E_{\mathrm{bad}}(t):=\mathbb E\left[
  \mathbf1_{\mathcal C_n(\Lambda+1)}e^{K_A}
  \sum_i\mathbf1_{\mathcal G_i^c}\{\rho_i(t)-T_i\}\right].
\]
We first keep the dependence on the row-tail level explicit. Under the
spectral hypotheses for \(2A\) at cutoff \(\Lambda+1\),
\eqref{eq:bad-row-probability} implies
\begin{equation}
\begin{aligned}
 |\mathcal E_{\mathrm{bad}}(t)|
 \le{}&Cn\exp\{C\|A\|_{\HS}^2/N-T_{\mathrm{loc}}/2\}\\
 &+Cn\left(\frac nN\right)^{3/2}R_2(A)^3
 e^{-T_{\mathrm{loc}}/(2p_{\mathrm H})}\max\{1,Z(t)\}.
\end{aligned}
 \label{eq:bad-row-finite-bound}
\end{equation}
Indeed, \(|\rho_i(t)|Z_i^{\mathrm{loc}}(t)=|O_i|
\le\mathbb E_{\bar\mu,i}e^{D_i-Q_i/2}\). Conditional integration and
Cauchy--Schwarz, using \eqref{eq:covariance-remainder-second-moment},
bound the score contribution by
\[
 n\exp\left\{\frac{C\|A\|_{\HS}^2+C(\Lambda+1)R_2(A)^2}{N}
                  -\frac{T_{\mathrm{loc}}}{2}\right\}.
\]
The spectral row hypothesis and \(d^2\asymp N\) bound the cutoff term
in this exponent by an absolute constant. The cubic contribution follows
from \eqref{eq:bad-row-cubic-Lp} and \eqref{eq:weighted-holder-closure}.
This proves \eqref{eq:bad-row-finite-bound} without an asymptotic assumption.

For the present parameters, \eqref{eq:environment-FR} bounds the cubic
prefactor by \(C|\beta|^6n^4/N^3\le Cn^4\), while
\(T_{\mathrm{loc}}/p_{\mathrm H}\ge c_{K_{\mathrm H}}K_1\log n\).
Taking \(K_1\) sufficiently large therefore gives
\begin{equation}
  |\mathcal E_{\mathrm{bad}}(t)|\le o(1)\max\{1,Z(t)\}.
  \label{eq:bad-row-weighted-error}
\end{equation}

Combining Lemmas~\ref{lem:buffered-cavity-complete} and
\ref{lem:cavity-score-expansion} with \eqref{eq:full-row-aggregation} yields
\[
\begin{aligned}
  |Z'(t)|\le{}&|\operatorname{Rem}_{\mathrm{cov}}(t)|
  +|\mathcal E_{\mathrm{bad}}(t)|
  +\mathbb E[\mathbf1_{\mathcal C_n(\Lambda+1)}e^{K_A}|\mathcal T_g|]\\
  &+Z(t)\left\{\sum_i\sup_{0\le s\le1}|\varepsilon_i(s)|
  +CN^{-3/2}\sum_{i,j}|A_{ij}|^3\right\}.
\end{aligned}
\]
The third term is $o(1)\max\{1,Z(t)\}$ by
\eqref{eq:weighted-holder-closure} and \eqref{eq:Tg-Lp-small}; the braces
are $o(1)$ by Lemma~\ref{lem:cavity-score-expansion} and
\eqref{eq:repeated-index-small}.  The other two terms have already been
bounded, so uniformly in $0\le t\le1$,
\begin{equation}
  |Z'(t)|\le o(1)\max\{1,Z(t)\}.
  \label{eq:Z-differential}
\end{equation}
Proposition~\ref{prop:symmetric-endpoint} gives $Z(0)\le1+o(1)$.
Writing the uniform error above as $\delta_n\to0$ and using
$\max\{1,Z(t)\}\le1+Z(t)$, Gronwall gives
$1+Z(t)\le(1+Z(0))e^{\delta_nt}$.  Hence
\begin{equation}
  Z(1)\le1+o(1).
  \label{eq:Z-one}
\end{equation}

\subsection{Return to the exact likelihood}
\label{sec:return-simple-edge-likelihood}

The symmetric-endpoint and interpolation bounds are now available.  It remains to record the quantitative consequences of the event
\(\mathcal E_n\) defined at the beginning of
Section~\ref{sec:dense-likelihood}, transfer the bound on \(Z(1)\) to
the exact product in \eqref{eq:pure-negative-conditional-target}, and remove
the truncation.

\begin{proof}[Proof of Proposition~\ref{prop:dense-branch-second-moment}]

Recall the matrix bounds of Proposition~\ref{prop:row-gram-environment} and write
\(\pi_n=\mathbb P(q\in\mathcal E_n)=1-o(1)\).

For completeness, the auxiliary localization events used inside the two
conditional arguments are different from this outer truncation.  Adaptive
deletion uses the hollow--Schur recursion in \eqref{eq:hollow-schur} and the
event \(\mathcal G^{\mathrm{rec}}\) in
\eqref{eq:recursive-good-event}.  Coordinate interpolation uses the row events
\(\mathcal G_i\) defined before \eqref{eq:bad-row-probability}.  Thus every
cutoff appearing in the proof is fixed before the final likelihood comparison.

The estimates assembled in \eqref{eq:Z-one} give \(Z(1)\le1+o(1)\).
The useful scale identities are
\begin{equation}
  \frac{\beta^6n^3}{N^4}=h_n^{-4},\qquad
  \frac{|\beta|^6n^2}{N^3}
  =\frac{|\beta|^{3/2}}{h_n^3n^{1/4}},\qquad
  n\beta^2\gtrsim h_n^4.
  \label{eq:dense-master-scale-identities}
\end{equation}
They follow from the definition of \(h_n\) and
\eqref{eq:dense-master-regime}.

By Lemma~\ref{lem:truncated-likelihood}, the exact truncated second moment for
two independent replicas is
\[
  \frac1{\pi_n^2}
  \mathbb E_{q^\ast,q}\left[
  \mathbf1_{\{q^\ast\in\mathcal E_n\}}
  \mathbf1_{\{q\in\mathcal E_n\}}
  \prod_{i<j}(1+A_{ij}B_{ij})
  \right].
\]
The sixth-row-moment conclusion of
Proposition~\ref{prop:row-gram-environment}, applied to each replica, gives on
\(q^\ast,q\in\mathcal E_n\)
\[
  \sum_{i<j}|B_{ij}^\ast|^6\lesssim\frac{n^2}{N^3},
  \qquad
  \sum_{i<j}|B_{ij}|^6\lesssim\frac{n^2}{N^3}.
\]
Hence Cauchy--Schwarz gives, uniformly for \(q^\ast,q\in\mathcal E_n\),
\[
  \sum_{i<j}|A_{ij}B_{ij}|^3
  \le
  |\beta|^6
  \left(\sum_{i<j}|B_{ij}^\ast|^6\right)^{1/2}
  \left(\sum_{i<j}|B_{ij}|^6\right)^{1/2}
  \lesssim\frac{|\beta|^6n^2}{N^3}=o(1).
\]
Here the last convergence follows from
\(|\beta|^6n^2/N^3=|\beta|^{3/2}/(h_n^3n^{1/4})\) and
the definition of \(h_n\).
If one factor in the product below is zero, the desired comparison is
immediate.  Otherwise every factor is positive, and \(|A_{ij}B_{ij}|\le1\)
because \(|\beta|\le1\) and both \(B_{ij}^\ast\) and \(B_{ij}\) are inner
products of unit Hilbert--Schmidt vectors.  We may therefore apply
\(\log(1+x)\le x-x^2/2+C|x|^3\), valid for \(-1<x\le1\), to each
\(x=A_{ij}B_{ij}\).  Summing over \(i<j\), exponentiating, and using
the $o(1)$ cubic remainder above gives
\[
  \prod_{i<j}(1+A_{ij}B_{ij})
  \le e^{K_A+o(1)}
\]
uniformly for \(q^\ast,q\in\mathcal E_n\).  Since
\(q\in\mathcal E_n\) implies \(q\in\mathcal C_n(\Lambda+1)\), this gives the pointwise
comparison
\[
  \mathbf1_{\{q\in\mathcal E_n\}}
  \prod_{i<j}(1+A_{ij}B_{ij})
  \le e^{o(1)}
  \mathbf1_{\mathcal C_n(\Lambda+1)}e^{K_A}.
\]
The error is uniform over \(q^\ast\in\mathcal E_n\).  Taking expectation over
the second replica and using \eqref{eq:Z-one} proves
\eqref{eq:pure-negative-conditional-target}, and hence proves
Proposition~\ref{prop:dense-branch-second-moment}.
\end{proof}

Together with Lemma~\ref{lem:truncated-likelihood}, the probability bound
\(\pi_n=1-o(1)\) converts Proposition~\ref{prop:dense-branch-second-moment}
into a vanishing truncated \(\chi^2\)-divergence.  This is the dense-branch
input used in the main-text proof of Proposition~\ref{prop:pure-negative}.
\section{The pure quadratic critical window}
\label{sec:critical-window}

This section proves part~(i) of Theorem~\ref{thm:growing-critical-main}. On the scale
\(N\asymp n^{3/4}|\beta|^{3/2}\) with \(N\to\infty\), the normalized
signed-triangle mean remains of constant order.  Identifying the testing
distance therefore requires an approximation to the full likelihood ratio.
We first isolate the triangle contribution to the exact second moment,
then prove the two-model CLT and identify the limiting likelihood.
Section~\ref{sec:quantitative-main} gives the finite-sample bounds.

\subsection{Replica coefficients and the truncation input}
\label{subsec:growing-critical}

Recall the signed triangle and put
\[
  T_n=\sum_{i<j<k}Y_{ij}Y_{jk}Y_{ki},
  \qquad
  \Delta_n=\frac{\sqrt{\binom n3}\,|\beta|^3}{N^2},
  \qquad
  W_n=\operatorname{sgn}(\beta)\frac{T_n}{\sqrt{\binom n3}}.
\]

The parameter \(\Delta_n\) is the mean of the signed triangle normalized by
its null standard deviation.  It diverges in the detectable strict regime
and vanishes in the indistinguishable strict regime.  The results below concern
sequences for which \(\Delta_n\) remains bounded above and away from zero.  On
such a scale, a central limit theorem for \(W_n\)
describes the signed-triangle statistic but does not by itself determine the
total variation between the full graph laws.  One must additionally show that
the likelihood ratio contains no other contribution of constant order.
Proposition~\ref{prop:critical-likelihood-representation} supplies this
triangle-dominance step through an asymptotic expansion of the conditioned
likelihood ratio.  The expansion then proves part~(i) of
Theorem~\ref{thm:growing-critical-main}.

Throughout this section, \(N\to\infty\) and \(\Delta_n\to\Delta\in(0,\infty)\).
Choose fixed \(0<a<b<\infty\) so that \(a\le\Delta_n\le b\) eventually.
Since $\Delta_n^2=\binom n3\beta^6/N^4$ and $\binom n3\asymp n^3$,
\begin{equation}
  n\beta^2\asymp N^{4/3}.
  \label{eq:critical-scaling-relation}
\end{equation}
In particular, $N\lesssim n^{3/4}$ because $|\beta|\le1$.

Write \(L_n^{\mathcal E}=d\Pp_n^{\mathcal E_n}/d\Qq_n\).
The proof separates the finite graph calculation from the uniform bound
that permits summing it over all edge sets.

For a prior event \(\mathcal E_n\subseteq V^n\), put
\(\pi_n=\mathbb P(q\in\mathcal E_n)\).  We call the following real
polynomial the truncated replica generating function:
\begin{equation}
  \mathcal Z_n^{\mathcal E}(s)
  :=\mathbb E\left[
    \mathbf1_{\mathcal E_n}(q^\ast)
    \mathbf1_{\mathcal E_n}(q)
    \prod_{i<j}\left\{1+s\beta^2
      \langle q_i^\ast,q_j^\ast\rangle_{\HS}
      \langle q_i,q_j\rangle_{\HS}\right\}
  \right],
  \label{eq:quadratic-generating-function}
\end{equation}
where \(q^\ast\) and \(q\) are independent samples from \(\mu^{\otimes n}\).
At \(s=1\), equation~\eqref{eq:truncated-second-moment} becomes
\begin{equation}
  \mathbb E_{\Qq_n}(L_n^{\mathcal E})^2
  =\pi_n^{-2}\mathcal Z_n^{\mathcal E}(1).
  \label{eq:quadratic-replica-L2}
\end{equation}
Let \(K_n\) be the complete graph on \([n]\).  For an edge set
\(F\subseteq E(K_n)\), write \(|F|\) for its cardinality and set
\[
  G_F(q):=\mathbf1_{\mathcal E_n}(q)
  \prod_{\{i,j\}\in F}\langle q_i,q_j\rangle_{\HS}.
\]
For a polynomial \(f\), write \([s^k]f(s)\) for the coefficient of \(s^k\).
With this notation, Lemma~\ref{lem:nonnegative-replica-coefficients}
identifies each coefficient \(\mathfrak c_{k,n}\) as a sum of squares.

\begin{lemma}[Nonnegative replica coefficients]
\label{lem:nonnegative-replica-coefficients}
For every \(0\le k\le\binom n2\),
\begin{equation}
  \mathfrak c_{k,n}:=[s^k]\mathcal Z_n^{\mathcal E}(s)
  =\beta^{2k}\sum_{\substack{F\subseteq E(K_n)\\ |F|=k}}
  \bigl\{\mathbb E G_F(q)\bigr\}^{2}\ge0.
  \label{eq:quadratic-edge-set-expansion}
\end{equation}
\end{lemma}

\begin{proof}
Expand the product over edge sets and use independence of the two replicas:
\[
  \mathcal Z_n^{\mathcal E}(s)
  =\sum_{F\subseteq E(K_n)}(s\beta^2)^{|F|}
  \mathbb E[G_F(q^\ast)G_F(q)]
  =\sum_{F\subseteq E(K_n)}(s\beta^2)^{|F|}\{\mathbb E G_F(q)\}^2.
\]
Extracting $[s^k]$ gives the claimed nonnegative sum of squares.
\end{proof}

Lemma~\ref{lem:critical-generating-truncation}, proved in
Appendix~\ref{app:critical-likelihood-analysis}, supplies events
\(\mathcal E_n\) with \(\mathbb P(\mathcal E_n^c)\le n^{-M_n}\),
where \(M_n\to\infty\), and constants \(s_\star>1\), \(C_{a,b}<\infty\)
such that \(\sup_{0\le s\le s_\star}\mathcal Z_n^{\mathcal E}(s)\le C_{a,b}\).
Its proof reuses the finite-parameter deletion and interpolation bounds
with a critical covariance cutoff. The concentration and moment-transfer
checks are independent of the graph calculation below.

\subsection{Connected graph activities and the exact second moment}
\label{sec:critical-graph-expansion}

We identify the fixed-order coefficients and control their tail using the
common truncation. Since $\pi_n\to1$, \eqref{eq:quadratic-replica-L2}
reduces Proposition~\ref{prop:exact-critical-second-moment} to
\begin{equation}
\begin{gathered}
  \mathfrak c_{k,n}\longrightarrow[s^k]e^{\Delta^2s^3}
  \quad\text{for every fixed }k,\\
  \lim_{K\to\infty}\limsup_{n\to\infty}\sum_{k>K}\mathfrak c_{k,n}=0.
\end{gathered}
  \label{eq:critical-coefficient-roadmap}
\end{equation}
The second assertion follows at once from nonnegativity
(Lemma~\ref{lem:nonnegative-replica-coefficients}) and
Lemma~\ref{lem:critical-generating-truncation}:
\[
  0\le\sum_{k>K}\mathfrak c_{k,n}
  \le s_\star^{-K}\mathcal Z_n^{\mathcal E}(s_\star)
  \le C_{a,b}s_\star^{-K}.
\]
For the first assertion, remove the truncation indicators and put
\begin{equation}
  \mathfrak c_{k,n}^{\circ}
  :=\beta^{2k}\sum_{\substack{F\subseteq E(K_n)\\|F|=k}}
  \left\{\mathbb E\prod_{\{i,j\}\in F}
                 \langle q_i,q_j\rangle_{\HS}\right\}^2.
  \label{eq:quadratic-unconditioned-coefficient}
\end{equation}
For a finite simple graph $H$ without isolated vertices, define
\begin{equation}
  t_d(H):=\mathbb E\prod_{\{u,v\}\in E(H)}\langle q_u,q_v\rangle_{\HS},
  \qquad q_u\stackrel{\mathrm{iid}}\sim\mu.
  \label{eq:quadratic-graph-moment}
\end{equation}
Let $H_F$ be the graph with edge set $F$ after deleting isolated vertices.
Relabeling the iid features gives
\begin{equation}
  \mathbb E\prod_{\{i,j\}\in F}\langle q_i,q_j\rangle_{\HS}=t_d(H)
  \quad\text{if }H_F\cong H.
  \label{eq:quadratic-copy-moment}
\end{equation}
Write $(n)_0=1$ and $(n)_v=n(n-1)\cdots(n-v+1)$ for $v\ge1$, with
$(n)_v=0$ for $v>n$. Dividing injective vertex maps by the order of the
automorphism group gives
\begin{equation}
  \#\{F\subseteq E(K_n):H_F\cong H\}
  =\frac{(n)_{v(H)}}{|\operatorname{Aut}(H)|}.
  \label{eq:quadratic-copy-count}
\end{equation}
Thus, with the \emph{labelled activity}
$\mathsf{Act}_n(H):=(n)_{v(H)}\beta^{2e(H)}t_d(H)^2/|\operatorname{Aut}(H)|$,
grouping \eqref{eq:quadratic-unconditioned-coefficient} by isomorphism type yields
\begin{equation}
  \mathfrak c_{k,n}^{\circ}
  =\sum_{\substack{H:\ e(H)=k,\ v(H)\le n}}\mathsf{Act}_n(H),\qquad k\ge1.
  \label{eq:quadratic-activity-decomposition}
\end{equation}
The sum runs over simple graphs without isolated vertices. In particular,
the cycle $C_3$ has $(n)_3/6=\binom n3$ copies.

To reduce $t_d(H)$, extend \eqref{eq:quadratic-graph-moment} to multigraphs,
counting parallel edges with multiplicity. Loops contribute $\|q_u\|_{\HS}^2=1$
and are deleted. Isotropy gives
\begin{equation}
  \int\langle x,q\rangle_{\HS}\langle q,y\rangle_{\HS}\,d\mu(q)
  =N^{-1}\langle x,y\rangle_{\HS}.
  \label{eq:quadratic-degree-two-contraction}
\end{equation}
Suppressing a degree-two vertex $w$ with incident edges $\{u,w\},\{w,v\}$
means replacing these edges by $\{u,v\}$ and deleting $w$, any loop,
and any resulting isolated vertex. For the resulting graph $G'$, integration
in $q_w$ gives
\begin{equation}
  t_d(G)=N^{-1}t_d(G').
  \label{eq:quadratic-one-suppression}
\end{equation}
Starting from $G_0=H$, fix any suppression order. Stop after $h$ steps
when $G_h$ has a degree-one vertex, has no edges, or has minimum degree
at least three; in the last case write $H_{\mathrm c}=G_h$.
Iteration and centering at a degree-one vertex yield
\begin{equation}
  t_d(H)=
  \begin{cases}
    0,&\text{if $G_h$ has a degree-one vertex},\\
    N^{-h},&\text{if $G_h$ has no edges},\\
    N^{-h}t_d(H_{\mathrm c}),&\text{if $G_h=H_{\mathrm c}\ne\varnothing$}.
  \end{cases}
  \label{eq:quadratic-reduction-outcomes}
\end{equation}
For example, $K_4$ is unchanged, while the union of three distinct internally
vertex-disjoint paths with common endpoints reduces to three parallel edges. A cycle $C_v$
reduces to the empty edge set.

\begin{lemma}[Connected graph activities]
\label{lem:quadratic-connected-activities}
Let \(H\) be connected, with \(v=v(H)\) vertices and \(k=e(H)\) edges.
If \(H\), or a graph obtained from it by the preceding reduction, has a vertex
of degree one, then \(t_d(H)=0\).  If \(H=C_v\) is a cycle, then
\begin{equation}
  t_d(C_v)=N^{1-v}.
  \label{eq:quadratic-cycle-moment}
\end{equation}
If the reduction leaves \(H_{\mathrm c}\ne\varnothing\), then
\begin{equation}
  |t_d(H)|=N^{-h}|t_d(H_{\mathrm c})|
  \le \exp\{Ck\log(ek)\}
       N^{-h-e(H_{\mathrm c})/2},
  \label{eq:quadratic-core-moment}
\end{equation}
In particular, for every \(0<C_{\mathrm{sc}}<\infty\), there is a constant
\(C_{\mathrm{act}}<\infty\) such that, whenever
\(n\beta^2\le C_{\mathrm{sc}}N^{4/3}\), every connected graph of minimum
degree at least two other than a triangle satisfies
\begin{equation}
  \mathsf{Act}_n(H)
  \le (n)_v\beta^{2k}t_d(H)^2
  \le \exp\{C_{\mathrm{act}}k\log(ek)\}
       N^{-\sqrt{2k}/6}.
  \label{eq:quadratic-noncycle-activity}
\end{equation}
For cycles of length \(v\ge4\),
\begin{equation}
  \mathsf{Act}_n(C_v)
  \le (n)_v\beta^{2v}t_d(C_v)^2
  \le C_{\mathrm{sc}}^vN^{2-2v/3}.
  \label{eq:quadratic-long-cycle-activity}
\end{equation}
\end{lemma}

\begin{proof}
We treat the three reduction outcomes, then convert the moment bounds to
activity bounds. The degree-one case follows from
\eqref{eq:quadratic-reduction-outcomes}.

If the reduction removes all edges, suppression preserves connectedness
while an edge remains. Only the last step, from two parallel edges to a
deleted loop, also discards an isolated vertex. Thus $h=v-1$ and
$t_d(H)=N^{1-v}$. This proves \eqref{eq:quadratic-cycle-moment} for cycles.
If $H$ is not a triangle, its minimum degree and simplicity give
$k\ge v\ge4$ and $v\ge\sqrt{2k}$. Hence
\[
  (n)_v\beta^{2k}t_d(H)^2
  \le(n\beta^2)^vN^{2-2v}
  \le\max\{1,C_{\mathrm{sc}}\}^kN^{2-2v/3}
  \le e^{Ck}N^{-\sqrt{2k}/6},
\]
where $2-2v/3\le-v/6$ for $v\ge4$.

Suppose now that $H_{\mathrm c}\ne\varnothing$. For independent
$q_1,q_2\sim\mu$, condition on $q_1$ and apply
\eqref{eq:local-quadratic-tail} with $H=q_1$. Since
$\|q_1\|_{\HS}=1$ and $\|q_1\|_{\mathrm{op}}\le1$, layer-cake integration
of that tail bound gives
\[
  \|\langle q_1,q_2\rangle_{\HS}\|_{L^p}
  \le C\left(\sqrt{\frac pN}+\frac pd\right)
  \le\frac{Cp}{\sqrt N},\qquad p\ge2.
\]
Here the last inequality uses $N\asymp d^2$. The loop-free core has
$3\le e(H_{\mathrm c})\le k$ edges, counted with multiplicity.
H\"older's inequality with exponent $e(H_{\mathrm c})$ for every edge
factor therefore gives
\[
\begin{aligned}
  |t_d(H_{\mathrm c})|
  &\le\prod_{\{u,v\}\in E(H_{\mathrm c})}
          \|\langle q_u,q_v\rangle_{\HS}\|_{L^{e(H_{\mathrm c})}}\\
  &\le(Ck)^{e(H_{\mathrm c})}N^{-e(H_{\mathrm c})/2}
  \le e^{Ck\log(ek)}N^{-e(H_{\mathrm c})/2},
\end{aligned}
\]
including when parallel edges yield repeated factors.
Together with \eqref{eq:quadratic-reduction-outcomes}, this proves
\eqref{eq:quadratic-core-moment}.

Write $v_{\mathrm c}=v(H_{\mathrm c})$. No isolated vertex is discarded
in this case, so $v=h+v_{\mathrm c}$. The degree-sum identity gives
$e(H_{\mathrm c})\ge3v_{\mathrm c}/2$ and $k\ge v$.
Thus \eqref{eq:quadratic-core-moment} yields
\[
\begin{aligned}
  (n)_v\beta^{2k}t_d(H)^2
  &\le e^{Ck\log(ek)}N^{4v/3-2h-e(H_{\mathrm c})}\\
  &\le e^{Ck\log(ek)}N^{-2h/3-v_{\mathrm c}/6}
   \le e^{Ck\log(ek)}N^{-\sqrt{2k}/6}.
\end{aligned}
\]
Here $C_{\mathrm{sc}}^v$ is absorbed into the exponential, and the last
inequality uses $2h/3+v_{\mathrm c}/6\ge v/6\ge\sqrt{2k}/6$.
This proves \eqref{eq:quadratic-noncycle-activity} in all cases.
Finally, the cycle identity gives directly
\[
  \mathsf{Act}_n(C_v)\le(n\beta^2)^vN^{2-2v}
  \le C_{\mathrm{sc}}^vN^{2-2v/3},\qquad v\ge4,
\]
which is \eqref{eq:quadratic-long-cycle-activity}.
\end{proof}

Only triangle components can retain nonzero activity at fixed order.
We combine this fact with the generating-function tail bound to obtain the
exact second moment.

\begin{proposition}[Exact critical truncated second moment]
\label{prop:exact-critical-second-moment}
Suppose \(N\to\infty\) and
\(\Delta_n\to\Delta\in(0,\infty)\).  There are
prior events \(\mathcal E_n\) with
\(\mathbb P(q\in\mathcal E_n)=1-o(1)\) such that
\begin{equation}
  \mathbb E_{\Qq_n}(L_n^{\mathcal E})^2
  \longrightarrow e^{\Delta^2}.
  \label{eq:exact-critical-second-moment}
\end{equation}
\end{proposition}

\begin{proof}
Take the events from Lemma~\ref{lem:critical-generating-truncation} and
fix $k\ge0$. Each inner product is bounded by one in absolute value,
so removing an indicator changes an edge-set moment by at most
$\mathbb P(\mathcal E_n^c)\le n^{-M_n}$. There are at most $n^{2k}$
edge sets, and $|x^2-y^2|\le2|x-y|$ for $|x|,|y|\le1$. Hence
\begin{equation}
  |\mathfrak c_{k,n}-\mathfrak c_{k,n}^{\circ}|
  \le2n^{2k-M_n}=o(1)\quad\text{for every fixed }k.
  \label{eq:quadratic-conditioning-fixed-coefficient}
\end{equation}
In \eqref{eq:quadratic-activity-decomposition}, the moment factors over
connected components, and the number of disjoint embeddings is at most
the product of their separate embedding counts. Thus a graph with a
nontriangle component has activity $o(1)$ by
Lemma~\ref{lem:quadratic-connected-activities}, since triangle activity is
$\binom n3\beta^6N^{-4}=\Delta_n^2=O(1)$.
There are finitely many graph types for fixed $k$, so only disjoint unions
of triangles contribute in the limit. For $k=3m$, their contribution is
\[
  \frac{(n)_{3m}}{m!(3!)^m}\beta^{6m}N^{-4m}
  =\frac{\Delta_n^{2m}}{m!}\{1+o(1)\}
  \longrightarrow\frac{\Delta^{2m}}{m!};
\]
it is zero when $3\nmid k$. The empty graph gives the same formula for
$k=m=0$. This proves the coefficient convergence in
\eqref{eq:critical-coefficient-roadmap}.

The tail bound preceding \eqref{eq:quadratic-unconditioned-coefficient},
together with the nonnegative exponential series, gives
\[
  \sum_{k>K}\mathfrak c_{k,n}
  +\sum_{3m>K}\frac{\Delta^{2m}}{m!}
  \le s_\star^{-K}\{C_{a,b}+e^{\Delta^2s_\star^3}\}.
\]
Taking $n\to\infty$ for the first $K$ coefficients and then $K\to\infty$
yields $\mathcal Z_n^{\mathcal E}(1)\to e^{\Delta^2}$.
Equation~\eqref{eq:quadratic-replica-L2} and $\pi_n\to1$ finish the proof.
\end{proof}

\subsection{Signed-triangle limits under the two graph laws}
\label{sec:critical-two-model-triangle}

The exact second moment is now available. We next identify the Gaussian
shift of the signed triangle. The overlap estimates needed below are
proved in Lemma~\ref{lem:triangle-overlap-moments}.

\begin{lemma}[Two-model signed-triangle CLT]
\label{lem:triangle-two-model-CLT}
Let
\[
  W_n=\operatorname{sgn}(\beta)\frac1{\sqrt{\binom{n}{3}}}
  \sum_{i<j<k}Y_{ij}Y_{jk}Y_{ki}.
\]
Then
\[
  W_n\Rightarrow \Normal(0,1)
  \quad\text{under }\Qq_n
\]
and
\[
  W_n-\Delta_n\Rightarrow \Normal(0,1)
  \quad\text{under }\Pp_n.
\]
The second conclusion is unchanged when \(\Pp_n\) is replaced by
\(\Pp_n^{\mathcal E_n}\), where \((\mathcal E_n)\) is the sequence of prior
events constructed in Lemma~\ref{lem:critical-generating-truncation}.
\end{lemma}

\begin{proof}
We first compute the null moments, then isolate the conditionally centered
triangle sum under the alternative. A conditional de Jong argument handles
this sum; explicit second-moment bounds identify the remaining mean shift.

\medskip
\noindent\textit{Step 1. Null moments.}
Let \(\mathsf{Tri}_n=\binom{[n]}3\), with \(E(\mathfrak t)\) the edge set
of a triangle. Under \(\Qq_n\), distinct triangle monomials are orthogonal,
so \(\mathbb E_{\Qq_n}W_n^2=1\). Each edge belongs to \(n-2\) triangles,
and the maximum influence is \((n-2)/\binom n3=O(n^{-2})\).

In a fourth-moment expansion, a singleton edge makes a term vanish.
The remaining four-tuples consist of two equal pairs or the four faces of
a \(K_4\). Indeed, multiplicities \((3,1)\) leave a singleton edge, while
\((2,1,1)\) would require the symmetric difference of two distinct triangles
(at least four edges) to lie in one triangle. For four distinct triangles,
the three edges of one must be covered by the other three; their remaining
edges then force a common fourth vertex. The paired configurations number
\(3\binom n3^2-2\binom n3\), and the \(K_4\) configurations number
\(24\binom n4\). Thus
\[
  \mathbb E_{\Qq_n}W_n^4
  =3-\frac2{\binom n3}+\frac{24\binom n4}{\binom n3^2}\longrightarrow3.
\]
De Jong's theorem \cite{deJong1990} gives the null CLT. The same
classification will control the conditional fourth moment below.

\medskip
\noindent\textit{Step 2. Conditional variance.}
Under \(\Pp_n\), condition on \(q\sim\mu^{\otimes n}\) and put
\(\widetilde Y_{ij}=Y_{ij}-\beta B_{ij}\). These variables are independent,
centered, and have variances \(1-\beta^2B_{ij}^2\). Define
\[
  S_n=\sum_{\mathfrak t\in\mathsf{Tri}_n}
        \prod_{e\in E(\mathfrak t)}\widetilde Y_e.
\]
Expanding the three edges gives
\begin{equation}
\begin{aligned}
 T_n&=S_n+R_{1,n}+R_{2,n}+R_{3,n},\\
 R_{1,n}&=\beta\sum_{\mathfrak t\in\mathsf{Tri}_n}
       \sum_{e\in E(\mathfrak t)}B_e
       \prod_{f\in E(\mathfrak t)\setminus\{e\}}\widetilde Y_f,\\
 R_{2,n}&=\beta^2\sum_{i<j}\widetilde Y_{ij}
                    \sum_{k\ne i,j}B_{ik}B_{jk},\qquad
 R_{3,n}=\beta^3\sum_{i<j<k}B_{ij}B_{jk}B_{ki}.
\end{aligned}
\label{eq:triangle-chaos-decomposition}
\end{equation}
Distinct centered triangle products are conditionally orthogonal, since
some centered edge occurs in exactly one of them. Consequently,
\begin{equation}
 V_n(q):=\mathbb E(S_n^2\mid q)
   =\sum_{\mathfrak t\in\mathsf{Tri}_n}
       \prod_{e\in E(\mathfrak t)}(1-\beta^2B_e^2).
 \label{eq:conditional-leading-variance}
\end{equation}
Using \(1-\prod_e(1-x_e)\le\sum_e x_e\) for \(x_e\in[0,1]\), and
counting the \(n-2\) triangles through each edge, gives
\begin{equation}
 0\le1-\frac{V_n(q)}{\binom n3}
 \le\frac{(n-2)\beta^2\sum_{i<j}B_{ij}^2}{\binom n3},\qquad
 \mathbb E\left[1-\frac{V_n(q)}{\binom n3}\right]\le\frac{3\beta^2}{N}.
 \label{eq:leading-variance-defect}
\end{equation}
Here \(\mathbb EB_{12}^2=N^{-1}\), so \(V_n(q)/\binom n3\to1\) in
probability. Let \(I_e(q)\) be the sum in
\eqref{eq:conditional-leading-variance} restricted to triangles containing
\(e\). On \(V_n(q)>0\),
\begin{equation}
 I_e(q)\le n-2,\qquad
 \max_e\frac{I_e(q)}{V_n(q)}\le\frac{n-2}{V_n(q)}
     =O_{\mathbb P}(n^{-2}).
 \label{eq:conditional-edge-negligibility}
\end{equation}

\medskip
\noindent\textit{Step 3. Conditional fourth moment.}
Remove triangles meeting
\(\mathcal D_n(q)=\{e:1-\beta^2B_e^2<1/2\}\), and denote the retained
sum and its conditional variance by \(S_n^\circ\) and \(V_n^\circ(q)\).
Since \(|\mathcal D_n(q)|\le2\beta^2\sum_eB_e^2\), conditional
orthogonality gives
\begin{equation}
 \frac{\mathbb E[(S_n-S_n^\circ)^2\mid q]}{\binom n3}
 \le\frac{(n-2)|\mathcal D_n(q)|}{\binom n3}=o_{\mathbb P}(1).
 \label{eq:deleted-leading-chaos}
\end{equation}
Thus \(V_n^\circ(q)/\binom n3\to1\) in probability. For retained edges
and triangles, respectively, set
\[
 \xi_e=\frac{\widetilde Y_e}{\sqrt{1-\beta^2B_e^2}},\qquad
 a_{\mathfrak t}=\prod_{e\in E(\mathfrak t)}\sqrt{1-\beta^2B_e^2}.
\]
Conditionally on \(q\), the \(\xi_e\)'s are independent, centered,
variance one, and bounded by \(2\sqrt2\). All sums below are over retained
triangles; then \(S_n^\circ=\sum_{\mathfrak t}a_{\mathfrak t}
\prod_{e\in E(\mathfrak t)}\xi_e\) and
\(V_n^\circ(q)=\sum_{\mathfrak t}a_{\mathfrak t}^2\). Write
\(I_e^\circ(q)=\sum_{\mathfrak t:e\in E(\mathfrak t)}a_{\mathfrak t}^2\).
We have \(I_e^\circ\le I_e\) and \(\sum_eI_e^\circ=3V_n^\circ\).

The classification in Step 1 applies also to these centered edge variables.
Relative to \(3(V_n^\circ)^2\), paired terms require corrections only for
coincident triangles or pairs sharing an edge. Their contributions are
bounded by \(C\sum_{\mathfrak t}a_{\mathfrak t}^4\) and
\(C\sum_e(I_e^\circ)^2\); the \(K_4\) terms contribute at most \(Cn^4\).
Hence
\begin{equation}
 \left|\mathbb E[(S_n^\circ)^4\mid q]-3V_n^\circ(q)^2\right|
 \le C\left\{\sum_{\mathfrak t}a_{\mathfrak t}^4
       +\sum_e I_e^\circ(q)^2+n^4\right\}.
 \label{eq:conditional-fourth-diagram-bound}
\end{equation}
Because \(a_{\mathfrak t}\le1\), the first two sums are at most
\(V_n^\circ\) and \(3V_n^\circ\max_eI_e^\circ\). On \(V_n^\circ>0\),
\[
 \left|\mathbb E\left[\left.
   \left(\frac{S_n^\circ}{\sqrt{V_n^\circ(q)}}\right)^4\right|q\right]-3\right|
 \le C\left\{\frac1{V_n^\circ(q)}+
       \frac{n^4}{V_n^\circ(q)^2}+
       \frac{\max_eI_e^\circ(q)}{V_n^\circ(q)}\right\}
 =o_{\mathbb P}(1).
\]
The standardized sum is defined as zero when \(V_n^\circ=0\), an event
of vanishing probability. From any subsequence, the variance, influence
and fourth-moment bounds admit a further subsequence on which they hold
almost surely. For almost every such latent sequence, de Jong's theorem
\cite{deJong1990} applies to the deterministic conditional array. Thus the
conditional law converges weakly in probability to \(\Normal(0,1)\).
Equations \eqref{eq:leading-variance-defect} and
\eqref{eq:deleted-leading-chaos} restore the normalization and deleted
terms. Taking expectations of bounded continuous test functions yields
\(S_n/\sqrt{\binom n3}\Rightarrow\Normal(0,1)\) under \(\Pp_n\).

\medskip
\noindent\textit{Step 4. Mean-factor remainders.}
Conditional orthogonality of products of two centered edges bounds
\(\mathbb ER_{1,n}^2\) by
\(3\beta^2\binom n3/N\). For \(R_{2,n}\), orthogonality over the outer
edge gives
\[
\begin{aligned}
 \mathbb E(R_{2,n}^2\mid q)
 &\le\beta^4\sum_{i<j}\left(\sum_{k\ne i,j}B_{ik}B_{jk}\right)^2,\\
 \mathbb E\left(\sum_{k\ne i,j}B_{ik}B_{jk}\right)^2
 &=\frac{n-2}{N^2}+\frac{(n-2)(n-3)}{N^3}.
\end{aligned}
\]
The equal-index terms in the last identity have expectation \(N^{-2}\)
by conditioning on \(q_k\). For distinct indices, conditional independence
given \((q_i,q_j)\) and
\(\mathbb E(B_{ik}B_{jk}\mid q_i,q_j)=B_{ij}/N\) give \(N^{-3}\).

Isotropy also gives \(\mathbb E(B_{ij}B_{jk}B_{ki}\mid q_i)=N^{-2}\).
Thus disjoint triangles and triangles sharing only one vertex have zero
covariance in \(R_{3,n}\). Lemma~\ref{lem:triangle-overlap-moments}
controls the \(O(n^3)\) identical pairs and \(O(n^4)\) pairs sharing an
edge. We obtain the finite-sample bounds
\begin{equation}
\begin{aligned}
 \frac{\mathbb ER_{1,n}^2}{\binom n3}&\le\frac{3\beta^2}{N},&
 \frac{\mathbb ER_{2,n}^2}{\binom n3}
 &\le3\beta^4\left(\frac1{N^2}+\frac{n-3}{N^3}\right),\\
 \frac{\operatorname{Var}(R_{3,n})}{\binom n3}
 &\le C\beta^6\left(\frac1{N^3}+\frac n{N^4}\right),&
 \frac{\operatorname{sgn}(\beta)\mathbb ER_{3,n}}{\sqrt{\binom n3}}
 &=\Delta_n.
\end{aligned}
\label{eq:triangle-remainder-moments}
\end{equation}
The three normalized variance bounds tend to zero under
\eqref{eq:critical-scaling-relation}:
\(n\beta^4/N^3\asymp\beta^2/N^{5/3}\) and
\(n\beta^6/N^4\asymp\beta^4/N^{8/3}\). Chebyshev's inequality and
\eqref{eq:triangle-chaos-decomposition} now give the alternative CLT.
Finally, conditioning the prior changes the graph law by at most
\(\mathbb P(\mathcal E_n^c)=o(1)\) in total variation, which preserves
weak convergence.
\end{proof}
\subsection{A Hilbert-space likelihood-projection lemma}
\label{sec:critical-likelihood-projection}

The next lemma shows that matching the limiting second moment leaves no
additional likelihood fluctuation beyond the observed statistic.

\begin{lemma}[Exact second moments identify the limiting likelihood]
\label{lem:hilbert-likelihood-projection}
Let \(P_n\ll Q_n\), \(L_n=dP_n/dQ_n\), and let \(S_n\) be a real-valued
statistic.
Suppose
\[
  S_n\Rightarrow Q_\infty\quad\text{under }Q_n,
  \qquad
  S_n\Rightarrow P_\infty\quad\text{under }P_n,
\]
where \(P_\infty\ll Q_\infty\) with density
\(\ell=dP_\infty/dQ_\infty\in L^2(Q_\infty)\).  Let \(S_\infty\) have law
\(Q_\infty\).  If
\[
  \mathbb E_{Q_n}L_n^2
  \longrightarrow
  \mathbb E_{Q_\infty}\ell^2,
\]
then, under \(Q_n\),
\[
  (S_n,L_n)\Rightarrow(S_\infty,\ell(S_\infty)).
\]
In addition,
\[
  \TV(P_n,Q_n)
  \longrightarrow
  \TV(P_\infty,Q_\infty).
\]
\end{lemma}

\begin{proof}
Weak convergence of \(S_n\) and the second-moment bound make \((S_n,L_n)\)
tight and \(L_n\) uniformly integrable. Let \((S,L)\) be any subsequential
limit under \(Q_n\). Along this subsequence, for every bounded continuous
\(f\), uniform integrability gives
\[
 \mathbb E[f(S)L]
 =\lim_n\mathbb E_{Q_n}[f(S_n)L_n]
 =\lim_n\mathbb E_{P_n}f(S_n)
 =\int f\ell\,dQ_\infty.
\]
Since \(S\sim Q_\infty\), this identifies
\(\mathbb E[L\mid S]=\ell(S)\). Lower semicontinuity of the second moment
and conditional Jensen imply
\[
 \mathbb E\ell(S)^2\le\mathbb EL^2
 \le\liminf_n\mathbb E_{Q_n}L_n^2
 =\mathbb E\ell(S)^2.
\]
Equality in conditional Jensen forces \(L=\ell(S)\) almost surely.
Every subsequential limit is therefore the stated joint law. The same
uniform integrability yields
\(\tfrac12\mathbb E_{Q_n}|L_n-1|
\to\tfrac12\mathbb E_{Q_\infty}|\ell-1|\), proving the TV limit.
\end{proof}

\subsection{The likelihood expansion and critical testing curve}
\label{sec:critical-LAN-closure}

\begin{proposition}[Critical likelihood representation]
\label{prop:critical-likelihood-representation}
Assume \(N\to\infty\) and
\[
  \Delta_n
  =
  \frac{\sqrt{\binom n3}\,|\beta|^3}{N^2}
  \longrightarrow\Delta\in(0,\infty).
\]
There are prior events \(\mathcal E_n\) with probability \(1-o(1)\) for
which the conditioned likelihood ratio satisfies, under \(\Qq_n\),
\begin{equation}
  \log\frac{d\Pp_n^{\mathcal E_n}}{d\Qq_n}
  =\Delta_nW_n-\frac12\Delta_n^2+o_{\Qq_n}(1).
  \label{eq:critical-LAN}
\end{equation}
\end{proposition}

\begin{proof}
We combine the second moment, two-model CLT and projection lemma
established above.
Set
\[
  L_n^{\mathcal E}
  =
  \frac{d\Pp_n^{\mathcal E_n}}{d\Qq_n}.
\]
Proposition~\ref{prop:exact-critical-second-moment} and
Lemma~\ref{lem:triangle-two-model-CLT} give, respectively,
\[
  \mathbb E_{\Qq_n}(L_n^{\mathcal E})^2\longrightarrow e^{\Delta^2}
\]
and
\[
  W_n\Rightarrow\Normal(0,1)\quad\text{under }\Qq_n,
  \qquad
  W_n\Rightarrow\Normal(\Delta,1)\quad\text{under }\Pp_n^{\mathcal E_n}.
\]
The likelihood ratio of \(\Normal(\Delta,1)\) with respect to
\(\Normal(0,1)\) is
\(
  \ell(z)=\exp\{\Delta z-\Delta^2/2\}
\), and its second moment under \(\Normal(0,1)\) is \(e^{\Delta^2}\).
Lemma~\ref{lem:hilbert-likelihood-projection} therefore yields, under
\(\Qq_n\),
\[
  (W_n,L_n^{\mathcal E})
  \Rightarrow
  \bigl(Z,\ell(Z)\bigr),
  \qquad Z\sim\Normal(0,1).
\]
Applying the continuous mapping theorem to
\((w,l)\mapsto l-\ell(w)\) gives
\[
  L_n^{\mathcal E}
  -\exp\{\Delta W_n-\Delta^2/2\}
  \longrightarrow0
  \quad\text{in }\Qq_n\text{-probability}.
\]
The exponential is bounded away from zero whenever \(W_n\) is restricted to
a fixed compact interval.  We may therefore take logarithms and use
\(W_n=O_{\Qq_n}(1)\) together with \(\Delta_n-\Delta=o(1)\) to obtain the
stated likelihood expansion.
\end{proof}

\begin{proof}[Proof of Theorem~\ref{thm:growing-critical-main}, part~(i)]
The likelihood-projection argument in the proof of
Proposition~\ref{prop:critical-likelihood-representation} also gives
\begin{equation}
  \TV(\Pp_n^{\mathcal E_n},\Qq_n)
  \longrightarrow
  \TV\bigl(\Normal(\Delta,1),\Normal(0,1)\bigr)
  =2\Phi(\Delta/2)-1.
  \label{eq:critical-TV-curve}
\end{equation}
Finally,
\(
  \TV(\Pp_n,\Pp_n^{\mathcal E_n})
  \le \mathbb P(\mathcal E_n^c)=o(1)
\), which transfers the limit to \(\Pp_n\).
\end{proof}

No small-amplitude assumption is needed: when \(|\beta|\) stays bounded
away from zero, the critical scale is \(N\asymp n^{3/4}\). At this scale
the triangle determines the limiting likelihood. The curve
\(2\Phi(\Delta/2)-1\) tends to zero as \(\Delta\downarrow0\) and to one
as \(\Delta\to\infty\), consistently with the two strict regimes of
Theorem~\ref{thm:pure-dichotomy}.

\section{Quantitative likelihood and testing bounds}
\label{sec:quantitative-main}

We now make the critical argument quantitative. A common error scale
collects the graph-expansion and triangle-approximation errors; a
quantitative projection converts them into likelihood and testing bounds.
The propositions in Appendix~\ref{app:quantitative-LAN} verify these inputs,
and the proof below carries out the projection and removes conditioning.

Recall
\[
  \Delta_n=\frac{\sqrt{\binom n3}\,|\beta|^3}{N^2}.
\]
Fix \(0<a<b<\infty\) and a sufficiently
small constant \(\gamma_0=\gamma_0(a,b)>0\), as specified in the proof of
Proposition~\ref{prop:quant-second-moment} in Section~S4.3
of the Supplementary Material. We assume
\begin{equation}
  a\le\Delta_n\le b,\qquad N\ge2,\qquad
  0<|\beta|\le1.
  \label{eq:quant-critical-range}
\end{equation}
All constants in this section and Appendix~\ref{app:quantitative-LAN} may depend on \(a,b\), but not on
\(n,N,\beta\).

Put
\begin{equation}
\begin{aligned}
  \varepsilon_{n,N,\beta}
  &:={}
  \log^{12}(en)
  \left\{
    \frac1{\sqrt n\,|\beta|}
    +\frac{|\beta|^{7/8}}{n^{1/16}}
  \right\}
  +N^{-\gamma_0}+\frac{\log^2(eN)}n.
\end{aligned}
\label{eq:quant-master-error}
\end{equation}
The two terms outside braces control the graph expansion. The logarithmically
weighted terms provide common bounds for the truncation and local likelihood
errors, as well as the signed-triangle approximation and its lower-order
edge chaoses.

Section~\ref{subsec:quant-truncation} constructs a prior event
\(\mathcal E_n\) whose complement has superpolynomially small probability.
For \(L_n^{\mathcal E}=d\Pp_n^{\mathcal E_n}/d\Qq_n\),
Sections~\ref{subsec:quant-partition} and \ref{subsec:quant-triangle} bound
the second-moment error and the two signed-triangle Gaussian approximation
errors by \(C_{a,b}\varepsilon_{n,N,\beta}\), provided
\(\varepsilon_{n,N,\beta}\le c_{a,b}\) for a sufficiently small constant. More precisely,
Propositions~\ref{prop:quant-second-moment} and~\ref{prop:quant-triangle} give
\[
\begin{gathered}
 \left|\mathbb E_{\Qq_n}(L_n^{\mathcal E})^2-e^{\Delta_n^2}\right|
 \le C_{a,b}\varepsilon_{n,N,\beta},\\
 d_W(\mathcal L_{\Qq_n}(W_n),\Normal(0,1))
 +d_W(\mathcal L_{\Pp_n^{\mathcal E_n}}(W_n),\Normal(\Delta_n,1))
 \le C_{a,b}\varepsilon_{n,N,\beta}.
\end{gathered}
\]
Here \(d_W\) is the \(1\)-Wasserstein distance, defined in
Section~\ref{subsec:quant-triangle}.
The projection lemma in Section~\ref{subsec:quant-projection} converts these
inputs into TV and log-likelihood bounds of order
\(\varepsilon_{n,N,\beta}^{1/4}\).
Section~\ref{subsec:quant-final} removes the prior conditioning and
concludes the theorem.

\subsection{The finite-sample theorem}
\label{subsec:quant-final}

Propositions~\ref{prop:quant-second-moment} and \ref{prop:quant-triangle}
provide the three inputs to Lemma~\ref{lem:quant-sufficiency}.
We apply it and then remove the prior conditioning using
\eqref{eq:quant-prior-failure}.

\begin{theorem}[Quantitative critical TV and LAN bounds]
\label{thm:quantitative-triangle-LAN}
Assume \eqref{eq:quant-critical-range}.  There are constants
\(c_{a,b},C_{a,b}>0\) such that, whenever
\(\varepsilon_{n,N,\beta}\le c_{a,b}\),
\begin{equation}
  \left|
  \TV(\Pp_n,\Qq_n)
  -
  \{2\Phi(\Delta_n/2)-1\}
  \right|
  \le
  C_{a,b}\varepsilon_{n,N,\beta}^{1/4}.
  \label{eq:quant-TV-final}
\end{equation}
For \(L_n=d\Pp_n/d\Qq_n\) and every \(0<t\le1\),
\begin{equation}
\begin{split}
  \Qq_n\left\{
  \left|
  \log L_n
  -
  \left(
  \Delta_nW_n-\frac12\Delta_n^2
  \right)
  \right|>t
  \right\}
  \le
  C_{a,b}t^{-2}
  \varepsilon_{n,N,\beta}^{1/4}.
  \label{eq:quant-LAN-final}
\end{split}
\end{equation}
The bound \eqref{eq:quant-LAN-final} also holds with \(L_n\) replaced by
\(L_n^{\mathcal E}\), under the same law \(\Qq_n\).
\end{theorem}

\begin{proof}
Set $\eta=C_{a,b}\varepsilon_{n,N,\beta}$, with the constant large enough
for Propositions~\ref{prop:quant-second-moment} and
\ref{prop:quant-triangle} to verify Lemma~\ref{lem:quant-sufficiency}
for
\[
  P=\Pp_n^{\mathcal E_n},\qquad Q=\Qq_n,\qquad
  W=W_n,\qquad\delta=\Delta_n.
\]
Choose $c_{a,b}$ so that $\eta\le e^{-1}$. The lemma gives the asserted
TV and LAN bounds for the conditioned model. By
\eqref{eq:quant-prior-failure},
\begin{equation}
  \mathbb E_{\Qq_n}|L_n-L_n^{\mathcal E}|
  =2\TV(\Pp_n,\Pp_n^{\mathcal E_n})
  \le C_{a,b}e^{-c_{a,b}\log^2(en)}.
  \label{eq:quant-likelihood-conditioning-L1}
\end{equation}
The triangle inequality transfers the TV bound to $\Pp_n$.

For LAN, set \(T_\eta=4b+1+\sqrt{\log(1/\eta)}\) and
\(g_\eta(w)=e^{\Delta_nw-\Delta_n^2/2}\chi_\eta(w)\), where
\(\chi_\eta\) equals one on \([-T_\eta,T_\eta]\), zero outside
\([-T_\eta-1,T_\eta+1]\), and is linear in between, as in the proof of
Lemma~\ref{lem:quant-sufficiency}.
On $\{|W_n|\le T_\eta\}$, $g_\eta(W_n)$ equals the Gaussian
likelihood and is at least $c_be^{-bT_\eta}$. Hence a logarithmic
deviation larger than $t$ implies that either
$|L_n^{\mathcal E}-g_\eta(W_n)|$ or
$|L_n-L_n^{\mathcal E}|$ exceeds $c_bt e^{-bT_\eta}/2$.
Using the central-tail and $L^2$ estimates from that proof for the first
two terms below, and Markov's inequality for the last, gives
\[
\begin{aligned}
  &\Qq_n\left\{\left|\log L_n
    -\left(\Delta_nW_n-\frac12\Delta_n^2\right)\right|>t\right\}\\
  &\quad\le\Qq_n\{|W_n|>T_\eta\}
    +C_bt^{-2}e^{2bT_\eta}\mathbb E_{\Qq_n}
      |L_n^{\mathcal E}-g_\eta(W_n)|^2\\
  &\qquad+C_bt^{-1}e^{bT_\eta}\mathbb E_{\Qq_n}
      |L_n-L_n^{\mathcal E}|\\
  &\quad\le C_{a,b}t^{-2}\eta^{1/4}
    +C_{a,b}t^{-1}e^{-c_{a,b}\log^2(en)+bT_\eta}.
\end{aligned}
\]
Since $\varepsilon_{n,N,\beta}\ge n^{-1/2}$,
$T_\eta=O_{a,b}(1+\sqrt{\log n})$; the last term is absorbed by
$C_{a,b}t^{-2}\varepsilon_{n,N,\beta}^{1/4}$.
\end{proof}

\begin{corollary}[Finite-sample rate under polynomial signal scaling]
\label{cor:quant-polynomial-TV}
Assume \eqref{eq:quant-critical-range} and
\(c_-n^{-r}\le\beta^2\le c_+n^{-r}\) for fixed \(0\le r<1\) and
\(0<c_-\le c_+<\infty\). Then there is a constant \(\gamma_r>0\) such that
\[
  \left|
  \TV(\Pp_n,\Qq_n)
  -
  \{2\Phi(\Delta_n/2)-1\}
  \right|
  \le
  C_{a,b,r,c_-,c_+}
  n^{-\gamma_r},
\]
for all sufficiently large \(n\).
The right-hand side in \eqref{eq:quant-LAN-final} is bounded by the same
quantity times \(C_{a,b,r,c_-,c_+}t^{-2}\).
\end{corollary}

\begin{proof}
Corollary~\ref{cor:quant-polynomial-rate} gives
\(\varepsilon_{n,N,\beta}\le C_{a,b,r,c_-,c_+}n^{-\eta_r}\).
Apply Theorem~\ref{thm:quantitative-triangle-LAN} with
\(\gamma_r=\eta_r/4\).
\end{proof}
The polynomial-scaling conclusion now follows directly.

\begin{proof}[Proof of Theorem~\ref{thm:growing-critical-main}, part~(ii)]
Apply the quantitative estimate proved above.
Since \(\binom n3\asymp n^3\), the assumptions on \(\Delta_n\) and \(\beta\)
imply \(N\asymp n^{3(1-r)/4}\), with comparison constants depending only on
\(a,b,c_-,c_+\).  The hypotheses of
Corollary~\ref{cor:quant-polynomial-TV} therefore hold uniformly over the
stated parameter range.  Its total-variation and likelihood bounds are exactly
the two asserted inequalities.
\end{proof}

The parameter range \(0\le r<1\) includes \(r=0\), where the signal
amplitude may remain bounded away from zero.  Under the critical scaling,
\(\beta_n^2\asymp n^{-r}\) gives \(N\asymp n^{3(1-r)/4}\to\infty\).
Part~(ii) therefore quantifies the likelihood approximation and the
total-variation limit uniformly for \(\Delta_n\in[a,b]\) throughout this
parameter range.
\section{Discussion}

For the pure quadratic kernel, the signed triangle determines both the
detection scale and the limiting testing risk at criticality.
Theorem~\ref{thm:pure-dichotomy} establishes the sharp boundary
\(N\asymp n^{3/4}|\beta|^{3/2}\), while
Theorem~\ref{thm:growing-critical-main} identifies the Gaussian shift
carried by the normalized triangle under its critical-window assumptions.
Thus a statistic built from vertex triples retains all asymptotic testing
information in the full graph in this window.  The finite-sample estimates
quantify the likelihood and total-variation approximations under polynomial
signal scaling.

The proof clarifies the roles of harmonic structure and latent concentration.
A single harmonic block prevents cancellation in the cubic trace, but
isotropy alone does not determine the higher moments of the quadratic
feature law.  Controlling those moments and atypical latent configurations is
essential to the likelihood argument.  The spectral deletion and cavity
estimates give complementary bounds for the same conditioned second moment;
the division between them introduces no additional statistical threshold.
At criticality, uniform control of the tail of the edge-set expansion is what
allows triangle moment calculations to determine the full likelihood limit.

For testing, the statistic \(T_n\) uses only the observed edge signs
\(Y_{ij}=2a_{ij}-1\).  Its computation and null distribution depend on
\(n\), but require neither \(d\) nor \(|\beta|\).
Critical optimality is stated for a specified alternative: its parameters
determine the orientation and the likelihood-optimal rejection threshold.
If only the sign of \(\beta\) is unspecified, the two-sided test based on
\(|T_n|\) retains strong consistency in the detectable regime, as noted
after Proposition~\ref{prop:single-positive}.  This separates the choice of
statistic from the parameter knowledge needed to attain the optimal critical
testing risk.

The quadratic result leaves concrete questions for broader kernel families.
An extension to higher even harmonic layers would require corresponding
feature-moment and concentration estimates.  For mixtures of layers, one must
also account for cancellation in the cubic trace and possible information in
longer cycles.  A separate extension concerns unknown edge density.  The
identity \(p=1/2\) is structural for the present kernel, so treating \(p\)
as a nuisance parameter first requires enlarging the model.  One would then
need to control how estimating \(p\) changes the centered cycle statistics
and the critical likelihood.  These questions concern the conditions under
which a low-order graph statistic continues to capture the available testing
information beyond the pure quadratic model.

\begin{appendix}
\section{Feature, matrix, and spectral preliminaries}
\label{app:complete-preliminaries}

This appendix collects the algebraic and analytic inputs to the likelihood
proofs and the spectral identities for the quadratic model. The feature and matrix estimates are shared by the
strict-regime and critical-window arguments. Throughout,
\[
\begin{aligned}
  V&=\mathbb V_d=\operatorname{Sym}_0(d)
  :=\{H\in\mathbb R^{d\times d}:H^\top=H,\ \operatorname{tr}H=0\},\\
  N&=\dim V=\frac{(d-1)(d+2)}2,
\end{aligned}
\]
with the Hilbert--Schmidt inner product \(\langle H_1,H_2\rangle_{\HS}
:=\operatorname{tr}(H_1H_2)\), and
\[
  q=\varphi_d(X)
  =
  \sqrt{\frac d{d-1}}
  \left(XX^\top-\frac1dI_d\right),
  \qquad X\sim\operatorname{Unif}(\Sph^{d-1}).
\]
Let \(\mu\) denote the law of \(q\).
A square matrix is called \emph{hollow} if all its diagonal entries are zero.
For a configuration \(u=(u_1,\ldots,u_n)\in V^n\), let \(B(u)\) denote its
hollow Gram matrix,
\[
  B_{ii}(u)=0,\qquad
  B_{ij}(u)=\langle u_i,u_j\rangle_{\HS}\quad(i\ne j).
\]
When the configuration is clear, we write \(B\) for \(B(u)\).
For a square matrix \(M\), let \(\operatorname{off}(M)\) denote the hollow
matrix obtained by setting the diagonal of \(M\) equal to zero.

Lemmas~\ref{lem:feature-lift} and \ref{lem:complete-feature-identities}
verify the feature representation and its first three moments.  Lemma~\ref{lem:local-subgaussian}
controls a one-site quadratic-form MGF.  The next two sections turn scalar
inner-product tails into simultaneous row bounds and then into estimates for
the random interaction matrix.  Lemma~\ref{lem:uniform-covariance-tail}
provides uniform covariance concentration, and
Lemma~\ref{lem:hollow-schur-update-bounds} records the deterministic matrix
update used by both spectral recursions.
Section~\ref{subsec:spherical-spectral-background} supplies the spectral
identities used in the triangle calculation in Section~\ref{sec:cycles}.

\subsection{Exact algebra and moment tensors}

\begin{lemma}[Quadratic feature representation]
\label{lem:feature-lift}
For every $x,y\in\Sph^{d-1}$,
\[
  \|\varphi_d(x)\|_{\HS}=1
\]
and
\[
  \langle\varphi_d(x),\varphi_d(y)\rangle_{\HS}
  =
  \frac{d\langle x,y\rangle^2-1}{d-1}
  =
  Z_{2,d}(\langle x,y\rangle).
\]
Moreover, if $X\sim\sigma_d$, then
\[
  \E\varphi_d(X)=0,\qquad
  \E[\varphi_d(X)\otimes\varphi_d(X)]
  =\frac1N I_{\mathbb V_d}.
\]
In the covariance identity,
\(\varphi_d(X)\otimes\varphi_d(X)\) denotes the rank-one operator
\[
  H\longmapsto
  \langle\varphi_d(X),H\rangle_{\HS}\varphi_d(X)
  \qquad (H\in\mathbb V_d).
\]
\end{lemma}

The proof is given in Section~S1.1 of the Supplementary Material.

For matrices in $V$, let
\(S_3\) denote the permutation group of \(\{1,2,3\}\), and define
\[
  \operatorname{tr}_{\mathrm{sym}}(H_1H_2H_3)
  =
  \frac1{6}\sum_{\pi\in S_3}
  \operatorname{tr}(H_{\pi(1)}H_{\pi(2)}H_{\pi(3)}).
\]

\begin{lemma}[Quadratic feature identities and moments]
\label{lem:complete-feature-identities}
Let $X\sim\operatorname{Unif}(\Sph^{d-1})$ and $q=\varphi_d(X)$.  Then, almost
surely,
\begin{equation}
\begin{aligned}
  \|q\|_{\HS}&=1,
  &q^2&=\frac1dI_d+b_dq,\\
  b_d&=\frac{d-2}{\sqrt{d(d-1)}},
  &|b_d|&\le1.
\end{aligned}
  \label{eq:q-square}
\end{equation}
The first two moments of $q$ are
\[
  \mathbb E q=0,\qquad
  \mathbb E(q\otimes q)=\frac1N I_V.
\]
Moreover, for every $H_1,H_2,H_3\in V$,
\begin{equation}
  \mathbb E\prod_{a=1}^3\langle q,H_a\rangle_{\HS}
  =
  \gamma_d\,
  \operatorname{tr}_{\mathrm{sym}}(H_1H_2H_3),
  \label{eq:third-moment-tensor}
\end{equation}
where
\[
  \gamma_d
  =
  \frac{8(d/(d-1))^{3/2}}{d(d+2)(d+4)}
  \asymp N^{-3/2}.
\]
\end{lemma}

The proof is given in Section~S1.2 of the Supplementary Material.

\subsection{A one-site quadratic-form MGF bound}

The spectral recursions integrate one feature at a time.  The following
estimate controls each such integral whenever the associated quadratic form
has sufficiently small operator norm.

\begin{lemma}[Local sub-Gaussian moment bound]
\label{lem:local-subgaussian}
There are absolute constants $c_0,C_0>0$ such that, for every
$H\in V=\operatorname{Sym}_0(d)$ satisfying
\[
  \frac{\|H\|_{\mathrm{op}}}{d}\le c_0,
\]
one has
\begin{equation}
  \log\mathbb E_\mu
  \exp\{\langle H,q\rangle_{\HS}\}
  \le
  \frac{C_0}{N}\|H\|_{\HS}^2.
  \label{eq:local-subgaussian}
\end{equation}
The same estimate holds with $q$ replaced by $-q$ and, consequently, under
any convex combination of the laws of $q$ and $-q$.
\end{lemma}

The same calculation yields, for every nonzero \(H\in V\)
and \(u>0\), without a smallness assumption on \(H\),
\begin{equation}
  \mathbb P_\mu\{|\langle H,q\rangle_{\HS}|\ge u\}
  \le2\exp\left[-c'\min\left\{
  \frac{Nu^2}{\|H\|_{\HS}^2},
  \frac{du}{\|H\|_{\mathrm{op}}}\right\}\right].
  \label{eq:local-quadratic-tail}
\end{equation}
The tail bound also holds under any convex combination
of the laws of \(q\) and \(-q\).
The proof is given in Section~S1.3 of the Supplementary Material.

\subsection{Quadratic inner-product tails and simultaneous row occupancy}

For independent $X_1,\ldots,X_n$ uniform on $\Sph^{d-1}$, set
\(q_i=\varphi_d(X_i)\) and \(B=B(q)\).  The feature identity gives, for
\(i\ne j\),
\[
  B_{ij}
  =\langle q_i,q_j\rangle_{\HS}
  =
  Z_{2,d}(\langle X_i,X_j\rangle)
  =
  \frac{d\langle X_i,X_j\rangle^2-1}{d-1}.
\]

\begin{lemma}[Quadratic inner-product tail]
\label{lem:quadratic-inner-product-tail}
There are absolute constants $c,C>0$ such that, for fixed
$x\in\Sph^{d-1}$ and $0\le u\le1$,
\[
  \mathbb P\{Z_{2,d}(\langle x,X\rangle)\ge u\}
  \le Ce^{-cdu}.
\]
For every $D>0$ there is $C_D$ such that
\[
  \mathbb P\left\{
  \max_{i<j}|B_{ij}|>
  C_D\min\left(1,\frac{\log n}{d}\right)
  \right\}\le n^{-D}.
\]
\end{lemma}

The proof is given in Section~S1.4 of the Supplementary Material.

For $u\ge0$, define the row occupancy
\[
  \mathcal N_i(u)=\#\{j\ne i:B_{ij}\ge u\}.
\]
Let
\[
  \mathcal U_d
  =
  \left\{\frac{2^m}{d}:0\le m\le\lceil\log_2d\rceil\right\}\cap[0,1].
\]

\begin{lemma}[Multiscale row occupancy]
\label{lem:multiscale-occupancy}
For every fixed $D>0$, there are $c,C_D>0$ such that, with probability
at least $1-n^{-D}$,
\begin{equation}
  \mathcal N_i(u)
  \le
  C_D\bigl(ne^{-cdu}+\log n\bigr)
  \label{eq:occupancy-dyadic}
\end{equation}
simultaneously for every $i\in[n]$ and $u\in\mathcal U_d$.
On the same event, after changing constants,
\begin{equation}
  \mathcal N_i(u)
  \le
  C_D\bigl(ne^{-cdu/2}+\log n\bigr)
  \label{eq:occupancy-continuous}
\end{equation}
for every $u\in[1/d,C_D\log n/d]$.
\end{lemma}

The proof is given in Section~S1.5 of the Supplementary Material.

The row-moment estimates used in Proposition~\ref{prop:row-gram-environment}
are given in Corollary~S1.1
of the Supplementary Material.

\subsection{First-replica matrix bounds}

Fix an independent first replica $q_1^\ast,\ldots,q_n^\ast$ and set
\begin{equation}
  B^\ast=B(q^\ast),\qquad
  A=\beta^2 B^\ast,\qquad 0<|\beta|\le1.
  \label{eq:first-replica-environment}
\end{equation}

\begin{proposition}[First-replica matrix-norm bounds]
\label{prop:row-gram-environment}
With probability $1-o(1)$, the Hilbert--Schmidt and row norms satisfy
\begin{equation}
\begin{aligned}
  \|A\|_{\HS}^2&\lesssim\frac{\beta^4n^2}{N},
  &\qquad
  \max_i\|A_{i\cdot}\|_2^2&\lesssim\frac{\beta^4n}{N}.
\end{aligned}
  \label{eq:environment-FR}
\end{equation}
On the same event, the sixth row moment obeys
\begin{equation}
  \max_i\sum_{j\ne i}|A_{ij}|^6
  \lesssim\frac{\beta^{12}n}{N^3}.
  \label{eq:environment-sixth-row}
\end{equation}
The maximum entry and operator norm satisfy
\begin{align}
  \|A\|_{\max}&\lesssim
  \frac{\beta^2\log n}{\sqrt N},
  \label{eq:environment-M}\\
  \|A\|_{\mathrm{op}}&\lesssim
  \beta^2\left(
  \frac nN+
  \sqrt{\frac{n\log N}{N}}+
  \log N
  \right).
  \label{eq:environment-L}
\end{align}
\end{proposition}

The proof is given in Section~S1.7 of the Supplementary Material.

\subsection{Covariance truncation and its uniform tail}
\label{sec:covariance-truncation}

Recall that $\mu$ is the law of $q$.  Let $\mu^-$ be the law of $-q$, and
let $\operatorname{conv}\{\mu,\mu^-\}$ denote their convex mixtures.  Because
$q\otimes q$ is invariant under $q\mapsto-q$, its distribution is the same
under every law in this class.  For a finite index set \(I\), define
\[
  \mathcal C_I(u)
  =
  \left\{
  \left\|\sum_{i\in I}q_i\otimes q_i\right\|_{\mathrm{op}}\le u
  \right\}.
\]
We abbreviate
\(\mathcal C_m(u)=\mathcal C_{[m]}(u)\) and
\(\mathcal C_{-i}(u)=\mathcal C_{[n]\setminus\{i\}}(u)\).

\begin{lemma}[Uniform covariance tail]
\label{lem:uniform-covariance-tail}
There are absolute constants $c,C>0$ such that, for every finite index set
\(I\) and all $T\ge0$,
\[
  \sup_{\nu\in\operatorname{conv}\{\mu,\mu^-\}}
  \mathbb P_{\nu^{\otimes |I|}}
  \left\{
  \mathcal C_I\left(
  C\left\{\frac{|I|}{N}+T+\log N\right\}
  \right)^c
  \right\}
  \le e^{-cT}.
\]
For the quadratic spherical features there is also the stronger bound
\begin{equation}
  \sup_{\nu\in\operatorname{conv}\{\mu,\mu^-\}}
  \mathbb P_{\nu^{\otimes |I|}}
  \left\{
    \mathcal C_I\left(
      C\left\{\frac{|I|}{N}+\log(eN)+\frac Td\right\}
    \right)^c
  \right\}
  \le e^{-cT}.
  \label{eq:covariance-tail-spherical}
\end{equation}
The bounds also hold when the independent coordinates have different laws
in this convex class.
\end{lemma}

The proof is given in Section~S1.8 of the Supplementary Material.

\subsection{Deterministic hollow--Schur update bounds}

Both spectral recursions use the same matrix update after a random
bipartition.  The next lemma records all of its norm bounds independently of
the probability measure used in either recursion.

\begin{lemma}[Hollow--Schur update bounds]
\label{lem:hollow-schur-update-bounds}
Let $A$ be a real symmetric hollow $m\times m$ matrix, and let $I,J$ partition
$[m]$.  Recall the submatrix notation
$A_{I,J}=(A_{ij})_{i\in I,\,j\in J}$, and set
\[
  D=A_{I,J}^{\mathsf T}A_{I,J},
  \qquad
  A'=\frac cN\operatorname{off}(D)
\]
for some $c>0$.  Then
\begin{align}
  \|A'\|_{\HS}
  &\le\frac cN\|A\|_{\mathrm{op}}\|A\|_{\HS},
  &
  \|A'\|_{\mathrm{op}}
  &\le\frac{2c}{N}\|A\|_{\mathrm{op}}^2,
  \label{eq:hollow-schur-common-norms}\\
  \max_j\sum_k|A'_{jk}|
  &\le\frac cN\left(\max_i\sum_j|A_{ij}|\right)^2,
  &
  \max_j\|A'_{j\cdot}\|_2
  &\le\frac cN\|A\|_{\mathrm{op}}
  \max_i\|A_{i\cdot}\|_2.
  \label{eq:hollow-schur-common-rows}
\end{align}
\end{lemma}

The proof is given in Section~S1.9 of the Supplementary Material.

The lemmas above provide the feature and matrix estimates used by the
branch-specific second-moment arguments and the later critical analyses.

\subsection{Cycle traces and the quadratic spectrum}
\label{subsec:spherical-spectral-background}

Retain the normalized kernel \(\kappa\), edge variables \(Y_{ij}\) and
integral operator \(\Kop_\kappa\) from Section~\ref{subsec:quadratic-spectral-scale}.
For a simple cycle on distinct vertices \(i_1,\ldots,i_\ell\), conditional
independence gives the product of the kernel values as its conditional
mean. Integrating the latent points and composing the bounded kernel yields
\begin{equation}
  \mathbb E_{\Pp_{n,d,K}}\prod_{r=1}^{\ell}Y_{i_ri_{r+1}}
  =\tr(\Kop_\kappa^\ell),
  \qquad i_{\ell+1}=i_1,\quad \ell\ge3.
  \label{eq:cycle-trace-identity}
\end{equation}
The corresponding expectation under \(\Qq_{n,p}\) is zero because the
distinct edges are independent and centered.

For \(\kappa=\beta Z_{2,d}\), set
\(f_H(x)=\langle\varphi_d(x),H\rangle_{\HS}\), \(H\in V\).
Lemma~\ref{lem:feature-lift} gives
\(\Kop_\kappa f_H=(\beta/N)f_H\) and
\(\langle f_H,f_G\rangle_{L^2(\sigma_d)}=\langle H,G\rangle_{\HS}/N\).
The range of \(\Kop_\kappa\) is contained in the span of these functions;
thus its only nonzero eigenvalue is \(\beta/N\), with multiplicity \(N\).
Consequently,
\begin{equation}
  \tr(\Kop_{\beta Z_{2,d}}^\ell)
  =\frac{\beta^\ell}{N^{\ell-1}},
  \qquad \ell\ge2.
  \label{eq:single-layer-traces}
\end{equation}
The full cycle calculation and the general harmonic expansion are given in
Section~S5 of the Supplementary Material.
\section{Likelihood second moments and the sparse-regime recursion}
\label{app:complete-spectral-recursions}

This appendix first converts the graph likelihood second moment into an
expectation over two latent configurations.  It then completes the sparse
branch of Proposition~\ref{prop:pure-negative}.  The main step is an
untruncated spectral recursion: a random bipartition permits one side to be
integrated using Lemma~\ref{lem:local-subgaussian}, and the resulting
lower-dimensional interaction contracts under row-\(\ell_1\) and operator-norm
control.  Proposition~\ref{prop:sparse-signal-complete} applies this recursion
when \(n\beta^2/N\to0\).

Throughout this appendix, \(V=\operatorname{Sym}_0(d)\), \(N=\dim V\),
\(q=\varphi_d(X)\), and \(\mu=\mathcal L(q)\) are as defined in
Appendix~\ref{app:complete-preliminaries}.

\subsection{Likelihood second moment and prior truncation}

For a measurable subset $\mathcal E$ of the latent configuration space with
$\pi=\mathbb P(q\in\mathcal E)>0$, let $\Pp^{\mathcal E}$ be the geometric
graph model obtained by conditioning the latent configuration on
$q\in\mathcal E$.
For a configuration \(q=(q_1,\ldots,q_n)\), use the hollow Gram matrix
\(B(q)\) defined in Appendix~\ref{app:complete-preliminaries}.  Thus
\(B_{ij}(q)=\langle q_i,q_j\rangle_{\HS}\) for \(i\ne j\).

\begin{lemma}[Truncated likelihood identity]
\label{lem:truncated-likelihood}
For the pure quadratic model, let
\(L_{\mathcal E}=d\Pp^{\mathcal E}/d\Qq_n\).  Then
\begin{equation}
  \mathbb E_{\Qq_n}L_{\mathcal E}^2
  =
  \frac1{\pi^2}\mathbb E_{q^\ast,q}\left[
  \mathbf1_{\{q^\ast\in\mathcal E\}}\mathbf1_{\{q\in\mathcal E\}}
  \prod_{i<j}\left\{1+\beta^2 B_{ij}(q^\ast)B_{ij}(q)\right\}
  \right].
  \label{eq:truncated-second-moment}
\end{equation}
In particular, if, uniformly over $q^\ast\in\mathcal E$,
\[
  \mathbb E_q
  \left[
  \mathbf1_{\{q\in\mathcal E\}}
  \prod_{i<j}
  \left\{1+\beta^2 B_{ij}(q^\ast)B_{ij}(q)\right\}
  \right]
  \le1+o(1),
\]
then $\chi^2(\Pp^{\mathcal E}\|\Qq_n)\to0$ whenever $\pi\to1$.
\end{lemma}

The proof is given in Section~S2.1 of the Supplementary Material.

\subsection{Untruncated spectral recursion and the sparse regime}

For a real symmetric hollow $m\times m$ matrix $A$, define
\[
  \Psi_m(A)=\log\mathbb E\exp\left\{
  \sum_{i<j}A_{ij}\langle q_i,q_j\rangle_{\HS}\right\},
  \qquad r_1(A)=\max_i\sum_j|A_{ij}|.
\]

\begin{lemma}[Untruncated spectral MGF bound]
\label{lem:l1-spectral-recursion}
There are absolute constants $c,C>0$ such that, for every \(m\) and every
real symmetric hollow \(m\times m\) matrix \(A\), the conditions
\[
  \frac{r_1(A)}d\le c,\qquad
  \frac{\|A\|_{\mathrm{op}}}{N}\le c
\]
imply
\begin{equation}
  \Psi_m(A)\le\frac C N\|A\|_{\HS}^2.
  \label{eq:l1-spectral-recursion}
\end{equation}
\end{lemma}

\begin{proof}
A random bipartition lets us integrate one group of features.  The resulting
hollow--Schur update preserves the hypotheses and contracts the
Hilbert--Schmidt norm, closing a strong induction on $m$.

Choose independent $\operatorname{Bernoulli}(1/2)$ variables $\delta_i$,
independent of the features, and set $I=\{i:\delta_i=1\}$, $J=[m]\setminus I$.
Write $\mathbb E_{\mathrm{cut}}$ for expectation over this partition and put
\[
  \mathsf H_{I,J}^{\mathrm{cut}}
  =\sum_{i\in I,j\in J}A_{ij}\langle q_i,q_j\rangle.
\]
Each unordered pair crosses the cut with probability $1/2$, so
$\mathbb E_{\mathrm{cut}}[2\mathsf H_{I,J}^{\mathrm{cut}}]
=\sum_{i<j}A_{ij}\langle q_i,q_j\rangle$.  Jensen's inequality gives
$e^{\Psi_m(A)}\le\mathbb E_{\mathrm{cut}}\mathbb E_q
e^{2\mathsf H_{I,J}^{\mathrm{cut}}}$; this is the standard random-partition
decoupling step \cite[proof of Theorem~6.1.1]{Vershynin2018}.
For a fixed cut, condition on $q_J$.  The fields
$H_i=2\sum_{j\in J}A_{ij}q_j$ satisfy $\|H_i\|_{\mathrm{op}}\le2r_1(A)$.
For small enough $c$, Lemma~\ref{lem:local-subgaussian} and independence give
\[
  \mathbb E_{q_I}[e^{2\mathsf H_{I,J}^{\mathrm{cut}}}\mid q_J]
  \le\exp\left\{\frac{C_1}{N}
  \sum_{i\in I}\left\|\sum_{j\in J}A_{ij}q_j\right\|_{\HS}^2\right\}.
\]
Set $D=A_{I,J}^{\mathsf T}A_{I,J}$, $C_2=2C_1$, and
$A'=(C_2/N)\operatorname{off}(D)$.  Since $\|q_j\|_{\HS}=1$,
\[
  \sum_{i\in I}\left\|\sum_{j\in J}A_{ij}q_j\right\|_{\HS}^2
  =\operatorname{tr}D+2\sum_{j<k}D_{jk}\langle q_j,q_k\rangle.
\]
Integrating over $q_J$ and averaging over the cut therefore yield
\[
  e^{\Psi_m(A)}\le\mathbb E_{\mathrm{cut}}
  \exp\left\{\frac{C_1}{N}\operatorname{tr}D+\Psi_{|J|}(A')\right\}.
\]
Lemma~\ref{lem:hollow-schur-update-bounds} and $N\asymp d^2$ give
\[
\begin{gathered}
  \frac{r_1(A')}{d}\le\frac{C_2d}{N}\left(\frac{r_1(A)}d\right)^2,
  \qquad
  \frac{\|A'\|_{\mathrm{op}}}{N}
  \le2C_2\left(\frac{\|A\|_{\mathrm{op}}}{N}\right)^2,\\
  \|A'\|_{\HS}\le\frac{C_2}{N}\|A\|_{\mathrm{op}}\|A\|_{\HS}
  \le C_2c\|A\|_{\HS}.
\end{gathered}
\]
Thus both hypotheses pass to $A'$ for sufficiently small $c$.

We prove the conclusion with $C=2C_1$ by strong induction, the case
$m\le1$ being immediate.  If $I=\varnothing$, then $D=A'=0$ and
$\Psi_m(0)=0$.  Otherwise $|J|<m$, so the induction hypothesis applies to
$A'$.  In either case, $\operatorname{tr}D\le\|A\|_{\HS}^2$ gives
\[
  \frac{C_1}{N}\operatorname{tr}D+\Psi_{|J|}(A')
  \le\frac{C_1+CC_2^2c^2}{N}\|A\|_{\HS}^2
  \le\frac C N\|A\|_{\HS}^2,
\]
after reducing $c$ so that $C_2^2c^2\le1/2$.  Substituting into the
random-cut bound and taking logarithms proves \eqref{eq:l1-spectral-recursion}.
\end{proof}

\subsection{Completion of the sparse branch}

We now verify the row-\(\ell_1\), operator-norm, and Hilbert--Schmidt conditions
of Lemma~\ref{lem:l1-spectral-recursion} for the random interaction generated
by the first latent replica.

\begin{proposition}[Sparse-signal indistinguishability]
\label{prop:sparse-signal-complete}
If
\[
  \frac{n\beta^2}{N}\to0,
\]
then
\[
  \TV(\Pp_n,\Qq_n)\to0.
\]
\end{proposition}

\begin{proof}
For $\beta=0$ the laws coincide, so assume $\beta\ne0$ and set $B=B(q^\ast)$.
Conditional on $q_i^\ast$, the variables $(|B_{ij}|)_{j\ne i}$ are independent,
bounded by one, and, by isotropy, satisfy
\[
  \mathbb E[|B_{ij}|\mid q_i^\ast]\le N^{-1/2},
  \qquad \mathbb E[B_{ij}^2\mid q_i^\ast]=N^{-1}.
\]
Bernstein's inequality gives, uniformly in $q_i^\ast$,
\[
  \mathbb P\left\{\sum_{j\ne i}|B_{ij}|>
  \frac n{\sqrt N}+C\left(\sqrt{\frac{nx}{N}}+x\right)
  \,\middle|\,q_i^\ast\right\}\le e^{-x}.
\]
Taking $x=3\log n$ and a union bound shows that
$r_1(B)\lesssim n/\sqrt N+\log n$ with probability $1-o(1)$.
Also $\mathbb E\|B\|_{\HS}^2=n(n-1)/N$, so Markov's inequality gives
\begin{equation}
  \|B\|_{\HS}^2\le\frac n{\beta^2}
  \label{eq:sparse-hs-bound}
\end{equation}
with probability at least $1-n\beta^2/N$.  Let $\mathcal E_n$ be the
intersection of these events and $\pi_n=\mathbb P(q^\ast\in\mathcal E_n)\to1$.
On $\mathcal E_n$, $N\asymp d^2$ implies
\[
  \frac{r_1(\beta^2B)}d\lesssim
  \frac{n\beta^2}{N}+\frac{\beta^2\log n}{d}=o(1),
  \qquad
  \frac{\|\beta^2B\|_{\mathrm{op}}}{N}
  \le\frac{r_1(\beta^2B)}N=o(1).
\]
Indeed, the remaining term $\beta^2\log n/d$ is at most $\log n/\sqrt n$
when $d\ge\sqrt n$, since $|\beta|\le1$; when $d<\sqrt n$, it is at most
$C(n\beta^2/N)\log n/\sqrt n=o(1)$.
Lemma~\ref{lem:l1-spectral-recursion} and \eqref{eq:sparse-hs-bound} now give,
uniformly on $\mathcal E_n$,
\[
  \log\mathbb E_q\prod_{i<j}
  (1+\beta^2B_{ij}\langle q_i,q_j\rangle)
  \le\Psi_n(\beta^2B)
  \le\frac{C\beta^4}{N}\|B\|_{\HS}^2
  \le\frac{Cn\beta^2}{N}=o(1).
\]
Here $1+u\le e^u$ applies to each factor, and all factors are nonnegative
because $\beta^2\le1$ and the inner products lie in $[-1,1]$.
Inserting $\mathbf1_{\{q\in\mathcal E_n\}}$ can therefore only decrease the
expectation.  Lemma~\ref{lem:truncated-likelihood} yields
$\chi^2(\Pp_n^{\mathcal E_n}\|\Qq_n)=o(1)$, and hence
\[
  \TV(\Pp_n,\Qq_n)\le1-\pi_n+
  \tfrac12\sqrt{\chi^2(\Pp_n^{\mathcal E_n}\|\Qq_n)}=o(1).
\]
\end{proof}

Proposition~\ref{prop:sparse-signal-complete} completes the sparse branch of
Proposition~\ref{prop:pure-negative}.
\section{Bounds for covariance-truncated deletion}
\label{app:complete-adaptive-even}

This appendix proves the estimates used in Section~\ref{sec:dense-likelihood}:
the covariance-truncated spectral bound, weighted feature tails, stability
of the hollow--Schur updates, and summability of the local remainders.
We retain the notation and uniformity convention of that section.
Unless a finite-parameter statement specifies otherwise, the dense regime
\eqref{eq:dense-master-regime} and cutoff \eqref{eq:adaptive-cutoff-parameters}
are in force. Expectations integrate the second replica.

\subsection{Covariance-truncated row-\texorpdfstring{$\ell_2$}{l2}
recursion}
\label{sec:covariance-truncated-recursion}

Throughout this subsection, $\Lambda>0$ is an arbitrary covariance cutoff.
The estimates below are uniform in $\Lambda$ subject to their displayed
assumptions.  Their application to the dense branch uses the value in
\eqref{eq:adaptive-cutoff-parameters}.  For \(I\subseteq[m]\), recall from
Appendix~\ref{app:complete-preliminaries} that
\[
  \mathcal C_I(\Lambda)
  =
  \left\{
  \left\|\sum_{i\in I}q_i\otimes q_i\right\|_{\mathrm{op}}\le\Lambda
  \right\},
  \qquad
  \mathcal C_m(\Lambda)=\mathcal C_{[m]}(\Lambda).
\]
For a real symmetric hollow \(m\times m\) matrix \(A\), let
\(q_1,\ldots,q_m\) be independent with respective laws
\(\mu_{t_1},\ldots,\mu_{t_m}\), where \(t_i\in[0,1]\), and define
\[
  \mathcal M_m(A;\Lambda)
  =
  \mathbb E\left[
  \mathbf1_{\mathcal C_m(\Lambda)}
  \exp\left\{
  \sum_{i<j}A_{ij}\langle q_i,q_j\rangle
  \right\}
  \right],
\]
and define
\[
  R_2(A)=\max_i\left(\sum_jA_{ij}^2\right)^{1/2}.
\]

\begin{lemma}[Deterministic covariance-cutoff weighted-sum bound]
\label{lem:cutoff-field-bound}
On $\mathcal C_m(\Lambda)$, for every $a\in\mathbb R^m$,
\[
  \left\|\sum_ja_jq_j\right\|_{\HS}
  \le\sqrt\Lambda\,\|a\|_2,
\]
and therefore the same upper bound holds for the matrix operator norm.
\end{lemma}

\begin{proof}
Let \(\mathsf S_q:\mathbb R^m\to V\) be
\(\mathsf S_qa=\sum_ja_jq_j\), and let \(\mathsf S_q^\ast\) denote its
Hilbert-space adjoint.  Then
\[
  \mathsf S_q\mathsf S_q^\ast=\sum_jq_j\otimes q_j.
\]
On $\mathcal C_m(\Lambda)$,
\(\|\mathsf S_q\|_{\mathrm{op}}^2
  =\|\mathsf S_q\mathsf S_q^\ast\|_{\mathrm{op}}\le\Lambda\).
Thus \(\|\mathsf S_qa\|_{\HS}\le\sqrt\Lambda\|a\|_2\).  Finally,
\(\|\mathsf S_qa\|_{\mathrm{op}}\le\|\mathsf S_qa\|_{\HS}\).
\end{proof}

\begin{theorem}[Covariance-truncated spectral recursion]
\label{thm:l2-spectral-recursion}
There are absolute constants $c,C>0$ such that
\[
  \frac{\sqrt\Lambda\,R_2(A)}d\le c,
  \qquad
  \frac{\|A\|_{\mathrm{op}}}{N}\le c
\]
imply
\begin{equation}
  \log\mathcal M_m(A;\Lambda)
  \le\frac C N\|A\|_{\HS}^2.
  \label{eq:l2-spectral-recursion}
\end{equation}
The conclusion is uniform when the coordinates independently have laws
$\mu_t$ with possibly different $t\in[0,1]$.
\end{theorem}

\begin{proof}
We decouple across a random cut, apply the one-site MGF with the covariance
cutoff on the unintegrated side, and close the resulting matrix recursion
by induction. Let \(I\subseteq[m]\) include each index independently with
probability \(1/2\), put \(J=[m]\setminus I\), and write
\(S_{I,J}=\sum_{i\in I,j\in J}A_{ij}\langle q_i,q_j\rangle\).
Each edge crosses the cut with probability \(1/2\), so Jensen gives
\[
 \mathcal M_m(A;\Lambda)
 \le\mathbb E_{\mathrm{cut}}\mathbb E[
       \mathbf1_{\mathcal C_m(\Lambda)}e^{2S_{I,J}}].
\]
Fix the cut. Positivity of the covariance summands implies
\(\mathcal C_m(\Lambda)\subseteq\mathcal C_J(\Lambda)\). We may therefore
drop the restriction on \(q_I\) before conditioning on \(q_J\). On
\(\mathcal C_J(\Lambda)\), the matrices
\(H_i=2\sum_{j\in J}A_{ij}q_j\) satisfy
\[
 \max_{i\in I}\frac{\|H_i\|_{\mathrm{op}}}{d}
 \le\frac{2\sqrt\Lambda R_2(A)}d\le2c
\]
by Lemma~\ref{lem:cutoff-field-bound}. For sufficiently small \(c\),
conditional independence and Lemma~\ref{lem:local-subgaussian} give
\[
 \mathbb E_{q_I}[e^{2S_{I,J}}\mid q_J]
 \le\exp\left\{\frac{C_1}{N}\sum_{i\in I}
            \left\|\sum_{j\in J}A_{ij}q_j\right\|_{\HS}^2\right\}.
\]
With \(D=A_{I,J}^{\mathsf T}A_{I,J}\), the sum in this exponent is
\(\tr D+2\sum_{j<k}D_{jk}\langle q_j,q_k\rangle\), since
\(\|q_j\|_{\HS}=1\). Thus, for \(C_2=2C_1\) and
\(A'=C_2\operatorname{off}(D)/N\),
\[
 \mathcal M_m(A;\Lambda)
 \le\mathbb E_{\mathrm{cut}}
       [e^{C_1\tr D/N}\mathcal M_{|J|}(A';\Lambda)].
\]
Lemma~\ref{lem:hollow-schur-update-bounds} yields
\[
\begin{gathered}
 \|A'\|_{\HS}\le C_2c\|A\|_{\HS},\qquad
 R_2(A')\le C_2cR_2(A),\qquad
 \|A'\|_{\mathrm{op}}\le2C_2c\|A\|_{\mathrm{op}}.
\end{gathered}
\]
Choose \(c\) so that \(2C_2c\le1\) and \(C_2^2c^2\le1/2\).
Then \(A'\) satisfies the same two smallness hypotheses.
We prove \(\mathcal M_m(A;\Lambda)\le
e^{C_{\mathcal M}\|A\|_{\HS}^2/N}\), with
\(C_{\mathcal M}=2C_1\), by strong induction on \(m\).
For \(m\le1\) the integral is at most one. If \(I\ne\varnothing\),
then \(|J|<m\), so the induction hypothesis applies to \(A'\);
if \(I=\varnothing\), then \(A'=0\) and its integral is again at most one.
In both cases, using \(\tr D\le\|A\|_{\HS}^2\),
\[
 \mathcal M_m(A;\Lambda)
 \le\exp\left\{\frac{C_1+C_{\mathcal M}C_2^2c^2}{N}
                         \|A\|_{\HS}^2\right\}
 \le e^{C_{\mathcal M}\|A\|_{\HS}^2/N}.
\]
This closes the induction. The one-site MGF is uniform over \(\mu_t\),
and \(q_i\otimes q_i\) has the same law for every \(t\); hence the proof
also allows a different interpolation parameter at each coordinate.
\end{proof}

\subsection{Weighted feature-sum tails}
\label{sec:weighted-feature-sum-tails}

For deterministic coefficients $a=(a_1,\ldots,a_m)\in\mathbb R^m$, put
\[
  H(a)=\sum_{j=1}^ma_jq_j.
\]
Thus $H(a)$ is the weighted feature sum with coefficient vector $a$.  Since
$\|q_j\|_{\HS}=1$ for every $j$,
\[
  \|H(a)\|_{\HS}^2-\|a\|_2^2
  =
  2\sum_{j<k}a_ja_k\langle q_j,q_k\rangle.
\]

\begin{corollary}[Weighted-sum MGF]
\label{cor:local-field-mgf}
There are absolute constants \(c,C>0\) such that the following holds.  Let
\(\theta\ge0\), and suppose
\[
  \frac{\theta\sqrt\Lambda\,\|a\|_\infty\|a\|_2}{d}\le c,
  \qquad
  \frac{\theta\|a\|_2^2}{N}\le c.
\]
Then
\begin{equation}
  \log\mathbb E\left[
  \mathbf1_{\mathcal C_m(\Lambda)}
  e^{\theta(\|H(a)\|_{\HS}^2-\|a\|_2^2)}
  \right]
  \le C\frac{\theta^2\|a\|_2^4}{N}.
  \label{eq:local-field-mgf}
\end{equation}
Consequently, if \(\|a\|_2>0\), \(1\le T\le cN\), and
\[
  \frac{T\sqrt\Lambda\,\|a\|_\infty}{d\|a\|_2}\le c,
\]
then for every \(u>0\),
\begin{equation}
  \mathbb P\left\{
  \mathcal C_m(\Lambda),\
  \|H(a)\|_{\HS}^2\ge(1+u)\|a\|_2^2
  \right\}
  \le
  \exp\left\{-uT+C\frac{T^2}{N}\right\}.
  \label{eq:local-field-tail}
\end{equation}
\end{corollary}

\begin{proof}
Set \(A_\theta=2\theta\operatorname{off}(aa^\top)\). The identity above
identifies the expectation in \eqref{eq:local-field-mgf} with
\(\mathcal M_m(A_\theta;\Lambda)\). Directly,
\[
\begin{aligned}
 R_2(A_\theta)&\le2\theta\|a\|_\infty\|a\|_2,
 &\|A_\theta\|_{\HS}^2
 &=4\theta^2\left(\|a\|_2^4-\sum_i a_i^4\right),\\
 \|A_\theta\|_{\mathrm{op}}&\le2\theta\|a\|_2^2.
\end{aligned}
\]
The operator bound follows from
\(-\operatorname{diag}(a_i^2)\preceq\operatorname{off}(aa^\top)
\preceq aa^\top\). After decreasing \(c\), the assumptions therefore
verify both hypotheses of Theorem~\ref{thm:l2-spectral-recursion}, which
gives \eqref{eq:local-field-mgf}.
For \eqref{eq:local-field-tail}, take \(\theta=T/\|a\|_2^2\).
The MGF hypotheses reduce exactly to the two stated restrictions on \(T\),
and exponential Markov inequality gives
\[
 \mathbb P\{\mathcal C_m(\Lambda),
       \|H(a)\|_{\HS}^2\ge(1+u)\|a\|_2^2\}
 \le e^{-\theta u\|a\|_2^2+C\theta^2\|a\|_2^4/N}
 =e^{-uT+CT^2/N}.
\]
\end{proof}

\subsection{Hollow--Schur recursion and squared-row-norm stability}
\label{sec:adaptive-hollow-schur}

Eliminating one feature creates a rank-one Schur correction in the remaining
interaction matrix.  We first show deterministically that repeated corrections
do not enlarge its squared row norms or squared Hilbert--Schmidt norm beyond
the scales needed in Section~\ref{sec:dense-likelihood}.
We use the matrices, column vectors and row norms defined in
Section~\ref{sec:two-matrix-compensation} and \eqref{eq:hollow-schur}.

\begin{lemma}[One-step row and Hilbert--Schmidt norm bounds]
\label{lem:one-step-row-hs-bounds}
If $r_m^2/N^2\le1/4$, then, for every $j<m$,
\[
  r_j^{(m-1)\,2}
  \le
  \left(1+5\frac{r_m^2}{N^2}\right)r_j^{(m)\,2}
  -\frac12(A_{jm}^{(m)})^2.
\]
If $\|A^{(m)}\|_{\HS}^2/N^2\le1/20$, then
\[
  \|A^{(m-1)}\|_{\HS}^2
  \le \|A^{(m)}\|_{\HS}^2-\frac54r_m^2.
\]
\end{lemma}

\begin{proof}
Write \(a=a^{(m)}\). For \(j<m\), expanding \eqref{eq:hollow-schur}
and applying Cauchy--Schwarz to the cross term gives
\[
\begin{aligned}
 r_j^{(m-1)\,2}
 &=r_j^{(m)\,2}-a_j^2
   +\frac{2a_j}{N}\sum_{k<m,\,k\ne j}A_{jk}^{(m)}a_k
   +\frac{a_j^2}{N^2}\sum_{k<m,\,k\ne j}a_k^2\\
 &\le r_j^{(m)\,2}-a_j^2
       +\frac{2|a_j|r_j^{(m)}r_m}{N}+\frac{a_j^2r_m^2}{N^2}.
\end{aligned}
\]
Young's inequality bounds the cross term by
\(a_j^2/4+4r_j^{(m)\,2}r_m^2/N^2\); the last term is at most
\(a_j^2/4\) when \(r_m^2/N^2\le1/4\). This proves the first assertion,
even with \(4\) in place of \(5\).

Under the second hypothesis the first assertion applies. Summing it over
\(j<m\), using \(\sum_{j<m}a_j^2=r_m^2\) and
\(\sum_{j<m}r_j^{(m)\,2}=\|A^{(m)}\|_{\HS}^2-r_m^2\), gives
\[
\begin{aligned}
 \|A^{(m-1)}\|_{\HS}^2
 &\le\left(1+\frac{5r_m^2}{N^2}\right)
        (\|A^{(m)}\|_{\HS}^2-r_m^2)-\frac12r_m^2\\
 &\le\|A^{(m)}\|_{\HS}^2
       -\left(\frac32-\frac{5\|A^{(m)}\|_{\HS}^2}{N^2}\right)r_m^2
 \le\|A^{(m)}\|_{\HS}^2-\frac54r_m^2.
\end{aligned}
\]
\end{proof}

\begin{proposition}[Global hollow--Schur bounds]
\label{prop:global-row-quartic}
Assume $\|A\|_{\HS}^2/N^2\le1/20$, and put
\[
  R_0^2=\max_j\sum_{k\ne j}A_{jk}^2,
  \qquad
  M_\ast=\max_{m,j<k\le m}|A_{jk}^{(m)}|.
\]
Then
\begin{equation}
  \|A^{(m)}\|_{\HS}^2\le\|A\|_{\HS}^2,\qquad
  \sum_{m=2}^nr_m^2\le\frac45\|A\|_{\HS}^2,
  \label{eq:recursive-hs-bounds}
\end{equation}
\[
  \max_{m,j\le m}r_j^{(m)\,2}
  \le
  R_0^2
  \exp\left\{5\frac{\|A\|_{\HS}^2}{N^2}\right\},
\]
and
\begin{equation}
  \sum_{m=2}^n r_m^4
  \le
  \frac45
  e^{5\|A\|_{\HS}^2/N^2}
  R_0^2\|A\|_{\HS}^2.
  \label{eq:quartic-row-sum}
\end{equation}
Moreover,
\begin{equation}
  M_\ast\le \|A\|_{\max}+\frac{CR_0^2}{N}.
  \label{eq:recursive-max-entry}
\end{equation}
\end{proposition}

\begin{proof}
The second bound in Lemma~\ref{lem:one-step-row-hs-bounds} propagates
\(\|A^{(m)}\|_{\HS}^2\le N^2/20\). Summing its decreases gives
\eqref{eq:recursive-hs-bounds}. Its first bound therefore applies at every
stage; discarding the negative term and iterating for a surviving row gives
\[
 r_j^{(m)\,2}
 \le R_0^2\prod_{\ell>m}\left(1+\frac{5r_\ell^2}{N^2}\right)
 \le R_0^2e^{5\|A\|_{\HS}^2/N^2}.
\]
Together with \(\sum_mr_m^4\le(\max_mr_m^2)\sum_mr_m^2\), this proves
\eqref{eq:quartic-row-sum}.

To control individual entries, retain the negative term in the same row
inequality and sum it for a surviving row \(j\):
\[
 \frac12\sum_{\ell>m}(A_{j\ell}^{(\ell)})^2
 \le r_j^{(n)\,2}-r_j^{(m)\,2}
       +\frac5{N^2}\sum_{\ell>m}r_\ell^2r_j^{(\ell)\,2}
 \le CR_0^2.
\]
The exact entry recursion and Cauchy--Schwarz now give
\[
 |A_{jk}^{(m)}-A_{jk}|
 =\frac1N\left|\sum_{\ell>m}
        A_{j\ell}^{(\ell)}A_{k\ell}^{(\ell)}\right|
 \le\frac1N\left\{\sum_{\ell>m}(A_{j\ell}^{(\ell)})^2
                    \sum_{\ell>m}(A_{k\ell}^{(\ell)})^2\right\}^{1/2}
 \le\frac{CR_0^2}{N}.
\]
This proves \eqref{eq:recursive-max-entry}.
\end{proof}

\subsection{Probability of the recursive weighted-sum event}
\label{sec:adaptive-local-fields}

The probability estimate for $\mathcal G^{\mathrm{rec}}$ follows by first
controlling one deterministic row and then taking a union bound over the
deletion sequence.  The row threshold must be large enough to give the desired
tail probability, but small enough that $k_{m,\mathrm{ad}}r_m^2/N$, $k_{m,\mathrm{ad}}r_m^2/d$, and the
sum of the $U_m$ terms remain negligible.  We begin with a deterministic
vector $a\in\mathbb R^m$ and the feature sum $H(a)$ defined in
Section~\ref{sec:weighted-feature-sum-tails}.

Corollary~\ref{cor:local-field-mgf} bounds the MGF of
$\|H(a)\|_{\HS}^2-\|a\|_2^2$ on the covariance event when both
$\theta\sqrt\Lambda\|a\|_\infty\|a\|_2/d$ and
$\theta\|a\|_2^2/N$ are sufficiently small.
Let $C_0$ be the absolute constant in \eqref{eq:local-field-mgf}.  For a
prescribed tail level $t\ge1$ and $\|a\|_2>0$, choose
\[
  \theta(a)
  =c_0\min\left\{
  \frac N{\|a\|_2^2},
  \frac d{\sqrt\Lambda\|a\|_\infty\|a\|_2}
  \right\},
  \qquad
  k_{\mathrm{ad}}(a,t)
  =1+2C_0\frac{\theta(a)\|a\|_2^2}{N}
  +\frac{2t}{\theta(a)\|a\|_2^2}.
\]
The choice of $\theta(a)$ stays within the MGF range, while the two
correction terms in $k_{\mathrm{ad}}(a,t)$ control the MGF contribution and the target tail
probability $e^{-t}$, respectively.  Lemma~\ref{lem:adaptive-threshold} gives
the resulting tail bound and records the upper bound on $k_{\mathrm{ad}}(a,t)$ used in the
remainder estimates.

\begin{lemma}[Adaptive weighted-sum threshold]
\label{lem:adaptive-threshold}
Let $a$ be a deterministic coefficient vector with $\|a\|_2>0$, and let
$t\ge1$.  For the threshold $k_{\mathrm{ad}}(a,t)$ defined above, sufficiently small $c_0$
gives
\[
  \mathbb P\left\{
  \mathcal C_m(\Lambda),\
  \|H(a)\|_{\HS}^2\ge k_{\mathrm{ad}}(a,t)\|a\|_2^2
  \right\}\le e^{-t},
\]
and
\begin{equation}
  k_{\mathrm{ad}}(a,t)
  \le
  C\left[
  1+\frac tN+
  \frac{t\sqrt\Lambda\|a\|_\infty}{d\|a\|_2}
  \right].
  \label{eq:adaptive-K}
\end{equation}
\end{lemma}

\begin{proof}
The definition of \(\theta=\theta(a)\) verifies both MGF restrictions in
Corollary~\ref{cor:local-field-mgf} when \(c_0\) is sufficiently small.
For \(k=k_{\mathrm{ad}}(a,t)\), exponential Markov inequality gives
\[
\begin{aligned}
 \log\mathbb P\{\mathcal C_m(\Lambda),\|H(a)\|_{\HS}^2\ge k\|a\|_2^2\}
 &\le-\theta(k-1)\|a\|_2^2+C_0\theta^2\|a\|_2^4/N\\
 &=-C_0\theta^2\|a\|_2^4/N-2t\le-t.
\end{aligned}
\]
Finally, \(\theta\|a\|_2^2/N\le c_0\) and
\[
 \frac1{\theta\|a\|_2^2}
 =\frac1{c_0}\max\left\{\frac1N,
       \frac{\sqrt\Lambda\|a\|_\infty}{d\|a\|_2}\right\}.
\]
Substituting these bounds into \(k\) proves \eqref{eq:adaptive-K}.
\end{proof}

We now return to the recursion.  Recall the deterministic tail level $T$ from
\eqref{eq:adaptive-cutoff-parameters}.  For every nonzero recursive row
$a^{(m)}$, set $k_{m,\mathrm{ad}}=k_{\mathrm{ad}}(a^{(m)},T)$, and retain $k_{m,\mathrm{ad}}=1$ when $r_m=0$.
The required size of $K_0$ is verified in
Section~\ref{sec:symmetric-endpoint-closure}.  Apply
Lemma~\ref{lem:adaptive-threshold} with $t=T$ and $m-1$ coordinates to each
nonzero row: $H(a^{(m)})=H_m$ and $\|a^{(m)}\|_2^2=r_m^2$.
Because the covariance summands are positive semidefinite,
$\mathcal C_n(\Lambda)$ implies $\mathcal C_{m-1}(\Lambda)$ for every
$2\le m\le n$.  Hence Lemma~\ref{lem:adaptive-threshold} and a union bound give
\begin{equation}
  \mathbb P\{\mathcal C_n(\Lambda)\cap(\mathcal G^{\mathrm{rec}})^c\}
  \le\sum_{m=2}^n\mathbb P\{\mathcal C_{m-1}(\Lambda)\cap\mathcal G_m^c\}
  \le ne^{-T}.
  \label{eq:all-bad-rows-probability}
\end{equation}
Thus the thresholds $k_{m,\mathrm{ad}}$ make the complement of
$\mathcal G^{\mathrm{rec}}$ cost at most $ne^{-T}$ on
$\mathcal C_n(\Lambda)$.  We next verify that these thresholds are also small
enough to make the deterministic remainder majorant summable.

\subsection{Order of the deterministic remainder majorant}
\label{sec:adaptive-quartic-sum}

We sum the two components of $\mathcal E_m=C_{\mathrm{rem}}(U_m+V_m)$
and verify the pointwise smallness required for the conditional moments.
Squaring \eqref{eq:adaptive-K}, using $(x+y+z)^2\le3(x^2+y^2+z^2)$,
and summing with $\sum_mr_m^2\le4\|A\|_{\HS}^2/5$ gives
\begin{equation}
  \sum_mk_{m,\mathrm{ad}}^2\left(\frac{r_m^2}{N}\right)^2
  \le C\left\{\left(1+\frac{T^2}{N^2}\right)
  \sum_m\left(\frac{r_m^2}{N}\right)^2
  +\frac{T^2\Lambda M_\ast^2\|A\|_{\HS}^2}{d^2N^2}\right\}.
  \label{eq:adaptive-quartic-sum}
\end{equation}

\begin{lemma}[Deterministic remainder scales at the cubic scale]
\label{lem:cubic-adaptive-remainder}
Under the standing assumptions \eqref{eq:dense-master-regime}, on the event
in Proposition~\ref{prop:row-gram-environment},
\begin{equation}
  \sum_m U_m=o(1),
  \qquad
  \max_m\frac{k_{m,\mathrm{ad}}r_m^2}{d}=o(1),
  \qquad
  \sum_m V_m=o(1).
  \label{eq:adaptive-remainder-small}
\end{equation}
Consequently,
\begin{equation}
  \sum_{m=2}^n\mathcal E_m=o(1).
  \label{eq:even-remainders}
\end{equation}
In particular, \(\max_mk_{m,\mathrm{ad}}r_m^2/N=o(1)\).
\end{lemma}

The proof is given in Section~S3.1 of the Supplementary Material.

\subsection{Comparison with the conditional remainder}
\label{sec:adaptive-exponential-remainders}

The order of the deterministic majorant was established in
\eqref{eq:even-remainders}.  To compare it with the random conditional
remainder in \eqref{eq:even-one-step}, we first establish in
Lemma~\ref{lem:square-exponential-moments} a square-exponential moment bound
for a generic linear form $D=\langle H,q\rangle_{\HS}$.  We then apply that
bound to $D_m=\langle H_m,q_m\rangle_{\HS}$ on $\mathcal G_m$.

\begin{lemma}[Square-exponential weighted moments]
\label{lem:square-exponential-moments}
Let $q\sim\bar\mu$, let $H\in V$ be deterministic, set
$D=\langle H,q\rangle_{\HS}$, and put
\[
  \sigma^2=\frac{\|H\|_{\HS}^2}{N}.
\]
For every fixed $K>0$ and integer $p\ge0$, there are constants
$c>0$ and $C_{K,p}<\infty$ such that, if
\[
  K\sigma^2\le c,\qquad
  \frac{K\|H\|_{\mathrm{op}}^2}{d}\le c,
\]
then
\[
  \mathbb E\left[|D|^pe^{KD^2}\right]
  \le C_{K,p}\sigma^p.
\]
\end{lemma}

The proof is given in Section~S2.2 of the Supplementary Material.

Set
\[
  Q_m=\sum_{j<m}(a_j^{(m)})^2B_{jm}^2,\qquad W_m=(D_m^2-Q_m)/2.
\]
On $\mathcal G_m$, both $\|H_m\|_{\HS}^2$ and
$\|H_m\|_{\mathrm{op}}^2$ are at most $k_{m,\mathrm{ad}}r_m^2$.
Lemma~\ref{lem:cubic-adaptive-remainder} therefore verifies the hypotheses of
Lemma~\ref{lem:square-exponential-moments}, uniformly in $m$.  Thus, for
each fixed integer $s\ge0$ and $\kappa>0$,
\[
  \mathbb E_m[|D_m|^se^{\kappa D_m^2}]
  \le C_{\kappa,s}(k_{m,\mathrm{ad}}r_m^2/N)^{s/2}.
\]
Applying \eqref{eq:local-quadratic-tail} with $H=q_j$ and integrating gives
$\mathbb E_m|B_{jm}|^{2p}\le C_p(N^{-p}+d^{-2p})\le C_pN^{-p}$.
Minkowski's inequality therefore yields, for every fixed integer $p\ge1$,
\begin{equation}
  \|Q_m\|_{L^p(q_m)}\le C_pr_m^2/N,\qquad
  \|Q_m^P\|_{L^p(q_m)}\le C_pp_m/N.
  \label{eq:compensation-Q-moments}
\end{equation}
Since $e^{(W_m)_+}\le e^{D_m^2/2}$, Cauchy--Schwarz combines these bounds
into the following estimate for fixed nonnegative integers $r,s$:
\[
\begin{aligned}
  \mathbb E_m[|D_m|^rQ_m^se^{(W_m)_+}]
  &\le\bigl(\mathbb E_m[|D_m|^{2r}e^{D_m^2}]\bigr)^{1/2}
  \bigl(\mathbb E_mQ_m^{2s}\bigr)^{1/2}\\
  &\le C_{r,s}(k_{m,\mathrm{ad}}r_m^2/N)^{r/2}(r_m^2/N)^s.
\end{aligned}
\]
Taking $(r,s)=(4,0),(0,2),(5,0),(3,1),(1,2)$ and using
$W_m^2\le C(D_m^4+Q_m^2)$, $k_{m,\mathrm{ad}}\ge1$, and
$\max_m k_{m,\mathrm{ad}}r_m^2/N=o(1)$ proves
\begin{equation}
\begin{split}
  \mathbb E_m\big[(&W_m^2+D_m^4+Q_m^2+|D_m|^5+|D_m|^3Q_m+|D_m|Q_m^2)
  e^{(W_m)_+}\big]  \le Ck_{m,\mathrm{ad}}^2(r_m^2/N)^2.
\end{split}
  \label{eq:GER-local}
\end{equation}
Its sum is $o(1)$ by Lemma~\ref{lem:cubic-adaptive-remainder}.

For the actual remainder in \eqref{eq:even-one-step},
$(W_m^P)^2e^{(W_m^P)_+}\le C(D_m^4+(Q_m^P)^2)e^{D_m^2/2}$ gives
\[
  \widetilde{\mathcal E}_m\le CU_m
  +C\|Q_m^P\|_{L^4(q_m)}^2(\mathbb E_me^{D_m^2})^{1/2}
  \le C(U_m+V_m)\le\mathcal E_m
\]
on $\mathcal G_m$, after fixing $C_{\mathrm{rem}}$ sufficiently large.
Together with \eqref{eq:even-remainders}, this supplies the required
summable deterministic bound.
\section{Estimates for interpolation from the symmetric prior}
\label{app:complete-buffered-cavity}

This appendix supplies the cutoff, local-remainder and cubic moment bounds
used in Section~\ref{sec:likelihood-interpolation}. We use the fixed first
replica, dense-regime assumptions and uniformity convention of
Section~\ref{sec:dense-likelihood}. The final differential inequality and
likelihood comparison are proved in the main text.

\subsection{The covariance-boundary remainder}
\label{sec:covariance-boundary-proof}
\begin{proof}[Proof of Lemma~\ref{lem:buffered-cavity-complete}]
We derive the identity, bound its two boundary errors by the same
conditional integral, and establish a second-moment estimate valid for
every cutoff in the finite-parameter assertion.  Write $\mathbb E_{-i}$
for expectation over $q_{-i}\sim\mu_t^{\otimes(n-1)}$.  Differentiation gives
\[
  Z'(t)=\sum_i\mathbb E_{-i}\int
  \mathbf1_{\mathcal C_n(\Lambda+1)}e^{K_A}\,d\eta(q_i).
\]
On $\mathcal C_{-i}(\Lambda)$, conditional integration and
$O_i=Z_i^{\mathrm{loc}}(t)\rho_i(t)$ give
\[
  \mathbb E_{-i}[\mathbf1_{\mathcal C_{-i}(\Lambda)}e^{K_{-i}}O_i]
  =\mathbb E_{\mu_t^{\otimes n}}
  [\mathbf1_{\mathcal C_{-i}(\Lambda)}e^{K_A}\rho_i(t)].
\]
Set $\mathcal E_i^{\mathrm{cut}}=
\mathcal C_n(\Lambda+1)\cap\mathcal C_{-i}(\Lambda)^c$.
Subtracting the last identity from the two full-cutoff integrals yields
\begin{equation}
\begin{aligned}
  \operatorname{Rem}_{\mathrm{cov}}(t)
  ={}&\sum_i\mathbb E_{-i}\int
  \mathbf1_{\mathcal E_i^{\mathrm{cut}}}e^{K_A}\,d\eta(q_i)\\
  &-\sum_i\mathbb E_{\mu_t^{\otimes n}}
  [\mathbf1_{\mathcal E_i^{\mathrm{cut}}}e^{K_A}\rho_i(t)].
\end{aligned}
  \label{eq:covariance-remainder-decomposition}
\end{equation}
Since $\mathcal E_i^{\mathrm{cut}}\subseteq
\mathcal C_{-i}(\Lambda+1)\cap\mathcal C_{-i}(\Lambda)^c$,
$|\eta|\le\bar\mu$, and
$|\rho_i(t)|Z_i^{\mathrm{loc}}(t)=|O_i|
\le\mathbb E_{\bar\mu,i}e^{D_i-Q_i/2}$, the absolute value of the
$i$-th summand in either sum is at most
\[
  \mathbb E_{-i}\left[
  \mathbf1_{\mathcal C_{-i}(\Lambda+1)\cap\mathcal C_{-i}(\Lambda)^c}
  e^{K_{-i}}\mathbb E_{\bar\mu,i}e^{D_i-Q_i/2}\right].
\]

For its second moment, Jensen's inequality and $Q_i\ge0$ give
\[
  \left(e^{K_{-i}}\mathbb E_{\bar\mu,i}e^{D_i-Q_i/2}\right)^2
  \le e^{2K_{-i}}\mathbb E_{\bar\mu,i}e^{2D_i}.
\]
Let $H=\sum_{j\ne i}A_{ij}q_j$.  On $\mathcal C_{-i}(\Lambda+1)$,
Lemma~\ref{lem:cutoff-field-bound} gives
$\|H\|_{\mathrm{HS}}\le\sqrt{\Lambda+1}\|A_{i\cdot}\|_2$.
The row hypothesis on $2A$ therefore verifies
Lemma~\ref{lem:local-subgaussian} for $2H$, and
\[
  \mathbb E_{\bar\mu,i}e^{2D_i}
  \le\exp\{C(\Lambda+1)\|A_{i\cdot}\|_2^2/N\}.
\]
Principal submatrices do not increase the row, Hilbert--Schmidt, or
operator norms.  Thus Theorem~\ref{thm:l2-spectral-recursion}, applied to
$2A_{-i,-i}$ after dropping the nonpositive compensation, gives
$\mathbb E_{-i}[\mathbf1_{\mathcal C_{-i}(\Lambda+1)}e^{2K_{-i}}]
\le\exp\{C\|A\|_{\HS}^2/N\}$.  Combining the estimates yields
\begin{equation}
\begin{aligned}
  &\mathbb E_{-i}\left[\mathbf1_{\mathcal C_{-i}(\Lambda+1)}
  \left(e^{K_{-i}}\mathbb E_{\bar\mu,i}e^{D_i-Q_i/2}\right)^2\right]\\
  &\qquad\le\exp\left\{\frac C N\|A\|_{\HS}^2
  +\frac{C(\Lambda+1)}N R_2(A)^2\right\}.
\end{aligned}
  \label{eq:covariance-remainder-second-moment}
\end{equation}
Cauchy--Schwarz, with boundary probability at most $p_\Lambda$, bounds
each summand of \eqref{eq:covariance-remainder-decomposition} by the
square root of this expression times $p_\Lambda^{1/2}$.  Summing the
$2n$ terms proves \eqref{eq:buffered-cavity-general-remainder}.

For the dense-regime parameters, \eqref{eq:dense-recursion-smallness}
verifies the spectral hypotheses uniformly in $i$:
\begin{equation}
  \frac{\|2H\|_{\mathrm{op}}}{d}
  \le\frac{2\sqrt{\Lambda+1}R_2(A)}d=o(1),
  \label{eq:local-field-spectral-small}
\end{equation}
\begin{equation}
  \frac{\|2A_{-i,-i}\|_{\mathrm{op}}}{N}
  \le\frac{2\|A\|_{\HS}}N\lesssim h_n^{-2}=o(1).
  \label{eq:environment-op-small}
\end{equation}
It also gives $(\Lambda+1)R_2(A)^2/N=o(1)$.  With
\eqref{eq:covariance-tail-dense}, the finite-parameter bound becomes
\[
  |\operatorname{Rem}_{\mathrm{cov}}(t)|
  \le2n\exp\{C\|A\|_{\HS}^2/N-cT/2+o(1)\}=o(1),
\]
since $T\ge K_0\{\|A\|_{\HS}^2/N+\log(nN)\}$ and $K_0$ is sufficiently large.
Finally, for a $q_{-i}$-measurable event $\mathcal U_i$, domination by
$\bar\mu$ bounds the corresponding signed integral by
\[
  \mathbb E_{-i}[\mathbf1_{\mathcal C_{-i}(\Lambda+1)\cap\mathcal U_i}
  e^{K_{-i}}\mathbb E_{\bar\mu,i}e^{D_i-Q_i/2}].
\]
Applying Cauchy--Schwarz and \eqref{eq:covariance-remainder-second-moment}
with probability $e^{-u}$ gives
$\exp\{C\|A\|_{\HS}^2/(2N)-u/2+O(1)\}$ per coordinate.
Here the row smallness hypothesis and $d^2\asymp N$ bound the cutoff
contribution by an absolute constant. Summation proves the weighted-event assertion.
\end{proof}

\subsection{Two-sided local expansion}
\label{sec:two-sided-local-expansion}

Lemma~\ref{lem:buffered-cavity-complete} reduces the interpolation derivative
to the one-coordinate scores \(\rho_i(t)\).  This subsection identifies the
cubic part of each score and proves that the remaining terms are summable.
Lemma~\ref{lem:global-local-remainders} supplies deterministic Taylor bounds
for the odd numerator and the even denominator.
Lemma~\ref{lem:cavity-score-expansion} applies those bounds after localizing the
weighted feature sum associated with the deleted coordinate.  It extracts the
cubic term \(T_i\) and controls the accumulated score error.  The same
localization also provides the tail estimate needed to remove the exceptional
rows before the cubic terms are aggregated.

We next record the deterministic odd and even Taylor estimates used on
\(\mathcal G_i\).  For generic $D\in\mathbb R$ and $Q\ge0$, put
$W=(D^2-Q)/2$ and $W_+=\max(W,0)$.

\begin{lemma}[Global odd and even remainders]
\label{lem:global-local-remainders}
For every $D\in\mathbb R$ and $Q\ge0$,
\[
\begin{aligned}
  \left|
  e^{-Q/2}\sinh D-D-\frac16D^3+\frac12DQ
  \right|
  \le
  C(&|D|^5+|D|^3Q+|D|Q^2)e^{W_+},
\end{aligned}
\]
and
\[
  \left|
  e^{-Q/2}\cosh D-1-W
  \right|
  \le C(D^4+Q^2)e^{W_+}.
\]
\end{lemma}

The proof is given in Section~S2.3 of the Supplementary Material.

We now apply Lemma~\ref{lem:global-local-remainders} with
\((D,Q)=(D_i,Q_i)\).  The event
\(\mathcal G_i\) supplies the localized bounds used below to sum its
remainders.  The first nonzero odd term in the local expansion is cubic.
Integrating out
the distinguished feature converts this term into its third moment.  Recall
from Section~\ref{sec:main-local-score} that, for \(q_1,q_2,q_3\in V\),
\[
  \mathcal K_3(q_1,q_2,q_3)
  =
  \mathbb E_{q\sim\mu}
  \prod_{a=1}^3\langle q,q_a\rangle_{\HS}.
\]
Equation~\eqref{eq:third-moment-tensor} gives
\[
  \mathcal K_3(q_1,q_2,q_3)
  =
  \gamma_d\operatorname{tr}_{\mathrm{sym}}(q_1q_2q_3).
\]
Here
\(\gamma_d\) is the third-moment coefficient from
Lemma~\ref{lem:complete-feature-identities}, and
\(\gamma_d\asymp N^{-3/2}\).
Lemma~\ref{lem:cavity-score-expansion} uses this kernel to express the cubic
contribution to the one-coordinate score and bounds the remaining Taylor and
denominator terms.

\begin{proof}[Proof of Lemma~\ref{lem:cavity-score-expansion}]
We first verify the local moment hypotheses at level $T_{\mathrm{loc}}$,
then identify the cubic term and control division by the local partition
function.  Since $T_{\mathrm{loc}}\lesssim T\log n$, the calculation in
Lemma~\ref{lem:cubic-adaptive-remainder} applies to the original rows,
with $\Lambda+1$ in place of $\Lambda$.  Indeed, squaring
\eqref{eq:adaptive-K} and summing gives
\[
\begin{aligned}
  \sum_i\left(\frac{k_{i,\mathrm{ad}}\|A_{i\cdot}\|_2^2}{N}\right)^2
  \le C\left\{
  \left(1+\frac{T_{\mathrm{loc}}^2}{N^2}\right)
  \frac{R_0^2\|A\|_{\HS}^2}{N^2}
  +\frac{T_{\mathrm{loc}}^2(\Lambda+1)\|A\|_{\max}^2\|A\|_{\HS}^2}{d^2N^2}
  \right\}.
\end{aligned}
\]
The part independent of $T_{\mathrm{loc}}$ is $O(\beta^2/h_n^4)$.
Every term involving $T$ in the proof of
\textup{(S3.5)}, including the preceding weighted
quartic estimate, has a negative power of $n$ after
$1\le n\beta^2\le n$ is used.  Multiplication by $\log^2n$ preserves its
convergence to zero.  The pointwise calculation following
\textup{(S3.6)} likewise survives multiplication of
the $T$ terms by $\log n$.  Consequently,
\[
  \max_i k_{i,\mathrm{ad}}\|A_{i\cdot}\|_2^2/d=o(1),
\]
and
\begin{equation}
  \max_i\frac{k_{i,\mathrm{ad}}\|A_{i\cdot}\|_2^2}{N}=o(1),\qquad
  \sum_i\left(\frac{k_{i,\mathrm{ad}}\|A_{i\cdot}\|_2^2}{N}\right)^2=o(1).
  \label{eq:local-row-scales}
\end{equation}
For the finite-parameter assertion, assume instead
\eqref{eq:local-score-finite-conditions}. On $\mathcal G_i$,
$\|H_i\|_{\mathrm{op}}^2\le\|H_i\|_{\HS}^2
\le k_{i,\mathrm{ad}}\|A_{i\cdot}\|_2^2$, so these conditions verify
Lemma~\ref{lem:square-exponential-moments} for all fixed moment orders
used below. The following calculation applies in both cases.

Integration against $\eta$ extracts the odd part, whereas $\bar\mu$
extracts the even part.  Thus
\[
  O_i=\mathbb E_{\mu,i}[e^{-Q_i/2}\sinh D_i],\qquad
  E_i:=\mathbb E_{\mu,i}[e^{-Q_i/2}\cosh D_i],\qquad
  Z_i^{\mathrm{loc}}(t)=E_i+tO_i.
\]
Since $\mathbb E_{\mu,i}D_i=0$, Lemma~\ref{lem:global-local-remainders}
and the conditional moment calculation \eqref{eq:GER-local} give
\[
\begin{gathered}
  O_i=\mathbb E_{\mu,i}[D_i^3/6-D_iQ_i/2]
  +\operatorname{Rem}_i^{\mathrm{odd}},\\
  E_i=1+\frac{\|H_i\|_{\HS}^2-\|A_{i\cdot}\|_2^2}{2N}
  +\operatorname{Rem}_i^{\mathrm{even}},\\
  |\operatorname{Rem}_i^{\mathrm{odd}}|+|\operatorname{Rem}_i^{\mathrm{even}}|
  \le C(k_{i,\mathrm{ad}}\|A_{i\cdot}\|_2^2/N)^2.
\end{gathered}
\]
The moment majorants are even in $q_i$, so their $\bar\mu$ and $\mu$
integrals agree.  The quadratic mean follows from isotropy in
Lemma~\ref{lem:complete-feature-identities}.

The cubic mean is \(T_i\) by the cancellation derived in
Section~\ref{sec:main-local-score}. To bound it, retain the unexpanded form:
\[
  T_i=\tfrac16\mathcal K_3(H_i,H_i,H_i)
  -\tfrac12\sum_{j\ne i}A_{ij}^2\mathcal K_3(H_i,q_j,q_j).
\]
Write $q_j=\varepsilon_j\varphi_d(X_j)$, $\varepsilon_j\in\{-1,1\}$.
By \eqref{eq:q-square}, $q_j^2=I/d+b_d\varepsilon_jq_j$ with $|b_d|\le1$.
Since $H_i$ is traceless, \eqref{eq:third-moment-tensor} gives
$\mathcal K_3(H_i,q_j,q_j)=\gamma_db_d\varepsilon_j\langle H_i,q_j\rangle_{\HS}$.
Using $|\gamma_d|\lesssim N^{-3/2}$ and $\|q_j\|_{\HS}=1$, we obtain
\[
  |T_i|\le CN^{-3/2}\left(\|H_i\|_{\HS}^3
  +\|A_{i\cdot}\|_2^2\|H_i\|_{\HS}\right)
  \le C(k_{i,\mathrm{ad}}\|A_{i\cdot}\|_2^2/N)^{3/2}
\]
on $\mathcal G_i$, since $k_{i,\mathrm{ad}}\ge1$.

The denominator expansion is
\begin{equation}
  Z_i^{\mathrm{loc}}(t)-1
  =\frac{\|H_i\|_{\HS}^2-\|A_{i\cdot}\|_2^2}{2N}
  +\operatorname{Rem}_i^{\mathrm{even}}+tT_i+t\operatorname{Rem}_i^{\mathrm{odd}}.
  \label{eq:local-denominator-expansion}
\end{equation}
The preceding estimates imply
$|Z_i^{\mathrm{loc}}(t)-1|\le Ck_{i,\mathrm{ad}}\|A_{i\cdot}\|_2^2/N
\le Cc_\ast$ uniformly on $\mathcal G_i$. Taking $c_\ast$ sufficiently
small gives $Z_i^{\mathrm{loc}}(t)\ge1/2$.
Therefore
\[
\begin{aligned}
  |\rho_i(t)-T_i|
  &=\left|\frac{\operatorname{Rem}_i^{\mathrm{odd}}
  -T_i(Z_i^{\mathrm{loc}}(t)-1)}{Z_i^{\mathrm{loc}}(t)}\right|\\
  &\le C(k_{i,\mathrm{ad}}\|A_{i\cdot}\|_2^2/N)^2,
\end{aligned}
\]
which implies \eqref{eq:score-error}; its sum is $o(1)$ by
\eqref{eq:local-row-scales}.  The additional repeated-index scale in the
stated bound is also summable: \eqref{eq:environment-FR},
\eqref{eq:environment-M}, and \textup{(S3.1)} give
\begin{equation}
  \sum_i\frac{\|A\|_{\max}\|A_{i\cdot}\|_2^2}{N^{3/2}}
  =\frac{\|A\|_{\max}\|A\|_{\HS}^2}{N^{3/2}}
  \lesssim\frac{|\beta|^{3/2}\log n}{h_n^3n^{1/4}}=o(1).
  \label{eq:local-repeated-scale}
\end{equation}
\end{proof}

The cubic contribution on exceptional rows admits a finite-parameter
bound. By \eqref{eq:third-moment-tensor} and
$|\operatorname{tr}(XYZ)|\le\|X\|_{\HS}\|Y\|_{\mathrm{op}}\|Z\|_{\HS}$,
the kernel is bounded by $CN^{-3/2}$ on the feature support. Hence
\eqref{eq:local-cubic-source} and
$\|A_{i\cdot}\|_1\le\sqrt{n-1}\|A_{i\cdot}\|_2$ give
\[
 |T_i|\le CN^{-3/2}\|A_{i\cdot}\|_1^3
 \le C\left(\frac nN\right)^{3/2}\|A_{i\cdot}\|_2^3.
\]
Since $\mathcal C_n(\Lambda+1)\subseteq\mathcal C_{-i}(\Lambda+1)$,
Cauchy--Schwarz and \eqref{eq:bad-row-probability} yield
\begin{equation}
 \sum_i\left\|\mathbf1_{\mathcal C_n(\Lambda+1)}
 \mathbf1_{\mathcal G_i^c}T_i\right\|_{p_{\mathrm H}}
 \le Cn\left(\frac nN\right)^{3/2}R_2(A)^3
 e^{-T_{\mathrm{loc}}/(2p_{\mathrm H})}.
 \label{eq:bad-row-cubic-Lp}
\end{equation}
Here $p_{\mathrm H}\ge2$, and all norms are under $\mu_t^{\otimes n}$;
the bound is uniform in $t\in[0,1]$. Together with the weighted score
estimate, it gives \eqref{eq:bad-row-finite-bound}, which controls the
replacement of the localized cubic sum by $\sum_iT_i$.

\subsection{Coefficient norms for the aggregated cubic term}
Recall \(g\) and \(\mathcal T_g\) from \eqref{eq:full-row-tensor}.
We define the coefficient norms used below.  For an array
\(u=(u_{jk\ell})_{j,k,\ell\in[n]}\), let
\[
  \|u\|_{\HS}:=\left(\sum_{j,k,\ell}u_{jk\ell}^2\right)^{1/2},
\]
and define the two partition norms by
\[
\begin{aligned}
  \|u\|_{1|23}
  &:={}
  \sup_{\substack{\|x\|_2\le1\\\|Y\|_{\HS}\le1}}
  \left|\sum_{j,k,\ell}u_{jk\ell}x_jY_{k\ell}\right|,\\
  \|u\|_{1|2|3}
  &:={}
  \sup_{\substack{\|x\|_2\le1,\ \|y\|_2\le1\\\|z\|_2\le1}}
  \left|\sum_{j,k,\ell}u_{jk\ell}x_jy_kz_\ell\right|.
\end{aligned}
\]
Here \(x,y,z\in\mathbb R^n\) and \(Y\in\mathbb R^{n\times n}\).
For later use, define the matrix and vector slices
\[
\begin{aligned}
  u_{j\cdot\cdot}&:=(u_{jk\ell})_{k,\ell\in[n]}
  \in\mathbb R^{n\times n},&
  u_{\cdot\cdot\ell}&:=(u_{jk\ell})_{j,k\in[n]}
  \in\mathbb R^{n\times n},\\
  u_{\cdot k\ell}&:=(u_{jk\ell})_{j\in[n]}\in\mathbb R^n.
\end{aligned}
\]
With this notation,
\[
  \|u\|_{1|23}
  =\sup_{\|x\|_2\le1}
  \left\|\sum_jx_ju_{j\cdot\cdot}\right\|_{\HS}.
\]
Thus \(\|u\|_{1|23}\) is the operator norm of the flattening associated with
the partition \(\{1\}|\{2,3\}\), whereas \(\|u\|_{1|2|3}\) is the injective
norm associated with the three singleton blocks.  These definitions apply to
the full array \(g\) in \eqref{eq:full-row-tensor}.  For the coefficient array
of \(\mathcal T_g\), set
\[
  g^\circ_{jk\ell}
  :=g_{jk\ell}\mathbf1\{j,k,\ell\ \text{are pairwise distinct}\}.
\]

\begin{lemma}[Full-row coefficient norms]
\label{lem:full-row-coefficient-norms}
The full coefficient array \(g\) and its restriction \(g^\circ\) satisfy
\[
\begin{aligned}
  \|g\|_{\HS}
  &\lesssim \left(\max_i\|A_{i\cdot}\|_2\right)\|A\|_{\HS}^2,&
  \|g\|_{1|23}
  &\lesssim \left(\max_i\|A_{i\cdot}\|_2\right)\|A\|_{\mathrm{op}}^2,\\
  \|g\|_{1|2|3}
  &\lesssim \left(\max_i\|A_{i\cdot}\|_2\right)\|A\|_{\mathrm{op}}^2,&
  \max_\ell\|g_{\cdot\cdot\ell}\|_{\HS}
  &\lesssim \|A\|_{\mathrm{op}}
  \left(\max_i\|A_{i\cdot}\|_2\right)^2,\\
  \max_\ell\|g_{\cdot\cdot\ell}\|_{\mathrm{op}}
  &\lesssim \|A\|_{\max}\|A\|_{\mathrm{op}}^2,&
  \max_{k,\ell}\|g_{\cdot k\ell}\|_2
  &\lesssim \|A\|_{\mathrm{op}}\|A\|_{\max}
  \max_i\|A_{i\cdot}\|_2.
\end{aligned}
\]
The same six inequalities hold with \(g^\circ\) in place of \(g\).
\end{lemma}

\begin{proof}
We bound the full array using its row decomposition and matrix slices, then
transfer the estimates to distinct indices.  Coordinate deletion handles the
Euclidean norms; the operator and partition norms require explicit diagonal
subtractions.
Set \(a_i=A_{i\cdot}\).  Since \(g=\sum_i a_i^{\otimes3}\), its
Hilbert--Schmidt norm satisfies
\[
\begin{aligned}
  \|g\|_{\HS}
  &\le\sum_i\|a_i^{\otimes3}\|_{\HS}
  =\sum_i\|a_i\|_2^3\\
  &\le\left(\max_i\|A_{i\cdot}\|_2\right)
  \sum_i\|A_{i\cdot}\|_2^2
  =\left(\max_i\|A_{i\cdot}\|_2\right)\|A\|_{\HS}^2.
\end{aligned}
\]

For the flattening norm, use
\(\sum_jx_jg_{j\cdot\cdot}=A^\top\operatorname{diag}(Ax)A\).  Therefore,
\[
\begin{aligned}
  \|g\|_{1|23}
  &=\sup_{\|x\|_2\le1}
  \|A^\top\operatorname{diag}(Ax)A\|_{\HS}\\
  &\le\|A\|_{\mathrm{op}}
  \sup_{\|x\|_2\le1}
  \left(\sum_i(Ax)_i^2\|A_{i\cdot}\|_2^2\right)^{1/2}\\
  &\le\left(\max_i\|A_{i\cdot}\|_2\right)\|A\|_{\mathrm{op}}
  \sup_{\|x\|_2\le1}\|Ax\|_2\\
  &\le\left(\max_i\|A_{i\cdot}\|_2\right)\|A\|_{\mathrm{op}}^2.
\end{aligned}
\]
For the injective norm, \(\|yz^\top\|_{\HS}=\|y\|_2\|z\|_2\), and hence
\[
\begin{aligned}
  \|g\|_{1|2|3}
  &=\sup_{\substack{\|x\|_2\le1,\ \|y\|_2\le1\\\|z\|_2\le1}}
  \left|\left\langle
  \sum_jx_jg_{j\cdot\cdot},yz^\top
  \right\rangle_{\HS}\right|\\
  &\le\sup_{\|x\|_2\le1}
  \left\|\sum_jx_jg_{j\cdot\cdot}\right\|_{\HS}
  =\|g\|_{1|23}\\
  &\le\left(\max_i\|A_{i\cdot}\|_2\right)\|A\|_{\mathrm{op}}^2.
\end{aligned}
\]

For a fixed \(\ell\), the matrix slice satisfies
\(g_{\cdot\cdot\ell}=A^\top\operatorname{diag}(A_{\cdot\ell})A\).  Its
Hilbert--Schmidt norm is bounded by
\[
\begin{aligned}
  \|g_{\cdot\cdot\ell}\|_{\HS}
  &\le\|A\|_{\mathrm{op}}
  \|\operatorname{diag}(A_{\cdot\ell})A\|_{\HS}\\
  &=\|A\|_{\mathrm{op}}
  \left(\sum_iA_{i\ell}^2\|A_{i\cdot}\|_2^2\right)^{1/2}\\
  &\le\|A\|_{\mathrm{op}}
  \left(\max_i\|A_{i\cdot}\|_2\right)\|A_{\cdot\ell}\|_2\\
  &\le\|A\|_{\mathrm{op}}
  \left(\max_i\|A_{i\cdot}\|_2\right)^2,
\end{aligned}
\]
where the last inequality uses the symmetry of \(A\).  The operator norm of
the same slice satisfies
\[
\begin{aligned}
  \|g_{\cdot\cdot\ell}\|_{\mathrm{op}}
  &\le\|A\|_{\mathrm{op}}^2
  \|\operatorname{diag}(A_{\cdot\ell})\|_{\mathrm{op}}\\
  &=\|A\|_{\mathrm{op}}^2\max_i|A_{i\ell}|
  \le\|A\|_{\max}\|A\|_{\mathrm{op}}^2.
\end{aligned}
\]
Finally, the vector slice satisfies
\[
\begin{aligned}
  \|g_{\cdot k\ell}\|_2
  &=\|A^\top(A_{\cdot k}\circ A_{\cdot\ell})\|_2\\
  &\le\|A\|_{\mathrm{op}}
  \|A_{\cdot k}\circ A_{\cdot\ell}\|_2\\
  &\le\|A\|_{\mathrm{op}}\|A\|_{\max}
  \|A_{\cdot k}\|_2
  \le\|A\|_{\mathrm{op}}\|A\|_{\max}
  \max_i\|A_{i\cdot}\|_2.
\end{aligned}
\]

It remains to transfer these estimates to \(g^\circ\).  Coordinate deletion
immediately gives
\[
  \|g^\circ\|_{\HS}\le\|g\|_{\HS},\qquad
  \|g^\circ_{\cdot\cdot\ell}\|_{\HS}
  \le\|g_{\cdot\cdot\ell}\|_{\HS},\qquad
  \|g^\circ_{\cdot k\ell}\|_2\le\|g_{\cdot k\ell}\|_2.
\]
For the operator norm of a matrix slice, let
  \(P_\ell=I-e_\ell e_\ell^\top\), where \(e_\ell\) is the \(\ell\)-th standard
  basis vector of \(\mathbb R^n\).  Since \(g_{\cdot\cdot\ell}\) is symmetric,
\[
\begin{aligned}
  g^\circ_{\cdot\cdot\ell}
  &=P_\ell\Bigl[
  g_{\cdot\cdot\ell}
  -\operatorname{diag}(g_{11\ell},\ldots,g_{nn\ell})
  \Bigr]P_\ell,\\
  \|g^\circ_{\cdot\cdot\ell}\|_{\mathrm{op}}
  &\le\|g_{\cdot\cdot\ell}\|_{\mathrm{op}}
  +\max_j|g_{jj\ell}|
  \le2\|g_{\cdot\cdot\ell}\|_{\mathrm{op}}.
\end{aligned}
\]

To control \(\|g^\circ\|_{1|23}\), define
\[
\begin{aligned}
  d^{12}_{jk\ell}&:=g_{jk\ell}\mathbf1\{j=k\},&
  d^{13}_{jk\ell}&:=g_{jk\ell}\mathbf1\{j=\ell\},\\
  d^{23}_{jk\ell}&:=g_{jk\ell}\mathbf1\{k=\ell\},&
  d^{123}_{jk\ell}&:=g_{jk\ell}\mathbf1\{j=k=\ell\}.
\end{aligned}
\]
Inclusion--exclusion gives
\[
  g^\circ=g-d^{12}-d^{13}-d^{23}+2d^{123}.
\]
For \(\|x\|_2\le1\), symmetry of \(g\) yields
\[
\begin{aligned}
  \left\|\sum_jx_jd^{12}_{j\cdot\cdot}\right\|_{\HS}^2
  &=\sum_{j,\ell}x_j^2g_{jj\ell}^2
  \le\max_j\|g_{\cdot jj}\|_2^2,\\
  \left\|\sum_jx_jd^{13}_{j\cdot\cdot}\right\|_{\HS}^2
  &\le\max_j\|g_{\cdot jj}\|_2^2,\\
  \left\|\sum_jx_jd^{123}_{j\cdot\cdot}\right\|_{\HS}^2
  &=\sum_jx_j^2g_{jjj}^2
  \le\max_j|g_{jjj}|^2.
\end{aligned}
\]
Let \(\operatorname{Diag}(M)\) denote the diagonal matrix with the same
diagonal entries as \(M\).  Then
\[
  \sum_jx_jd^{23}_{j\cdot\cdot}
  =\operatorname{Diag}\left(
  \sum_jx_jg_{j\cdot\cdot}
  \right),
  \qquad
  \|d^{23}\|_{1|23}\le\|g\|_{1|23}.
\]
Consequently,
\[
\begin{aligned}
  \|g^\circ\|_{1|23}
  &\le2\|g\|_{1|23}
  +4\max_{k,\ell}\|g_{\cdot k\ell}\|_2\\
  &\lesssim
  \left(\max_i\|A_{i\cdot}\|_2\right)\|A\|_{\mathrm{op}}^2,
\end{aligned}
\]
where the last step also uses \(\|A\|_{\max}\le\|A\|_{\mathrm{op}}\).
Finally,
\[
  \|g^\circ\|_{1|2|3}
  \le\|g^\circ\|_{1|23}
  \lesssim
  \left(\max_i\|A_{i\cdot}\|_2\right)\|A\|_{\mathrm{op}}^2.
\]
\end{proof}

\subsection{The maximum coefficient}
\label{sec:maximum-coefficient}

The statistic \(\mathcal T_g\) in \eqref{eq:full-row-aggregation} contains only
\(g_{jk\ell}\) with pairwise distinct indices.
Lemma~\ref{lem:full-row-coefficient-norms} bounds the global coefficient norms
used below.  The order-three moment inequality also needs the largest
coefficient.  We therefore set
\[
  G_\infty:=\max_{j,k,\ell\ \mathrm{pairwise\ distinct}}|g_{jk\ell}|.
\]
Conditioning on the three indexed first-replica features makes the summands
over the remaining indices independent and permits a Bernstein estimate.

The following tail-parameter form will also supply the stronger truncations
in Appendices~\ref{app:critical-likelihood-analysis} and \ref{app:quantitative-LAN}.
Its constants do not depend on the chosen tail level.

\begin{lemma}[Maximum cubic coefficient at a prescribed tail level]
\label{lem:gmax-tail-parameter}
Let \(q_1,\ldots,q_n\) be independent with law \(\mu\), and form
\(G_\infty(q)\) from \(A(q)=\beta^2B(q)\) as above.
There are absolute constants \(K,c,C>0\) such that, for
\(x\ge K\log(enN)\),
\[
 \mathbb P\left\{G_\infty(q)>C\beta^6
   \left[\frac{nx^{3/2}}{N^{9/4}}
        +\frac{\sqrt{nx}+x^4}{N^{3/2}}\right]\right\}
 \le Cn^3e^{-cx}.
\]
\end{lemma}

\begin{proof}
We condition on the three indexed features, bound the conditional mean,
and apply Bernstein's inequality after truncating the remaining summands.
Put \(u_x=\min\{1,C_{\max}x/d\}\). The lower bound
\(B_{ij}\ge-1/(d-1)\) and Lemma~\ref{lem:quadratic-inner-product-tail}
give
\[
 \mathbb P\{\max_{i<j}|B_{ij}|>u_x\}\le Cn^2e^{-cx}.
\]
Here \(C_{\max}\) is large enough that \(u_x\ge1/(d-1)\); when
\(u_x=1\), the event on the left is empty.

Fix distinct \(j,k,\ell\), set
\(\mathcal F_{jk\ell}=\sigma(q_j,q_k,q_\ell)\) and
\(I_{jk\ell}=[n]\setminus\{j,k,\ell\}\), and write
\(W_i=B_{ij}B_{ik}B_{i\ell}\). Conditionally on \(\mathcal F_{jk\ell}\),
these summands are independent and have mean
\(\mathcal K_3(q_j,q_k,q_\ell)\). Their sum is
\(\beta^{-6}g_{jk\ell}(q)\).
For \(r\in\{j,k,\ell\}\), the spherical quadratic-form formula
\cite[Section~1.5]{Muirhead1982} has sixth-moment denominator
\(d(d+2)\cdots(d+10)\). Its numerator is a fixed sum of trace products
of total degree six. Tracelessness eliminates singleton traces, and
Schatten H\"older bounds each remaining product by \(\|q_r\|_{\HS}^6=1\).
Consequently,
\[
 \mathbb E(|B_{ir}|^6\mid\mathcal F_{jk\ell})\le CN^{-3},\qquad
 \mathbb E(W_i^2\mid\mathcal F_{jk\ell})
 \le\prod_{r\in\{j,k,\ell\}}
       \mathbb E(|B_{ir}|^6\mid\mathcal F_{jk\ell})^{1/3}
 \le CN^{-3}.
\]

Write \(q_r=\varphi_d(X_r)\). If the three conditioned pairwise entries
are at most \(u_x\) in absolute value, then
\(\lvert\langle X_r,X_s\rangle\rvert\le C\sqrt{x/d}\), by
\(d\langle X_r,X_s\rangle^2=1+(d-1)B_{rs}\). Direct multiplication gives
\[
 \operatorname{tr}(q_jq_kq_\ell)
 =\left(\frac d{d-1}\right)^{3/2}\left\{
  \langle X_j,X_k\rangle\langle X_k,X_\ell\rangle\langle X_\ell,X_j\rangle
  -\frac1d\sum_{\{r,s\}\subset\{j,k,\ell\}}\langle X_r,X_s\rangle^2
  +\frac2{d^2}\right\}.
\]
The identity \eqref{eq:third-moment-tensor} and \(d\asymp\sqrt N\) imply
\begin{equation}
 \left|\sum_{i\in I_{jk\ell}}\mathbb E(W_i\mid\mathcal F_{jk\ell})\right|
 \le Cnx^{3/2}N^{-9/4}.
 \label{eq:fixed-M-coefficient-conditional-mean}
\end{equation}

Set \(\widetilde W_i=W_i\mathbf1\{
\max_{r\in\{j,k,\ell\}}|B_{ir}|\le u_x\}\). Then
\(\lvert\widetilde W_i\rvert\le Cx^3N^{-3/2}\) and the sum of their
conditional variances is at most \(CnN^{-3}\).
Rotational invariance makes the entry-tail bound uniform in each conditioned
feature. A union bound, with \(x\ge K\log(enN)\), gives
\begin{equation}
\begin{aligned}
 \mathbb P(\exists i:\widetilde W_i\ne W_i\mid\mathcal F_{jk\ell})
 &\le Ce^{-cx},\\
 \sum_{i\in I_{jk\ell}}
  |\mathbb E(\widetilde W_i-W_i\mid\mathcal F_{jk\ell})|
 &\le Ce^{-cx}.
\end{aligned}
\label{eq:fixed-M-coefficient-truncation}
\end{equation}
Both left sides are zero if \(u_x=1\); otherwise use
\(ne^{-cd u_x}\le Ce^{-cx}\) and \(|W_i|\le1\).
Conditional Bernstein gives
\begin{equation}
 \mathbb P\left(\left|\sum_{i\in I_{jk\ell}}
   \{\widetilde W_i-\mathbb E(\widetilde W_i\mid\mathcal F_{jk\ell})\}\right|
   >C\frac{\sqrt{nx}+x^4}{N^{3/2}}\,\middle|\,\mathcal F_{jk\ell}\right)
 \le2e^{-cx}.
 \label{eq:fixed-M-coefficient-bernstein}
\end{equation}
The four contributions satisfy the exact identity
\begin{equation}
\begin{aligned}
 \sum_iW_i={}&\sum_i(W_i-\widetilde W_i)
 +\sum_i\{\widetilde W_i-\mathbb E(\widetilde W_i\mid\mathcal F_{jk\ell})\}\\
 &+\sum_i\mathbb E(\widetilde W_i-W_i\mid\mathcal F_{jk\ell})
 +\sum_i\mathbb E(W_i\mid\mathcal F_{jk\ell}),
\end{aligned}
\label{eq:fixed-M-coefficient-decomposition}
\end{equation}
where the sums run over \(I_{jk\ell}\). The first term vanishes on the
simultaneous entry event. The truncation bias is absorbed because, after
increasing the absolute constant \(K\),
\begin{equation}
 e^{-cx}\le C\frac{\sqrt{nx}}{N^{3/2}}.
 \label{eq:fixed-M-truncation-absorption}
\end{equation}
Apply \eqref{eq:fixed-M-coefficient-conditional-mean} on that entry event,
combine the other three bounds, and take a union bound over fewer than
\(n^3\) triples. Multiplication by \(\beta^6\) proves the lemma.
\end{proof}

\begin{lemma}[Conditional Bernstein bound for \(G_\infty\)]
\label{lem:gmax}
With probability $1-o(1)$ over the first replica,
\begin{equation}
  G_\infty\lesssim\beta^6\left[\frac{n(\log n)^{3/2}}{N^{9/4}}+\frac{\sqrt{n\log n}+(\log n)^4}{N^{3/2}}\right].
  \label{eq:gmax-bound}
\end{equation}
\end{lemma}

\begin{proof}
Apply Lemma~\ref{lem:gmax-tail-parameter} with \(x=K'\log(en)\).
Under \eqref{eq:dense-master-regime}, \(N\lesssim n\beta^2\le n\), so
\(\log(enN)\lesssim\log(en)\). A sufficiently large absolute \(K'\)
makes the lemma applicable and its failure bound \(o(1)\), giving
\eqref{eq:gmax-bound}.
\end{proof}

\subsection{Partition norms of the cubic kernel}
\label{sec:cubic-kernel-partition-norms}

As recorded in the roadmap, the cubic estimate needed for the differential
inequality is
\[
  \sup_{0\le t\le1}
  \|\mathcal T_g\|_{L^{p_{\mathrm H}}(\mu_t^{\otimes n})}=o(1).
\]
Adamczak's order-three moment inequality uses seven mixed norm types.
Lemmas~\ref{lem:full-row-coefficient-norms}
and~\ref{lem:gmax} provide the coefficient bounds.  This subsection supplies
the feature-kernel bounds uniformly along the interpolation; the joint
three-singleton estimate is proved in
Lemma~\ref{lem:mixed-cubic-singletons}.  Since \(\mathcal K_3\) is symmetric, one
representative partition of each type suffices.  For functions normalized in
the indicated \(L^2(\mu_t)\) spaces, define the seven norms below.  In the
last four definitions, every displayed fixed variable ranges over the support
of \(\mu_t\).
\begingroup
\allowdisplaybreaks[4]
\begin{align*}
  \nu_{3,1}(t)
  &:={}
  \|\mathcal K_3\|_{L^2(\mu_t^{\otimes3})},\\
  \nu_{3,2}(t)
  &:={}
  \sup_{\substack{\|f\|_{L^2(\mu_t)}\le1\\
  \|F\|_{L^2(\mu_t^{\otimes2})}\le1}}
  \left|\mathbb E_{\mu_t^{\otimes3}}
  [\mathcal K_3(q_1,q_2,q_3)f(q_1)F(q_2,q_3)]\right|,\\
  \nu_{3,3}(t)
  &:={}
  \sup_{\substack{\|f_a\|_{L^2(\mu_t)}\le1\\a=1,2,3}}
  \left|\mathbb E_{\mu_t^{\otimes3}}
  [\mathcal K_3(q_1,q_2,q_3)\textstyle\prod_{a=1}^3f_a(q_a)]\right|,\\
  \nu_{2,1}(t)
  &:={}
  \sup_{q_3}\sup_{\|F\|_{L^2(\mu_t^{\otimes2})}\le1}
  \left|\mathbb E_{\mu_t^{\otimes2}}
  [\mathcal K_3(q_1,q_2,q_3)F(q_1,q_2)]\right|,\\
  \nu_{2,2}(t)
  &:={}
  \sup_{q_3}\sup_{\substack{\|f_a\|_{L^2(\mu_t)}\le1\\a=1,2}}
  \left|\mathbb E_{\mu_t^{\otimes2}}
  [\mathcal K_3(q_1,q_2,q_3)f_1(q_1)f_2(q_2)]\right|,\\
  \nu_{1,1}(t)
  &:={}
  \sup_{q_2,q_3}\sup_{\|f\|_{L^2(\mu_t)}\le1}
  \left|\mathbb E_{\mu_t}
  [\mathcal K_3(q_1,q_2,q_3)f(q_1)]\right|,\\
  \nu_{0,0}(t)
  &:={}
  \sup_{q_1,q_2,q_3}|\mathcal K_3(q_1,q_2,q_3)|.
\end{align*}
\endgroup

\begin{lemma}[Partition norms of the cubic kernel]
\label{lem:cubic-kernel-partition-norms}
Uniformly in \(t\in[0,1]\),
\begin{equation}
\begin{aligned}
  (\nu_{3,1}(t),\nu_{3,2}(t),\nu_{3,3}(t),\nu_{2,1}(t),\nu_{2,2}(t))
  &\lesssim (N^{-9/4},N^{-11/4},N^{-3},N^{-9/4},N^{-5/2}),\\
  (\nu_{1,1}(t),\nu_{0,0}(t))
  &\lesssim (N^{-2},N^{-3/2}).
\end{aligned}
  \label{eq:seven-kernel-norms}
\end{equation}
\end{lemma}

\begin{proof}
Let
\[
  M_3:=\mathbb E_{q\sim\mu}[q^{\otimes3}]\in V^{\otimes3}.
\]
We first bound the Hilbert--Schmidt, flattening, and injective norms of this
moment tensor.  We then contract each integrated feature coordinate using
isotropy; the remaining fixed coordinates have unit norm.  This yields all
seven kernel bounds uniformly along the interpolation.
Then
\(\mathcal K_3(q_1,q_2,q_3)
=\langle M_3,q_1\otimes q_2\otimes q_3\rangle\).
For independent \(q,q'\sim\mu\), Lemma~\ref{lem:complete-feature-identities}
and \(\operatorname{tr}(q^3)=b_d\), which follows from
\eqref{eq:q-square}, give
\begin{equation}
  \|M_3\|_{\HS}^2
  =\mathbb E\langle q,q'\rangle_{\HS}^3
  =\mathbb E\mathcal K_3(q,q,q)
  =\gamma_db_d
  \asymp N^{-3/2}.
  \label{eq:m3-hs-size}
\end{equation}
We use the same symbol \(M_3\) for the associated flattening
\(V\to V\otimes V\), defined by
\[
  \langle M_3H,H_2\otimes H_3\rangle
  :=\langle M_3,H\otimes H_2\otimes H_3\rangle,
  \qquad H,H_2,H_3\in V.
\]
For \(U\in O(d)\), let \(\operatorname{Ad}_U\) denote the orthogonal
conjugation action on
\(V\),
\[
  \operatorname{Ad}_U(H):=UHU^\top.
\]
Since \(\varphi_d(UX)=\operatorname{Ad}_U(\varphi_d(X))\) and \(UX\) has the
same law as
\(X\),
the law \(\mu\) is invariant under \(\operatorname{Ad}_U\).  The definition
of the
flattening then gives
\[
  M_3\operatorname{Ad}_U
  =(\operatorname{Ad}_U\otimes\operatorname{Ad}_U)M_3,
  \qquad U\in O(d).
\]
Taking adjoints and using the orthogonality of \(\operatorname{Ad}_U\) shows
that
\[
  (M_3^\ast M_3)\operatorname{Ad}_U
  =\operatorname{Ad}_U(M_3^\ast M_3),
  \qquad U\in O(d).
\]

The conjugation action is irreducible on \(V\).  To see this, let
\(\mathcal W\subseteq V\) be a nonzero invariant subspace and choose
\(0\ne H\in\mathcal W\).  An orthogonal conjugation diagonalizes \(H\), so
\(\mathcal W\) contains a nonzero traceless diagonal matrix
\(H_{\mathrm{diag}}\).  Let \(E_{ab}\) denote the \(d\times d\)
matrix whose \((a,b)\)-entry is \(1\) and whose other entries are zero.  For
two unequal diagonal entries of \(H_{\mathrm{diag}}\), subtracting the
conjugate obtained by
exchanging those entries produces a nonzero multiple of
\(E_{aa}-E_{bb}\).  Permutation conjugations therefore show that \(\mathcal W\)
contains every traceless diagonal matrix.  For a suitable rotation \(U\)
through \(\pi/4\) in the \((a,b)\)-coordinate plane,
\[
  \operatorname{Ad}_U(E_{aa}-E_{bb})=E_{ab}+E_{ba}.
\]
Hence \(\mathcal W\) contains every symmetric off-diagonal matrix unit.  These
matrices together with the traceless diagonal matrices span \(V\), so
\(\mathcal W=V\).

The operator \(M_3^\ast M_3\) is self-adjoint, and the preceding commutation
relation makes each of its eigenspaces invariant under every
\(\operatorname{Ad}_U\).
Irreducibility therefore forces \(M_3^\ast M_3=cI_V\) for some \(c\ge0\).
Taking traces and using
\(\operatorname{tr}(M_3^\ast M_3)=\|M_3\|_{\HS}^2\) gives
\begin{equation}
  M_3^\ast M_3=\frac{\|M_3\|_{\HS}^2}{N}I_V,
  \qquad
  \|M_3\|_{1|23}=\frac{\|M_3\|_{\HS}}{\sqrt N}\asymp N^{-5/4}.
  \label{eq:m3-flattening-size}
\end{equation}
Equation~\eqref{eq:third-moment-tensor} also gives the injective bound
\begin{equation}
  \|M_3\|_{1|2|3}
  =\sup_{\|H_a\|_{\HS}\le1}
  \gamma_d|\operatorname{tr}_{\mathrm{sym}}(H_1H_2H_3)|
  \le\gamma_d\lesssim N^{-3/2}.
  \label{eq:m3-injective-size}
\end{equation}

The laws \(\mu_t\) all have covariance \(I_V/N\).  For
\(f\in L^2(\mu_t)\) and \(F\in L^2(\mu_t^{\otimes2})\),
Cauchy--Schwarz and this covariance identity imply
\begin{equation}
\begin{aligned}
  \left\|\mathbb E_{\mu_t}[f(q)q]\right\|_{\HS}
  &\le N^{-1/2}\|f\|_{L^2(\mu_t)},\\
  \left\|\mathbb E_{\mu_t^{\otimes2}}
  [F(q_1,q_2)q_1\otimes q_2]\right\|_{\HS}
  &\le N^{-1}\|F\|_{L^2(\mu_t^{\otimes2})}.
\end{aligned}
  \label{eq:first-second-contraction-bounds}
\end{equation}
For independent \(q_1,q_2\sim\mu_t\), duality against
\(Y\in V\otimes V\) gives the second inequality:
\[
  \left|\mathbb E
  [F(q_1,q_2)\langle q_1\otimes q_2,Y\rangle]\right|
  \le\|F\|_2
  \left\{\mathbb E
  \langle q_1\otimes q_2,Y\rangle^2\right\}^{1/2}
  =N^{-1}\|F\|_2\|Y\|_{\HS}.
\]

We now evaluate the seven norms in their displayed order.  Integrating all
three coordinates and inserting the covariance operator in each coordinate,
then using \eqref{eq:m3-hs-size}, gives
\[
  \nu_{3,1}(t)^2=N^{-3}\|M_3\|_{\HS}^2\lesssim N^{-9/2}.
\]
For the two nontrivial partitions of all three coordinates, combine
\eqref{eq:m3-flattening-size}, \eqref{eq:m3-injective-size}, and
\eqref{eq:first-second-contraction-bounds} to obtain
\[
\begin{aligned}
  \nu_{3,2}(t)
  &\le N^{-3/2}\|M_3\|_{1|23}
  \lesssim N^{-11/4},\\
  \nu_{3,3}(t)
  &\le N^{-3/2}\|M_3\|_{1|2|3}
  \lesssim N^{-3}.
\end{aligned}
\]
Every \(q\) in the support of \(\mu_t\) is a quadratic feature or its
negative, and hence has Hilbert--Schmidt norm one by
Lemma~\ref{lem:complete-feature-identities}.  Symmetry of \(M_3\), together
with \eqref{eq:m3-flattening-size} and \eqref{eq:m3-injective-size}, therefore
gives
\begin{equation}
\begin{aligned}
  \|M_3(\cdot,\cdot,q_3)\|_{\HS}
  &\le\|M_3\|_{1|23},
  &\|M_3(\cdot,\cdot,q_3)\|_{\mathrm{op}}
  &\le\|M_3\|_{1|2|3},\\
  \|M_3(\cdot,q_2,q_3)\|_{\HS}
  &\le\|M_3\|_{1|2|3}.
\end{aligned}
  \label{eq:m3-fixed-coordinate-bounds}
\end{equation}
Combining \eqref{eq:m3-fixed-coordinate-bounds} with
\eqref{eq:first-second-contraction-bounds} yields
\[
\begin{aligned}
  \nu_{2,1}(t)
  &\le N^{-1}\sup_{q_3}\|M_3(\cdot,\cdot,q_3)\|_{\HS}
  \lesssim N^{-9/4},\\
  \nu_{2,2}(t)
  &\le N^{-1}\sup_{q_3}\|M_3(\cdot,\cdot,q_3)\|_{\mathrm{op}}
  \lesssim N^{-5/2}.
\end{aligned}
\]
Equations~\eqref{eq:m3-fixed-coordinate-bounds}
and~\eqref{eq:first-second-contraction-bounds} give the bound for
\(\nu_{1,1}(t)\), while \eqref{eq:m3-injective-size} and
\(\|q_a\|_{\HS}=1\) give the bound for \(\nu_{0,0}(t)\):
\[
\begin{aligned}
  \nu_{1,1}(t)
  &\le N^{-1/2}\sup_{q_2,q_3}\|M_3(\cdot,q_2,q_3)\|_{\HS}
  \lesssim N^{-2},\\
  \nu_{0,0}(t)
  &\le\|M_3\|_{1|2|3}\sup_{q_1,q_2,q_3}
  \prod_{a=1}^3\|q_a\|_{\HS}
  \lesssim N^{-3/2}.
\end{aligned}
\]
These are the seven bounds in \eqref{eq:seven-kernel-norms}.
\end{proof}

\subsection{Adamczak moment bound and exponent audit}
\label{sec:adamczak-exponent-audit}

We now combine the coefficient bounds from
Lemma~\ref{lem:full-row-coefficient-norms} and Lemma~\ref{lem:gmax} with
\eqref{eq:seven-kernel-norms}, and prove the mixed three-singleton bound
directly.  Decoupling makes the
canonical statistic compatible with Adamczak's inequality, and the subsequent
power count verifies that every resulting term is $o(1)$.

Write \(\mathbf i=(i_1,i_2,i_3)\in[n]^3\),
\(\mathbf q=(q_1,q_2,q_3)\), and
\(h_{\mathbf i}(\mathbf q)=g^\circ_{i_1i_2i_3}\mathcal K_3(\mathbf q)\).
For \(I\subseteq\{1,2,3\}\), let \(\mathcal P(I)\) be the set of partitions
of \(I\), and let \(\mathbb E_I\) integrate the coordinates \(\mathbf q_I\)
under \(\mu_t^{\otimes I}\).  With \(\mathbf i_{I^c}\) and
\(\mathbf q_{I^c}\) fixed, the mixed norm is
\cite[Definitions~2--3]{Adamczak2006}
\[
\begin{gathered}
  \|(h_{\mathbf i})_{\mathbf i_I}\|_{\mathcal J}
  :=\sup_{(f_J)_{J\in\mathcal J}}
  \left|\sum_{\mathbf i_I}\mathbb E_I
  \left[h_{\mathbf i}(\mathbf q)\prod_{J\in\mathcal J}f_J\right]\right|,
  \qquad \mathcal J\in\mathcal P(I),\\
  f_J=f_J(\mathbf i_J,\mathbf q_J),\qquad
  \sum_{\mathbf i_J}\mathbb E_J|f_J|^2\le1
  \quad(J\in\mathcal J).
\end{gathered}
\]
For \(I=\varnothing\), set \(\mathcal P(\varnothing)=\{\varnothing\}\)
and take the mixed norm to be \(|h_{\mathbf i}(\mathbf q)|\).
Since \(\mathbb E_{\mu_t}q=0\), the kernel is canonical.  Decoupling
\cite[Theorem~1]{deLaPenaMontgomerySmith1994} and
\cite[Theorem~4]{Adamczak2006} therefore give, for \(p\ge2\),
\[
  \|\mathcal T_g\|_{L^p(\mu_t^{\otimes n})}
  \lesssim
  \sum_{I\subseteq\{1,2,3\}}\sum_{\mathcal J\in\mathcal P(I)}
  p^{|I^c|+|\mathcal J|/2}
  \max_{\mathbf i_{I^c}}\sup_{\mathbf q_{I^c}}
  \|(h_{\mathbf i})_{\mathbf i_I}\|_{\mathcal J}.
\]
Here the feature suprema are over \(\operatorname{supp}(\mu_t)\).
The seven types \((|I|,|\mathcal J|)\), up to coordinate permutations,
and their \(p\)-factors are listed in
Table~\ref{tab:adamczak-factor-pairing}.

For \(|\mathcal J|\le2\), coefficient and kernel norms separate by
multiplicativity of \(L^2\) and Hilbert-space operator norms under tensor
products; the empty partition gives a pointwise bound.  For example,
\[
\begin{aligned}
  \left(\sum_{i,j,k}\mathbb E_{\mu_t^{\otimes3}}h_{ijk}^2\right)^{1/2}
  &=\|g^\circ\|_{\HS}\,\nu_{3,1}(t),\\
  \|(h_{ijk})\|_{\{1\}|\{2,3\}}
  &\le\|g^\circ\|_{1|23}\,\nu_{3,2}(t).
\end{aligned}
\]
The same identities with complementary indices and features fixed give the
other four bounds after taking maxima and suprema.  The remaining partition
\(\{\{1\},\{2\},\{3\}\}\) requires the following trilinear estimate.

\begin{lemma}[Mixed three-singleton norm]
\label{lem:mixed-cubic-singletons}
For every fixed coefficient matrix \(A\), the associated kernel array
\((h_{ijk})\) satisfies, uniformly in \(t\in[0,1]\),
\begin{equation}
  \|(h_{ijk})\|_{1|2|3}
  \le C N^{-3}R_2(A)\|A\|_{\mathrm{op}}^2.
  \label{eq:mixed-cubic-singletons}
\end{equation}
\end{lemma}

\begin{proof}
We first contract the feature coordinates, then use the row decomposition
of \(g\).  An auxiliary random partition of \([n]\) into three classes
removes repeated indices at the cost of an absolute factor.
For the three singleton blocks, the admissible test-function families are
\[
  \mathcal F_t:=\left\{
  f=(f_i^{(a)})_{\substack{i\in[n]\\a\in\{1,2,3\}}}:
  f_i^{(a)}\in L^2(\mu_t),\quad
  \sum_{i=1}^n\|f_i^{(a)}\|_{L^2(\mu_t)}^2\le1
  \quad(a=1,2,3)\right\}.
\]
We bound the defining expression for every \(f\in\mathcal F_t\);
no maximizing family is required.  For such a family, put
\[
  H_i^{(a)}=\sqrt N\,\mathbb E_{\mu_t}[f_i^{(a)}(q)q].
\]
By \eqref{eq:first-second-contraction-bounds},
\(\sum_i\|H_i^{(a)}\|_{\HS}^2\le1\).
By the mixed-norm definition, \eqref{eq:third-moment-tensor}, and
independence of \(q_1,q_2,q_3\sim\mu_t\),
\[
\begin{aligned}
  \|(h_{ijk})\|_{1|2|3}
  &=\sup_{f\in\mathcal F_t}
  \left|\sum_{i,j,k}g^\circ_{ijk}\,
    \mathbb E_{\mu_t^{\otimes3}}\!\left[
    \mathcal K_3(q_1,q_2,q_3)
    f_i^{(1)}(q_1)f_j^{(2)}(q_2)f_k^{(3)}(q_3)
    \right]\right|\\
  &=|\gamma_d|\sup_{f\in\mathcal F_t}
  \left|\sum_{i,j,k}g^\circ_{ijk}
    \operatorname{tr}_{\mathrm{sym}}\!\left(
    \mathbb E_{\mu_t}[f_i^{(1)}(q)q]\,
    \mathbb E_{\mu_t}[f_j^{(2)}(q)q]\,
    \mathbb E_{\mu_t}[f_k^{(3)}(q)q]
    \right)\right|\\
  &=|\gamma_d|N^{-3/2}\sup_{f\in\mathcal F_t}
  \left|\sum_{i,j,k}g^\circ_{ijk}
    \operatorname{tr}_{\mathrm{sym}}
    (H_i^{(1)}H_j^{(2)}H_k^{(3)})\right|.
\end{aligned}
\]
Fix an arbitrary \(f\in\mathcal F_t\) and introduce independent auxiliary labels
\(\omega_i\sim\operatorname{Unif}\{1,2,3\}\), so that
\[
  27\,\mathbb E_\omega
  \mathbf1\{\omega_i=1,\omega_j=2,\omega_k=3\}
  =\mathbf1\{i,j,k\text{ are pairwise distinct}\}.
\]
Thus the sum with \(g^\circ\) is 27 times the \(\omega\)-expectation
of the sum with \(g\) and with \(H_i^{(a)}\) replaced by
\(\mathbf1\{\omega_i=a\}H_i^{(a)}\).

For every realization of \(\omega\), define
\[
  \widehat H_\ell^{(a)}
  =\sum_i A_{\ell i}\mathbf1\{\omega_i=a\}H_i^{(a)}.
\]
Cauchy--Schwarz in each row gives, for every \(\ell\) and \(a\),
\[
  \|\widehat H_\ell^{(a)}\|_{\HS}
  \le\left(\sum_i A_{\ell i}^2\right)^{1/2}
  \left(\sum_i\mathbf1\{\omega_i=a\}\|H_i^{(a)}\|_{\HS}^2\right)^{1/2}
  \le R_2(A).
\]
Applying the scalar operator norm bound for \(A\) to each coordinate in an
orthonormal basis of \(V\), and then summing over the coordinates, gives
\[
\begin{aligned}
  \sum_\ell\|\widehat H_\ell^{(a)}\|_{\HS}^2
  &=\sum_\ell\left\|\sum_i
    A_{\ell i}\mathbf1\{\omega_i=a\}H_i^{(a)}\right\|_{\HS}^2\\
  &\le\|A\|_{\mathrm{op}}^2
    \sum_i\mathbf1\{\omega_i=a\}\|H_i^{(a)}\|_{\HS}^2
  \le\|A\|_{\mathrm{op}}^2.
\end{aligned}
\]
The identity \(g_{ijk}=\sum_\ell A_{\ell i}A_{\ell j}A_{\ell k}\)
and \(|\operatorname{tr}_{\mathrm{sym}}(UVW)|
\le\|U\|_{\HS}\|V\|_{\HS}\|W\|_{\HS}\) therefore imply
\[
  \left|\sum_\ell\operatorname{tr}_{\mathrm{sym}}
    (\widehat H_\ell^{(1)}\widehat H_\ell^{(2)}
     \widehat H_\ell^{(3)})\right|
  \le\max_\ell\|\widehat H_\ell^{(1)}\|_{\HS}
    \prod_{a=2}^3
    \left(\sum_\ell\|\widehat H_\ell^{(a)}\|_{\HS}^2\right)^{1/2}
  \le R_2(A)\|A\|_{\mathrm{op}}^2.
\]
This bound is uniform in \(f\in\mathcal F_t\) and \(\omega\).
Averaging over \(\omega\), taking the supremum over \(f\), and using
\(|\gamma_d|\le CN^{-3/2}\) proves \eqref{eq:mixed-cubic-singletons}
deterministically for each \(A\).
\end{proof}

The six factorized bounds and
Lemma~\ref{lem:mixed-cubic-singletons} give the seven estimates recorded in
Table~\ref{tab:adamczak-factor-pairing}.

\begin{table}[tbp]
\centering
\caption{Adamczak's seven norm types and their \(L^p\) weights.
The last two columns give a product upper bound for the mixed norm, up to an
absolute constant; the \((3,3)\) row uses
\eqref{eq:mixed-cubic-singletons}.}
\label{tab:adamczak-factor-pairing}
\begin{tabular}{@{}cccc@{}}
\toprule
$(|I|,|\mathcal J|)$ & $p$-factor & Coefficient factor & Kernel factor\\
\midrule
$(3,1)$ & $p^{1/2}$ & $\|g^\circ\|_{\HS}$ & $\nu_{3,1}$\\
$(3,2)$ & $p$ & $\|g^\circ\|_{1|23}$ & $\nu_{3,2}$\\
$(3,3)$ & $p^{3/2}$ & $R_2(A)\|A\|_{\mathrm{op}}^2$ & $N^{-3}$\\
$(2,1)$ & $p^{3/2}$ & $\max_\ell\|g^\circ_{\cdot\cdot\ell}\|_{\HS}$ & $\nu_{2,1}$\\
$(2,2)$ & $p^2$ & $\max_\ell\|g^\circ_{\cdot\cdot\ell}\|_{\mathrm{op}}$ & $\nu_{2,2}$\\
$(1,1)$ & $p^{5/2}$ & $\max_{k,\ell}\|g^\circ_{\cdot k\ell}\|_2$ & $\nu_{1,1}$\\
$(0,0)$ & $p^3$ & $G_\infty$ & $\nu_{0,0}$\\
\bottomrule
\end{tabular}
\end{table}

Substituting Lemmas~\ref{lem:full-row-coefficient-norms}
and~\ref{lem:cubic-kernel-partition-norms}, together with
\eqref{eq:mixed-cubic-singletons} and
\eqref{eq:gmax-bound}, into the seven estimates in
Table~\ref{tab:adamczak-factor-pairing} gives
\begin{equation}
\begin{aligned}
  \|\mathcal T_g\|_{L^{p_{\mathrm H}}(\mu_t^{\otimes n})}
  \lesssim{}&
  p_{\mathrm H}^{1/2}N^{-9/4}
  \left(\max_i\|A_{i\cdot}\|_2\right)\|A\|_{\HS}^2\\
  &+p_{\mathrm H}N^{-11/4}
  \left(\max_i\|A_{i\cdot}\|_2\right)\|A\|_{\mathrm{op}}^2\\
  &+p_{\mathrm H}^{3/2}N^{-3}
  \left(\max_i\|A_{i\cdot}\|_2\right)\|A\|_{\mathrm{op}}^2\\
  &+p_{\mathrm H}^{3/2}N^{-9/4}\|A\|_{\mathrm{op}}
  \left(\max_i\|A_{i\cdot}\|_2\right)^2\\
  &+p_{\mathrm H}^2N^{-5/2}\|A\|_{\max}\|A\|_{\mathrm{op}}^2\\
  &+p_{\mathrm H}^{5/2}N^{-2}\|A\|_{\mathrm{op}}\|A\|_{\max}
  \max_i\|A_{i\cdot}\|_2
  +p_{\mathrm H}^3N^{-3/2}G_\infty.
  \label{eq:full-row-adamczak}
\end{aligned}
\end{equation}
The estimate is uniform in $t$.

Substituting the bounds of Proposition~\ref{prop:row-gram-environment}
and Lemma~\ref{lem:gmax} into \eqref{eq:full-row-adamczak} gives
the nineteen terms listed in Table~S3.1
of the Supplementary Material. Section~S3.2
there verifies that each tends to zero, using
\(n\beta^2\gtrsim h_n^4\) for the two terms with negative powers of
\(|\beta|\). Consequently,
\begin{equation}
  \sup_{0\le t\le1}
  \|\mathcal T_g\|_{L^{p_{\mathrm H}}(\mu_t^{\otimes n})}=o(1).
  \label{eq:Tg-Lp-small}
\end{equation}
The repeated-index source also satisfies
\begin{equation}
  \sum_{i,j}N^{-3/2}|A_{ij}|^3
  \le N^{-3/2}\|A\|_{\max}\|A\|_{\HS}^2
  \lesssim\frac{|\beta|^6n^2\log n}{N^3}
  =\frac{|\beta|^{3/2}\log n}{h_n^3n^{1/4}}
  =o(1).
  \label{eq:repeated-index-small}
\end{equation}
Here the final convergence uses \(|\beta|\le1\), \(h_n\to\infty\), and
\(\log n=o(n^{1/4})\).
\section{Critical truncation and uniform replica bounds}
\label{app:critical-likelihood-analysis}

This appendix records the truncation input to the critical graph expansion.
The finite-parameter deletion and interpolation bounds from Sections
\ref{sec:dense-likelihood}--\ref{sec:likelihood-interpolation} remain
applicable at criticality: the local smallness conditions still hold,
while the cumulative remainder and trace potential need only stay bounded.
The detailed transfer, concentration and scale calculations are given in
Section~S6 of the Supplementary Material.
We retain the uniform moment bound and the choice of a common sequence
of prior events, followed by the overlap estimates used in the
signed-triangle CLT.
Throughout, \(0<a<b<\infty\), \(a\le\Delta_n\le b\), and \(N\to\infty\),
with the notation of Section~\ref{sec:critical-window}.

\subsection{A uniform replica bound and a common truncation}
\label{sec:critical-uniform-closure}

Lemma~S6.4 in the Supplementary
Material constructs events \(\mathcal E_{n,M}\) with
\(\mathbb P(\mathcal E_{n,M}^c)\le n^{-2M}\).
For each fixed \(M\), the moment estimates below have a bound independent
of \(M\), although the required lower bound on \(n\) may depend on it.
This permits a slowly increasing choice of \(M\).

\begin{lemma}[Uniform bound for the truncated replica generating function]
\label{lem:critical-generating-truncation}
Fix \(0<a<b<\infty\).  Under
\(a\le\Delta_n\le b\) and \(N\to\infty\), the events in
Lemma~S6.2 may be chosen so that, for some
sequence \(M_n\to\infty\),
\begin{equation}
  \mathbb P(q\notin\mathcal E_n)\le n^{-M_n}.
  \label{eq:critical-superpolynomial-event}
\end{equation}
For the resulting events \(\mathcal E_n\), there are constants
\(s_\star>1\), \(C_{a,b}<\infty\), and \(n_0<\infty\) such that
\begin{equation}
  \sup_{n\ge n_0}\sup_{0\le s\le s_\star}
  \mathcal Z_n^{\mathcal E}(s)\le C_{a,b}.
  \label{eq:quadratic-uniform-generating-bound}
\end{equation}
\end{lemma}

\begin{proof}
Fix an integer \(M\ge1\), set \(s_\star=3/2\), and take
\(\mathcal E_{n,M}\) and \(\Lambda_M\) from
Lemma~S6.4, with
\(x_M=K_{\mathrm{tail}}(M+5)\log(en)\) as in
\textup{(S6.24)}.
For \(q^\ast\in\mathcal E_{n,M}\), write \(B^\ast=B(q^\ast)\),
put \(A_s=s\beta^2B^\ast\), and
recall the interpolating moment
\[
 Z_s(t)=\mathbb E_{\mu_t^{\otimes n}}
 [\mathbf1_{\mathcal C_n(\Lambda_M+1)}e^{K_{A_s}}].
\]
All bounds below are uniform in this first replica and in
\(0\le s\le s_\star\), \(0\le t\le1\).
Write \(o_M(1)\) for deterministic errors tending to zero with \(M\) fixed.

\noindent\textit{Step 1. Transfer the critical moment bounds.}
Section~S6.4 verifies the deletion and score
hypotheses at cutoff \(\Lambda_M+1\) and establishes
\[
 \sup_{0\le s\le s_\star}Z_s(0)\le C_{a,b},\qquad
 |Z_s'(t)|\le C_{a,b}\max\{1,Z_s(t)\}.
\]
The constants are independent of \(M\), for all sufficiently large \(n\)
depending on \(M\). The accumulated row remainders are bounded by
\(C_{a,b}\beta^2+o_M(1)\), and the trace potential by
\(C_{a,b}+o_M(1)\); the local smallness and weighted covariance-boundary
terms still vanish. Gronwall's inequality therefore gives
\(\sup_{s,t}Z_s(t)\le C_{a,b}\), with a constant independent of \(M\).

\noindent\textit{Step 2. Compare products and choose the diagonal sequence.}
Every edge factor is positive for all sufficiently large \(n\), since
\(B_{ij}^\ast B_{ij}\ge-1/(d-1)\) and \(s_\star<2\).
Uniformly over \(q^\ast,q\in\mathcal E_{n,M}\), the sixth-row bounds and
Cauchy--Schwarz give
\[
 \sup_{s\le s_\star}\sum_{i<j}|(A_s)_{ij}B_{ij}|^3
 \le C_{a,b}\frac{\beta^6n(n+x_M^7)}{N^3}
 \le C_{a,b}\frac Nn\left(1+\frac{x_M^7}{n}\right)=o_M(1).
\]
The product assertion of Lemma~S6.1 now gives
\begin{equation}
 \sup_{0\le s\le s_\star}\mathcal Z_n^{\mathcal E_{n,M}}(s)\le C_{a,b}.
 \label{eq:fixed-M-stability-bound}
\end{equation}
The constant is independent of \(M\); only the required lower bound on
\(n\) depends on \(M\). Choose increasing integers \(n_M\) so that, for
\(n\ge n_M\), this bound, \textup{(S6.34)}, and all input
bounds above hold, with each \(o_M(1)\) error at most \(M^{-1}\).
Take nondecreasing \(M_n\to\infty\) with \(n\ge n_{M_n}\), and set
\(\mathcal E_n=\mathcal E_{n,M_n}\). Then
\eqref{eq:critical-superpolynomial-event} holds and
\begin{equation}
 \sup_{0\le s\le s_\star}
 \mathbb E[\mathbf1_{\mathcal E_n}(q^\ast)\mathbf1_{\mathcal E_n}(q)
       \prod_{i<j}\{1+s\beta^2B_{ij}^\ast B_{ij}\}]
 \le C_{a,b}.
 \label{eq:real-amplitude-coarse-bound}
\end{equation}
By \eqref{eq:quadratic-generating-function}, this is
\eqref{eq:quadratic-uniform-generating-bound}.
The verified recursive, covariance, and product estimates also retain all
conclusions of Lemma~S6.2 for these events.
\end{proof}
\subsection{Triangle-kernel overlap estimates}
\label{sec:triangle-overlap-estimates}

The following two estimates are used in the remainder calculation in
the proof of Lemma~\ref{lem:triangle-two-model-CLT}.

\begin{lemma}[Triangle-kernel overlap moments]
\label{lem:triangle-overlap-moments}
For independent \(q_1,q_2,q_3,q_4\sim\mu\), write
\(B_{ij}=\langle q_i,q_j\rangle_{\HS}\).  Then
\[
  \mathbb E(B_{12}^2B_{23}^2B_{31}^2)\le CN^{-3},
  \qquad
  \left|\operatorname{Cov}(B_{12}B_{23}B_{31},
  B_{12}B_{24}B_{41})\right|\le CN^{-4}.
\]
\end{lemma}

The proof is given in Section~S6.5
of the Supplementary Material.
\section{Estimates for the finite-sample likelihood bounds}
\label{app:quantitative-LAN}

This appendix collects the quantitative inputs used in
Section~\ref{sec:quantitative-main}: a common truncation event, the
second-moment and triangle approximation bounds, and the quantitative
projection lemma. The proofs of the truncation and approximation bounds
are in Section~S4 of the Supplementary Material;
the projection argument is given here. We retain the critical range
\eqref{eq:quant-critical-range}, error scale \eqref{eq:quant-master-error}
and constants from Section~\ref{sec:quantitative-main}.
Let \(c_\ast\) be the smallness constant
in Theorem~\ref{thm:l2-spectral-recursion}, decreased if necessary to
satisfy the pointwise hypotheses of the deletion estimates.

\subsection{A quantitative truncation event}
\label{subsec:quant-truncation}

We construct a latent event used throughout this appendix.  On this event, the
matrix and coefficient estimates hold
uniformly and the truncated replica generating function remains bounded.  The
complement has superpolynomially small probability, which is sufficient for
the graph expansion with a logarithmically growing number of edges.

Recall that \(q=(q_1,\ldots,q_n)\sim\mu^{\otimes n}\), and that \(B(u)\)
denotes the hollow Gram matrix of a latent configuration \(u\), as defined in
Appendix~\ref{app:complete-preliminaries}.  We also retain the notation
\(A(u)=\beta^2B(u)\) from Section~\ref{sec:dense-likelihood}.

\begin{lemma}[Quantitative prior truncation]
\label{lem:quant-critical-bounds}
There exist constants \(c_{a,b},C_{a,b}>0\) and \(s_\star>1\) such that, under
\eqref{eq:quant-critical-range}, whenever
\(\varepsilon_{n,N,\beta}\le c_{a,b}\), one can choose a measurable event
\(\mathcal E_n\subseteq V^n\) with the following simultaneous properties.
Uniformly over \(u\in\mathcal E_n\), the matrix \(A(u)=\beta^2B(u)\) satisfies
\begin{equation}
\begin{aligned}
  \|A(u)\|_{\HS}^2
  &\le C_{a,b}\frac{\beta^4n^2}N,\\
  \max_i\sum_{j\ne i}A_{ij}(u)^2
  &\le C_{a,b}\frac{\beta^4n}N,\\
  \max_i\sum_{j\ne i}|B_{ij}(u)|^6
  &\le C_{a,b}\frac n{N^3},\\
  \|A(u)\|_{\max}
  &\le C_{a,b}\beta^2
    \min\left\{1,\frac{\log^2(en)}{\sqrt N}\right\},\\
  \|A(u)\|_{\mathrm{op}}
  &\le C_{a,b}\beta^2\left\{
    \frac nN+\sqrt{\frac nN}\log(en)+\log^2(en)
  \right\}.
\end{aligned}
\label{eq:quant-matrix-bounds}
\end{equation}
For the same \(u\), let \(G_\infty(u)\) denote the maximum coefficient from
Section~\ref{sec:maximum-coefficient} evaluated at \(A(u)\).  Then
\begin{equation}
  G_\infty(u)
  \le C_{a,b}\beta^6\left[
    \frac{n\log^3(en)}{N^{9/4}}
    +\frac{\sqrt n\log(en)+\log^8(en)}{N^{3/2}}
  \right].
  \label{eq:quant-gmax-bound}
\end{equation}
Moreover, with \(\mathcal Z_n^{\mathcal E}\) defined by
\eqref{eq:quadratic-generating-function},
\begin{equation}
  \sup_{0\le s\le s_\star}\mathcal Z_n^{\mathcal E}(s)\le C_{a,b}.
  \label{eq:quant-uniform-generating-bound}
\end{equation}
Finally, the complement of \(\mathcal E_n\) satisfies
\begin{equation}
  \mathbb P(q\notin\mathcal E_n)
  \le C_{a,b}\exp\{-c_{a,b}\log^2(en)\}.
  \label{eq:quant-prior-failure}
\end{equation}
\end{lemma}

The event is \(\mathcal E_n=\mathcal E_{n,M}\) from
Lemma~S6.4, with
\(M\asymp\log(en)\). Substituting its concentration bounds
into Lemma~S6.1 verifies the
generating-function bound uniformly for \(0\le s\le3/2\).
The parameter checks and the proof are given in
Section~S4.1 of the Supplementary Material.

Corollary~\ref{cor:quant-polynomial-rate} specializes
Lemma~\ref{lem:quant-critical-bounds} to polynomial signal scaling and records
an explicit decay rate.  This corollary is used only after the general
finite-sample theorem has been proved.

\begin{corollary}[Rate under polynomial signal scaling]
\label{cor:quant-polynomial-rate}
Suppose, in addition, that for some fixed \(0\le r<1\) and
\(0<c_-\le c_+<\infty\),
\[
  c_-n^{-r}\le\beta^2\le c_+n^{-r}.
\]
Then \eqref{eq:quant-critical-range} implies
\(N\asymp n^{3(1-r)/4}\), with comparison constants depending only on
\(a,b,c_-,c_+\).  There are constants
\(\eta_r,C_{a,b,r,c_-,c_+}>0\) such that, for all sufficiently large \(n\),
\begin{equation}
  \varepsilon_{n,N,\beta}
  \le C_{a,b,r,c_-,c_+}
  n^{-\eta_r}.
  \label{eq:quant-polynomial-error}
\end{equation}
\end{corollary}

The proof is given in Section~S4.2 of the Supplementary Material.

\subsection{Quantitative graph expansion and second moment}
\label{subsec:quant-partition}

We make the graph expansion of Section~\ref{sec:critical-graph-expansion}
quantitative by retaining degrees up to $K\asymp\log N$.
The errors come from conditioning, nontriangle components, intersections
between triangles, and the higher-degree tail.

\begin{proposition}[Truncated second moment]
\label{prop:quant-second-moment}
Recall that \(L_n^{\mathcal E}=d\Pp_n^{\mathcal E_n}/d\Qq_n\), where
\(\mathcal E_n\) is the event in
Lemma~\ref{lem:quant-critical-bounds}.  Under
\eqref{eq:quant-critical-range}, if
\(\varepsilon_{n,N,\beta}\le c_{a,b}\), then
\begin{equation}
  \left|
  \mathbb E_{\Qq_n}(L_n^{\mathcal E})^2-e^{\Delta_n^2}
  \right|
  \le C_{a,b}\varepsilon_{n,N,\beta}.
  \label{eq:quant-second-moment}
\end{equation}
\end{proposition}

The proof is given in Section~S4.3 of the Supplementary Material.
\subsection{A quantitative two-model triangle theorem}
\label{subsec:quant-triangle}

Proposition~\ref{prop:quant-second-moment} controls the second moment of the
truncated likelihood ratio.  To obtain the finite-sample total-variation and
likelihood approximations, we also need Gaussian bounds for the signed
triangle under both graph laws.
For probability measures \(\nu,\eta\) on \(\mathbb R\) with finite first
moments, recall their \(1\)-Wasserstein distance as
\[
  d_W(\nu,\eta)
  :=\sup_{\|f\|_{\mathrm{Lip}}\le1}
  \left|\int f\,d\nu-\int f\,d\eta\right|.
\]
Here
\(\|f\|_{\mathrm{Lip}}=\sup_{x\ne y}|f(x)-f(y)|/|x-y|\).
Recall
from Section~\ref{subsec:growing-critical} that
\begin{equation}
  W_n
  =
  \frac{\operatorname{sgn}(\beta)}{\sqrt{\binom n3}}
  \sum_{i<j<k}Y_{ij}Y_{jk}Y_{ki}.
  \label{eq:quant-signed-triangle}
\end{equation}
Condition \eqref{eq:quant-critical-range} excludes \(\beta=0\).
Proposition~\ref{prop:quant-triangle} gives the two quantitative Gaussian
approximations that will be combined with
Proposition~\ref{prop:quant-second-moment}.

\begin{proposition}[Finite-sample signed-triangle approximation]
\label{prop:quant-triangle}
One has
\begin{equation}
\begin{split}
  d_W\left(
  \mathcal L_{\Qq_n}(W_n),\Normal(0,1)
  \right)
  &\le\frac Cn,\\
  d_W\left(
  \mathcal L_{\Pp_n}(W_n),\Normal(\Delta_n,1)
  \right)
  &\le C_{a,b}\varepsilon_{n,N,\beta}.
  \label{eq:quant-two-model-W}
\end{split}
\end{equation}
The second bound also holds, after changing \(C\), under
\(\Pp_n^{\mathcal E_n}\) whenever
\(\varepsilon_{n,N,\beta}\le c_{a,b}\), with \(\mathcal E_n\) as in
Lemma~\ref{lem:quant-critical-bounds}.
\end{proposition}

We apply the quantitative de Jong bound under the null law and,
conditionally on the latent features, under the alternative law.
The moment estimates in Lemma~\ref{lem:triangle-two-model-CLT}
control the fourth moments and influences; the remaining errors
come from deleting small-variance edges, restoring the normalization
and removing prior conditioning.
The proof is given in Section~S4.4 of the Supplementary Material.

\subsection{Quantitative likelihood projection}
\label{subsec:quant-projection}

Proposition~\ref{prop:quant-second-moment} bounds the second moment of the
conditioned likelihood ratio.  Proposition~\ref{prop:quant-triangle} controls
in \(1\)-Wasserstein distance the signed-triangle laws under the null and the
conditioned alternative.  Lemma~\ref{lem:quant-sufficiency} converts these
three quantitative inputs into total-variation and log-likelihood bounds.

\begin{lemma}[Quantitative likelihood-ratio projection]
\label{lem:quant-sufficiency}
Let \(P\ll Q\), \(L=dP/dQ\), and let \(W\) be real-valued.  Fix
\(\delta\in[0,b]\) and \(0<\eta\le e^{-1}\).  Suppose
\[
  \mathbb E_QL^2\le e^{\delta^2}+\eta,
\]
\[
  d_W(\mathcal L_Q(W),\Normal(0,1))\le\eta,
  \qquad
  d_W(\mathcal L_P(W),\Normal(\delta,1))\le\eta.
\]
Then
\begin{equation}
  \left|
  \TV(P,Q)-\{2\Phi(\delta/2)-1\}
  \right|
  \le C_b\eta^{1/4},
  \label{eq:quant-TV-abstract}
\end{equation}
and, for \(0<t\le1\),
\begin{equation}
  Q\left\{
  \left|
  \log L-\left(\delta W-\frac12\delta^2\right)
  \right|>t
  \right\}
  \le
  C_bt^{-2}\eta^{1/4}.
  \label{eq:quant-LAN-abstract}
\end{equation}
\end{lemma}

\begin{proof}
We approximate $L$ in $L^2(Q)$ by a Lipschitz truncation of the Gaussian
likelihood. This gives the TV bound; for LAN, we also use a lower bound
for the truncation on a central interval.
Put
\[
  \ell_\delta(w)=e^{\delta w-\delta^2/2},\qquad
  T_\eta=4b+1+\sqrt{\log(1/\eta)},\qquad
  g_\eta(w)=\ell_\delta(w)\chi_\eta(w),
\]
where $\chi_\eta$ is one on $[-T_\eta,T_\eta]$, zero outside
$[-T_\eta-1,T_\eta+1]$, and linear in between. The Lipschitz constants
of $g_\eta$ and $g_\eta^2$ are at most $C_be^{bT_\eta}$ and
$C_be^{2bT_\eta}$, respectively.

For $Z\sim\Normal(0,1)$, Gaussian tilting gives
\[
  \mathbb E\ell_\delta(Z)^2=e^{\delta^2},\qquad
  \mathbb Eg_\eta(Z+\delta)=\mathbb E[\ell_\delta(Z)^2\chi_\eta(Z)].
\]
Consequently, both
$e^{\delta^2}-\mathbb Eg_\eta(Z+\delta)$ and
$e^{\delta^2}-\mathbb Eg_\eta(Z)^2$ lie in the interval
\[
  [0,e^{\delta^2}\mathbb P\{|Z+2\delta|>T_\eta\}]
  \subseteq[0,C_b\eta^{1/2}].
\]
The Wasserstein assumptions and
$\eta e^{2bT_\eta}\le C_b\eta^{3/4}$ therefore imply
\[
  \mathbb E_Pg_\eta(W)\ge e^{\delta^2}-C_b\eta^{1/2},\qquad
  \mathbb E_Qg_\eta(W)^2\le e^{\delta^2}+C_b\eta^{3/4}.
\]
Together with the second-moment hypothesis, these bounds yield
\begin{equation}
  \mathbb E_Q(L-g_\eta(W))^2
  =\mathbb E_QL^2-2\mathbb E_Pg_\eta(W)+\mathbb E_Qg_\eta(W)^2
  \le C_b\eta^{1/2}.
  \label{eq:quant-L2-dominance}
\end{equation}

Since $w\mapsto|g_\eta(w)-1|$ has the same Lipschitz bound as $g_\eta$,
the triangle inequality gives
\[
\begin{aligned}
  \left|\TV(P,Q)-\frac12\mathbb E|\ell_\delta(Z)-1|\right|
  &\le\frac12\mathbb E_Q|L-g_\eta(W)|+C_b\eta e^{bT_\eta}
     +\frac12\mathbb E|g_\eta(Z)-\ell_\delta(Z)|\\
  &\le C_b\eta^{1/4}.
\end{aligned}
\]
Here the last term is at most
$\mathbb P\{|Z+\delta|>T_\eta\}/2\le C_b\eta^{1/2}$ by Gaussian
tilting. The identity
$\frac12\mathbb E|\ell_\delta(Z)-1|=2\Phi(\delta/2)-1$
proves \eqref{eq:quant-TV-abstract}.

For LAN, the $1$-Lipschitz ramp
$w\mapsto\min\{1,(|w|-T_\eta+1)_+\}$ gives
\[
  Q\{|W|>T_\eta\}\le\mathbb P\{|Z|>T_\eta-1\}+\eta
  \le C_b\eta^{1/2}.
\]
On $\{|W|\le T_\eta\}$ we have
$g_\eta(W)=\ell_\delta(W)\ge e^{-bT_\eta-b^2/2}$.
For $0<t\le1$, the logarithmic deviation in
\eqref{eq:quant-LAN-abstract} thus implies
$|L-g_\eta(W)|\ge c_bt e^{-bT_\eta}$, including when $L=0$.
Chebyshev's inequality and \eqref{eq:quant-L2-dominance} give
\[
\begin{aligned}
  Q\left\{\left|\log L-\left(\delta W-\frac12\delta^2\right)\right|>t\right\}
  &\le C_b\eta^{1/2}
  +C_bt^{-2}e^{2bT_\eta}\mathbb E_Q(L-g_\eta(W))^2\\
  &\le C_bt^{-2}\eta^{1/4},
\end{aligned}
\]
where $e^{2bT_\eta}\eta^{1/2}\le C_b\eta^{1/4}$. This proves
\eqref{eq:quant-LAN-abstract}.
\end{proof}

\end{appendix}

\begin{acks}[Acknowledgments]
Ziqian Lin is the corresponding author.

Generative AI tools were used to assist with manuscript editing,
organization, and polishing the presentation of proofs. The authors are
responsible for the mathematical arguments and all other content.
\end{acks}

\begin{supplement}
\stitle{Supplement to ``Sharp detection and triangle likelihood limits for quadratic spherical kernel graphs''}
\sdescription{Spherical spectral background, standard feature and matrix estimates, the truncated
likelihood identity, elementary weighted-moment and Taylor bounds, and
critical truncation constructions and parameter checks for the dense,
critical and finite-sample arguments, together with
the quantitative graph-expansion and signed-triangle approximation proofs.
The supplementary material is included after the references in this arXiv version.}
\end{supplement}

\bibliographystyle{imsart-number}
\bibliography{references}
\end{document}


\begin{frontmatter}
\title{Supplement to ``Sharp detection and triangle likelihood limits for quadratic spherical kernel graphs''}
\runtitle{Supplement: quadratic spherical geometry}
\begin{aug}
\author[A]{\fnms{Ziqian}~\snm{Lin}\ead[label=e1]{linziqian@fzu.edu.cn}}
\author[B]{\fnms{Huangang}~\snm{Lu}\ead[label=e2]{luhuangang@pku.edu.cn}}
\address[A]{School of Mathematics and Statistics, Fuzhou University\printead[presep={,\ }]{e1}}
\address[B]{School of Physics, Peking University\printead[presep={,\ }]{e2}}
\end{aug}
\begin{abstract}
This supplement provides standard feature and matrix estimates, auxiliary
likelihood and critical truncation calculations, quantitative approximation
proofs and general spherical spectral background.
\end{abstract}
\end{frontmatter}

\tableofcontents

We use the notation of the main article, which precedes this supplement
in the arXiv version. Section, result, equation and table
references with an S prefix point to this supplement; all others point to
the main article.

\section{Standard feature and matrix estimates}
\label{supp:features}
We use the feature and matrix notation of
Appendix~A.

\subsection{Quadratic feature representation}
\label{supp:proof-feature-lift}
\begin{proof}[Proof of Lemma~A.1]
Since \(\|x\|=\|y\|=1\),
\[
  \tr\left(xx^\top-\frac1dI_d\right)^2
  =1-\frac1d
\]
and
\[
  \tr\left[
    \left(xx^\top-\frac1dI_d\right)
    \left(yy^\top-\frac1dI_d\right)
  \right]
  =\langle x,y\rangle^2-\frac1d.
\]
Multiplication by \(d/(d-1)\) proves the norm and inner-product identities.
Rotational invariance gives \(\E(XX^\top)=I_d/d\), and hence
\(\E\varphi_d(X)=0\).  Finally, for every \(H\in\mathbb V_d\), the spherical
fourth-moment identity \cite[Section~1.5]{Muirhead1982} gives
\[
  \E\langle\varphi_d(X),H\rangle_{\HS}^2
  =\frac d{d-1}\E(X^\top HX)^2
  =\frac d{d-1}\frac{2\tr(H^2)}{d(d+2)}
  =\frac1N\|H\|_{\HS}^2.
\]
This quadratic-form identity is equivalent to
\(\E[\varphi_d(X)\otimes\varphi_d(X)]=I_{\mathbb V_d}/N\).
\end{proof}

\subsection{Feature identities and moments}
\label{supp:proof-complete-feature-identities}
\begin{proof}[Proof of Lemma~A.2]
Lemma~A.1 gives the unit norm, centering, and covariance
identities. Since \((XX^\top)^2=XX^\top\) almost surely, the matrix square
satisfies
\[
\begin{aligned}
  q^2
  &=
  \frac d{d-1}\left[
  \left(1-\frac2d\right)XX^\top+\frac1{d^2}I_d
  \right]\\
  &=
  \frac1dI_d+\frac{d-2}{\sqrt{d(d-1)}}q,
\end{aligned}
\]
which is \textup{(A.1)}.

For traceless symmetric matrices \(A_1,A_2,A_3\), the corresponding
sixth-order moment identity \cite[Section~1.5]{Muirhead1982} yields directly
\[
  \mathbb E\prod_{a=1}^3(X^\top A_aX)
  =
  \frac{
  8\operatorname{tr}_{\mathrm{sym}}(A_1A_2A_3)
  }{d(d+2)(d+4)}
\]
Since $\langle q,H\rangle_{\HS}=\sqrt{d/(d-1)}X^\top HX$ for $H\in V$,
\textup{(A.2)} follows.  Finally,
$N\asymp d^2$ gives $\gamma_d\asymp d^{-3}\asymp N^{-3/2}$.
\end{proof}

\subsection{Quadratic-form moment and tail bounds}
\label{supp:proof-local-subgaussian}
\begin{proof}[Proof of Lemma~A.3]
The case $H=0$ is immediate.  For $H\ne0$, put
$a_d=\sqrt{d/(d-1)}$ and $W=\langle H,q\rangle_{\HS}=a_dX^\top HX$.
Write a standard Gaussian vector as $G=RX$, where $R$ and $X$ are
independent and $\mathbb ER^2=d$.  For
$2a_d|\lambda|\|H\|_{\mathrm{op}}/d<1$, conditional Jensen's inequality
in $R$, followed by the Gaussian quadratic-form integral, gives
\[
\begin{aligned}
  \mathbb E_\mu e^{\lambda W}
  &\le\mathbb E_{X,R}
       \exp\{\lambda a_dR^2X^\top HX/d\}\\
  &=\det(I_d-2\lambda a_dH/d)^{-1/2}.
\end{aligned}
\]
Since $\operatorname{tr}H=0$ and
$|\operatorname{tr}(H^k)|\le\|H\|_{\HS}^2\|H\|_{\mathrm{op}}^{k-2}$
for $k\ge2$, expansion of the log determinant yields
\[
\begin{aligned}
  \log\mathbb E_\mu e^{\lambda W}
  &\le\sum_{k\ge2}\frac{(2\lambda a_d/d)^k}{2k}
                    \operatorname{tr}(H^k)\\
  &\le\frac{a_d^2\lambda^2\|H\|_{\HS}^2/d^2}
                {1-2a_d|\lambda|\|H\|_{\mathrm{op}}/d}.
\end{aligned}
\]
Taking $\lambda=1$, using $N\asymp d^2$, and choosing $c_0$ sufficiently
small proves \textup{(A.3)}.

For later use, the same calculation also gives a tail bound without any
smallness assumption on $H$.  Apply exponential Markov inequality to $W$
and $-W$ with
\[
  \lambda=c\min\left\{
  \frac{Nu}{\|H\|_{\HS}^2},\frac d{\|H\|_{\mathrm{op}}}\right\}
\]
for a sufficiently small absolute $c>0$;
this gives \textup{(A.4)}.
Replacing $q$ by $-q$ replaces $H$ by $-H$ and leaves both bounds unchanged.
Averaging the bounds gives them under any convex combination of the two laws.
\end{proof}

\subsection{Inner-product tails}
\label{supp:proof-quadratic-inner-product-tail}
\begin{proof}[Proof of Lemma~A.4]
By rotation invariance take $x=e_1$ and put $t_X=X_1$.  A coordinate of a
uniform spherical vector has the sub-Gaussian tail
\[
  \mathbb P\{|t_X|\ge s\}\le2e^{-cds^2},
  \qquad 0\le s\le1.
\]
See, for example, \cite[Theorem~3.4.6]{Vershynin2018}.  The event
\(\{Z_{2,d}(\langle x,X\rangle)\ge u\}\) is exactly
\[
  t_X^2\ge\frac{1+(d-1)u}{d},
\]
which proves the exponential bound after changing constants.
Furthermore $B_{ij}\ge-1/(d-1)$.  A union bound over at most $n^2/2$
pairs proves the second assertion when $d\gtrsim\log n$.  When
$d\lesssim\log n$ the displayed threshold is an absolute constant and
the claim follows from $|B_{ij}|\le1$ after enlarging $C_D$.
\end{proof}

\subsection{Simultaneous row occupancy}
\label{supp:proof-multiscale-occupancy}
\begin{proof}[Proof of Lemma~A.5]
By Lemma~A.4, except on an event of
probability $n^{-D-2}$,
\[
  \max_{i<j}B_{ij}\le u_{\max}
  :=C_{D+2}\min\left(1,\frac{\log n}{d}\right).
\]
On this event, $\mathcal N_i(u)=0$ for $u>u_{\max}$.  There are at most
$2+\log_2(C_{D+2}\log n)$ remaining levels in $\mathcal U_d$.
For each such level, conditional on $X_i$, rotational invariance and the
same lemma give
\[
  \mathcal N_i(u)\sim\operatorname{Bin}(n-1,p_u),
  \qquad p_u\le Ce^{-cdu},
\]
where $p_u$ does not depend on $X_i$.  A binomial Chernoff bound yields
\[
  \mathbb P\{\mathcal N_i(u)>C_D(np_u+\log n)\mid X_i\}
  \le n^{-D-3}.
\]
A union bound over the $n$ rows and these levels has error at most
$n^{-D-1}$ for large $n$.  Together with the maximum event this proves
\textup{(A.5)}; enlarging $C_D$ covers the remaining $n$.
For $u\in[1/d,1]$, the largest dyadic level $u'\le u$ satisfies
$u/2\le u'\le u$, so
\[
  \mathcal N_i(u)\le\mathcal N_i(u')
  \le C_D(ne^{-cdu'}+\log n)
  \le C_D(ne^{-cdu/2}+\log n).
\]
For $u>1$ the occupancy is zero.  This proves \textup{(A.6)}.
\end{proof}

\subsection{Uniform row moments}
\label{supp:proof-uniform-row-moments}
Fix \(D>0\), as in the preceding occupancy estimate.
\begin{corollary}[Uniform row moments]
\label{cor:uniform-row-moments}
On an event of probability $1-O(n^{-D})$, for every integer $r\ge2$,
\begin{equation}
  \max_i\sum_{j\ne i}|B_{ij}|^r
  \le
  \frac{(C_Dr)^r}{d^r}
  \left(n+(\log n)^{r+1}\right).
  \label{eq:row-r-moment}
\end{equation}
In particular, for each fixed $r$ and all sufficiently large $n$, there is a
constant \(C_{D,r}\) such that
\[
  \max_i\sum_{j\ne i}|B_{ij}|^r
  \le C_{D,r}\frac n{N^{r/2}}.
\]
\end{corollary}
\begin{proof}
The negative part obeys
$0\le(-B_{ij})_+\le1/(d-1)$, and hence contributes at most
$2^rn/d^r$.  For the positive part, layer cake gives
\[
  \sum_{j\ne i}(B_{ij})_+^r
  =
  \int_0^{u_{\max}}ru^{r-1}\mathcal N_i(u)\,du,
  \qquad
  u_{\max}=C_D\min\left(1,\frac{\log n}{d}\right).
\]
Use $\mathcal N_i(u)\le n$ for $u\le1/d$ and
\textup{(A.6)} above $1/d$.  The exponential term is
bounded by
\[
  n\int_0^\infty ru^{r-1}e^{-cdu/2}\,du
  \le\frac{n(Cr)^r}{d^r},
\]
and the additive $\log n$ term contributes at most
$C_D^r(\log n)^{r+1}/d^r$.  This proves
\eqref{eq:row-r-moment}.  Since $N\asymp d^2$ and
$(\log n)^{r+1}=o(n)$ for fixed $r$, the stated fixed-$r$ bound follows.
\end{proof}

\subsection{First-replica matrix bounds}
\label{supp:proof-row-gram-environment}
\begin{proof}[Proof of Proposition~A.6]
Corollary~\ref{cor:uniform-row-moments} with $r=2,6$ gives the row bounds in
\textup{(A.8)} and \textup{(A.9)}; summing the
squared row norms gives $\|A\|_{\HS}^2\le n\max_i\|A_{i\cdot}\|_2^2$.
Lemma~A.4 and $N\asymp d^2$ give
$\|A\|_{\max}\lesssim\beta^2\min(1,\log n/d)$, proving
\textup{(A.10)}.

For the spectral bound, set
$\widehat\Sigma=\sum_iq_i^\ast\otimes q_i^\ast$ and
$Y_i=q_i^\ast\otimes q_i^\ast-I_V/N$.  Unit norm and isotropy imply
\[
  \mathbb EY_i=0,\qquad \|Y_i\|_{\mathrm{op}}\le1,\qquad
  \mathbb EY_i^2=\frac{N-1}{N^2}I_V\preceq\frac1N I_V.
\]
Applying matrix Bernstein to $\pm\sum_iY_i$
\cite[Theorem~6.1]{Tropp2012} gives
\begin{equation}
  \mathbb P\left\{\left\|\widehat\Sigma-\frac nN I_V\right\|_{\mathrm{op}}>t\right\}
  \le2N\exp\left[-c\min\left\{\frac{Nt^2}{n},t\right\}\right].
  \label{eq:environment-bernstein-tail}
\end{equation}
Take $t=C\{n/N+\sqrt{n\log N/N}+\log N\}$ with $C$ large.
The probability is $o(1)$: for bounded $N$, the $n/N$ term makes the
exponent diverge; as $N\to\infty$, the other terms absorb the prefactor.
The full Gram matrix $B^\ast+I_n$ and $\widehat\Sigma$ have the same nonzero
eigenvalues, so
\[
  \|B^\ast\|_{\mathrm{op}}\le\|\widehat\Sigma\|_{\mathrm{op}}+1
  \lesssim\frac nN+\sqrt{\frac{n\log N}{N}}+\log N.
\]
Multiplying by $\beta^2$ proves \textup{(A.11)}.
\end{proof}

\subsection{Uniform covariance concentration}
\label{supp:proof-uniform-covariance-tail}
\begin{proof}[Proof of Lemma~A.7]
Put \(m=|I|\).
Every summand is positive semidefinite, has operator norm one, and has
expectation $I_V/N$.  The matrix Chernoff inequality
\cite[Corollary~5.2]{Tropp2012} therefore gives, for every \(u\ge0\),
\[
  \mathbb P\left\{
  \left\|\sum_{i\in I}q_i\otimes q_i\right\|_{\mathrm{op}}
  >
  C\left(\frac mN+u\right)
  \right\}
  \le
  N e^{-cu}.
\]
Taking $u=T+C\log N$ absorbs the prefactor.  Uniformity over the convex
mixtures follows from the sign invariance just noted.  Because \(I\) is
arbitrary, the same
bound gives the stated estimates for \(\mathcal C_m\) and, with
\(I=[n]\setminus\{i\}\), for every \(\mathcal C_{-i}\).

To prove \textup{(A.12)}, assume \(m\ge1\), since the
empty-set case is immediate.  For \(x=(x_1,\ldots,x_m)\), define
\(\mathsf S_x:\mathbb R^m\to V\) by
\(\mathsf S_x a=\sum_i a_i\varphi_d(x_i)\), and put
\(F(x)=\|\mathsf S_x\|_{\mathrm{op}}\).  Then
\[
  F(x)^2=\left\|\sum_i\varphi_d(x_i)\otimes\varphi_d(x_i)
       \right\|_{\mathrm{op}}.
\]
The identity \(xx^\top-yy^\top=(x-y)x^\top+y(x-y)^\top\) gives
\(\|\varphi_d(x)-\varphi_d(y)\|_{\HS}\le2\sqrt2\|x-y\|_2\).
Consequently,
\[
\begin{aligned}
  |F(x)-F(y)|
  &\le\|\mathsf S_x-\mathsf S_y\|_{\mathrm{op}}
   \le\left(\sum_i
     \|\varphi_d(x_i)-\varphi_d(y_i)\|_{\HS}^2\right)^{1/2}\\
  &\le2\sqrt2\left(\sum_i\|x_i-y_i\|_2^2\right)^{1/2}.
\end{aligned}
\]
Thus \(F\) has an absolute Lipschitz constant for the product Riemannian
metric.  Let \(\sigma_d\) be normalized surface measure on
\(\mathbb S^{d-1}\).  The spherical logarithmic Sobolev inequality
\cite[(LS)]{BrigatiDolbeaultSimonov2024}, with sphere dimension \(d-1\),
and entropy tensorization give
\[
  \operatorname{Ent}_{\sigma_d^{\otimes m}}(f^2)
  \le\frac{2}{d-1}\mathbb E\sum_i\|\nabla_i f\|_2^2,
  \qquad d\ge2.
\]
Here \(\operatorname{Ent}(G)=\mathbb E[G\log(G/\mathbb EG)]\), and
\(\nabla_i\) is the spherical gradient in the \(i\)-th coordinate.
Applying this inequality to \(f=e^{\lambda F/2}\), by smooth approximation
if necessary, and using the Lipschitz bound above yields
\[
  \operatorname{Ent}(e^{\lambda F})
  \le\frac{4\lambda^2}{d-1}\mathbb E e^{\lambda F}.
\]
Integrating the resulting differential inequality for
\(\lambda^{-1}\log\mathbb E e^{\lambda F}\) gives
\(\log\mathbb E e^{\lambda(F-\mathbb EF)}\le4\lambda^2/(d-1)\).
Exponential Markov inequality therefore implies
\[
  \mathbb P\{F>\mathbb EF+r\}\le e^{-cd r^2},\qquad r\ge0.
\]
Integrating the preceding matrix Chernoff bound gives
\(\mathbb EF^2\le C\{m/N+\log(eN)\}\), so
\(\mathbb EF\le C\sqrt{m/N+\log(eN)}\).  Take \(r=\sqrt{T/d}\)
and square the threshold to obtain \textup{(A.12)},
after changing absolute constants.  Finally, independent sign changes of
the columns of \(\mathsf S_x\) leave its operator norm unchanged.  This
proves the asserted uniformity, including different mixture laws at
different coordinates.
\end{proof}

\subsection{Hollow--Schur update bounds}
\label{supp:proof-hollow-schur-update-bounds}
\begin{proof}[Proof of Lemma~A.8]
Removing a diagonal cannot increase the Hilbert--Schmidt norm, and
\[
  \|D\|_{\HS}
  \le\|A_{I,J}\|_{\mathrm{op}}\|A_{I,J}\|_{\HS}
  \le\|A\|_{\mathrm{op}}\|A\|_{\HS}.
\]
Because $D$ is positive semidefinite,
\[
  \|\operatorname{off}(D)\|_{\mathrm{op}}
  \le\|D\|_{\mathrm{op}}+
  \|\operatorname{diag}(D)\|_{\mathrm{op}}
  \le2\|A\|_{\mathrm{op}}^2.
\]
For every $j\in J$, symmetry of $A$ gives
\[
  \sum_{k\in J}|D_{jk}|
  \le
  \sum_{i\in I}|A_{ij}|\sum_{k\in J}|A_{ik}|
  \le\left(\max_i\sum_j|A_{ij}|\right)^2.
\]
For the final bound, the vector identity
\(D_{j\cdot}^{\mathsf T}=A_{I,J}^{\mathsf T}A_{I,j}\) implies
\[
  \|D_{j\cdot}\|_2
  \le\|A_{I,J}\|_{\mathrm{op}}\|A_{I,j}\|_2
  \le\|A\|_{\mathrm{op}}\max_i\|A_{i\cdot}\|_2.
\]
Removing the diagonal can only decrease a row norm.  Multiplication by $c/N$
proves \textup{(A.13)} and
\textup{(A.14)}.
\end{proof}

\section{Likelihood identity and elementary remainder bounds}
\label{supp:elementary}

\subsection{The truncated likelihood identity}
\label{supp:proof-truncated-likelihood}
\begin{proof}[Proof of Lemma~B.1]
Conditional on $q$, an observed edge $a_{ij}\in\{0,1\}$ has likelihood
ratio $1+\beta Y_{ij}B_{ij}(q)$, where $Y_{ij}=2a_{ij}-1$: the alternative
and null success probabilities are $(1+\beta B_{ij}(q))/2$ and $1/2$.
Conditional independence of the edges and prior conditioning give
\[
  L_{\mathcal E}=\frac1\pi\mathbb E_q\left[
  \mathbf1_{\{q\in\mathcal E\}}
  \prod_{i<j}\{1+\beta Y_{ij}B_{ij}(q)\}\right].
\]
Squaring introduces an independent copy $q^\ast$.  Under $\Qq_n$, the
$Y_{ij}$ are independent, centered, and have variance one, so
\[
  \mathbb E_{\Qq_n}\left[
  \{1+\beta Y_{ij}B_{ij}(q)\}\{1+\beta Y_{ij}B_{ij}(q^\ast)\}\right]
  =1+\beta^2B_{ij}(q)B_{ij}(q^\ast).
\]
Multiplication over edges and averaging over $q,q^\ast$ prove
\textup{(B.1)}.  The uniform bound in the lemma gives
$\mathbb E_{\Qq_n}L_{\mathcal E}^2\le(1+o(1))/\pi$.  Since
$\mathbb E_{\Qq_n}L_{\mathcal E}=1$ and $\pi\to1$,
\[
  0\le\chi^2(\Pp^{\mathcal E}\|\Qq_n)
  =\mathbb E_{\Qq_n}L_{\mathcal E}^2-1=o(1).
\]
\end{proof}

\subsection{Square-exponential weighted moments}
\label{supp:proof-square-exponential-moments}
\begin{proof}[Proof of Lemma~C.8]
The case $H=0$ is immediate.  Otherwise \textup{(A.4)},
which also holds under $\bar\mu$, gives
\begin{equation}
  \mathbb P\{|D|\ge u\}\le2\left(e^{-cu^2/\sigma^2}
  +e^{-cdu/\|H\|_{\mathrm{op}}}\right).
  \label{eq:square-exp-tail}
\end{equation}
The feature representation in Appendix~A
also gives $|D|\le C\|H\|_{\mathrm{op}}$.  On this support, the two
smallness assumptions absorb $Ku^2$ into the respective tail exponents:
\[
  Ku^2-cu^2/\sigma^2\le-c'u^2/\sigma^2,\qquad
  Ku^2-cdu/\|H\|_{\mathrm{op}}\le-c'du/\|H\|_{\mathrm{op}}.
\]
For $p\ge1$, layer-cake integration and then the substitutions
$u=\sigma v$ and $u=\|H\|_{\mathrm{op}}v/d$ yield
\[
\begin{aligned}
  \mathbb E[|D|^pe^{KD^2}]
  &\le C\int_0^\infty(pu^{p-1}+2Ku^{p+1})
  \left(e^{-c'u^2/\sigma^2}+e^{-c'du/\|H\|_{\mathrm{op}}}\right)\,du\\
  &\le C_{K,p}\left\{\sigma^p+K\sigma^{p+2}
  +(\|H\|_{\mathrm{op}}/d)^p+K(\|H\|_{\mathrm{op}}/d)^{p+2}\right\}
  \le C_{K,p}\sigma^p.
\end{aligned}
\]
The last step uses $\|H\|_{\mathrm{op}}/d\lesssim\sigma$ and
$K\sigma^2\le c$.  For $p=0$, the same calculation starts from
$\mathbb E e^{KD^2}=1+\int_0^{C\|H\|_{\mathrm{op}}}
2Ku e^{Ku^2}\mathbb P\{|D|\ge u\}\,du$ and gives a constant bound.
\end{proof}

\subsection{Odd and even Taylor remainders}
\label{supp:proof-global-local-remainders}
For $D\in\mathbb R$ and $Q\ge0$, put $W=(D^2-Q)/2$ and
$W_+=\max(W,0)$, as in the main article.
\begin{proof}[Proof of Lemma~D.1]
Taylor's formula gives
\[
  \sinh D=D+\frac16D^3+R_s,\qquad
  |R_s|\le C|D|^5e^{|D|},
\]
while
\[
  |e^{-Q/2}-1+Q/2|\le CQ^2,\qquad
  |e^{-Q/2}-1|\le CQ.
\]
Subtracting the cubic polynomial from the product gives the exact identity
\[
  e^{-Q/2}\sinh D-D-\tfrac16D^3+\tfrac12DQ
  =e^{-Q/2}R_s+\tfrac16D^3(e^{-Q/2}-1)
  +D(e^{-Q/2}-1+\tfrac12Q).
\]
Since \(e^{W_+}\ge1\), the last two terms are bounded by
\(C(|D|^3Q+|D|Q^2)e^{W_+}\).  The first term is at most
\(C|D|^5e^{|D|-Q/2}\).  To control this factor, put \(y=Q/2\).
If \(y\le D^2/2\), then
\(e^{|D|-y}\le e^{1/2}e^{D^2/2-y}\).  If \(y>D^2/2\), then
\(e^{|D|-y}\le e^{|D|-D^2/2}\le e^{1/2}\).
Thus \(e^{|D|-Q/2}\le Ce^{W_+}\), which proves the odd bound.

For the even bound, the scalar Taylor remainder gives
\[
  |e^W-1-W|\le CW^2e^{W_+}.
\]
Also,
\[
  1+\frac{D^2}{2}\le\cosh D\le e^{D^2/2},
\]
and hence
\[
  0\le e^{D^2/2}-\cosh D
  \le e^{D^2/2}-1-\frac{D^2}{2}
  \le CD^4e^{D^2/2}.
\]
Multiplying this inequality by \(e^{-Q/2}\) gives
\[
  0\le
  e^W-e^{-Q/2}\cosh D
  =
  e^{-Q/2}(e^{D^2/2}-\cosh D)
  \le CD^4e^W
  \le CD^4e^{W_+}.
\]
Since $W^2\lesssim D^4+Q^2$, the result follows.
\end{proof}

\section{Scale calculations for the dense branch}
\label{supp:dense-scales}
Throughout this section we assume \textup{(6.1)}
and work on the first-replica event of
Proposition~A.6. All recursive quantities
and cutoff parameters are those of Appendix~C.

\subsection{Adaptive deletion remainders}
\label{supp:proof-cubic-adaptive-remainder}
\begin{proof}[Proof of Lemma~C.7]
We first bound the two terms in \textup{(C.9)}, then
deduce pointwise smallness and sum the original compensation coefficients.
The standing assumptions imply
\begin{equation}
  N=h_nn^{3/4}|\beta|^{3/2},\qquad n\beta^2\gtrsim h_n^4,
  \qquad 1\le n\beta^2\le n.
  \label{eq:adaptive-regime-relations}
\end{equation}
For the parameter calculations in this proof, put $x=n\beta^2$.
Then $N=h_nx^{3/4}$, $1\le x\le n$, and $x\gtrsim h_n^4$.
Proposition~A.6 gives
$\|A\|_{\HS}^2/N^2\lesssim h_n^{-3}x^{-1/4}\lesssim h_n^{-4}=o(1)$,
so Proposition~C.5 applies.  The environment
bounds and its maximum-entry estimate \textup{(C.6)} give
\begin{equation}
  \|A\|_{\HS}^2\lesssim\frac{n^{5/4}|\beta|^{5/2}}{h_n},\qquad
  R_0^2\lesssim\frac{n^{1/4}|\beta|^{5/2}}{h_n},\qquad
  M_\ast\lesssim\frac{|\beta|^{5/4}\log(nN)}{h_n^{1/2}n^{3/8}}
  +\frac{|\beta|}{h_n^2n^{1/2}}.
  \label{eq:adaptive-environment-scales}
\end{equation}
The quartic estimate in that proposition yields
\begin{equation}
  \sum_m(r_m^2/N)^2\lesssim
  \frac{R_0^2\|A\|_{\HS}^2}{N^2}
  \lesssim\frac{x}{h_n^4n}=\frac{\beta^2}{h_n^4}=o(1).
  \label{eq:unweighted-quartic-small}
\end{equation}
Equations~\textup{(6.4)} and
\eqref{eq:adaptive-environment-scales} yield
\begin{equation}
\begin{aligned}
  T^2&\lesssim x/h_n^4+\log^2(nN),\qquad
  \Lambda\lesssim\frac{n}{h_nx^{3/4}}+\frac{x^{1/2}}{h_n^2}+\log(nN),\\
  M_\ast^2&\lesssim\frac{x^{5/4}\log^2(nN)}{h_nn^2}
  +\frac{x}{h_n^4n^2}
  \lesssim\frac{x^{5/4}\log^2(nN)}{h_nn^2}.
\end{aligned}
  \label{eq:adaptive-parameter-scales}
\end{equation}
The first term of \textup{(C.9)} is therefore at most
\[
  C\left\{\frac{x}{h_n^4n}+\frac{x^{1/2}}{h_n^{10}n}
  +\frac{\log^2(nN)}{h_n^6nx^{1/2}}\right\}=o(1).
\]
For the second, $d^2\asymp N$ gives
\begin{equation}
\begin{aligned}
  \frac{T^2\Lambda M_\ast^2\|A\|_{\HS}^2}{d^2N^2}
  &\lesssim\frac{T^2\Lambda x^{1/4}\log^2(nN)}{h_n^5n^2}\\
  &\lesssim\frac{C\{1+\log^5(nN)\}}{h_n^5n^{1/4}}=o(1).
\end{aligned}
  \label{eq:adaptive-mixed-product}
\end{equation}
For completeness, multiplying the two bounds for $T^2$ and $\Lambda$
in \eqref{eq:adaptive-parameter-scales} produces, before the last bound,
\[
\begin{gathered}
  \frac{x^{1/2}\log^2(nN)}{h_n^{10}n}
  +\frac{x^{7/4}\log^2(nN)}{h_n^{11}n^2}
  +\frac{x^{5/4}\log^3(nN)}{h_n^9n^2}\\
  {}+\frac{\log^4(nN)}{h_n^6nx^{1/2}}
  +\frac{x^{3/4}\log^4(nN)}{h_n^7n^2}
  +\frac{x^{1/4}\log^5(nN)}{h_n^5n^2}.
\end{gathered}
\]
Each term has the stated bound because $1\le x\le n$ and $h_n\ge1$;
also $\log(nN)=O(\log n)$ since $N\lesssim n$.  This proves $\sum_mU_m=o(1)$.

The row bound in Proposition~C.5 and
\textup{(C.7)} also give
\begin{equation}
  \max_m\frac{k_{m,\mathrm{ad}}r_m^2}{d}
  \le C\left\{\left(1+\frac TN\right)\frac{R_0^2}{d}
  +\frac{T\sqrt\Lambda M_\ast R_0}{N}\right\}.
  \label{eq:pointwise-adaptive-bound}
\end{equation}
Its first term is bounded by
\[
  C\left\{\frac{x^{7/8}}{h_n^{3/2}n}
  +\frac{x^{5/8}}{h_n^{9/2}n}
  +\frac{x^{1/8}\log(nN)}{h_n^{5/2}n}\right\}=o(1).
\]
For the second, the same environment bounds give
\[
  \left(\frac{T\sqrt\Lambda M_\ast R_0}{N}\right)^2
  \lesssim\frac{N}{n}\,
  \frac{T^2\Lambda x^{1/4}\log^2(nN)}{h_n^5n^2}=o(1),
\]
by \eqref{eq:adaptive-mixed-product} and $N\lesssim n$.
This proves pointwise smallness, including the version with denominator
$N$, since $d/N\asymp N^{-1/2}$.

Finally, restriction of $P^{(n)}=A^{\circ2}$ counts each original edge once:
\begin{equation}
  \sum_{m=2}^np_m=\sum_{i<j}A_{ij}^2=\tfrac12\|A\|_{\HS}^2.
  \label{eq:compensation-edge-count}
\end{equation}
Since $\max_mp_m\le R_0^2$, \eqref{eq:unweighted-quartic-small} implies
$\sum_mV_m\le R_0^2\|A\|_{\HS}^2/(2N^2)=o(1)$.
The definition of $\mathcal E_m$ now gives \textup{(C.11)}.
\end{proof}

\subsection{The nineteen cubic-moment terms}
\label{supp:adamczak-scales}
The component indices are the pairs \((|I|,|\mathcal J|)\)
in Table~D.1 of the main article.
Proposition~A.6 provides all the bounds on \(A\)
used in \textup{(D.22)}.  More precisely,
\textup{(A.8)} controls the maximum row norm and the
Hilbert--Schmidt norm, \textup{(A.10)} controls the maximum entry,
and \textup{(A.11)} controls the operator norm.  In addition, the
definition \(p_{\mathrm H}=2\vee K_{\mathrm H}\|A\|_{\HS}^2/N\) and
\textup{(A.8)} give
\(p_{\mathrm H}\lesssim1+\beta^4n^2/N^2\).  Since
\(N\lesssim n\beta^2\) in \textup{(6.1)}, we have
\(\beta^4n^2/N^2\gtrsim1\), so this quantity absorbs the constant term.
Consequently,
\[
\begin{gathered}
  p_{\mathrm H}\lesssim\frac{\beta^4n^2}{N^2},\qquad
  \max_i\|A_{i\cdot}\|_2\lesssim\beta^2\sqrt{\frac nN},\qquad
  \|A\|_{\HS}^2\lesssim\frac{\beta^4n^2}{N},\\
  \|A\|_{\max}\lesssim\frac{\beta^2\log n}{\sqrt N},\qquad
  \|A\|_{\mathrm{op}}\lesssim
  \frac{\beta^2 n}{N}
  +\beta^2\sqrt{\frac{n\log N}{N}}
  +\beta^2\log N.
\end{gathered}
\]
The \((3,1)\) component contributes one term.  Each of the five
components containing \(\|A\|_{\mathrm{op}}\) contributes three terms, and
the three summands in \textup{(D.14)} contribute the final three terms.
Thus there are nineteen bounds to check.
For example, before substituting for \(N\), the \((3,1)\) term and the
\(\beta^2n/N\) part of the \((3,2)\) term satisfy
\[
\begin{aligned}
  p_{\mathrm H}^{1/2}N^{-9/4}
  \max_i\|A_{i\cdot}\|_2\|A\|_{\HS}^2
  &\lesssim |\beta|^8n^{7/2}N^{-19/4},\\
  p_{\mathrm H}N^{-11/4}
  \max_i\|A_{i\cdot}\|_2(\beta^2n/N)^2
  &\lesssim |\beta|^{10}n^{9/2}N^{-29/4}.
\end{aligned}
\]
Using the definition of \(h_n\), these
become the first two corresponding
rows of Table~\ref{tab:adamczak-exponent-patterns}.  Applying the same
substitution to the other seventeen terms gives patterns of the form
$n^a|\beta|^bh_n^{-c}(\log n)^u$.

\begin{table}[tbp]
\centering
\caption{Exponent patterns for the nineteen terms obtained from
\textup{(D.22)}.  The source-term column identifies the quantity
expanded to obtain each row.}
\label{tab:adamczak-exponent-patterns}
\small
\renewcommand{\arraystretch}{0.92}
\begin{tabular}{@{}ccc@{}}
\toprule
Adamczak component & Source term & $(a,b,-c;u)$\\
\midrule
$(3,1)$ & first term of \textup{(D.22)} & $(-1/16,7/8,-19/4;0)$\\
$(3,2)$ & $\beta^2 n/N$ & $(-15/16,-7/8,-29/4;0)$\\
$(3,3)$ & $\beta^2 n/N$ & $(-7/8,-3/4,-17/2;0)$\\
$(2,1)$ & $\beta^2 n/N$ & $(-7/16,9/8,-29/4;0)$\\
$(2,2)$ & $\beta^2 n/N$ & $(-3/4,1/2,-9;1)$\\
$(1,1)$ & $\beta^2 n/N$ & $(-1/4,5/2,-9;1)$\\
$(3,2)$ & $\beta^2\sqrt{n\log N/N}$ & $(-19/16,5/8,-25/4;1)$\\
$(3,3)$ & $\beta^2\sqrt{n\log N/N}$ & $(-9/8,3/4,-15/2;1)$\\
$(2,1)$ & $\beta^2\sqrt{n\log N/N}$ & $(-9/16,15/8,-27/4;1/2)$\\
$(2,2)$ & $\beta^2\sqrt{n\log N/N}$ & $(-1,2,-8;2)$\\
$(1,1)$ & $\beta^2\sqrt{n\log N/N}$ & $(-3/8,13/4,-17/2;3/2)$\\
$(3,2)$ & $\beta^2\log N$ & $(-23/16,17/8,-21/4;2)$\\
$(3,3)$ & $\beta^2\log N$ & $(-11/8,9/4,-13/2;2)$\\
$(2,1)$ & $\beta^2\log N$ & $(-11/16,21/8,-25/4;1)$\\
$(2,2)$ & $\beta^2\log N$ & $(-5/4,7/2,-7;3)$\\
$(1,1)$ & $\beta^2\log N$ & $(-1/2,4,-8;2)$\\
$(0,0)$ & $\beta^6n(\log n)^{3/2}/N^{9/4}$ & $(-5/16,27/8,-39/4;3/2)$\\
$(0,0)$ & $\beta^6\sqrt{n\log n}/N^{3/2}$ & $(-1/4,9/2,-9;1/2)$\\
$(0,0)$ & $\beta^6(\log n)^4/N^{3/2}$ & $(-3/4,9/2,-9;4)$\\
\bottomrule
\end{tabular}
\end{table}

As Table~\ref{tab:adamczak-exponent-patterns} shows, the only negative powers
of \(|\beta|\) occur in the second and third rows.
Since \(n\beta^2\gtrsim h_n^4\),
\[
\begin{aligned}
  n^{-15/16}|\beta|^{-7/8}h_n^{-29/4}
  &\lesssim n^{-1/2}h_n^{-9},\\
  n^{-7/8}|\beta|^{-3/4}h_n^{-17/2}
  &\lesssim n^{-1/2}h_n^{-10}.
\end{aligned}
\]
Both right-hand sides tend to zero because \(n\to\infty\) and
\(h_n\to\infty\) by \textup{(6.1)}.  In every other row
of Table~\ref{tab:adamczak-exponent-patterns}, the exponent of \(|\beta|\) is
nonnegative, so \(|\beta|^b\le1\) by \(0<|\beta|\le1\) in
\textup{(A.7)}.  Each exponent of \(n\) is negative and
the logarithmic exponent is fixed, which gives
\(n^a(\log n)^u\to0\).  The remaining factor \(h_n^{-c}\) also tends to zero
by the definition of \(h_n\).  Thus every term on the right-hand side of
\textup{(D.22)} is \(o(1)\), proving \textup{(D.23)}.

\section{Proofs for the finite-sample bounds}
\label{supp:quantitative}
We retain the critical range \textup{(9.1)},
the error scale \textup{(9.2)}, and the smallness
constant \(c_\ast\) specified in Appendix~F.

\subsection{Quantitative prior truncation}
\label{supp:proof-quant-critical-bounds}
\begin{proof}[Proof of Lemma~F.1]
We use the concentration family at \(M\asymp\log n\), verify the
inputs of Lemma~\ref{lem:replica-moment-transfer}, and apply its product
conclusion. The smallness assumption is used in the quantitative checks
below; no fixed-\(M\) limit is applied to a varying \(M\).

\noindent\textit{Step 1. Choose the concentration event.}
The critical relation and a sufficiently small \(c_{a,b}\) give
\begin{equation}
 N\asymp n^{3/4}|\beta|^{3/2},\qquad
 \sqrt n\,|\beta|\asymp N^{2/3},\qquad \log^{14}(en)\le n.
 \label{eq:quant-critical-substitution}
\end{equation}
Choose an integer \(M\asymp\log(en)\), so that
\(x_M\asymp\log^2(en)\), and take \(\mathcal E_n=\mathcal E_{n,M}\)
from Lemma~\ref{lem:critical-concentration-family}. The row bounds have
\(n+x_M^7\le Cn\), and \(\log N\le C_{a,b}\log(en)\). Thus the
maximum-entry, row, operator, and coefficient estimates of that lemma give
\textup{(F.1)} and \textup{(F.2)}. Its failure
bound gives
\(\mathbb P(\mathcal E_n^c)\le n^{-2M}
\le C_{a,b}e^{-c_{a,b}\log^2(en)}\), proving \textup{(F.4)}.
Use the corresponding covariance cutoff
\[
 \Lambda^\sharp=C_{\mathrm{cov}}
      \{n/N+\log^2(en)+\log N\},
\]
with an absolute adjustment of constants so that
\(\mathcal E_n\subset\mathcal C_n(\Lambda^\sharp+1)\).

\noindent\textit{Step 2. Verify the spectral and recursive inputs.}
Fix \(q^\ast\in\mathcal E_n\), write \(A=\beta^2B(q^\ast)\),
\(A_s=sA\), and take \(s_\star=3/2\). Substitution of
\textup{(F.1)} and \eqref{eq:quant-critical-substitution}
gives, for \(0\le s\le2s_\star\),
\begin{equation}
\begin{aligned}
  \frac{\sqrt{\Lambda^\sharp+1}\,R_2(sA)}d
  &\le C_{a,b}s_\star\left(
    N^{-1/6}
    +\frac{\log(en)}{N^{1/3}}
    \right)\le c_\ast,\\
  \frac{\|sA\|_{\mathrm{op}}}{N}
  &\le s\frac{\|A\|_{\mathrm{op}}}{N}\le c_\ast.
\end{aligned}
\label{eq:quant-spectral-smallness}
\end{equation}
\begin{equation}
  \frac{(\Lambda^\sharp+1)R_2(A)^2}{N}
  \le C_{a,b}\left(
    N^{-1/6}+
    \frac{\log(en)}{N^{1/3}}
    \right)^2
  \le c_\ast.
  \label{eq:quant-cutoff-replacement}
\end{equation}

Equation~\eqref{eq:quant-spectral-smallness} verifies both hypotheses of
Theorem~C.2, including for the moment with
exponent \(2K_{A_s}\). Equation~\eqref{eq:quant-cutoff-replacement} bounds
the local exponential factor in \textup{(7.3)}.
The spherical boundary tail is
\(\exp[-cd\{n/N+\log^2(en)+\log N\}]\), and
\[
 \frac{\|A_s\|_{\HS}^2/N+\log(2n)}
 {d\{n/N+\log^2(en)+\log N\}}
 \le C_{a,b}(N^{-1/6}+N^{-1/2}).
\]
Taking \(c_{a,b}\) small, using \(\varepsilon_{n,N,\beta}\ge N^{-\gamma_0}\),
absorbs the weight and proves
\begin{equation}
  \sup_{s,t}|\operatorname{Rem}_{\mathrm{cov}}(t)|
  \le C_{a,b}\exp\{-c_{a,b}\log^2(en)\}.
  \label{eq:quant-weighted-covariance-remainder}
\end{equation}

Choose the deletion threshold
\(T^\sharp=K_0\{C_{\mathrm{mat}}\beta^4n^2/N^2+\log^2(en)\}\),
with its absolute constant enlarged to cover \(s\le s_\star\).
The adaptive bounds \textup{(C.7)} and \textup{(C.9)},
applied to \(A_s\), yield
\begin{equation}
\begin{aligned}
  \max_m\frac{k_{m,\mathrm{ad}}r_m^2}{N}
  +\max_m\frac{k_{m,\mathrm{ad}}r_m^2}{d}
  &\le C_{a,b}\varepsilon_{n,N,\beta}
  \le c_\ast,\\
  \sum_m\left[
    k_{m,\mathrm{ad}}^2\left(\frac{r_m^2}{N}\right)^2
    +\left(\frac{p_m}{N}\right)^2
  \right]
  &\le C_{a,b}.
\end{aligned}
\label{eq:quant-recursion-scales}
\end{equation}

These verify the local deletion conditions and the summed remainder in
\eqref{eq:replica-recursive-input}. Also,
\(\|A_s\|_{\HS}^2/N^2\le C_{a,b}N^{-1/3}\le1/20\), and its
trace-approximation term is at most \(C_{a,b}\beta^2/N\).
The trace potential satisfies
\[
 |\mathcal P_N(A_s)|\le C\frac{\|A_s\|_{\HS}^2}{N}
                       \frac{\|A_s\|_{\mathrm{op}}}{N}
 \le C_{a,b}\left\{1+\sqrt{\frac Nn}\log(en)
                         +\frac{N\log^2(en)}n\right\}\le C_{a,b}.
\]
Moreover, \textup{(C.8)} and the spectral second moment
give \(\eta_{\mathrm{rec}}(s)
\le C\sqrt n\exp\{C_0\|A_s\|_{\HS}^2/N-T^\sharp/2\}
\le C_{a,b}e^{-c_{a,b}\log^2(en)}\). Thus the endpoint input is bounded.
For the local scores, use \(T_{\mathrm{loc}}\) from
Section~D.2 in the same adaptive estimate:
\begin{equation}
\begin{aligned}
  \max_i\frac{k_{i,\mathrm{ad}}\|A_{i\cdot}\|_2^2}{N}
  +\max_i\frac{k_{i,\mathrm{ad}}\|A_{i\cdot}\|_2^2}{d}
  &\le C_{a,b}\varepsilon_{n,N,\beta}\le c_\ast,\\
  \sum_i\left(\frac{k_{i,\mathrm{ad}}\|A_{i\cdot}\|_2^2}{N}\right)^2
  &\le C_{a,b}.
\end{aligned}
  \label{eq:quant-local-row-scales}
\end{equation}
The $d$-normalization follows from \eqref{eq:critical-pointwise-local-d}
with the quantitative logarithmic levels: by
\eqref{eq:quant-critical-substitution}, its three powers
$N^{7/6}/n$, $N^{5/6}/n$, and $N^{5/3}/n^{3/2}$ are bounded by
$C_{a,b}n^{-1/16}|\beta|^{7/8}$, and their logarithmic factors are
absorbed by \textup{(9.2)}.
These bounds verify \textup{(7.7)}.
For exceptional rows, \textup{(7.13)}, with $A_s$ and $p_s$,
gives
\[
 |\mathcal E_{\mathrm{bad}}(t)|
 \le C_{a,b}n^{-2}\max\{1,Z_s(t)\}.
\]
Indeed, $R_2(A_s)\le\sqrt n\,s_\star$ bounds the cubic prefactor by $Cn^4$,
and $T_{\mathrm{loc}}/p_s\ge c_{K_{\mathrm H}}K_1\log n$ absorbs it
when $K_1$ is sufficiently large.

\noindent\textit{Step 3. Verify the cubic, score, and product inputs.}
Uniformly over \(q^\ast,q\in\mathcal E_n\), the sixth-row bound gives the
first term below, while the maximum-entry and Hilbert--Schmidt bounds give
the other two:
\begin{equation}
\begin{aligned}
 \sup_{0\le s\le s_\star}\bigg\{
  &\sum_{i<j}|s\beta^2B_{ij}^\ast B_{ij}|^3
  +\frac{\|A_s\|_{\max}\|A_s\|_{\HS}^2}{N^{3/2}}
  +\frac1{N^{3/2}}\sum_{i,j}|(A_s)_{ij}|^3
 \bigg\}
 \le C_{a,b}\varepsilon_{n,N,\beta}.
\end{aligned}
 \label{eq:quant-product-interpolation-ledger}
\end{equation}

More explicitly, they are bounded by \(C_{a,b}N/n\) and
\(C_{a,b}\{\log^2(en)N/n+\beta^2/\sqrt N\}\), respectively. Use
\begin{equation}
  \frac Nn\lesssim\frac{|\beta|^{7/8}}{n^{1/16}},
  \qquad
  \frac{\beta^2}{\sqrt N}\lesssim
  \frac{|\beta|^{7/8}}{n^{1/16}},
  \label{eq:quant-product-critical-ratios}
\end{equation}

and \(\log^2(en)\le\log^{12}(en)\) to obtain the stated bound.
The three terms control, respectively, the product Taylor remainder, the
entry part of \textup{(7.6)}, and the repeated-index correction.

For \(g_s,\mathcal T_{g_s}\) formed from \(A_s\), use the exponent
\(p_{\mathrm H}=p_s\) of Lemma~\ref{lem:replica-moment-transfer}.
In Table~\ref{tab:adamczak-exponent-patterns}, the ratio
\(h=N/(n^{3/4}|\beta|^{3/2})\) is bounded above and away from zero.
The only negative powers of \(|\beta|\) occur in
\[
 n^{-15/16}|\beta|^{-7/8}
 =n^{-1/2}(\sqrt n|\beta|)^{-7/8},\qquad
 n^{-7/8}|\beta|^{-3/4}
 =n^{-1/2}(\sqrt n|\beta|)^{-3/4}.
\]
Here \(\sqrt n|\beta|\ge c_{a,b}^{-1}\). Every other row has a
nonnegative power of \(|\beta|\), at most twelve logarithmic factors after
the quantitative substitutions, and a power of \(n\) at most \(-1/16\).
The nineteen terms therefore give
\begin{equation}
    \sup_{\substack{0\le s\le s_\star\\0\le t\le1}}
    \|\mathcal T_{g_s}\|_{L^{p_{\mathrm H}}(\mu_t^{\otimes n})}
    \le C_{a,b}\left\{
      n^{-1/2}+n^{-1/16}\log^{12}(en)
    \right\}
    \le C_{a,b}.
    \label{eq:quant-nineteen-term-audit}
  \end{equation}

For the scores, \((x+y)^2\le2x^2+2y^2\) in \textup{(7.6)} and
\(\sum_i\|A_{i\cdot}\|_2^4\le\|A\|_{\HS}^4\) give
\begin{equation}
    \sup_{\substack{0\le s\le s_\star\\0\le t\le1}}
    \sum_i|\varepsilon_{i,s}(t)|
    \le C_{a,b}\left\{
      \sum_i\left(
        \frac{k_{i,\mathrm{ad}}\|A_{i\cdot}\|_2^2}{N}
      \right)^2
      +\left(
        \frac{\|A\|_{\max}\|A\|_{\HS}^2}{N^{3/2}}
      \right)^2
    \right\}
    \le C_{a,b}.
    \label{eq:quant-score-error-sum}
  \end{equation}

The last inequality uses \eqref{eq:quant-local-row-scales} and
\eqref{eq:quant-product-interpolation-ledger}; it also bounds
\(\sum_i\varepsilon_{i,s}^{\ast}\), since the rowwise majorants are
independent of \(t\) and the good configuration. Combining the two estimates,
\begin{equation}
   \sup_{\substack{0\le s\le s_\star\\0\le t\le1}}
   \left\{
   \|\mathcal T_{g_s}\|_{L^{p_{\mathrm H}}(\mu_t^{\otimes n})}
   +\sum_i|\varepsilon_{i,s}(t)|
 \right\}
   \le C_{a,b}.
   \label{eq:quant-odd-aggregate-ledger}
  \end{equation}

Together with Step 2, these are all inputs of
Lemma~\ref{lem:replica-moment-transfer}, with \(B_0,B_1\le C_{a,b}\).
The product factors are positive once \(c_{a,b}\) is small enough that
\(d\ge3\), since \(B_{ij}^\ast B_{ij}\ge-1/(d-1)\) and \(s_\star=3/2\).
Its product assertion, using \eqref{eq:quant-product-interpolation-ledger},
proves \textup{(F.3)}.
\end{proof}

\subsection{Polynomial signal scaling}
\label{supp:proof-quant-polynomial-rate}
\begin{proof}[Proof of Corollary~F.2]
The two terms inside braces in \textup{(9.2)} have orders
\(n^{-(1-r)/2}\) and \(n^{-(1+7r)/16}\).  Moreover,
\(N^{-\gamma_0}\asymp n^{-3\gamma_0(1-r)/4}\), while
\(\log^2(eN)/n\) has order \(n^{-1}\) up to logarithms.  The factor
\(\log^{12}(en)\) can be absorbed by decreasing the smallest of these four
positive exponents.  This gives some \(\eta_r>0\) and proves
\textup{(F.5)}.
\end{proof}

\subsection{Quantitative graph expansion and second moment}
\label{supp:proof-quant-second-moment}
The four errors are conditioning, nontriangle components,
intersections between triangles, and the higher-degree tail.
\begin{proof}[Proof of Proposition~F.3]
Fix $K=\lfloor c_1\log(eN)\rfloor$, with $c_1>0$ sufficiently small.
We bound the four errors in the order just described, using
$\mathfrak c_{k,n}^{\circ}$ and $\mathsf{Act}_n(H)$ from
\textup{(8.6)} and
\textup{(8.10)}.

\noindent\textit{Step 1. Remove conditioning up to degree $K$.}
Bounded edge factors and $|x^2-y^2|\le2|x-y|$ for $|x|,|y|\le1$ give
\begin{equation}
  |\mathfrak c_{k,n}-\mathfrak c_{k,n}^{\circ}|
  \le2n^{2k}\mathbb P(\mathcal E_n^c),\qquad 0\le k\le K.
  \label{eq:quant-coefficient-conditioning}
\end{equation}
Since $N\le C_{a,b}n^{3/4}$, \textup{(F.4)} makes the sum
over $k\le K$ at most $C_{a,b}e^{-c\log^2(en)}$ when $c_1$ is small.

\medskip\noindent\textit{Step 2. Sum nontriangle components.}
Lemma~8.2 applies because
$n\beta^2\le C_{a,b}N^{4/3}$. It gives zero for a component with a leaf,
and bounds every other connected nontriangle graph with $j$ edges by
$e^{C_{\mathrm{act}}j\log(ej)}N^{-\sqrt{2j}/6}$.
Such a graph has at most $2j$ vertices, so the number of isomorphism types is
at most
\[
  \sum_{v=2}^{2j}\binom{\binom v2}{j}
  \le2j\binom{2j^2}{j}\le e^{C_0j\log(ej)}.
\]
Consequently, with $c=\sqrt2/6$,
\begin{equation}
  \sum_{\substack{H\ \mathrm{connected},\ e(H)=j\\H\not\cong C_3}}
  \mathsf{Act}_n(H)
  \le e^{Cj\log(ej)}N^{-c\sqrt j}.
  \label{eq:quant-nontriangle-connected-sum}
\end{equation}
The sum is zero for $j\le3$. For $4\le j\le K$, decreasing $c_{a,b}$
ensures $C\sqrt j\log(ej)\le(c/2)\log N$, since $K=O(\log N)$.
Thus, for some absolute $c_2>0$,
\begin{equation}
  \sum_{j=1}^K\sum_{\substack{H\ \mathrm{connected},\ e(H)=j\\H\not\cong C_3}}
  \mathsf{Act}_n(H)
  \le\sum_{j=4}^K N^{-c\sqrt j/2}\le C_{a,b}N^{-c_2}.
  \label{eq:quant-nontriangle-connected-total}
\end{equation}

Let $m_G$ be the multiplicity of connected type $G$ in a graph $H$.
Independence, the component automorphisms, and disjointness of the vertex
images give
\[
\begin{gathered}
  t_d(H)=\prod_Gt_d(G)^{m_G},\qquad
  |\operatorname{Aut}(H)|=\prod_G|\operatorname{Aut}(G)|^{m_G}m_G!,\\
  (n)_{v(H)}\le\prod_G(n)_{v(G)}^{m_G},\qquad
  \mathsf{Act}_n(H)\le\prod_G\frac{\mathsf{Act}_n(G)^{m_G}}{m_G!}.
\end{gathered}
\]
Here the falling-factorial inequality compares disjoint embeddings with
separately chosen embeddings. Sum over component multiplicities and drop
the restriction $\sum_Gm_Ge(G)\le K$. The nonnegative sum then factorizes:
\[
  \sum_{(m_G)}\prod_G\frac{\mathsf{Act}_n(G)^{m_G}}{m_G!}
  =\prod_G e^{\mathsf{Act}_n(G)}
  =\exp\left\{\sum_G\mathsf{Act}_n(G)\right\},
\]
where $G$ ranges over connected types with at most $K$ edges.
This is the exponential formula for components
\cite[Section~II.5]{FlajoletSedgewick2009}; the embedding inequality makes it
an upper bound for the graph sum. Decrease $c_{a,b}$ so that $K\ge3$.
Since $\mathsf{Act}_n(C_3)=\Delta_n^2$, restricting the multiplicity sum
to families with at least one nontriangle type gives
\begin{equation}
\begin{aligned}
  \sum_{\substack{H:\ e(H)\le K,\ v(H)\le n\\
                  H\ \text{has a nontriangle component}}}\mathsf{Act}_n(H)
  &\le e^{\Delta_n^2+C_{a,b}N^{-c_2}}-e^{\Delta_n^2}\\
  &\le C_{a,b}N^{-c_2},
\end{aligned}
  \label{eq:quant-nontriangle-total}
\end{equation}
by \eqref{eq:quant-nontriangle-connected-total} and $\Delta_n\le b$.

\medskip\noindent\textit{Step 3. Bound intersections between triangles.}
The count of $m$ vertex-disjoint triangles and their moment give
\[
  \mathfrak c_{3m,n}^{\triangle}
  :=\frac{(n)_{3m}}{m!(3!)^m}\beta^{6m}N^{-4m}
  =\frac{\Delta_n^{2m}}{m!}\frac{(n)_{3m}}{[(n)_3]^m}.
\]
The term $\log^2(eN)/n$ in \textup{(9.2)} ensures $K\le n/2$
after decreasing $c_{a,b}$. For $0\le3m\le K$, all factors below lie
in $[0,1]$, and $1-\prod_i(1-x_i)\le\sum_ix_i$ gives
\[
  0\le1-\frac{(n)_{3m}}{[(n)_3]^m}
  =1-\prod_{j=0}^{m-1}\prod_{r=0}^2\left(1-\frac{3j}{n-r}\right)
  \le\sum_{j=0}^{m-1}\sum_{r=0}^2\frac{3j}{n-r}
  \le C\frac{m^2}{n}.
\]
Therefore
\begin{equation}
  \sum_{0\le3m\le K}\left|
  \mathfrak c_{3m,n}^{\triangle}-\frac{\Delta_n^{2m}}{m!}\right|
  \le\frac{CK^2}{n}\sum_{m\ge0}\frac{\Delta_n^{2m}}{m!}
  \le C_{a,b}\frac{K^2}{n}.
  \label{eq:quant-triangle-collision}
\end{equation}

\medskip\noindent\textit{Step 4. Control the tail and combine the errors.}
For fixed $1<R\le s_\star$, coefficient nonnegativity
(Lemma~8.1) and
\textup{(F.3)} yield
\[
  \sum_{k>K}\mathfrak c_{k,n}
  +\sum_{3m>K}\frac{\Delta_n^{2m}}{m!}
  \le R^{-K}\{\mathcal Z_n^{\mathcal E}(R)+e^{\Delta_n^2R^3}\}
  \le C_{a,b}R^{-K}.
\]
By \textup{(8.10)}, the unconditioned sum
up to degree $K$ consists of the triangle unions in Step 3 and the graphs
in \eqref{eq:quant-nontriangle-total}; the empty graph is counted by $m=0$.
Splitting $\mathcal Z_n^{\mathcal E}(1)$ and $e^{\Delta_n^2}$ at degree $K$
and applying Steps 1--3 therefore gives
\begin{equation}
\begin{aligned}
  |\mathcal Z_n^{\mathcal E}(1)-e^{\Delta_n^2}|
  &\le C_{a,b}\left\{e^{-c\log^2(en)}+N^{-c_2}+K^2/n+R^{-K}\right\}\\
  &\le C_{a,b}\left\{N^{-\gamma_0}+\log^2(eN)/n\right\},
\end{aligned}
  \label{eq:quant-partition-expansion}
\end{equation}
where $0<\gamma_0<\min\{c_2,c_1\log R\}$. The last line uses
$R^{-K}\le C_RN^{-c_1\log R}$ and $N\le C_{a,b}n^{3/4}$, which absorbs
the superpolynomial conditioning error. Finally,
\textup{(F.4)} gives
$|\pi_n^{-2}-1|\le C_{a,b}\varepsilon_{n,N,\beta}$.
Since $\Delta_n\le b$, division by $\pi_n^2$ in
\textup{(8.3)} proves \textup{(F.6)}.
\end{proof}

\subsection{Quantitative signed-triangle approximation}
\label{supp:proof-quant-triangle}
We use the quantitative de Jong bound in Theorem~1.3 of
\cite{DoblerPeccati2017}. For independent variables indexed by
\([m]=\{1,\ldots,m\}\), let \(\binom{[m]}{k}\) denote the
\(k\)-element subsets. If \(U=\sum_{J\in\binom{[m]}{k}}U_J\) is a
unit-variance degenerate \(U\)-statistic of fixed order \(k\) with finite
fourth moment, define
\[
  \rho^2
  :=\max_{1\le i\le m}
  \sum_{\substack{J\in\binom{[m]}{k}\\ i\in J}}\mathbb EU_J^2.
\]
That theorem gives
\begin{equation}
  d_W(\mathcal L(U),\Normal(0,1))
  \le
  C_k\left(
  |\mathbb EU^4-3|^{1/2}+\rho
  \right).
  \label{eq:quant-deJong}
\end{equation}

We apply \eqref{eq:quant-deJong} to the signed triangle under the null law and,
after conditioning on the latent features, under the alternative law.
In both applications, \(m=\binom n2\) counts the independent edge variables,
\(k=3\), and \(U_J=0\) unless the three edges in \(J\) form a retained triangle.
\begin{proof}[Proof of Proposition~F.4]
We reuse the finite-sample moment calculations in the proof of
Lemma~8.4. The quantitative steps are to apply
\eqref{eq:quant-deJong}, track the normalization error, and remove the
prior conditioning.

\medskip
\noindent\textit{Step 1. Null approximation.}
Step 1 of that proof gives maximum influence
\((n-2)/\binom n3=O(n^{-2})\) and
\(|\mathbb E_{\Qq_n}W_n^4-3|\le Cn^{-2}\).
Equation \eqref{eq:quant-deJong} proves the first bound.

\medskip
\noindent\textit{Step 2. Conditional leading term.}
For the alternative bound, it suffices to consider
\(\varepsilon_{n,N,\beta}\le c_{a,b}\). Indeed,
Proposition~4.1 gives
\(\mathbb E_{\Pp_n}W_n=\Delta_n\) and
\[
 \operatorname{Var}_{\Pp_n}(W_n)
 \le1+\frac{3(n-3)\beta^4}{N^3}\le C_{a,b},
\]
since \textup{(9.1)} gives
\(n\beta^2\le3b^{2/3}N^{4/3}\), \(N\ge2\), and \(|\beta|\le1\).
The triangle inequality through the
point mass at \(\Delta_n\) bounds the Wasserstein distance by
\(\sqrt{\operatorname{Var}_{\Pp_n}(W_n)}+\sqrt{2/\pi}\le C_{a,b}\).
This proves the alternative bound when
\(\varepsilon_{n,N,\beta}>c_{a,b}\), after enlarging the constant.
We henceforth assume \(\varepsilon_{n,N,\beta}\le c_{a,b}\).

Use \(S_n,V_n,S_n^\circ,V_n^\circ,\mathcal D_n\) and \(I_e^\circ\) from
Steps 2--3 of Lemma~8.4. On
\(q\in\mathcal E_n\), \textup{(F.1)} gives
\(\beta^2\sum_{i<j}B_{ij}^2\le C_{a,b}\beta^2n^2/N\).
Equations \textup{(8.22)} and
\textup{(8.24)} therefore yield
\begin{equation}
 0\le1-\frac{V_n(q)}{\binom n3}\le C_{a,b}\frac{\beta^2}{N},
 \label{eq:quant-variance-defect}
\end{equation}
\begin{equation}
 \frac{\mathbb E[(S_n-S_n^\circ)^2\mid q]}{\binom n3}
 \le\frac{(n-2)|\mathcal D_n(q)|}{\binom n3}
 \le C_{a,b}\frac{\beta^2}{N}.
 \label{eq:quant-deleted-chaos}
\end{equation}
Consequently, \(V_n^\circ(q)/\binom n3=1+O_{a,b}(\beta^2/N)\).
By \eqref{eq:quant-critical-substitution} and
\textup{(9.2)}, this ratio lies in \([1/2,1]\) after
decreasing \(c_{a,b}\). The retained edge variables are independent,
centered, variance one and bounded by \(2\sqrt2\), as in the cited proof.
Their maximum influence is
\(\rho_n^2\le(n-2)/V_n^\circ(q)\le Cn^{-2}\).
Equation \textup{(8.25)}, together with
\(\sum_{\mathfrak t}a_{\mathfrak t}^4\le V_n^\circ\) and
\(\sum_e(I_e^\circ)^2\le3\rho_n^2(V_n^\circ)^2\), gives
\[
 \left|\mathbb E\left[\left.
  \left(\frac{S_n^\circ}{\sqrt{V_n^\circ(q)}}\right)^4\right|q\right]-3\right|
 \le C\left\{\frac1{V_n^\circ(q)}+
       \frac{n^4}{V_n^\circ(q)^2}+\rho_n^2\right\}\le\frac C{n^2}.
\]
Thus \eqref{eq:quant-deJong} gives conditional Wasserstein error \(C/n\)
for \(S_n^\circ/\sqrt{V_n^\circ(q)}\). Put
\(\sigma_n(q)=\sqrt{V_n^\circ(q)/\binom n3}\). Coupling the deleted
and retained sums on the same edges, and using the same
\(Z\sim\Normal(0,1)\) at the two Gaussian scales, gives
\begin{equation}
\begin{aligned}
 &d_W\left(\mathcal L\left(\frac{S_n}{\sqrt{\binom n3}}\Bigm|q\right),
                  \Normal(0,1)\right)\\
 &\quad\le
 \left\{\frac{\mathbb E[(S_n-S_n^\circ)^2\mid q]}{\binom n3}\right\}^{1/2}
 +\frac{C\sigma_n(q)}n+|\sigma_n(q)-1|\mathbb E|Z|\\
 &\quad\le C_{a,b}\left(\frac1n+\frac{|\beta|}{\sqrt N}
                                  +\frac{\beta^2}{N}\right).
\end{aligned}
\label{eq:quant-conditional-leading}
\end{equation}
The last inequality uses \eqref{eq:quant-variance-defect},
\eqref{eq:quant-deleted-chaos} and
\(|\sigma_n-1|\le|\sigma_n^2-1|\).
Outside \(\mathcal E_n\), the pointwise bound
\(|S_n|\le8\binom n3\) bounds the conditional Wasserstein distance by
\(8\sqrt{\binom n3}+\mathbb E|Z|\).
Integrating \eqref{eq:quant-conditional-leading} and using
\textup{(F.4)} yields
\begin{equation}
 d_W\left(\mathcal L_{\Pp_n}\left(\frac{S_n}{\sqrt{\binom n3}}\right),
                  \Normal(0,1)\right)
 \le C_{a,b}\left(\frac1n+\frac{|\beta|}{\sqrt N}
                         +\frac{\beta^2}{N}\right)+Cn^{-18}.
 \label{eq:quant-unconditional-leading}
\end{equation}
Indeed, the omitted contribution is at most
\(C n^{3/2}\mathbb P(\mathcal E_n^c)\), which is absorbed by \(Cn^{-18}\)
in the range where \(\varepsilon_{n,N,\beta}\le c_{a,b}\).

\medskip
\noindent\textit{Step 3. Remainders and conditioning.}
The decomposition \textup{(8.20)} and the
finite-sample bounds \textup{(8.26)} hold without a
limit assumption on \(\Delta_n\). Taking square roots, then applying
Cauchy--Schwarz and \(d_W(\mathcal L(X),\mathcal L(Y))\le\mathbb E|X-Y|\),
gives
\[
\begin{aligned}
 &d_W\left(\mathcal L_{\Pp_n}(W_n),\Normal(\Delta_n,1)\right)\\
 &\quad\le C_{a,b}\bigg[\frac1n+\frac{|\beta|}{\sqrt N}+\frac{\beta^2}{N}
       +\beta^2\left(\frac1N+\frac{\sqrt n}{N^{3/2}}\right)\\
 &\hspace{4em}+|\beta|^3\left(\frac1{N^{3/2}}+\frac{\sqrt n}{N^2}\right)
       +n^{-18}\bigg]\le C_{a,b}\varepsilon_{n,N,\beta}.
\end{aligned}
\]
Here normal symmetry handles \(\operatorname{sgn}(\beta)S_n\),
\textup{(8.26)} identifies the mean shift as
\(\Delta_n\), and the last inequality follows from
\eqref{eq:quant-critical-substitution} and \textup{(9.2)}.

Finally, decomposing the prior over \(\mathcal E_n\) and its complement
gives \(\TV(\Pp_n,\Pp_n^{\mathcal E_n})\le\mathbb P(\mathcal E_n^c)\).
Since \(|W_n|\le\sqrt{\binom n3}\), subtracting \(f(0)\) in the dual
definition of \(d_W\) shows that
\[
 d_W\left(\mathcal L_{\Pp_n}(W_n),
          \mathcal L_{\Pp_n^{\mathcal E_n}}(W_n)\right)
 \le2\sqrt{\binom n3}\,\mathbb P(\mathcal E_n^c)
 \le C_{a,b}n^{3/2}e^{-c_{a,b}\log^2(en)}.
\]
This is absorbed by \(\varepsilon_{n,N,\beta}\ge n^{-1/2}\), proving
the conditioned bound.
\end{proof}

\section{Spectral background for spherical kernel graphs}
\label{supp:spectral-background}

For a measurable connection function \(K:[-1,1]\to[0,1]\), retain
the graph laws \(\Pp_{n,d,K}\) and \(\Qq_{n,p}\) from the introduction,
where \(p=\mathbb E K(\langle X_1,X_2\rangle)\). Assume \(0<p<1\).
Define the centered normalized kernel and edge variables by
\[
  \kappa(t)=\frac{K(t)-p}{\sqrt{p(1-p)}},
  \qquad
  Y_{ij}=\frac{a_{ij}-p}{\sqrt{p(1-p)}}.
\]
The integral operator on \(L^2(\Sph^{d-1},\sigma_d)\) associated with
$\kappa$ is
\[
  (\Kop_\kappa f)(x)
  =
  \int_{\Sph^{d-1}}
  \kappa(\langle x,y\rangle)f(y)\,d\sigma_d(y).
\]
The bounded symmetric kernel makes $\Kop_\kappa$ self-adjoint and
Hilbert--Schmidt; hence $\Kop_\kappa^\ell$ is trace class for every integer
$\ell\ge2$.  Rotational invariance and the definition of $p$ give
$\int\kappa(\langle x,y\rangle)\,d\sigma_d(x)=0$ for every $y$, so constants
belong to the zero eigenspace.  Writing the eigenvalues with multiplicity as
$(\lambda_j)_j$, set
\[
  \tau:=\tr(\Kop_\kappa^3)=\sum_j\lambda_j^3,
  \qquad \zeta:=\tr(\Kop_\kappa^4)=\sum_j\lambda_j^4.
\]
The Funk--Hecke formula makes $\Kop_\kappa$ scalar on each spherical harmonic
subspace \cite[Chapters~1--2]{DaiXu2013}.  These traces therefore decompose
by harmonic degree; the cubic contributions may cancel, while the fourth
trace is nonnegative.

For a simple cycle on distinct vertices $i_1,\ldots,i_\ell$, $\ell\ge3$,
put $i_{\ell+1}=i_1$ and $\mathcal X=\sigma(X_1,\ldots,X_n)$.
Its edges are conditionally independent given $\mathcal X$, with
$\E(Y_{ij}\mid\mathcal X)=\kappa(\langle X_i,X_j\rangle)$, so
\[
  \E\!\left(\prod_{r=1}^\ell Y_{i_ri_{r+1}}\,\middle|\,\mathcal X\right)
  =\prod_{r=1}^\ell\kappa(\langle X_{i_r},X_{i_{r+1}}\rangle).
\]
Integrating the independent latent points, with $x_{\ell+1}=x_1$, gives
\begin{equation}
  \E_{\Pp_{n,d,K}}\prod_{r=1}^{\ell}Y_{i_ri_{r+1}}
  =\int\prod_{r=1}^{\ell}\kappa(\langle x_r,x_{r+1}\rangle)
  \,d\sigma_d(x_1)\cdots d\sigma_d(x_\ell)
  =\tr(\Kop_\kappa^\ell).
  \label{eq:supp-cycle-trace-identity}
\end{equation}
The last equality follows by kernel composition and the trace identity for
products of Hilbert--Schmidt operators; boundedness justifies the integrals.
Under $\Qq_{n,p}$ the moment is zero, since the distinct cycle edges are
independent and centered.  In particular, the signed-triangle means are
$\tau$ under $\Pp_{n,d,K}$ and zero under $\Qq_{n,p}$.

\subsection{Spherical harmonic layers}

For $d\ge2$, let $\cH_{k,d}$ be the real spherical harmonics of degree $k$
on $\Sph^{d-1}$.  The dimension formula \cite[Chapter~1]{DaiXu2013} is
\begin{equation}
  N_{k,d}:=\dim\cH_{k,d}
  =\binom{d+k-1}{k}-\binom{d+k-3}{k-2},
  \label{eq:harmonic-multiplicity}
\end{equation}
with the second binomial interpreted as zero for $k<2$.  Thus $N_{0,d}=1$
and, for $k\ge1$,
\[
  N_{k,d}=\frac{(2k+d-2)(k+d-3)!}{k!(d-2)!}.
\]
For an orthonormal basis $(\psi_{k,m})_{m=1}^{N_{k,d}}$ in $L^2(\sigma_d)$,
the addition formula \cite[Chapter~1]{DaiXu2013} defines the zonal polynomial
\begin{equation}
  Z_{k,d}(\langle x,y\rangle)
  :=\frac1{N_{k,d}}\sum_{m=1}^{N_{k,d}}\psi_{k,m}(x)\psi_{k,m}(y).
  \label{eq:spherical-harmonic-addition-formula}
\end{equation}
The sum is basis independent and rotationally invariant; its normalization
gives $Z_{k,d}(1)=1$ and, by Cauchy--Schwarz, $|Z_{k,d}(t)|\le1$.
Completeness in the zonal $L^2$ space \cite[Chapter~1]{DaiXu2013} gives
\[
  \kappa(t)=\sum_{k\ge1}b_{k,d}Z_{k,d}(t)\quad\text{in }L^2.
\]
The constant term vanishes because $\kappa$ has spherical mean zero.
For $d\ge3$, $Z_{k,d}$ are normalized Gegenbauer polynomials; for $d=2$,
$Z_{k,2}(t)=\cos(k\arccos t)$, with the same addition and expansion formulas.

\begin{definition}[Active harmonic layers and spectral blocks]
\label{def:harmonic-layers}
A degree \(k\ge1\) is an \emph{active harmonic layer} of \(\kappa\) if
\(b_{k,d}\ne0\).  The corresponding invariant subspace
\(\cH_{k,d}\) is its \emph{spectral block}.  A kernel is
\emph{single-layer}, or \emph{pure of degree \(k\)}, if exactly one degree is
active, so that \(\kappa=\beta Z_{k,d}\) for some \(\beta\ne0\).  It is
\emph{multi-layer} if at least two degrees are active.  The number of active
degrees may be finite or infinite.
\end{definition}

By the addition formula,
$\left.\Kop_\kappa\right|_{\cH_{k,d}}=(b_{k,d}/N_{k,d})I_{\cH_{k,d}}$,
with rank $N_{k,d}$ whenever $b_{k,d}\ne0$.  For a single layer this gives
\begin{equation}
  \left.\Kop_{\beta Z_{k,d}}\right|_{\cH_{k,d}}
  =\frac{\beta}{N_{k,d}}I_{\cH_{k,d}},
  \qquad
  \left.\Kop_{\beta Z_{k,d}}\right|_{\cH_{k,d}^{\perp}}=0.
  \label{eq:single-layer-spectrum}
\end{equation}
It follows that
\begin{equation}
  \tr(\Kop_{\beta Z_{k,d}}^\ell)
  =\frac{\beta^\ell}{N_{k,d}^{\ell-1}},
  \qquad \ell\ge2.
  \label{eq:supp-single-layer-traces}
\end{equation}
Here $\beta$ is the pointwise signal amplitude and $N_{k,d}$ is the
multiplicity fixed by $(k,d)$.  All nonzero eigenvalues in a single layer
have the same sign; cancellation of odd traces can occur between layers.

\section{Critical truncation: transfer, concentration and scale checks}
\label{supp:critical-checks}

We use the notation of Appendix~E
and Sections~6--7.
Unless a finite-parameter statement specifies otherwise,
\(0<a<b<\infty\), \(a\le\Delta_n\le b\), and \(N\to\infty\).
The covariance cutoffs are chosen below for the critical scale.
The first two subsections give the common moment-transfer argument and
its preliminary critical application. The concentration family then
permits the fixed-tail verification used in the proof of
Lemma~E.1 of the main article.
\subsection{A common moment-transfer argument}
\label{sec:replica-moment-transfer}

We use the coefficient identity \textup{(8.4)}
from the main text. The coarse, generating-function and finite-sample
truncations will be checked separately against the following common inputs.

For a fixed first-replica matrix \(A=\beta^2B(q^\ast)\), write
\(A_s=sA\) and retain \(K_{A_s}\) from \textup{(6.3)}.
For a covariance cutoff \(\Lambda\), define
\begin{equation}
 Z_s(t)=\mathbb E_{\mu_t^{\otimes n}}
 [\mathbf1_{\mathcal C_n(\Lambda+1)}e^{K_{A_s}}],\qquad 0\le t\le1.
 \label{eq:fixed-M-interpolating-partition}
\end{equation}
All recursive quantities below are formed from \(A_s\). Let
\(\mathcal G_s^{\mathrm{rec}}\) be the event in
\textup{(6.15)}, and set
\[
 \eta_{\mathrm{rec}}(s)=\mathbb E_{\bar\mu^{\otimes n}}
 [\mathbf1_{\mathcal C_n(\Lambda+1)}
  \mathbf1_{(\mathcal G_s^{\mathrm{rec}})^c}e^{K_{A_s}}].
\]
Assume \(\|A_s\|_{\HS}^2/N^2\le1/20\) and the local hypotheses of the
compensated deletion estimates. Then
\textup{(6.14)} and the remainder bound from
Section~C.6 can be iterated on
\(\mathcal G_s^{\mathrm{rec}}\), using the nested covariance events and
conditional integrals of Section~6.4.
The compensation calculation
\textup{(6.17)}, with the finite error bound
\textup{(6.22)}, bounds the defect sum by the trace
potential, the quartic remainder, and the trace-approximation error.
Adding the complementary integral gives
\begin{equation}
\begin{aligned}
 Z_s(0)\le
 \exp\bigg\{&|\mathcal P_N(A_s)|+
 C\sum_m\left[k_{m,\mathrm{ad}}^2\left(\frac{r_m^2}{N}\right)^2
                +\left(\frac{p_m}{N}\right)^2\right]\\
 &+C\frac{R_2(A_s)^2\|A_s\|_{\HS}^2}{N^3}\bigg\}
 +\eta_{\mathrm{rec}}(s).
\end{aligned}
 \label{eq:replica-recursive-input}
\end{equation}
The local smallness conditions and every term on this right side will be
verified in each application; no asymptotic regime is imposed in this
implication. The second moment used below follows from
Theorem~C.2 applied to \(2A_s\), since
\(2K_{A_s}\le2\sum_{i<j}(A_s)_{ij}B_{ij}\).

Under the coordinate-deletion and local-score hypotheses, the buffered
identity \textup{(7.2)}, the score expansion in
Lemma~7.2, and the aggregation formula
\textup{(7.9)} give the derivative bound below.
Here \(\varepsilon_{i,s}^{\ast}\) is the deterministic right side of
\textup{(7.6)}, formed with \(A_s\); it bounds the score remainder
uniformly in \(t\) and on the good row event:
\begin{equation}
\begin{aligned}
 |Z_s'(t)|\le{}&|\operatorname{Rem}_{\mathrm{cov}}(t)|
 +|\mathcal E_{\mathrm{bad}}(t)|
 +\mathbb E_{\mu_t^{\otimes n}}
   [\mathbf1_{\mathcal C_n(\Lambda+1)}e^{K_{A_s}}|\mathcal T_{g_s}|]\\
 &+Z_s(t)\left\{\sum_i\varepsilon_{i,s}^{\ast}
       +CN^{-3/2}\sum_{i,j}|(A_s)_{ij}|^3\right\}.
\end{aligned}
\label{eq:fixed-M-derivative-decomposition}
\end{equation}
As in Section~7.4, \(\mathcal E_{\mathrm{bad}}\)
accounts for the complementary row events. Dependence on \(s\) is suppressed
in the two boundary terms.

\begin{lemma}[Transfer of uniform replica-moment bounds]
\label{lem:replica-moment-transfer}
Suppose \eqref{eq:replica-recursive-input} and
\eqref{eq:fixed-M-derivative-decomposition} hold for amplitudes
\(s\in S\subset[0,s_\star]\). Assume, uniformly over these amplitudes,
that the right side of \eqref{eq:replica-recursive-input} is at most
\(B_0\ge1\), and that the spectral second-moment estimate
\[
 \mathbb E_{\mu_t^{\otimes n}}
 [\mathbf1_{\mathcal C_n(\Lambda+1)}e^{2K_{A_s}}]
 \le\exp\{C\|A_s\|_{\HS}^2/N\}
\]
holds for \(0\le t\le1\). Put
\(p_s=2\vee K_{\mathrm H}\|A_s\|_{\HS}^2/N\). Suppose also that, for
some \(B_1\ge1\),
\[
\begin{aligned}
 |\operatorname{Rem}_{\mathrm{cov}}(t)|+|\mathcal E_{\mathrm{bad}}(t)|
 &\le B_1\max\{1,Z_s(t)\},\\
 \|\mathcal T_{g_s}\|_{L^{p_s}(\mu_t^{\otimes n})}
 +\sum_i\varepsilon_{i,s}^{\ast}
 +N^{-3/2}\sum_{i,j}|(A_s)_{ij}|^3&\le B_1.
\end{aligned}
\]
Then \(\sup_{s\in S,\,0\le t\le1}Z_s(t)\le C(B_0,B_1)\).
If, in addition, a common prior event \(\mathcal E\) is contained in
\(\mathcal C_n(\Lambda+1)\), all factors \(1+(A_s)_{ij}B_{ij}\) are
nonnegative, and
\[
 \sup_{q^\ast,q\in\mathcal E}\sup_{s\in S}
 \sum_{i<j}|(A_s)_{ij}B_{ij}|^3\le R,
\]
with the preceding bounds uniform over \(q^\ast\in\mathcal E\), then
\(\sup_{s\in S}\mathcal Z_n^{\mathcal E}(s)
\le e^{CR}C(B_0,B_1)\).
\end{lemma}

\begin{proof}
The weighted H\"older calculation \textup{(7.12)}, with
\(A_s\) and \(p_s\), bounds the weighted cubic term in
\eqref{eq:fixed-M-derivative-decomposition} by
\(C\max\{1,Z_s(t)\}\|\mathcal T_{g_s}\|_{p_s}\). Indeed, the second-moment
factor is at most \(\exp\{C\|A_s\|_{\HS}^2/(Np_s)\}\le e^{C/K_{\mathrm H}}\).
Thus \(|Z_s'(t)|\le CB_1\max\{1,Z_s(t)\}\). Gronwall gives
\[
 1+Z_s(t)\le(1+B_0)e^{CB_1t}.
\]
For the product assertion, a zero factor makes the upper bound immediate.
Otherwise \(\log(1+v)\le v-v^2/2+C|v|^3\), valid for \(v>-1\), gives
\[
 \mathbf1_{\mathcal E}(q)\prod_{i<j}\{1+(A_s)_{ij}B_{ij}\}
 \le e^{CR}\mathbf1_{\mathcal C_n(\Lambda+1)}e^{K_{A_s}}.
\]
Integrate in \(q\), then in \(q^\ast\in\mathcal E\), and apply the bound
for \(Z_s(1)\).
\end{proof}

\subsection{Critical truncation and preliminary recursive bounds}
\label{sec:preliminary-critical-second-moment}

This section establishes the high-probability framework for the critical
analysis.  It constructs the prior event \(\mathcal E_n\),
checks that the matrix and cutoff estimates from Appendices
C and
D remain valid at criticality, and verifies the inputs of Lemma~\ref{lem:replica-moment-transfer}.
Section~\ref{sec:critical-truncation-family} strengthens the prior event;
Section~8.2 then identifies the limiting second moment.

\begin{lemma}[Critical truncation and recursive bounds]
\label{lem:critical-truncation-recursion}
For every \(0<a<b<\infty\), suppose
\[
  a\le\Delta_n\le b,
  \qquad N\to\infty.
\]
There are prior events \(\mathcal E_n\) with
\(\mathbb P(\mathcal E_n)=1-o(1)\), a covariance cutoff \(\Lambda\), and a
constant \(C_{a,b}<\infty\) such that the following estimates hold uniformly
over \(q^\ast\in\mathcal E_n\).  For the matrices \(B^\ast\) and
\(A=A(q^\ast)=\beta^2B^\ast\) in \textup{(A.7)},
\[
  \frac{\|A\|_{\HS}^2}{N^2}=o(1),
  \qquad
  \frac{\|A\|_{\mathrm{op}}}{N}=o(1),
  \qquad
  |\mathcal P_N(A)|\le C_{a,b}.
\]
For the covariance-truncated compensated deletion recursion given by
\textup{(6.12)} and
\textup{(6.14)}, let \(k_{m,\mathrm{ad}}\) be the thresholds from
Section~C.4, and let
\(\mathcal G^{\mathrm{rec}}\) be the localization event in
\textup{(6.15)}, and recall \(p_m\) from
\textup{(6.9)}.  Then
\[
\begin{gathered}
  \max_m\frac{k_{m,\mathrm{ad}}r_m^2}{N}=o(1),
  \qquad
  \max_m\frac{k_{m,\mathrm{ad}}r_m^2}{d}=o(1),\\
  \sum_m\left[
    k_{m,\mathrm{ad}}^2\left(\frac{r_m^2}{N}\right)^2
    +\left(\frac{p_m}{N}\right)^2
  \right]\le C_{a,b}.
\end{gathered}
\]
For the configuration \(q\) integrated independently of \(q^\ast\) in
\textup{(B.1)}, write \(B=B(q)\).  Thus the compensated
exponent in \textup{(6.3)} is
\[
  K_A=\sum_{i<j}\left(A_{ij}B_{ij}
     -\frac12A_{ij}^2B_{ij}^2\right).
\]
The truncated exponential moment satisfies
\begin{equation}
  \mathbb E_{q\sim\mu^{\otimes n}}
  \left[
    \mathbf1_{\mathcal C_n(\Lambda+1)}e^{K_A}
  \right]
  \le C_{a,b}.
  \label{eq:critical-recursion-summary}
\end{equation}
Uniformly over \(q^\ast,q\in\mathcal E_n\),
\begin{equation}
  \sum_{i<j}|A_{ij}B_{ij}|^3=o(1),
  \qquad
  \mathbf1_{\{q\in\mathcal E_n\}}
  \prod_{i<j}(1+A_{ij}B_{ij})
  \le
  e^{o(1)}
  \mathbf1_{\mathcal C_n(\Lambda+1)}e^{K_A}.
  \label{eq:critical-product-comparison}
\end{equation}
The \(o(1)\) terms and the bounded recursive sums are uniform over
\(q^\ast\in\mathcal E_n\).  Both estimates in
\eqref{eq:critical-product-comparison} are uniform over pairs
\(q^\ast,q\in\mathcal E_n\).
\end{lemma}

\begin{proof}
We verify the common transfer inputs in five steps: matrix bounds,
local recursion and boundary estimates, the symmetric endpoint,
interpolation, and the exact-product comparison.  All constants below
may depend on $a,b$.

\noindent\textit{Step 1. Derive the matrix bounds.}
Use \textup{(8.1)} and \(N\lesssim n^{3/4}\).
Let $\mathcal G_n^{\mathrm{env}}$ be the event in
Proposition~A.6 giving
\textup{(A.8)}--\textup{(A.11)}.  It has probability
$1-o(1)$, and on it
\begin{equation}
\begin{aligned}
  \frac{\|A\|_{\HS}^2}{N^2}&\lesssim N^{-1/3}=o(1),
  &\frac{\|A\|_{\HS}^2}{N}&\lesssim N^{2/3},\\
  \frac{\|A\|_{\mathrm{op}}}{N}
  &\lesssim N^{-2/3}+\frac{\sqrt{\log N}}{\sqrt n\,N^{1/6}}
  +\frac{N^{1/3}\log N}{n}=o(1).
\end{aligned}
  \label{eq:critical-environment-scales}
\end{equation}
The first bound permits Proposition~C.5 for
all sufficiently large $n$.

\medskip\noindent\textit{Step 2. Verify the recursive local bounds.}
Keep the deletion level from \textup{(6.4)}, but
use the spherical covariance tail to choose the smaller cutoff
\begin{equation}
  T=K_0\{C_{\mathrm{mat}}\beta^4n^2/N^2+\log(nN)\},\qquad
  \Lambda=C_{\mathrm{cov}}\{n/N+\log(nN)\}.
  \label{eq:critical-cutoff-parameters}
\end{equation}
Taking tail parameter $d\{n/N+\log(nN)\}$ in
\textup{(A.12)} gives
\begin{equation}
  \sup_{i,\,0\le t\le1}\mathbb P_{\mu_t^{\otimes(n-1)}}
  \{\mathcal C_{-i}(\Lambda)^c\}
  \le\exp[-cd\{n/N+\log(nN)\}].
  \label{eq:critical-covariance-boundary-tail}
\end{equation}
Define $\mathcal E_n$ as the intersection of $\mathcal G_n^{\mathrm{env}}$,
$\mathcal C_n(\Lambda+1)$, and the event in Lemma~D.4.
That lemma, Proposition~A.6, and
Lemma~A.7 give
$\mathbb P(\mathcal E_n)=1-o(1)$.
On this event, Proposition~C.5 gives
\begin{equation}
  \sum_m(r_m^2/N)^2\lesssim\frac{R_2(A)^2\|A\|_{\HS}^2}{N^2}
  \lesssim\frac{\beta^8n^3}{N^4}\asymp\beta^2\le1.
  \label{eq:critical-quartic-sum}
\end{equation}
For the remaining substitutions, \textup{(C.6)} and
the environment bounds imply
\[
\begin{gathered}
  T\lesssim N^{2/3}+\log(nN),\qquad \Lambda\lesssim n/N+\log(nN),\\
  M_\ast\lesssim\frac{N^{5/6}\log n+N^{2/3}}n
  \lesssim\frac{N^{5/6}\log(nN)}n,\qquad
  \max_mr_m\lesssim N^{5/6}/\sqrt n.
\end{gathered}
\]
Both local normalizations follow from \textup{(C.7)}:
\begin{equation}
  \max_m\frac{k_{m,\mathrm{ad}}r_m^2}{N}
  \lesssim N^{-1/2}\max_m\frac{k_{m,\mathrm{ad}}r_m^2}{d}=o(1),
  \label{eq:critical-pointwise-local}
\end{equation}
where the stronger bound is
\begin{equation}
\begin{aligned}
  \max_m\frac{k_{m,\mathrm{ad}}r_m^2}{d}
  &\lesssim\left(1+\frac TN\right)\frac{R_2(A)^2}{d}
  +\frac{T\sqrt\Lambda M_\ast R_2(A)}{d^2}\\
  &\lesssim\frac{N^{7/6}}n
  +\log^{5/2}(nN)\left(\frac{N^{5/6}}n
  +\frac{N^{5/3}}{n^{3/2}}\right)=o(1).
\end{aligned}
  \label{eq:critical-pointwise-local-d}
\end{equation}
To obtain the second line, use $R_2(A)\lesssim N^{5/6}/\sqrt n$,
$d\asymp\sqrt N$, and
$\sqrt\Lambda\lesssim\sqrt{n/N}+\sqrt{\log(nN)}$.
Since $N\lesssim n^{3/4}$, its three powers of $n$ are at most
$n^{-1/8}$, $n^{-3/8}$, and $n^{-1/4}$, respectively.

For the sum in \textup{(C.9)}, the row term is at most
\[
  C\left\{\beta^2+\frac{N^{2/3}}n
  +\frac{\log^2(nN)}{nN^{2/3}}\right\}.
\]
The mixed term has the explicit bound
\[
\begin{aligned}
  \frac{T^2\Lambda M_\ast^2\|A\|_{\HS}^2}{d^2N^2}
  &\lesssim\frac{N^{1/3}\log^2(nN)}{n^2}
  (N^{4/3}+\log^2(nN))(n/N+\log(nN))\\
  &\lesssim\frac{N^{2/3}\log^2(nN)}n
  +\frac{N^{5/3}\log^3(nN)}{n^2}
  +\frac{\log^4(nN)}{nN^{2/3}}
  +\frac{N^{1/3}\log^5(nN)}{n^2}\\
  &\lesssim\frac{\log^5(nN)}{\sqrt n}=o(1).
\end{aligned}
\]
Thus
\begin{equation}
  \sum_m k_{m,\mathrm{ad}}^2(r_m^2/N)^2
  \le C\{\beta^2+\log^5(nN)/\sqrt n\}\le C_{a,b}.
  \label{eq:critical-GER-sum}
\end{equation}
Separately, \eqref{eq:compensation-edge-count} gives
$\sum_mp_m=\|A\|_{\HS}^2/2$ and $\max_mp_m\le R_2(A)^2$, so
\begin{equation}
  \sum_m(p_m/N)^2\le\frac{R_2(A)^2\|A\|_{\HS}^2}{2N^2}
  \lesssim\beta^2\le1.
  \label{eq:critical-compensation-sum}
\end{equation}

The spectral MGF hypotheses hold for every fixed multiple $sA$:
\begin{equation}
  \frac{\sqrt{\Lambda+1}R_2(sA)}d
  \lesssim_s N^{-1/6}+\frac{N^{1/3}\sqrt{\log(nN)}}{\sqrt n}=o(1),
  \qquad \frac{\|sA\|_{\mathrm{op}}}{N}=o(1).
  \label{eq:critical-spectral-smallness}
\end{equation}
Here the logarithmic term is $O(n^{-1/4}\sqrt{\log(nN)})$.
In particular,
\begin{equation}
  (\Lambda+1)R_2(A)^2/N=o(1).
  \label{eq:critical-cutoff-replacement-small}
\end{equation}
To control the weighted covariance boundary, combine
\textup{(7.3)} with
\eqref{eq:critical-covariance-boundary-tail}.  The exponent of the weight
and the factor $2n$ are absorbed because
\[
  \frac{\|A\|_{\HS}^2/N+\log(2n)}{d\{n/N+\log(nN)\}}
  \lesssim\frac{N^{7/6}}n+\frac1d
  \lesssim N^{-1/6}+N^{-1/2}=o(1).
\]
Consequently,
\begin{equation}
  \sup_t|\operatorname{Rem}_{\mathrm{cov}}(t)|
  \le C_{a,b}\exp[-cd\{n/N+\log(nN)\}]=o(1).
  \label{eq:critical-weighted-covariance-remainder}
\end{equation}
The same calculation is uniform over bounded ranges of $sA$.

\medskip\noindent\textit{Step 3. Verify the endpoint input.}
The potential series and \eqref{eq:critical-environment-scales} give
\begin{equation}
  |\mathcal P_N(A)|\lesssim
  \frac{\|A\|_{\HS}^2}{N}\frac{\|A\|_{\mathrm{op}}}{N}
  \lesssim1+\sqrt{N\log N/n}+N\log N/n\le C_{a,b}.
  \label{eq:critical-trace-potential-bound}
\end{equation}
The trace-approximation term in \eqref{eq:replica-recursive-input} is
$R_2(A)^2\|A\|_{\HS}^2/N^3\lesssim\beta^2/N=o(1)$.
Cauchy--Schwarz, \textup{(C.8)}, and the spectral
MGF, applicable by \eqref{eq:critical-spectral-smallness}, yield
\[
  \eta_{\mathrm{rec}}(1)\le C\sqrt n
  \exp\{C_0\|A\|_{\HS}^2/N-T/2\}=o(1),
\]
because $T\ge K_0\{\|A\|_{\HS}^2/N+\log(nN)\}$ with $K_0>2C_0+2$.
Together with \eqref{eq:critical-GER-sum} and
\eqref{eq:critical-compensation-sum}, these bounds verify
\eqref{eq:replica-recursive-input} and give
\begin{equation}
  Z(0)\le C_{a,b}.
  \label{eq:critical-symmetric-endpoint}
\end{equation}
Replacing $\Lambda$ by a number in $[\Lambda,\Lambda+1]$ changes the
bounds only by absolute factors, so this estimate holds at the required
enlarged cutoff.

\medskip\noindent\textit{Step 4. Interpolate to the original prior.}
For the local scores, use
$T_{\mathrm{loc}}=K_1(\|A\|_{\HS}^2/N+1)\log n\lesssim T\log n$.
The original rows satisfy the same row and entry bounds as the recursive
rows.  Thus \textup{(C.7)} and the preceding calculations give
\begin{equation}
\begin{aligned}
  \max_i\frac{k_{i,\mathrm{ad}}\|A_{i\cdot}\|_2^2}{N}
  +\max_i\frac{k_{i,\mathrm{ad}}\|A_{i\cdot}\|_2^2}{d}
  &\lesssim\frac{\log^{7/2}(nN)}{n^{1/8}}=o(1),\\
  \sum_i(k_{i,\mathrm{ad}}\|A_{i\cdot}\|_2^2/N)^2
  &\lesssim\beta^2+\frac{\log^7(nN)}{\sqrt n}\le C_{a,b}.
\end{aligned}
  \label{eq:critical-local-row-scales}
\end{equation}
Indeed, in \eqref{eq:critical-pointwise-local-d} the $T$ terms acquire
at most one extra factor $\log n$; in \eqref{eq:critical-GER-sum} the
$T$-dependent terms acquire at most $\log^2n$.  The unweighted term
$C\beta^2$ is unchanged.  The maximum-entry contribution satisfies
\begin{equation}
  \sum_i\frac{\|A\|_{\max}\|A_{i\cdot}\|_2^2}{N^{3/2}}
  =\frac{\|A\|_{\max}\|A\|_{\HS}^2}{N^{3/2}}
  \lesssim N\log n/n=o(1).
  \label{eq:critical-local-entry-sum}
\end{equation}
These estimates verify \textup{(7.7)} and bound the sum
of the deterministic score errors in \textup{(7.6)} by $C_{a,b}$.
For the cubic term, substitute \textup{(8.1)} into
\textup{(D.22)}.  The ratio
$h_n=N/(n^{3/4}|\beta|^{3/2})$ is now bounded above and below; the
negative powers of $n$ in Table~\ref{tab:adamczak-exponent-patterns}
of the Supplementary Material,
including the two negative-$|\beta|$ terms treated just after that table,
still give $\|\mathcal T_g\|_{p_{\mathrm H}}=o(1)$.
The repeated-index sum is bounded by
\eqref{eq:critical-local-entry-sum}.
The boundary estimate \eqref{eq:critical-weighted-covariance-remainder}
and \textup{(7.13)} supply the remaining transfer inputs.
For the latter, $R_2(A)\le\sqrt n$ bounds the cubic prefactor by $Cn^4$,
which is absorbed by $e^{-T_{\mathrm{loc}}/(2p_{\mathrm H})}$ for
sufficiently large $K_1$. Hence
Lemma~\ref{lem:replica-moment-transfer}, with $s=1$ and
$B_0,B_1\le C_{a,b}$, gives
\begin{equation}
  Z(1)\le C_{a,b},
  \label{eq:critical-Z-one}
\end{equation}
which is \eqref{eq:critical-recursion-summary}.

\medskip\noindent\textit{Step 5. Verify the product-comparison error.}
The sixth-row bound of Proposition~A.6
and Cauchy--Schwarz give, uniformly over $q^\ast,q\in\mathcal E_n$,
\begin{equation}
  \sum_{i<j}|A_{ij}B_{ij}|^3
  \le |\beta|^6\Big(\sum_{i<j}|B_{ij}^\ast|^6\Big)^{1/2}
  \Big(\sum_{i<j}|B_{ij}|^6\Big)^{1/2}
  \lesssim\frac{\beta^6n^2}{N^3}\asymp N/n=o(1).
  \label{eq:critical-simple-edge-error}
\end{equation}
Since every edge factor is nonnegative and
$\mathcal E_n\subseteq\mathcal C_n(\Lambda+1)$, the logarithmic
comparison used in \textup{(6.5)} proves
\eqref{eq:critical-product-comparison}.
\end{proof}

The estimates in Lemma~\ref{lem:critical-truncation-recursion} are uniform in
the fixed configuration \(q^\ast\in\mathcal E_n\).  Combining
\eqref{eq:critical-product-comparison} with
\eqref{eq:critical-recursion-summary}, and then averaging over \(q^\ast\) in
\textup{(B.1)}, gives the following \(L^2\) consequence.

\begin{corollary}[Bounded truncated second moment]
\label{cor:critical-truncated-second-moment}
Under the assumptions of Lemma~\ref{lem:critical-truncation-recursion}, the
events constructed there satisfy
\[
  \mathbb E_{\Qq_n}\left(L_n^{\mathcal E}\right)^2\le C_{a,b}
\]
for all sufficiently large \(n\), after increasing \(C_{a,b}\) if necessary.
\end{corollary}

\begin{proof}
With $\pi_n=\mathbb P(\mathcal E_n)=1-o(1)$,
\textup{(8.3)}, \eqref{eq:critical-product-comparison},
and \eqref{eq:critical-recursion-summary} give
\[
  \mathbb E_{\Qq_n}(L_n^{\mathcal E})^2
  \le\pi_n^{-1}e^{o(1)}C_{a,b}\le2C_{a,b}
\]
for all sufficiently large $n$. Renaming the constant proves the claim.
\end{proof}

\subsection{A common family of critical truncations}
\label{sec:critical-truncation-family}

The estimates below retain the integer tail parameter \(M\) explicitly. Fixed values give
the qualitative construction after a diagonal choice; a value of order
\(\log n\) will give the quantitative construction in
Section~F.1.
For a generic latent configuration \(u\), write
\begin{equation}
 g_{jk\ell}(u)=\beta^6\sum_iB_{ij}(u)B_{ik}(u)B_{i\ell}(u),\qquad
 G_\infty(u)=\max_{j,k,\ell\ \mathrm{pairwise\ distinct}}|g_{jk\ell}(u)|.
 \label{eq:fixed-M-cubic-coefficient}
\end{equation}
Fix \(s_\star=3/2\). For an integer \(M\ge1\), put
\begin{equation}
 x_M=K_{\mathrm{tail}}(M+5)\log(en),\qquad
 \Lambda_M=C_{\mathrm{cov}}\left\{\frac nN+x_M+\log N\right\}.
 \label{eq:fixed-M-tail-levels}
\end{equation}

\begin{lemma}[Concentration with an adjustable tail level]
\label{lem:critical-concentration-family}
For every \(0<a<b<\infty\), there is \(n_0(a,b)\) such that, whenever
\(n\ge n_0(a,b)\), \(N\ge2\), \(0<|\beta|\le1\), and
\(a\le\Delta_n\le b\), the events defined in \eqref{eq:fixed-M-maximum-entry-event},
\eqref{eq:fixed-M-row-moment-event}, \eqref{eq:fixed-M-operator-norm-event},
\eqref{eq:fixed-M-covariance-event}, and
\eqref{eq:fixed-M-maximum-coefficient-event} satisfy their displayed bounds
with constants independent of \(M\ge1\). Their intersection
\(\mathcal E_{n,M}\) satisfies
\(\mathbb P(\mathcal E_{n,M}^c)\le n^{-2M}\).
The threshold \(n_0(a,b)\) is independent of \(N,\beta,M\).
\end{lemma}

\begin{proof}
We construct the entry and row events together, then add the operator,
covariance, and coefficient bounds. Set
\(u_M=\min\{1,C_{\max}x_M/d\}\) and use the dyadic grid \(\mathcal U_d\)
from Lemma~A.5. Lemma
A.4 gives
\begin{equation}
 \mathbb P(B_{ij}\ge u)\le Ce^{-cdu},\qquad 0\le u\le1.
 \label{eq:fixed-M-entry-tail}
\end{equation}
Choose \(C_{\max}\) so that \(u_M\ge1/(d-1)\). Since
\(B_{ij}\ge-1/(d-1)\), a union bound on the positive tail gives
\begin{equation}
 \mathcal E_{n,M}^{\max}=\{\max_{i<j}|B_{ij}|\le u_M\},\qquad
 \mathbb P((\mathcal E_{n,M}^{\max})^c)\le Ce^{-cx_M}.
 \label{eq:fixed-M-maximum-entry-event}
\end{equation}
If \(u_M=1\) the event holds deterministically; otherwise absorb the
\(n^2\) union factor using \(x_M\ge K_{\mathrm{tail}}\log(en)\).

Recall \(\mathcal N_i(u)=\#\{j\ne i:B_{ij}\ge u\}\).
Conditionally on \(q_i\), the indicators in this count are independent;
rotational invariance makes their success probabilities equal to the
unconditional probability in \eqref{eq:fixed-M-entry-tail}.
The Chernoff argument of Lemma~A.5 therefore gives
\begin{equation}
\begin{aligned}
 \mathcal E_{n,M}^{\mathrm{occ}}
 &=\bigcap_{\substack{i\in[n]\\u\in\mathcal U_d\cap[0,u_M]}}
   \{\mathcal N_i(u)\le C(ne^{-cdu}+x_M)\},\\
 \mathbb P((\mathcal E_{n,M}^{\mathrm{occ}})^c)&\le Ce^{-cx_M}.
\end{aligned}
\label{eq:fixed-M-occupancy-event}
\end{equation}
There are \(O(n\log d)\) pairs \((i,u)\). Since
\(N\lesssim n^{3/4}\) by \textup{(8.1)}, their union
factor is absorbed by the same tail level, uniformly in \(M\).
For \(1/d\le u\le u_M\), choose \(u'\in\mathcal U_d\) with
\(u/2\le u'\le u\). Monotonicity gives
\(\mathcal N_i(u)\le C(ne^{-cdu/2}+x_M)\); for \(u\le1/d\), use
\(\mathcal N_i(u)\le n\).

On the entry event, layer cake gives, for \(r=2,6\),
\begin{equation}
 \sum_{j\ne i}(B_{ij})_+^r
 =\int_0^{u_M}ru^{r-1}\mathcal N_i(u)\,du.
 \label{eq:fixed-M-layer-cake}
\end{equation}
The negative contribution is at most \(n/(d-1)^r\). Thus, on the
intersection of the entry and occupancy events,
\[
 \sum_{j\ne i}|B_{ij}|^r
 \le\frac{(Cr)^rn}{d^r}+Cx_Mu_M^r
 \le\frac{(Cr)^r}{d^r}(n+x_M^{r+1}).
\]
Consequently the event
\begin{equation}
 \mathcal E_{n,M}^{\mathrm{row}}
 =\bigcap_{r\in\{2,6\}}\left\{
   \max_i\sum_{j\ne i}|B_{ij}|^r
       \le\frac{(Cr)^r}{d^r}(n+x_M^{r+1})\right\}
 \label{eq:fixed-M-row-moment-event}
\end{equation}
has complement probability at most \(Ce^{-cx_M}\).

For \(\mathsf L_qa=\sum_i a_iq_i\),
\begin{equation}
 B+I_n=\mathsf L_q^\ast\mathsf L_q,\qquad
 \sum_iq_i\otimes q_i=\mathsf L_q\mathsf L_q^\ast.
 \label{eq:fixed-M-gram-covariance-factorization}
\end{equation}
Apply the matrix Bernstein tail \eqref{eq:environment-bernstein-tail}
with \(t=C\{\sqrt{n(x_M+\log N)/N}+x_M+\log N\}\).
The two operators in \eqref{eq:fixed-M-gram-covariance-factorization}
have the same nonzero eigenvalues, so \(\|B\|_{\mathrm{op}}\le1+n/N+t\)
outside an event of probability \(Ce^{-cx_M}\). This proves the bound for
\begin{equation}
 \mathcal E_{n,M}^{\mathrm{op}}
 =\left\{\|A(q)\|_{\mathrm{op}}\le C\beta^2
   \left(\frac nN+\sqrt{\frac{n(x_M+\log N)}N}+x_M+\log N\right)\right\}.
 \label{eq:fixed-M-operator-norm-event}
\end{equation}
Lemma~A.7, with a constant multiple of
\(x_M\) as tail parameter, also gives complement probability
\(Ce^{-cx_M}\) for
\begin{equation}
 \mathcal E_{n,M}^{\mathrm{cov}}
 =\mathcal C_n(\Lambda_M+1)
 =\left\{\left\|\sum_iq_i\otimes q_i\right\|_{\mathrm{op}}
                      \le\Lambda_M+1\right\}.
 \label{eq:fixed-M-covariance-event}
\end{equation}
Finally, \(\log(enN)\lesssim\log(en)\), so a sufficiently large
\(K_{\mathrm{tail}}\) permits Lemma~D.3 at
\(x=x_M\). Hence
\begin{equation}
\begin{aligned}
 \mathcal E_{n,M}^{\mathrm{coef}}
 &=\left\{G_\infty(q)\le C\beta^6\left[
   \frac{nx_M^{3/2}}{N^{9/4}}
   +\frac{\sqrt{nx_M}+x_M^4}{N^{3/2}}\right]\right\},\\
 \mathbb P((\mathcal E_{n,M}^{\mathrm{coef}})^c)&\le Cn^3e^{-cx_M}.
\end{aligned}
\label{eq:fixed-M-maximum-coefficient-event}
\end{equation}
Let \(\mathcal E_{n,M}\) be the intersection of the five events with
superscripts \(\max,\mathrm{row},\mathrm{op},\mathrm{cov},\mathrm{coef}\).
For sufficiently large absolute \(K_{\mathrm{tail}}\), their failure bounds give
\begin{equation}
 \mathbb P(\mathcal E_{n,M}^c)
 \le C(1+n^3)e^{-cx_M}
 \le n^{-2M-4}\le n^{-2M}.
 \label{eq:fixed-M-prior-failure}
\end{equation}
For example, choose \(cK_{\mathrm{tail}}\ge4\); then
\(3-cK_{\mathrm{tail}}(M+5)\le-4M-17\), which absorbs the fixed
prefactor for all sufficiently large \(n\), independently of \(M\).
\end{proof}

\subsection{Uniform moment estimates at a fixed tail level}
\label{supp:critical-fixed-M}

Fix an integer \(M\ge1\), set \(s_\star=3/2\), and use the events of
Lemma~\ref{lem:critical-concentration-family}.
In \eqref{eq:fixed-M-interpolating-partition}, put \(\Lambda=\Lambda_M\)
and \(A_s=s\beta^2B(q^\ast)\). Write \(o_M(1)\) for deterministic errors
tending to zero with \(M\) fixed. We prove
\eqref{eq:fixed-M-stable-endpoint} and
\eqref{eq:fixed-M-derivative-bound} uniformly over
\(q^\ast\in\mathcal E_{n,M}\), \(0\le s\le s_\star\), and \(0\le t\le1\).
The final constants are independent of \(M\); the lower bound on \(n\)
may depend on it.

Fix \(q^\ast\in\mathcal E_{n,M}\), write \(B^\ast=B(q^\ast)\), and form
\(A_s=s\beta^2B^\ast\). All bounds below are uniform over this choice.
Since \(x_M=O_M(\log n)\), the row bounds give
\[
 R_2(A_s)^2\le C_{a,b}\frac{\beta^4n}{N},\qquad
 \|A_s\|_{\HS}^2\le C_{a,b}\frac{\beta^4n^2}{N},\qquad
 \max_i\sum_{j\ne i}|B_{ij}^\ast|^6\le C\frac n{N^3}
\]
for all sufficiently large \(n\), uniformly over \(s\in[0,s_\star]\).
The constants here are independent of \(M\), because \(x_M^7/n=o_M(1)\).
In particular, \(\|A_s\|_{\HS}^2/N^2\le C_{a,b}N^{-1/3}\le1/20\)
for large \(n\), and the trace-approximation term in
\eqref{eq:replica-recursive-input} is at most \(C_{a,b}\beta^2/N\).
Use the deletion threshold in \eqref{eq:critical-cutoff-parameters},
enlarging its absolute constant to cover \(s\le s_\star\).
In the calculations of Lemma~\ref{lem:critical-truncation-recursion},
replace the maximum-entry level \(\log n\) by \(x_M\) and the covariance
level \(\log(nN)\) by \(x_M+\log N\). The spectral calculation is then
\[
 \frac{\sqrt{\Lambda_M+1}R_2(2A_s)}d
 \le C_{a,b}\left\{N^{-1/6}
       +\frac{N^{1/3}\sqrt{x_M+\log N}}{\sqrt n}\right\}=o_M(1).
\]
The operator bound gives \(\|2A_s\|_{\mathrm{op}}/N=o_M(1)\), so
Theorem~C.2 supplies the required second moment.
For the trace potential, the same substitution gives the explicit bound
\[
 |\mathcal P_N(A_s)|
 \le C_{a,b}\left\{1+\sqrt{\frac{N(x_M+\log N)}n}
                         +\frac{N(x_M+\log N)}n\right\}
 \le C_{a,b}+o_M(1).
\]
In \eqref{eq:critical-pointwise-local}--\eqref{eq:critical-GER-sum} and
\eqref{eq:critical-local-row-scales}, this replacement multiplies only the
vanishing logarithmic terms by constants depending on \(M\). The
nonvanishing terms are bounded by \(C_{a,b}\beta^2\le C_{a,b}\).
Thus the pointwise deletion and score hypotheses hold, and
\[
\begin{gathered}
 \max_m\frac{k_{m,\mathrm{ad}}r_m^2}{N}
 +\max_m\frac{k_{m,\mathrm{ad}}r_m^2}{d}=o_M(1),\\
 \sum_m\left[k_{m,\mathrm{ad}}^2\left(\frac{r_m^2}{N}\right)^2
                    +\left(\frac{p_m}{N}\right)^2\right]
 +\sum_i\left(\frac{k_{i,\mathrm{ad}}\|(A_s)_{i\cdot}\|_2^2}{N}\right)^2
 \le C_{a,b}+o_M(1).
\end{gathered}
\]
The compensation part uses \eqref{eq:critical-compensation-sum} directly.
The exceptional-recursion estimate from Step 3 of that proof is unchanged,
giving \(\sup_s\eta_{\mathrm{rec}}(s)=o_M(1)\). Hence
\begin{equation}
 \sup_{0\le s\le s_\star}Z_s(0)\le C_{a,b}.
 \label{eq:fixed-M-stable-endpoint}
\end{equation}

For the covariance boundary, the spherical tail
\textup{(A.12)} and
\textup{(7.3)} give
\[
 \sup_{s,t}|\operatorname{Rem}_{\mathrm{cov}}(t)|
 \le C_{a,b}\exp\left[-cd\left\{\frac nN+x_M+\log N\right\}\right]=o_M(1).
\]
Indeed, \(\|A_s\|_{\HS}^2/N\le C_{a,b}N^{2/3}\) and
\(N^{2/3}/(dn/N)\le C_{a,b}N^{-1/6}\), so the weight is absorbed.
The local-row bounds verify \textup{(7.7)}.
For exceptional rows, apply \textup{(7.13)} to $A_s$.
Since $\|A_s\|_{\max}\le s_\star$, its cubic prefactor is at most $Cn^4$;
the choice of $T_{\mathrm{loc}}$ gives
$|\mathcal E_{\mathrm{bad}}(t)|\le Cn^{-2}\max\{1,Z_s(t)\}$ uniformly
in $s,t$, with $C$ independent of $M$.
The entry estimate \eqref{eq:critical-local-entry-sum} and the coefficient
bound \eqref{eq:fixed-M-maximum-coefficient-event}, inserted into
\textup{(D.22)}, give the remaining cubic and score inputs.
Here again each logarithmic term in
Table~\ref{tab:adamczak-exponent-patterns} of the Supplementary Material
is multiplied by a fixed power
of \(M\); its negative powers still give \(o_M(1)\), while the score sum
is bounded by \(C_{a,b}+o_M(1)\).
The score bound is a sum of the deterministic rowwise majorants
\(\varepsilon_{i,s}^{\ast}\) from \textup{(7.6)}.
Consequently \eqref{eq:fixed-M-derivative-decomposition} and
Lemma~\ref{lem:replica-moment-transfer} yield
\begin{equation}
 |Z_s'(t)|\le C_{a,b}\max\{1,Z_s(t)\},\qquad
 \sup_{s,t}Z_s(t)\le C_{a,b}.
 \label{eq:fixed-M-derivative-bound}
\end{equation}
This argument uses only the parameter calculations in
Lemma~\ref{lem:critical-truncation-recursion}; its endpoint and interpolation
are supplied by the common transfer lemma.

\subsection{Triangle-kernel overlap moments}
\label{supp:proof-triangle-overlap-moments}

\begin{proof}[Proof of Lemma~E.2]
The spherical quadratic-form estimate used in Lemma~D.4 gives
\(\mathbb E|B_{12}|^6\le CN^{-3}\).  H\"older's inequality then proves the
first assertion.  For the second, condition on \((q_1,q_2)\).  Isotropy gives
\[
  \mathbb E(B_{23}B_{31}\mid q_1,q_2)=B_{12}/N,
  \qquad
  \mathbb E(B_{24}B_{41}\mid q_1,q_2)=B_{12}/N.
\]
Hence the joint expectation equals \(N^{-2}\mathbb EB_{12}^4\), whereas the
product of the two means equals \(N^{-4}\).  The fourth-moment bound
\(\mathbb EB_{12}^4\le CN^{-2}\) proves the claim.
\end{proof}

\bibliographystyle{imsart-number}
\bibliography{references}